\documentclass{article}

     \usepackage[final,nonatbib]{neurips_2020}

\usepackage{amssymb, amsmath, amsthm, latexsym}
\usepackage{fullpage, color}
\usepackage{url}
\usepackage{algorithm}
\usepackage{algpseudocode}
\usepackage{tabularx}
\usepackage{paralist}
\usepackage{mathtools}
\usepackage{nicefrac}

\newcommand{\mytextstyle}{\textstyle} % include in a conference paper if we need to squeeze the paper...
 \usepackage{adjustbox}

\usepackage[utf8]{inputenc} % allow utf-8 input

\usepackage[T1]{fontenc}    % use 8-bit T1 fonts
\usepackage{lmodern}

\usepackage{hyperref}       % hyperlinks
\usepackage{url}            % simple URL typesetting
\usepackage{booktabs}       % professional-quality tables
\usepackage{amsfonts}       % blackboard math symbols
\usepackage{nicefrac}       % compact symbols for 1/2, etc.
\usepackage{microtype}      % microtypography

\usepackage{graphicx}
\usepackage{grffile}

\usepackage{tabularx}

\usepackage{amssymb, amsmath, amsthm, latexsym}
\usepackage{color}
\usepackage{url}
\usepackage{algorithm}
\usepackage{algpseudocode}
\usepackage{tabularx}
\usepackage{paralist}
\usepackage{mathtools}

\usepackage[numbers]{natbib}
\usepackage{sidecap}

\usepackage[dvipsnames]{xcolor}
\usepackage{pifont}

\usepackage{subfigure} 
\newcommand{\Exp}{\mathbf{E}}

\newcommand{\R}{\mathbb{R}}
\newcommand{\eqdef}{\stackrel{\text{def}}{=}}

\def\<#1,#2>{\left\langle #1,#2\right\rangle}

\newtheorem{lemma}{Lemma}[section]
\newtheorem{theorem}{Theorem}[section]
\newtheorem{definition}{Definition}[section]

\newtheorem{assumption}{Assumption}[section]
\newtheorem{corollary}{Corollary}[section]
\theoremstyle{plain}

\usepackage[colorinlistoftodos,bordercolor=orange,backgroundcolor=orange!20,linecolor=orange,textsize=scriptsize]{todonotes}

\newcommand{\cC}{{\cal C}}
\newcommand{\cD}{{\cal D}}

\newcommand{\cL}{{\cal L}}
\newcommand{\cM}{{\cal M}}

\newcommand{\cO}{{\cal O}}

\newcommand{\cQ}{{\cal Q}}

\newcommand{\Var}{\mathrm{Var}}

\newcommand{\EE}{\mathbf{E}}

\newcommand{\tx}{\tilde{x}}

\usepackage{hyperref}
\usepackage{makecell}

\usepackage{accents}
\newlength{\dhatheight}

\title{\bf Linearly Converging Error Compensated SGD}

\author{Eduard Gorbunov%\thanks{\texttt{eduard.gorbunov@phystech.edu}, \url{https://eduardgorbunov.github.io}} 
\\
  MIPT, Yandex and Sirius, Russia \\ KAUST, Saudi Arabia\\
  \And
   Dmitry Kovalev%\thanks{\texttt{dmitry.kovalev@kaust.edu.sa}, \url{https://www.dmitry-kovalev.com}} 
   \\
   KAUST, Saudi Arabia \\
   \AND
   Dmitry Makarenko %\thanks{\texttt{devjiu@gmail.com}} 
   \\
   MIPT, Russia\\
   \And
   Peter Richt\'{a}rik %\thanks{\texttt{peter.richtarik@kaust.edu.sa}, \url{https://richtarik.org}} 
   \\
    KAUST, Saudi Arabia\\
}

\begin{document}

\maketitle

\begin{abstract}
  In this paper, we propose a unified analysis of variants of distributed {\tt SGD} with arbitrary compressions and delayed updates. Our framework is general enough to cover different variants of quantized {\tt SGD}, Error-Compensated {\tt SGD} ({\tt EC-SGD}) and {\tt SGD} with delayed updates ({\tt D-SGD}). Via a single theorem, we derive the complexity results for all the methods that fit our framework. For the existing methods, this theorem gives the best-known complexity results. Moreover, using our general scheme, we develop new variants of {\tt SGD} that combine variance reduction or arbitrary sampling with error feedback and quantization and derive the convergence rates for these methods beating the state-of-the-art results. In order to illustrate the strength of our framework, we develop $16$ new methods that fit this. In particular, we propose the first method called {\tt EC-SGD-DIANA} that is based on error-feedback for biased compression operator and quantization of gradient differences and prove the convergence guarantees showing that {\tt EC-SGD-DIANA} converges to the exact optimum asymptotically in expectation with constant learning rate for both convex and strongly convex objectives when workers compute full gradients of their loss functions. Moreover, for the case when the loss function of the worker has the form of finite sum, we modified the method and got a new one called {\tt EC-LSVRG-DIANA} which is the first distributed stochastic method with error feedback and variance reduction that converges to the exact optimum asymptotically in expectation with a constant learning rate.
\end{abstract}

\section{Introduction}\label{sec:intro}
We consider distributed optimization problems of the form
\begin{equation}
\mytextstyle
\min\limits_{x\in\R^d} \left\{ f(x) = \frac{1}{n}\sum\limits_{i=1}^n f_i(x) \right\}, \label{eq:main_problem}
\end{equation}
where  $n$ is the number of  workers/devices/clients/nodes. The  information about function $f_i$ is stored on the $i$-th worker only.  Problems of this form appear in the distributed or federated training of supervised machine learning models \cite{shamir2014communication, konecny2016federated}.  In such applications, $x\in \R^d$ describes the parameters identifying a statistical model we wish to train, and $f_i$ is the (generalization or empirical) loss of model $x$ on the data accessible by worker $i$. If worker $i$ has access to data with distribution $\cD_i$, then $f_i$ is assumed to have the structure
\begin{equation}
	f_i(x) = \EE_{\xi_i\sim \cD_i}\left[f_{\xi_i}(x)\right]. \label{eq:f_i_expectation}
\end{equation}
Dataset $\cD_i$ may or may not be available to worker $i$ in its entirety. Typically, we assume that worker $i$ has  only access to samples from $\cD_i$. If the dataset is fully available, it is typically finite, in which case we can assume that $f_i$ has the finite-sum form\footnote{The implicit assumption that each worker contains exactly $m$ data points is for convenience/simplicity only; all our results direct analogues in the general setting with $m_i$ data points on worker $i$.}:
\begin{equation}
	\mytextstyle f_i(x) = \frac{1}{m}\sum\limits_{j=1}^m f_{ij}(x). \label{eq:f_i_sum}
\end{equation}

% While our main focus is on the setting $n>1$, some of our methods and results are new and of interest even in the $n=1 $ setting.

 % Distributed optimization is an  extremely popular and increasingly necessary tool for training supervised machine learning models \cite{goyal2017accurate}, and is of particular relevance in  Federated Learning \cite{konecny2016federated,kairouz2019advances}. 

\textbf{Communication bottleneck.}
  The key bottleneck in practical distributed \cite{goyal2017accurate} and federated \cite{konecny2016federated,kairouz2019advances} systems comes from the high cost of communication of messages among the clients needed to find a solution of sufficient qualities. Several approaches to addressing this communication bottleneck have been proposed in the literature.
  
In the very rare situation when it is possible to adjust the network architecture connecting the clients, one may consider a fully decentralized setup \cite{bertsekas1989parallel}, and allow each client in each iteration to  communicate to their neighbors only. One can argue that in some circumstances and in a certain sense,  decentralized architecture may be preferable to centralized architectures \cite{lian2017can}.  Another natural way to address the communication bottleneck is to do more meaningful (which typically means more expensive) work on each client before each communication round. This is done in the hope that such extra work will produce more valuable messages to be communicated, which hopefully results in the need for fewer communication rounds. A popular technique of this type which is particularly relevant to Federated Learning is based in applying multiple {\em local updates} instead of a single update only.  This is the main idea behind  {\tt Local-SGD} \cite{Stich18local}; see also \cite{basu2019qsparse, haddadpour2019convergence, karimireddy2019scaffold, khaled2020tighter, koloskova2020unified, stich2019error, woodworth2020local}.  However, in this paper, we contribute to the line work which aims to  resolve the communication bottleneck issue via {\em communication compression}. That is, the  information that is normally exchanged---be it iterates, gradients or some more sophisticated vectors/tensors---is compressed in a lossy manner before communication. By applying compression,  fewer bits are transmitted  in each communication round, and one hopes that the increase in the number of communication rounds necessary to solve the problem, if any, is compensated by the savings, leading to a more efficient method overall.

\textbf{Error-feedback framework.} In order to address these issues, in this paper we study a broad class of distributed stochastic first order methods for solving problem \eqref{eq:main_problem} described by the  iterative framework
\begin{eqnarray}
	x^{k+1} &=& \mytextstyle x^k - \frac{1}{n} \sum \limits_{i=1}^n v_i^k, \label{eq:x^k+1_update}\\
	e_i^{k+1} &=&  e_i^k +  \gamma g_i^k -  v_i^k , \qquad i=1,2,\dots , n.\label{eq:error_update}
\end{eqnarray}
In this scheme,  $x^k$ represents the key  iterate, $v_i^k$ is the contribution of worker $i$ towards the update in iteration $k$, $g_i^k$ is an unbiased estimator of $\nabla f_i(x^k)$ computed by worker $i$, $\gamma>0$ is a fixed stepsize and $e_i^k$ is the error accumulated at node $i$ prior to iteration $k$ (we set to $e_i^0
= 0$ for all $i$). In order to better understand the role of the vectors $v_i^k$ and $e_i^k$, first consider the simple special case with $v_i^k \equiv \gamma g_i^k$. In this case, $e_i^k=0$ for all $i$ and $k$, and method \eqref{eq:x^k+1_update}--\eqref{eq:error_update} reduces to distributed {\tt SGD}:
\begin{equation}\label{eq:SGD-ss}\mytextstyle x^{k+1} = x^k - \frac{ \gamma}{n}\sum \limits_{i=1}^n g_i^k. \end{equation}
However, by allowing to chose the vectors $v_i^k$ in a different manner,  we obtain a more general update rule  than what the {\tt SGD} update \eqref{eq:SGD-ss} can offer. \citet{stich2019error},  who studied  \eqref{eq:x^k+1_update}--\eqref{eq:error_update} in the $n=1$ regime, show that this  flexibility allows to capture several types of methods,  including those employing  i) compressed communication, ii)  delayed gradients, and iii) minibatch gradient updates. While our general results apply to all these special cases and more, in order to not dilute the focus of the paper,  in the main body of this paper we concentrate on the first use case---compressed communication---which we now describe.

\textbf{Error-compensated compressed gradient methods.}
% The first use case, and the one we pay most attention to in the main part of this paper, has to do with communication compression. 
Note that in distributed {\tt SGD} \eqref{eq:SGD-ss}, each worker needs to know the aggregate gradient $g^k = \frac{1}{n}\sum_{i=1}^n g_i^k$ to form $x^{k+1}$, which is needed before the next iteration can start. This can be achieved, for example, by each worker $i$ communicating their gradient $g_i^k$ to all other workers. Alternatively, in a parameter server setup, a dedicated master node  collects the gradients from all workers, and broadcasts their average $g^k$  to all workers.  Instead of communicating the gradient vectors $g_i^k$, which is  expensive in distributed learning in general and in federated learning in particular, and especially if $d$ is large, we wish to communicate other but  closely related vectors which can be represented with fewer bits. To this effect, each worker $i$  sends  the vector 
\begin{equation}v_i^k = \cC(e_i^k +  \gamma g_i^k), \qquad \forall i\in [n]\label{eq:bu98gf}\end{equation} instead, where $\cC:\R^d\to \R^d$ is a (possibly randomized, and in such a case, drawn independently of all else in iteration $k$)  compression operator used  to reduce communication. We assume throughout that there exists $\delta \in (0,1]$ such  that  the following inequality holds for all $x\in\R^d$
	\begin{equation}
		\EE\left[\|\cC(x) - x\|^2\right] \le (1 - \delta)\|x\|^2. \label{eq:compression_def}
	\end{equation}

For any $k \geq 0$, the vector $e_i^{k+1} = \sum_{t=0}^{k} \gamma g_i^t - v_i^t $ captures the {\em error} accumulated by worker $i$ up to iteration $k$. This is the difference between the ideal {\tt SGD} update $\sum_{t=0}^{k} \gamma g_i^t$ and the applied update $ \sum_{t=0}^{k} v_i^t $. As we see in \eqref{eq:bu98gf},  at iteration $k$ the current error $e_i^k$ is added to the gradient update $\gamma g_i^k$---this is referred to as {\em error feedback}---and subsequently compressed, which defines the update vector $v_i^k$. Compression introduces additional error, which is added to $e_i^k$, and the process is repeated.

% \subsection{Compression operators}

\textbf{Compression operators.} For a rich collection of specific  operators satisfying \eqref{eq:compression_def}, we refer the reader to \citet{stich2019error} and \citet{beznosikov2020biased}. These include various unbiased or contractive sparsification operators such as RandK and TopK, respectively,  and quantization operators  such as natural compression  and natural dithering \cite{Cnat}. Several additional comments related to compression operators are included in Section~\ref{sec:compressions}.

% The term $e_i^k$ does not play any role at the start of the iterative process since we set $e_i^0=0$ for all $i$. Notice that while we ideally {\em wanted} to send the vector $\gamma g_i^0$ at $k=0$, we instead send $v_i^0$, its compressed version, thus incurring the error $\gamma g_i^0 - v_i^0$ which captures the remaining update we did {\em not} apply. The idea of error compensation is to add this error to the next (scaled) gradient before compression is applied. 

%\paragraph{Further applications.} \peter{This framework includes  Error-Compensated {\tt SGD} ({\tt EC-SGD}) with arbitrary contractive compressors \cite{stich2019error}, {\tt SGD} with delayed updates [???], and all methods  }
%
%\cite{stich2018sparsified, karimireddy2019error,stich2019error}.

\section{Summary of Contributions}\label{sec:contrib}

We now summarize the key contributions of this paper.

\textbf{$\diamond$ General theoretical framework.} 
In this work we propose a  {\em general theoretical framework} for analyzing a wide class of methods that can be written in the the error-feedback form \eqref{eq:x^k+1_update}-\eqref{eq:error_update}. We perform {\em complexity analysis}  under $\mu$-strong quasi convexity (Assumption~\ref{ass:quasi_strong_convexity})  and $L$-smoothness (Assumption~\ref{ass:L_smoothness})  assumptions on the functions $f$ and $\{f_i\}$, respectively. Our analysis is based on an additional {\em parametric assumption}  (using parameters $A$, $A'$, $B_1$, $B_1'$, $B_2$, $B_2'$, $C_1$, $C_2$, $D_1$, $D_1'$, $D_2$, $D_3$, $\eta$, $\rho_1$, $\rho_2$, $F_1$, $F_2$, $G$) on the relationship between the iterates $x^k$,  stochastic gradients $g^k$, errors $e^k$ and a few other quantities (see Assumption~\ref{ass:key_assumption_new}, and the stronger Assumption~\ref{ass:key_assumption_finite_sums_new}). We prove a single theorem (Theorem~\ref{thm:main_result_new}) from which all our complexity results follow as special cases. That is, for each existing or new specific method, we {\em prove} that one (or both) of our parametric assumptions holds, and specify the parameters for which it holds. These parameters have direct impact on the theoretical rate of the method. A summary of the values of the parameters for all methods developed in this paper is provided in Table~\ref{tbl:special_cases-parameters} in the appendix. We remark that the values of the parameters $A, A', B_1, B_1', B_2, B_2', C_1, C_2$ and $\rho_1, \rho_2$ influence the theoretical stepsize. 

\textbf{$\diamond$ Sharp rates.} For  existing methods covered by our general framework, our main convergence result (Theorem~\ref{thm:main_result_new}) recovers the best known rates for these methods up to constant factors. 

\textbf{$\diamond$ Eight new error-compensated (EC) methods.} We study several specific EC methods for solving problem \eqref{eq:main_problem}. First, we recover  the {\tt EC-SGD} method first analyzed in the $n=1$ case by \citet{stich2019error} and later in the general $n\geq 1$ case by \citet{beznosikov2020biased}. More importantly, we develop {\em eight new methods}: {\tt EC-SGDsr}, {\tt EC-GDstar}, {\tt EC-SGD-DIANA}, {\tt EC-SGDsr-DIANA}, {\tt EC-GD-DIANA}, {\tt EC-LSVRG}, {\tt EC-LSVRGstar} and {\tt EC-LSVRG-DIANA}. 
Some of these methods are designed to work with the expectation  structure of the local functions $f_i$ given in \eqref{eq:f_i_expectation}, and others are specifically designed
 to exploit the  finite-sum structure \eqref{eq:f_i_sum}. All these methods follow the error-feedback framework \eqref{eq:x^k+1_update}--\eqref{eq:error_update}, with $v_i^k$ chosen as in \eqref{eq:bu98gf}. They differ in how the gradient estimator $g_i^k$ is {\em constructed} (see Table~\ref{tbl:EC_methods_summary} for a compact description of all these methods;  formal descriptions can be found in the appendix). As we shall see, the existing {\tt EC-SGD} method uses a rather  naive gradient estimator, which renders it less efficient in theory and practice when compared to the best of our new methods. A key property of our parametric assumption described above is that it allows for the construction and modeling of more elaborate gradient estimators, which leads to new EC methods with superior theoretical and practical properties when compared to prior state of the art.

\begin{table*}[!t]
\caption{Complexity of Error-Compensated SGD methods established in this paper. Symbols: $\varepsilon = $ error tolerance; $\delta = $ contraction factor of compressor $\cC$; $\omega = $ variance parameter of compressor $\cQ$; $\kappa = \nicefrac{L}{\mu}$; $\cL =$ expected smoothness constant; $\sigma_*^2 = $ variance of the stochastic gradients in the solution; $\zeta_*^2 =$ average of $\|\nabla f_i(x^*)\|^2$; $\sigma^2 =$ average of the uniform bounds for the variances of stochastic gradients of workers. {\tt EC-GDstar}, {\tt EC-LSVRGstar} and {\tt EC-LSVRG-DIANA} are the first EC methods with a linear convergence rate without assuming that $\nabla f_i(x^*)=0$ for all $i$. {\tt EC-LSVRGstar} and {\tt EC-LSVRG-DIANA} are the first EC methods with a linear convergence rate which do not require the computation of the full gradient $\nabla f_i(x^k)$ by all workers in each iteration. Out of these three methods, only {\tt EC-LSVRG-DIANA} is practical. $^\dagger${\tt EC-GD-DIANA} is a special case of {\tt EC-SGD-DIANA} where each worker $i$ computes the full gradient  $\nabla f_i(x^k)$.
}
\label{tbl:special_cases2}
\begin{center}
\tiny
\begin{tabular}{|c|l|c|c|c|c|}
\hline
\bf Problem & \bf Method &   \bf Alg \# &  \bf Citation &  \bf  Sec \#  
%&  \bf Thm \# 
& \bf Rate (constants ignored)\\
\hline
\eqref{eq:main_problem}+\eqref{eq:f_i_sum} & {\tt EC-SGDsr}  & Alg \ref{alg:ec-SGDsr} & {\color{red}\bf new} & \ref{sec:ec_SGDsr} 
%&  \ref{thm:ec_SGDsr} 
& {\color{red}$\widetilde{\cO}\left(\frac{\cL}{\mu} + \frac{L+\sqrt{\delta L\cL}}{\delta\mu} + \frac{\sigma_*^2}{n\mu\varepsilon} + \frac{\sqrt{L(\sigma_*^2 + \nicefrac{\zeta_*^2}{\delta})}}{\mu\sqrt{\delta\varepsilon}}\right)$}\\
%%%%%%%%%%%%%%%%%%%%
%%%%%%%%%%%%%%%%%%%%
\hline
\eqref{eq:main_problem}+\eqref{eq:f_i_expectation} & {\tt EC-SGD}  & Alg \ref{alg:ec-sgd} & {\cite{stich2019error}}  & \ref{sec:ec_sgd_pure} 
%&  \ref{thm:ec_sgd_pure} 
& $\widetilde{\cO}\left(\frac{\kappa}{\delta} + \frac{\sigma_*^2}{n\mu\varepsilon} + \frac{\sqrt{L(\sigma_*^2 + \nicefrac{\zeta_*^2}{\delta})}}{\delta\mu\sqrt{\varepsilon}}\right)$\\
%%%%%%%%%%%%%%%%%%%%
%%%%%%%%%%%%%%%%%%%%
%\hline
%\eqref{eq:main_problem}+\eqref{eq:f_i_expectation} & {\tt D-SGD}  & Alg \ref{alg:d-sgd} & {\cite{agarwal2011distributed}}  & \xmark &  \xmark & \xmark 
%% &  \xmark 
%& \ref{sec:d_sgd_pure} &  \ref{thm:d_sgd_pure} \\
%%%%%%%%%%%%%%%%%%%%
%%%%%%%%%%%%%%%%%%%%
\hline
\eqref{eq:main_problem}+\eqref{eq:f_i_sum} & {\tt EC-GDstar}  & Alg \ref{alg:EC-GDstar} & {\color{red}\bf new}
& \ref{sec:ec_SGDstar} 
%& \ref{thm:ec_sgd_star} 
& {\color{red} $\cO\left( \frac{\kappa}{\delta}  \log \frac{1}{\varepsilon} \right)$ } \\
%%%%%%%%%%%%%%%%%%%%
%%%%%%%%%%%%%%%%%%%%
\hline
\eqref{eq:main_problem}+\eqref{eq:f_i_expectation} & {\tt EC-SGD-DIANA}  & Alg \ref{alg:EC-SGD-DIANA} & {\color{red}\bf new}
& \ref{sec:ec_diana} 
%&  \ref{thm:ec_diana} 
& \begin{tabular}{c}
	{\color{red}Opt.\ I: $\widetilde{\cO}\left(\omega + \frac{\kappa}{\delta} + \frac{\sigma^2}{n\mu\varepsilon} + \frac{\sqrt{L\sigma^2}}{\delta\mu\sqrt{\varepsilon}}\right)$}\\
	{\color{red}Opt.\ II: $\widetilde{\cO}\left(\frac{1+\omega}{\delta} + \frac{\kappa}{\delta} + \frac{\sigma^2}{n\mu\varepsilon} + \frac{\sqrt{L\sigma^2}}{\mu\sqrt{\delta\varepsilon}}\right)$}
\end{tabular}\\
%%%%%%%%%%%%%%%%%%%%
%%%%%%%%%%%%%%%%%%%%
\hline
\eqref{eq:main_problem}+\eqref{eq:f_i_sum} & {\tt EC-SGDsr-DIANA}  & Alg \ref{alg:EC-SGDsr-DIANA} & {\color{red}\bf new} & \ref{sec:ec_sgdsr_DIANA} 
%&  \ref{thm:ec_sgdsr_diana} 
& \begin{tabular}{c}
	{\color{red} Opt.\ I: $\widetilde{\cO}\left(\omega + \frac{\cL}{\mu} + \frac{\sqrt{L\cL}}{\delta\mu} + \frac{\sigma_*^2}{n\mu\varepsilon} + \frac{\sqrt{L\sigma_*^2}}{\delta\mu\sqrt{\varepsilon}}\right)$}\\
	{\color{red} Opt.\ II: $\widetilde{\cO}\left(\frac{1+\omega}{\delta} + \frac{\cL}{\mu} + \frac{\sqrt{L\cL}}{\delta\mu} + \frac{\sigma_*^2}{n\mu\varepsilon} + \frac{\sqrt{L\sigma_*^2}}{\mu\sqrt{\delta\varepsilon}}\right)$}
\end{tabular}\\
%%%%%%%%%%%%%%%%%%%%
%%%%%%%%%%%%%%%%%%%%
\hline
\eqref{eq:main_problem}+\eqref{eq:f_i_expectation} & {\tt EC-GD-DIANA}$^\dagger$  & Alg \ref{alg:EC-SGD-DIANA} & {\color{red}\bf new}
& \ref{sec:ec_diana} 
%&  \ref{thm:ec_diana} 
& {\color{red}$\cO\left(\left(\omega + \frac{\kappa}{\delta}\right) \log \frac{1}{\varepsilon}\right)$} \\
%%%%%%%%%%%%%%%%%%%%
%%%%%%%%%%%%%%%%%%%%
%\hline
%\eqref{eq:main_problem}+\eqref{eq:f_i_expectation} & {\tt D-QSGD}  & Alg \ref{alg:d-qsgd} & NEW???  & \xmark &  \xmark & \cmark 
%% &  \xmark 
%& \ref{sec:d_qsgd} &  \ref{thm:d_qsgd} \\
%%%%%%%%%%%%%%%%%%%%
%%%%%%%%%%%%%%%%%%%%
%\hline
%\eqref{eq:main_problem}+\eqref{eq:f_i_expectation} & {\tt D-QSGDstar}  & Alg \ref{alg:d-qSGDstar} & NEW  & \cmark &  \cmark & \xmark 
%% &  \xmark 
%& \ref{sec:d_qsgd_star} & \ref{thm:d_qsgd_star} \\
%%%%%%%%%%%%%%%%%%%%%
%%%%%%%%%%%%%%%%%%%%%
%\hline
%\eqref{eq:main_problem}+\eqref{eq:f_i_expectation} & {\tt D-DIANA}  & Alg \ref{alg:d-diana} & NEW  & \xmark\cmark &  \xmark & \cmark 
%% &  \xmark 
%& \ref{sec:d_diana} & \ref{thm:d_diana} \\
%%%%%%%%%%%%%%%%%%%%
%%%%%%%%%%%%%%%%%%%%
%\hline
%\eqref{eq:main_problem}+\eqref{eq:f_i_sum} & {\tt D-SGDsr}  & Alg \ref{alg:d-SGDsr} & NEW  & \xmark & \cmark & \xmark 
%% &  \xmark 
%& \ref{sec:d_SGDsr} &  \ref{thm:d_SGDsr} \\
%%%%%%%%%%%%%%%%%%%%
%%%%%%%%%%%%%%%%%%%%
\hline
\eqref{eq:main_problem}+\eqref{eq:f_i_sum} & {\tt EC-LSVRG}  & Alg \ref{alg:ec-LSVRG} & {\color{red}\bf new}
& \ref{sec:ec_LSVRG} 
%&  \ref{thm:ec_LSVRG} 
& {\color{red}$\widetilde{\cO}\left(m + \frac{\kappa}{\delta} + \frac{\sqrt{L\zeta_*^2}}{\delta\mu\sqrt{\varepsilon}}\right)$}\\
%%%%%%%%%%%%%%%%%%%%
%%%%%%%%%%%%%%%%%%%%
\hline
\eqref{eq:main_problem}+\eqref{eq:f_i_sum} & {\tt EC-LSVRGstar}  & Alg \ref{alg:ec-LSVRGstar} & {\color{red}\bf new}
 & \ref{sec:ec_LSVRGstar} 
% & \ref{thm:ec_LSVRGstar} 
 & {\color{red}$\cO\left(\left(m + \frac{\kappa}{\delta}\right)  \log \frac{1}{\varepsilon}\right)$} \\
%%%%%%%%%%%%%%%%%%%%
%%%%%%%%%%%%%%%%%%%%
\hline
\eqref{eq:main_problem}+\eqref{eq:f_i_sum} & {\tt EC-LSVRG-DIANA}  & Alg \ref{alg:ec-LSVRG-diana} & {\color{red}\bf new}
& \ref{sec:ec_LSVRG-diana} 
%& \ref{thm:ec_LSVRG-diana} 
& {\color{red}$\cO \left(\left( \omega + m + \frac{\kappa }{\delta} \right) \log \frac{1}{\varepsilon}\right)$} \\
%%%%%%%%%%%%%%%%%%%%
%%%%%%%%%%%%%%%%%%%%
%\hline
%\eqref{eq:main_problem}+\eqref{eq:f_i_sum} & {\tt D-LSVRG}  & Alg \ref{alg:d-LSVRG} & NEW  & \cmark & \xmark & \xmark 
%% &  \xmark 
%& \ref{sec:d_LSVRG} & \ref{thm:d_LSVRG} \\
%%%%%%%%%%%%%%%%%%%%%
%%%%%%%%%%%%%%%%%%%%%
%\hline
%\eqref{eq:main_problem}+\eqref{eq:f_i_sum} & {\tt D-QLSVRG}  & Alg \ref{alg:d-qLSVRG} & NEW  & \xmark\cmark & \xmark & \cmark  
%% & \xmark 
%& \ref{sec:d_qLSVRG} & \ref{thm:d_qLSVRG} \\
%%%%%%%%%%%%%%%%%%%%%
%%%%%%%%%%%%%%%%%%%%%
%\hline
%\eqref{eq:main_problem}+\eqref{eq:f_i_sum} & {\tt D-QLSVRGstar}  & Alg \ref{alg:d-qLSVRGstar} & NEW  & \cmark & \xmark & \cmark 
%% &  \xmark 
%& \ref{sec:d_qLSVRGstar} & \ref{thm:d_qLSVRGstar} \\
%%%%%%%%%%%%%%%%%%%%%
%%%%%%%%%%%%%%%%%%%%%
%\hline
%\eqref{eq:main_problem}+\eqref{eq:f_i_sum} & {\tt D-LSVRG-DIANA}  & Alg \ref{alg:d-LSVRG-diana} & NEW  & \cmark & \xmark & \cmark 
%% &  \xmark 
%& \ref{sec:d_LSVRG-diana} & \ref{thm:d_LSVRG-diana} \\
%%%%%%%%%%%%%%%%%%%%%
%%%%%%%%%%%%%%%%%%%%%
\hline
\end{tabular}
\end{center}
\end{table*}

\begin{table}[h]\caption{Error compensated methods developed in this paper. In all cases, $v_i^k = \cC(e_i^k + \gamma g_i^k)$. The full descriptions of the algorithms are included in the appendix.} \label{tbl:EC_methods_summary}
\begin{center}
\footnotesize
\begin{tabular}{|c|c|c|c|c|} 
%%%%%%%%%%%%
%%%%%%%%%%%%
\hline
%%%%%%%%%%%%
%%%%%%%%%%%%
Problem & Method & $g_i^k $ & Comment \\
%%%%%%%%%%%%
%%%%%%%%%%%%
\hline
%%%%%%%%%%%%
%%%%%%%%%%%%
\eqref{eq:main_problem} + \eqref{eq:f_i_sum}       &   {\tt EC-SGDsr}           & $\frac{1}{m} 
\sum \limits_{j=1}^m \xi_{ij}\nabla f_{ij}(x^k)$ & \begin{tabular}{c}$\Exp\left[ \xi_{ij} \right] = 1 $ \\  $\EE_{\cD_i}\left[\|\nabla f_{\xi_i}(x) - \nabla f_{\xi_i}(x^*)\|^2\right]$ \\
	$~~~~~\le 2\cL D_{f_i}(x,x^*)$
\end{tabular}  \\  
%%%%%%%%%%%%
%%%%%%%%%%%%
\hline
%%%%%%%%%%%%
%%%%%%%%%%%%
\eqref{eq:main_problem} + \eqref{eq:f_i_expectation}       &   {\tt EC-SGD}           & $\nabla f_{\xi_i}(x^k)$ &  \\  
%%%%%%%%%%%%
%%%%%%%%%%%%
\hline
%%%%%%%%%%%%
%%%%%%%%%%%%

\eqref{eq:main_problem}             & {\tt EC-GDstar}           & $\nabla f_i(x^k) - \nabla f_i(x^*)$ & known $\nabla f_i(x^*) \; \forall i$  \\  
%%%%%%%%%%%%
%%%%%%%%%%%%
\hline
%%%%%%%%%%%%
%%%%%%%%%%%%
\eqref{eq:main_problem} + \eqref{eq:f_i_expectation}              & {\tt EC-SGD-DIANA}   & $\hat{g}_i^k - h_i^k + h^k$ & \begin{tabular}{c}$\EE\left[\hat{g}_i^k\right] = \nabla f_i(x^k)$ \\
 			$\Exp_k \left[ \|\hat{g}_i^k - \nabla f_i(x^k)\|^2\right] \leq D_{1,i}$\\
             $h_i^{k+1} = h_i^k + \alpha \cQ(\hat{g}_i^k - h_i^k)$\\
             $h^k = \frac{1}{n} \sum \limits_{i=1}^n h_i^k$ \end{tabular}  \\
%%%%%%%%%%%%
%%%%%%%%%%%%
\hline
%%%%%%%%%%%%
%%%%%%%%%%%%
\eqref{eq:main_problem} +  \eqref{eq:f_i_sum}               & {\tt EC-SGDsr-DIANA}   & $\nabla f_{\xi_i^k}(x^k) - h_i^k + h^k$ & \begin{tabular}{c}$\EE\left[\nabla f_{\xi_i^k}(x^k)\right] = \nabla f_i(x^k)$ \\
 			$\EE_{\cD_i}\left[\|\nabla f_{\xi_i}(x) - \nabla f_{\xi_i}(x^*)\|^2\right]$ \\
	$~~~~~\le 2\cL D_{f_i}(x,x^*)$\\
             $h_i^{k+1} = h_i^k + \alpha \cQ(\nabla f_{\xi_i^k}(x^k) - h_i^k)$\\
             $h^k = \frac{1}{n} \sum \limits_{i=1}^n h_i^k$ \end{tabular}  \\
%%%%%%%%%%%%
%%%%%%%%%%%%
\hline
%%%%%%%%%%%%
%%%%%%%%%%%%   
\eqref{eq:main_problem} +  \eqref{eq:f_i_sum}         &   {\tt EC-LSVRG}           & {$ \begin{aligned} \nabla & f_{il}(x^k) - \nabla f_{il}(w_i^k) \\ &+ \nabla f_i(w_i^k) \end{aligned} $}  & \begin{tabular}{c}$l$ chosen uniformly from $[m]$\\ $w_i^{k+1} = \begin{cases}x^k,& \text{with prob. } p,\\ w_i^k,& \text{with prob. } 1-p\end{cases}$\end{tabular}  \\  
%%%%%%%%%%%%
%%%%%%%%%%%%
\hline
%%%%%%%%%%%%
%%%%%%%%%%%%        
\eqref{eq:main_problem} +  \eqref{eq:f_i_sum}         &   {\tt EC-LSVRGstar}           & { $ \begin{aligned}\nabla & f_{il}(x^k) - \nabla f_{il}(w_i^k) \\ & + \nabla f_i(w_i^k) - \nabla f_i(x^*) \end{aligned}$ }  & \begin{tabular}{c}$l$ chosen uniformly from $[m]$\\ $w_i^{k+1} = \begin{cases}x^k,& \text{with prob. } p,\\ w_i^k,& \text{with prob. } 1-p\end{cases}$\end{tabular}  \\  
%%%%%%%%%%%%
%%%%%%%%%%%%
\hline
%%%%%%%%%%%%
%%%%%%%%%%%%  
\eqref{eq:main_problem} +  \eqref{eq:f_i_sum}         &   {\tt EC-LSVRG-DIANA}           & \begin{tabular}{c} $\hat{g}_i^k - h_i^k + h^k$ \\ where \\ { $ \begin{aligned}& \hat{g}_i^k =  \nabla  f_{il}(x^k) \\ &- \nabla f_{il}(w_i^k)  + \nabla f_i(w_i^k)  \end{aligned}$ } \end{tabular} & \begin{tabular}{c}  $h_i^{k+1} = h_i^k + \alpha \cQ(\hat{g}_i^k - h_i^k)$\\
             $h^k = \frac{1}{n} \sum \limits_{i=1}^n h_i^k$ \\ $l$ chosen uniformly from $[m]$\\ $w_i^{k+1} = \begin{cases}x^k,& \text{with prob. } p,\\ w_i^k,& \text{with prob. } 1-p\end{cases}$\end{tabular} \\  
%%%%%%%%%%%%
%%%%%%%%%%%%
\hline
%%%%%%%%%%%%
%%%%%%%%%%%%  
\end{tabular}
\end{center}
\end{table}

\textbf{$\diamond$ First linearly converging EC methods.}  The key theoretical consequence of our general framework is the development of the {\em first linearly converging} error-compensated {\tt SGD}-type methods for distributed training with biased communication compression. In particular, we design four such methods: two simple  but impractical methods, {\tt EC-GDstar} and {\tt EC-LSVRGstar}, with rates $\cO\left(\frac{\kappa}{\delta}\ln\frac{1}{\varepsilon}\right)$ and $\cO\left(\left(m+\frac{\kappa}{\delta}\right)\ln\frac{1}{\varepsilon}\right)$, respectively,
and two  practical but more elaborate methods, {\tt EC-GD-DIANA}, with rate  $\cO\left(\left(\omega + \frac{\kappa }{\delta}\right) \ln \frac{1}{\varepsilon}\right)$, and  {\tt EC-LSVRG-DIANA}, with rate $\cO\left(\left(\omega + m + \frac{\kappa }{\delta}\right) \ln \frac{1}{\varepsilon}\right).$
In these rates, $\kappa=\nicefrac{L}{\mu}$ is the condition number,  $0<\delta \leq 1$ is the contraction parameter associated with the compressor $\cC$ used in \eqref{eq:bu98gf}, and $\omega$ is the variance parameter associated with a {\em secondary unbiased compressor\footnote{We assume that $\Exp{\cQ(x)}=x$ and $\Exp{\|\cQ(x)-x\|^2} \leq \omega \|x\|^2$ for all $x\in \R^d$.} $\cQ$} which plays a key role in the construction of the gradient estimator $g_i^k$. The complexity of the  first and third methods does not depend on $m$ as they require the computation of the full gradient $\nabla f_i(x^k)$ for each $i$. The remaining two methods only need to compute $\cO(1)$ stochastic gradients $\nabla f_{ij}(x^k)$ on each worker $i$.  

The first two methods, while impractical, provided us with the intuition which enabled us to develop the practical variant. We include them in this paper due to their simplicity, because of the added insights they offer, and to showcase the flexibility of our general theoretical framework, which is able to describe them. {\tt EC-GDstar} and {\tt EC-LSVRGstar} are impractical since they require the knowledge of the gradients $\{\nabla f_i(x^*)\}$, where $x^*$ is an optimal solution of \eqref{eq:main_problem}, which are obviously not known since $x^*$ is not known. % Their rates are $\cO\left(\frac{\kappa}{\delta}\ln\frac{1}{\varepsilon}\right)$ and $\cO\left(\left(m+\frac{\kappa}{\delta}\right)\ln\frac{1}{\varepsilon}\right)$, respectively.

The only known linear convergence result for an error compensated {\tt SGD} method is due to \citet{beznosikov2020biased}, who require the computation of the full gradient of $f_i$ by each  machine $i$ (i.e., $m$ stochastic gradients), and the additional assumption that $\nabla f_i(x^*) = 0$ for all $i$. We do not need such assumptions, thereby resolving a major theoretical issue with EC methods.

\textbf{$\diamond$ Results in the convex case.} 
Our theoretical analysis goes beyond distributed optimization and recovers the results from \citet{gorbunov2019unified,khaled2020unified} (without regularization) in the  special case when $v_i^k \equiv \gamma g_i^k$. As we have seen, in this case $e_i^k\equiv 0$ for all $i$ and $k$, and  the error-feedback framework \eqref{eq:x^k+1_update}--\eqref{eq:error_update}  reduces to distributed {\tt SGD} \eqref{eq:SGD-ss}. In this regime, the relation \eqref{eq:sum_of_errors_bound_new} in Assumption~\ref{ass:key_assumption_new} becomes void, while relations \eqref{eq:second_moment_bound_new} and \eqref{eq:sigma_k+1_bound_1} with $\sigma_{2,k}^2\equiv 0$ are precisely those used by \citet{gorbunov2019unified} to analyze a wide array of {\tt SGD} methods, including  vanilla {\tt SGD} \cite{robbins1985stochastic}, {\tt SGD} with arbitrary sampling \cite{gower2019sgd}, as well as variance reduced methods such as {\tt SAGA} \cite{SAGA}, {\tt SVRG} \cite{SVRG}, {\tt LSVRG} \cite{hofmann2015variance, kovalev2019don}, {\tt JacSketch} \cite{gower2018stochastic},  {\tt SEGA} \cite{hanzely2018sega} and {\tt DIANA} \cite{mishchenko2019distributed, horvath2019stochastic}. Our theorem recovers the rates of all the methods just listed in both the convex case $\mu = 0$ \citet{khaled2020unified} and the strongly-convex case $\mu > 0$ \citet{gorbunov2019unified} under the more general Assumption~\ref{ass:key_assumption_new}. 

%\paragraph{Unification of EC and DIANA.} To the best of our knowledge we propose the first framework that recovers the best known results for the methods based on biased compressors (via error compensation) either error-feedback or quantization. Using our theory one can easily obtain a clear theoretical comparison of error-feedback and quantization. 

\textbf{$\diamond$ {\tt DIANA} with bi-directional quantization.}
To illustrate how our framework can be used even in the case when $v_i^k \equiv \gamma g_i^k$, $e_i^k \equiv 0$, we develop analyze a new version of {\tt DIANA} called {\tt DIANAsr-DQ} that uses arbitrary sampling on every node and double quantization\footnote{In the concurrent work (which appeared on arXiv after we have submitted our paper to NeurIPS) a similar method was independently proposed under the name of {\tt Artemis} \cite{philippenko2020artemis}. However, our analysis is  more general, see all the details on this method in the appendix. This footnote was added to the paper during the preparation of the camera-ready version of our paper.}, i.e., unbiased compression not only on the workers' side but also on the master's one.

\textbf{$\diamond$ Methods with delayed updates.}
Following \citet{stich2019unified}, we also show that our approach covers {\tt SGD} with delayed updates \cite{agarwal2011distributed, arjevani2018tight, feyzmahdavian2016asynchronous} ({\tt D-SGD}), and our analysis shows the best-known rate for this method. Due to the flexibility of our framework, we are able develop several new variants of {\tt D-SGD} with and without quantization, variance reduction, and arbitrary sampling. Again, due to space limitations, we put these methods together with their convergence analyses in the appendix.

\section{Main Result}\label{sec:main_res}
In this section we present the main theoretical result of our paper.  First,  we introduce our assumption on $f$, which is a relaxation of $\mu$-strong convexity.
\begin{assumption}[$\mu$-strong quasi-convexity]\label{ass:quasi_strong_convexity}
	Assume that function $f$ has a unique minimizer $x^*$. We say that function $f$ is strongly quasi-convex with parameter $\mu\ge 0$ if for all $x\in\R^d$
	\begin{equation}
	\mytextstyle	f(x^*) \ge f(x) + \langle\nabla f(x), x^* - x\rangle + \frac{\mu}{2}\|x - x^*\|^2. \label{eq:str_quasi_cvx}
	\end{equation}
\end{assumption}
We allow $\mu$ to be zero, in which case  $f$ is sometimes called {\em weakly quasi-convex} in the literature (see \cite{stich2019unified} and references therein). Second, we introduce the classical $L$-smoothness assumption.
\begin{assumption}{$L$-smoothness}\label{ass:L_smoothness}
	We say that function $f$ is $L$-smooth if it is differentiable and its gradient is $L$-Lipschitz continuous, i.e., for all $x,y\in\R^d$
	\begin{equation}
		\|\nabla f(x) - \nabla f(y)\| \le L\|x-y\|. \label{eq:L_smoothness}
	\end{equation}
\end{assumption}
It is a well-known fact \cite{nesterov2018lectures} that $L$-smoothness of convex function $f$ implies that 
\begin{equation}
	\|\nabla f(x) - \nabla f(y)\|^2 \le 2L(f(x) - f(y) - \langle\nabla f(y), x-y\rangle) \eqdef 2LD_{f}(x,y). \label{eq:L_smoothness_cor}
\end{equation}

We now introduce our key parametric assumption on  the stochastic gradient $g^k$. This is a generalization of the assumption introduced by \citet{gorbunov2019unified} for the particular class of methods described covered by the EF framework \eqref{eq:x^k+1_update}--\eqref{eq:error_update}.

\begin{assumption}\label{ass:key_assumption_finite_sums_new}
	For all $k\ge 0$, the stochastic gradient $g^k$ is an average of stochastic gradients $g_i^k$ such that
	\begin{equation}
\mytextstyle		g^k = \frac{1}{n}\sum\limits_{i=1}^ng_i^k,\qquad \EE\left[g^k\mid x^k\right] = \nabla f(x^k). \label{eq:unbiasedness_g_i^k_new}
	\end{equation}
	Moreover,  there exist  constants $A,\widetilde{A}, A', B_1, B_2, \widetilde{B}_1, \widetilde{B}_2, B_1', B_2', C_1, C_2, G, D_1, \widetilde{D}_1, D_1', D_2, D_3 \ge 0$, and $\rho_1,\rho_2 \in [0,1]$ and two sequences of (probably random) variables $\{\sigma_{1,k}\}_{k\ge 0}$ and $\{\sigma_{2,k}\}_{k\ge 0}$, such  that the following recursions hold:
	\begin{eqnarray}
	\mytextstyle	\frac{1}{n}\sum\limits_{i=1}^n\left\|\bar{g}_i^k\right\|^2 &\le& 2A(f(x^k) - f(x^*)) + B_1\sigma_{1,k}^2 + B_2\sigma_{2,k}^2 + D_1, \label{eq:second_moment_bound_g_i^k_new}\\
\mytextstyle		\frac{1}{n}\sum\limits_{i=1}^n\EE\left[\left\|g_i^k-\bar{g}_i^k\right\|^2\mid x^k\right] &\le& 2\widetilde{A}(f(x^k) - f(x^*)) + \widetilde{B}_1\sigma_{1,k}^2 + \widetilde{B}_2\sigma_{2,k}^2 + \widetilde{D}_1, \label{eq:variance_bound_g_i^k_new}\\
		\EE\left[\|g^k\|^2\mid x^k\right] &\le& 2A'(f(x^k) - f(x^*)) + B_1'\sigma_{1,k}^2 + B_2'\sigma_{2,k}^2 + D_1', \label{eq:second_moment_bound_new}\\
		\EE\left[\sigma_{1,k+1}^2\mid \sigma_{1,k}^2, \sigma_{2,k}^2\right] &\le& (1-\rho_1)\sigma_{1,k}^2 + 2C_1\left(f(x^k) - f(x^*)\right) + G\rho_1 \sigma_{2,k}^2 + D_2,\label{eq:sigma_k+1_bound_1}\\
		\EE\left[\sigma_{2,k+1}^2\mid \sigma_{2,k}^2\right] &\le& (1-\rho_2)\sigma_{2,k}^2 + 2C_2\left(f(x^k) - f(x^*)\right),\label{eq:sigma_k+1_bound_2}
	\end{eqnarray}
	where $\bar{g}_i^k = \EE\left[g_i^k\mid x^k\right]$.
\end{assumption}

Let us briefly explain the intuition behind the assumption and the meaning of the introduced parameters. First of all, we assume that the stochastic gradient at iteration $k$ is conditionally unbiased estimator of $\nabla f(x^k)$, which is a natural and commonly used assumption on the stochastic gradient in the literature. However, we explicitly do {\em not} require unbiasedness of $g_i^k$, which is very useful in some special cases. Secondly, let us consider the simplest special case when $g^k \equiv \nabla f(x^k)$ and $f_1 = \ldots = f_n = f$, i.e., there is no stochasticity/randomness in the method and the workers have the same functions. Then due to $\nabla f(x^*) = 0$, we have that
\begin{equation*}
	\|\nabla f(x^k)\|^2 \overset{\eqref{eq:L_smoothness_cor}}{\le} 2L(f(x^k) - f(x^*)),
\end{equation*}
which implies that Assumption~\ref{ass:key_assumption_finite_sums_new} holds in this case with $A = A' = L$, $\widetilde{A}=0$ and $B_1 = B_2 = \widetilde{B}_1 = \widetilde{B}_2 = B_1' = B_2' = C_1 = C_2 = D_1 = \widetilde{D}_1 = D_1' = D_2 = 0$, $\rho = 1$, $\sigma_{1,k}^2 \equiv \sigma_{2,k}^2  \equiv 0$.

In general, if $g^k$ satisfies Assumption~\ref{ass:key_assumption_new}, then  parameters $A$, $\widetilde{A}$ and $A'$ are usually connected with the smoothness properties of $f$ and typically they are just multiples of $L$, whereas terms $B_1\sigma_{1,k}^2$, $B_2\sigma_{2,k}^2$, $\widetilde{B}_1\sigma_{1,k}^2$, $\widetilde{B}_2\sigma_{2,k}^2$, $B_1'\sigma_{1,k}^2$, $B_2'\sigma_{2,k}^2$ and $D_1$, $\widetilde{D}_1$, $D_1'$ appear due to the stochastic nature of $g_i^k$. Moreover, $\{\sigma_{1,k}^2\}_{k\ge 0}$ and $\{\sigma_{2,k}^2\}_{k\ge 0}$ are sequences connected with variance reduction processes and for the methods; without any kind of variance reduction these sequences contains only zeros. Parameters $B_1$ and $B_2$ are often $0$ or small positive constants, e.g., $B_1 = B_2 = 2$, and $D_1$ characterizes the remaining variance in the estimator $g^k$ that is not included in the first two terms. % Likewise, $D_1'$ and $D_2$ capture excess variance that is not reduced. 

Inequalities \eqref{eq:sigma_k+1_bound_1} and \eqref{eq:sigma_k+1_bound_2} describe the variance reduction processes: one can interpret $\rho_1$ and $\rho_2$ as the {\em rates} of the variance reduction processes, $2C_1(f(x^k) - f(x^*))$ and $2C_2(f(x^k) - f(x^*))$ are ``optimization'' terms and, similarly to $D_1$, $D_2$ represents the remaining variance that is not included in the first two terms. Typically, $\sigma_{1,k}^2$ controls the variance coming from compression and $\sigma_{2,k}^2$ controls the variance taking its origin in finite-sum type randomization (i.e., subsampling) by each worker. In the case $\rho_1 = 1$ we assume that $B_1 = B_1' = C_1 = G = 0, D_2 = 0$ (for $\rho_2 = 1$ analogously), since inequality \eqref{eq:sigma_k+1_bound_1} becomes superfluous.

However, in our main result we need a slightly different assumption.
\begin{assumption}\label{ass:key_assumption_new}
	For all $k\ge 0$, the stochastic gradient $g^k$ is an unbiased estimator of $\nabla f(x^k)$:
	\begin{equation}
		\EE\left[g^k\mid x^k\right] = \nabla f(x^k). \label{eq:unbiasedness_new}
	\end{equation}
	Moreover, there exist non-negative constants $A',B_1', B_2',C_1, C_2,F_1, F_2, G, D_1',D_2, D_3 \ge 0, \rho_1, \rho_2 \in [0,1]$ and two sequences of (probably random) variables $\{\sigma_{1,k}\}_{k\ge 0}$ and $\{\sigma_{2,k}\}_{k\ge 0}$ such that inequalities \eqref{eq:second_moment_bound_new}, \eqref{eq:sigma_k+1_bound_1} and \eqref{eq:sigma_k+1_bound_2} hold and
	\begin{eqnarray}
	\mytextstyle	3L\sum\limits_{k=0}^K w_k\EE\|e^k\|^2 &\le& \mytextstyle \frac{1}{4}\sum\limits_{k=0}^K w_k\EE\left[f(x^k) - f(x^*)\right] + F_1\sigma_{1,0}^2 + F_2\sigma_{2,0}^2 + \gamma D_3 W_K \label{eq:sum_of_errors_bound_new}
	\end{eqnarray}
for all $k,	K\ge 0$, where $e^k = \frac{1}{n}\sum_{i=1}^n e_i^k$ and $\{W_K\}_{K\ge 0}$ and $\{w_k\}_{k\ge 0}$ are defined as 
	\begin{equation}
	\mytextstyle			W_K = \sum\limits_{k=0}^K w_k,\quad w_k = (1 - \eta)^{-(k+1)},\quad \eta = \min\left\{\frac{\gamma\mu}{2}, \frac{\rho_1}{4}, \frac{\rho_2}{4}\right\}. \label{eq:w_k_definition_new}
	\end{equation}		
\end{assumption}
This assumption is more flexible than Assumption~\ref{ass:key_assumption_finite_sums_new} and helps us to obtain a unified analysis of all methods falling in the error-feedback framework. We emphasize that in this assumption we do not assume that \eqref{eq:second_moment_bound_g_i^k_new} and \eqref{eq:variance_bound_g_i^k_new} hold \textit{explicitly}. Instead of this, we introduce inequality \eqref{eq:sum_of_errors_bound_new}, which is the key tool that helps us to analyze the effect of error-feedback and comes from the analysis from \cite{stich2019error} with needed adaptations connected with the first three inequalities. As we show in the appendix, this inequality can be derived for {\tt SGD} with error compensation and delayed updates under Assumption~\ref{ass:key_assumption_finite_sums_new} and, in particular, using \eqref{eq:second_moment_bound_g_i^k_new} and \eqref{eq:variance_bound_g_i^k_new}. As before, $D_3$ hides a variance that is not handled by variance reduction processes and $F_1$ and $F_2$ are some constants that typically depend on $L, B_1, B_2, \rho_1, \rho_2$ and $\gamma$.

We now proceed to stating our main theorem.

\begin{theorem}\label{thm:main_result_new}
	Let Assumptions~\ref{ass:quasi_strong_convexity},~\ref{ass:L_smoothness} and ~\ref{ass:key_assumption_new} be satisfied and~$\gamma \le \nicefrac{1}{4(A'+C_1M_1+C_2M_2)}$. Then for all $K\ge 0$ we have
	\begin{equation}
	\mytextstyle	\EE\left[f(\bar x^K) - f(x^*)\right] \le \left(1 - \eta\right)^K\frac{4(T^0 + \gamma F_1 \sigma_{1,0}^2+ \gamma F_2 \sigma_{2,0}^2)}{\gamma} + 4\gamma\left(D_1' + M_1D_2 + D_3\right) \label{eq:main_result_new}
	\end{equation}	
	when $\mu > 0$ and
	\begin{equation}
	\mytextstyle	\EE\left[f(\bar x^K) - f(x^*)\right] \le \frac{4(T^0 + \gamma F_1 \sigma_{1,0}^2+ \gamma F_2 \sigma_{2,0}^2)}{\gamma K} + 4\gamma\left(D_1' + M_1D_2 + D_3\right) \label{eq:main_result_new_cvx}
	\end{equation}
	when $\mu = 0$, where $\eta = \min\left\{\nicefrac{\gamma\mu}{2},\nicefrac{\rho_1}{4},\nicefrac{\rho_2}{4}\right\}$, $T^k \eqdef \|\tx^k - x^*\|^2 + M_1\gamma^2 \sigma_{1,k}^2 + M_2\gamma^2 \sigma_{2,k}^2$ and $M_1 = \frac{4B_1'}{3\rho_1}$, $M_2 = \frac{4\left(B_2' + \frac{4}{3}G\right)}{3\rho_2}$.
\end{theorem}

All the complexity results summarized in Table~\ref{tbl:special_cases2} follow from this theorem; the detailed proofs are included in the appendix. Furthermore, in the appendix we include similar results but for methods employing {\em delayed} updates. The methods, and all associated theory is included there, too.

\section{Numerical Experiments}\label{sec:numerical_exp}
To justify our theory, we conduct several numerical experimentson logistic regression problem with $\ell_2$-regularization:
\begin{equation}
\mytextstyle	\min\limits_{x\in\R^d} \left\{ f(x) = \frac{1}{N}\sum\limits_{i=1}^N \log\left(1 + \exp\left(-y_i\cdot (Ax)_i\right)\right) + \frac{\mu}{2}\|x\|^2 \right\}, \label{eq:log_loss}
\end{equation}
where $N$ is a number of features, $x\in\R^d$ represents the weights of the model, $A\in\R^{N\times d}$ is a feature matrix, vector $y\in\{-1,1\}^N$ is a vector of labels and $(Ax)_i$ denotes the $i$-th component of vector $Ax$. Clearly, this problem is $L$-smooth and $\mu$-strongly convex with $L = \mu + \nicefrac{\lambda_{\max}(A^\top A)}{4N}$, where $\lambda_{\max}(A^\top A)$ is a largest eigenvalue of $A^\top A$. The datasets were taken from LIBSVM library \cite{chang2011libsvm}, and the code was written in Python 3.7 using standard libraries. Our code is available at  \url{https://github.com/eduardgorbunov/ef_sigma_k}.

We simulate parameter-server architecture using one machine with Intel(R) Core(TM) i7-9750 CPU \@ 2.60 GHz in the following way. First of all, we always use such $N$ that $N = n\cdot m$ and consider  $n = 20$ and $n=100$ workers. The choice of $N$ for each dataset that we consider is stated in Table~\ref{tab:logreg_datasets}.
\begin{table}[h]
    \centering \footnotesize
    \caption{Summary of datasets: $N =$  total \# of data samples; $d =$ \# of features.}
    \label{tab:logreg_datasets}
    \begin{tabular}{|c|c|c|c|c|c|c|}
        \hline
         & {\tt a9a} & {\tt w8a} & {\tt gisette} & {\tt mushrooms} & {\tt madelon} & {\tt phishing} \\
        \hline
        $N$ & $32,000$ & $49,700$ & $6,000$ & $8,000$ & $2,000$ & $11,000$ \\
        \hline
        $d$ & $123$ & $300$ & $5,000$ & $112$ & $500$ & $68$ \\
        \hline
    \end{tabular}
\end{table}
Next, we shuffle the data and split in $n$ groups of size $m$. To emulate the work of workers, we use a single machine and run the methods with the parallel loop in series. Since in these experiments we study sample complexity and number of bits used for communication, this setup is identical to the real parameter-server setup in this sense.

In all experiments we use the stepsize $\gamma = \nicefrac{1}{L}$ and $\ell_2$-regularization parameter $\mu = \nicefrac{10^{-4}\lambda_{\max}(A^\top A)}{4N}$. The starting point $x^0$ for each dataset was chosen so that $f(x^0) - f(x^*) \sim 10$. In experiments with stochastic methods we used batches of size $1$ and uniform sampling for simplicity. For {\tt LSVRG}-type methods we choose $p = \nicefrac{1}{m}$. 

\textbf{Compressing stochastic gradients.}
The results for {\tt a9a}, {\tt madelon} and {\tt phishing} can be found in Figure~\ref{fig:sgd_logreg} (included here) and for {\tt w8a}, {\tt mushrooms} and {\tt gisette} in Figure~\ref{fig:sgd_logreg_extra} (in the Appendix). We choose number of components for TopK operator of the order $\max\{1,\nicefrac{d}{100}\}$. Clearly, in these experiments we see two levels of noise. For some datasets, like {\tt a9a}, {\tt phishing} or {\tt mushrooms}, the noise that comes from the stochasticity of the gradients dominates the noise coming from compression. Therefore,  methods such as {\tt EC-SGD} and {\tt EC-SGD-DIANA} start to oscillate around a larger value of the loss function than other methods we consider. {\tt EC-LSVRG} reduces the largest source of noise and, as a result, finds a better approximation of the solution. However, at some point, it reaches another level of the loss function and starts to oscillate there due to the noise coming from compression. Finally, {\tt EC-LSVRG-DIANA} reduces the variance of both types, and as a result, finds an even better approximation of the solution. In contrast, for  the {\tt madelon} dataset, both noises are of the same order, and therefore, {\tt EC-LSVRG} and {\tt EC-SGD-DIANA} behave similarly to {\tt EC-SGD}. However, {\tt EC-LSVRG-DIANA} again reduces both types of noise effectively and finds a better approximation of the solution after a given number of epochs.   In the experiments with {\tt w8a} and {\tt gisette} datasets, the noise produced by compression is dominated by the noise coming from the stochastic gradients. As a result, we see that the {\tt DIANA}-trick is not needed here.

\begin{figure}[h]
    \centering
    \includegraphics[width=0.32\textwidth]{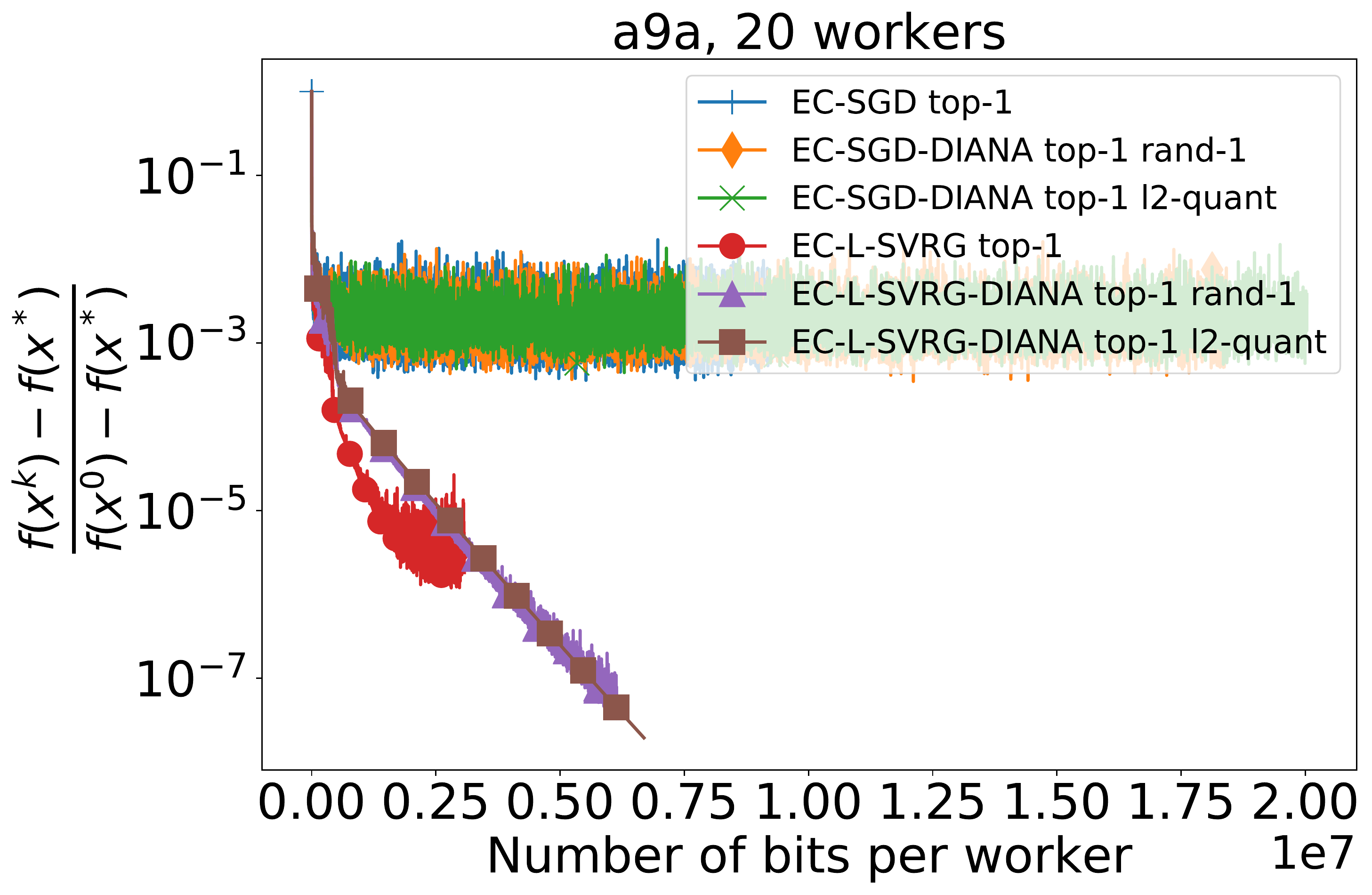}
	\includegraphics[width=0.32\textwidth]{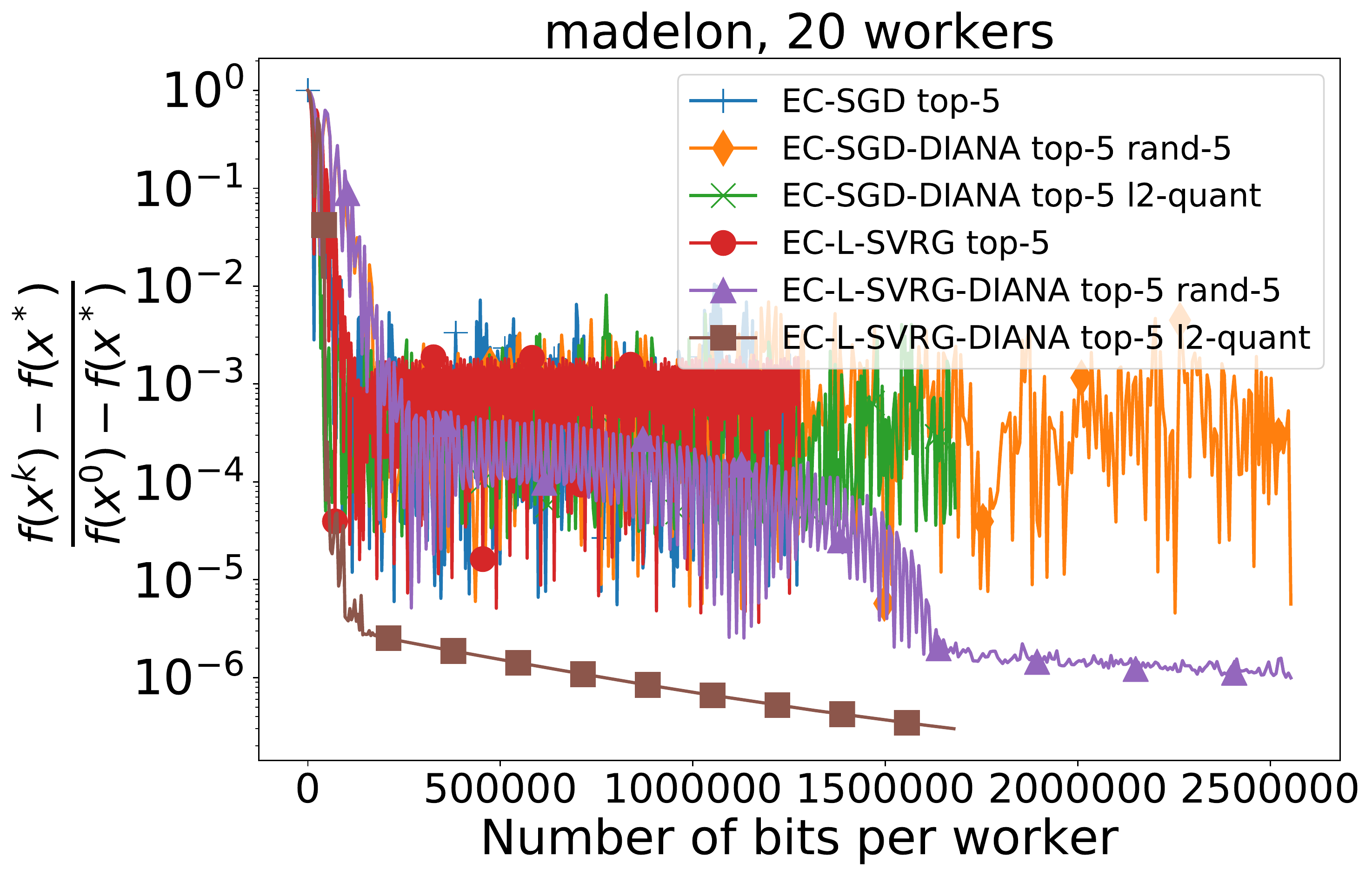}    
	\includegraphics[width=0.32\textwidth]{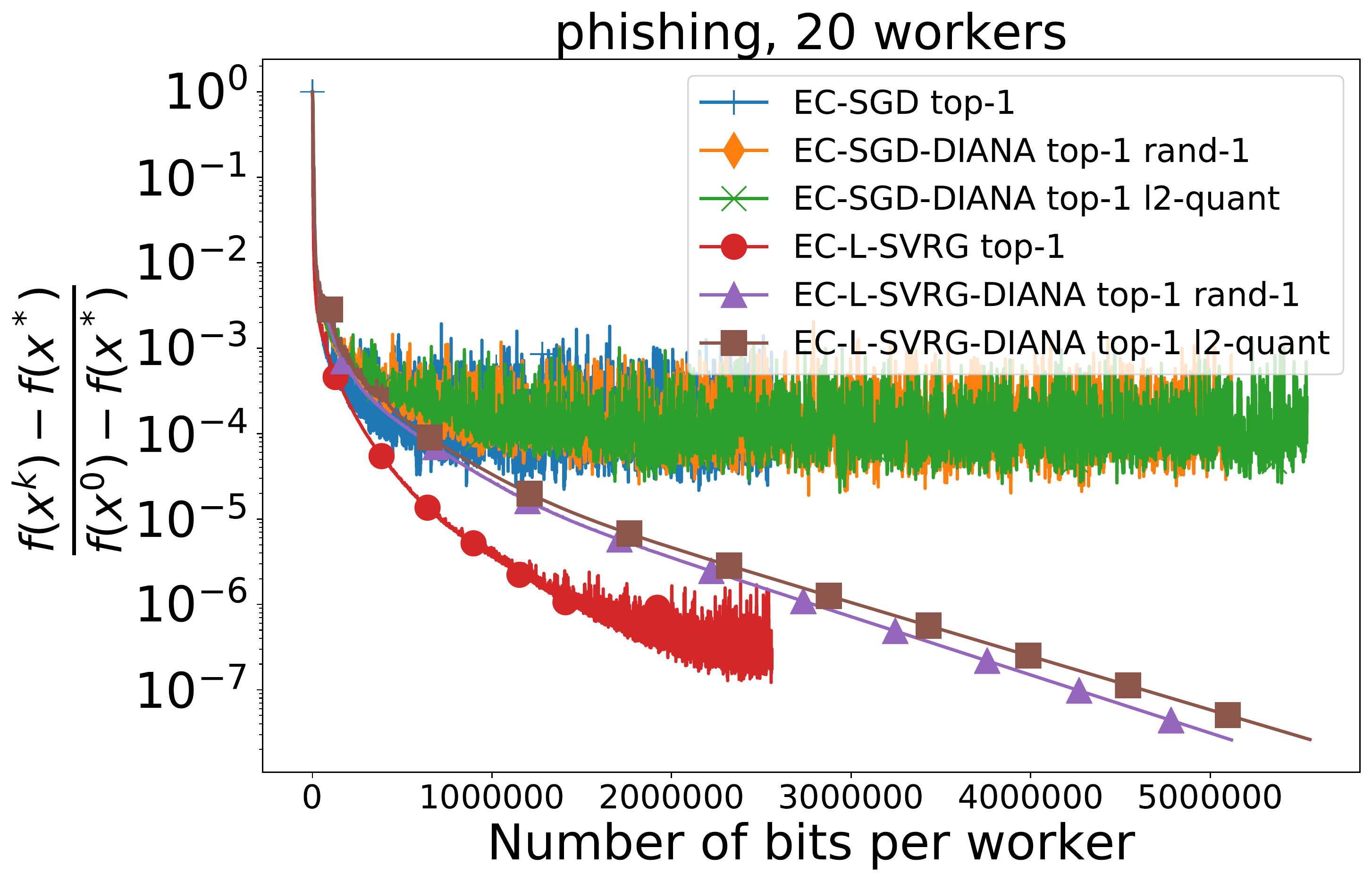}
    \\
    \includegraphics[width=0.32\textwidth]{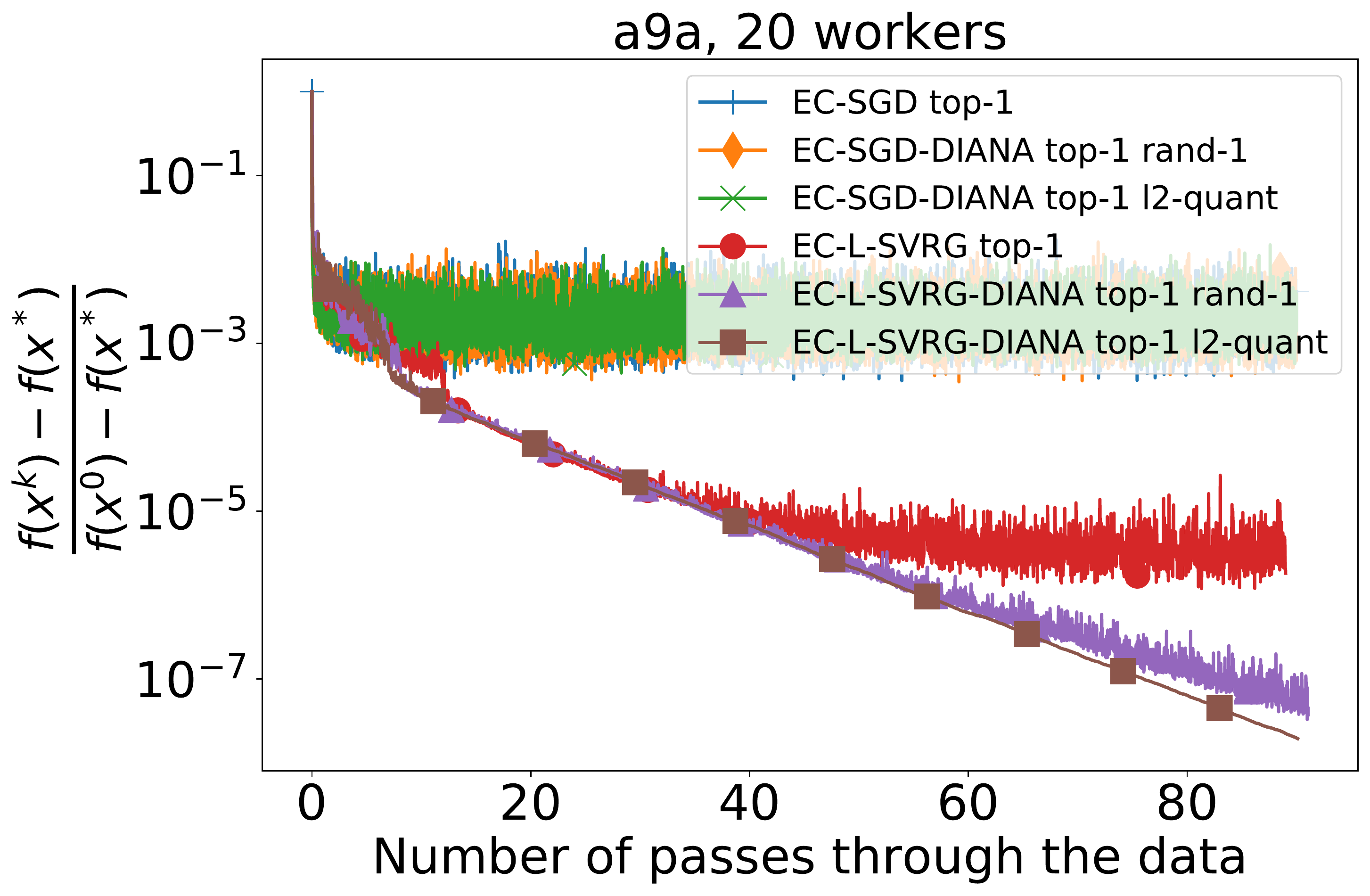}
    \includegraphics[width=0.32\textwidth]{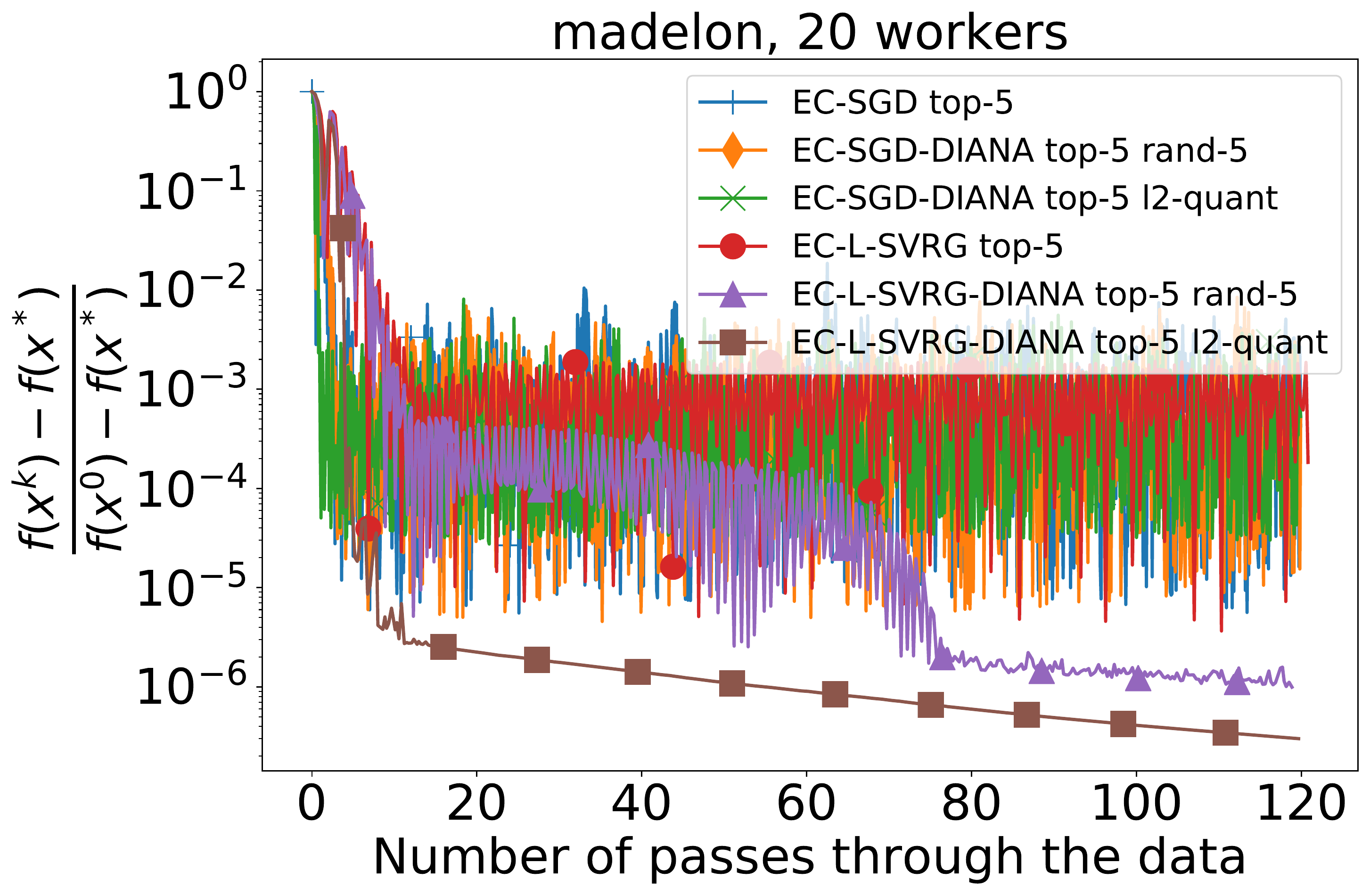}    
	\includegraphics[width=0.32\textwidth]{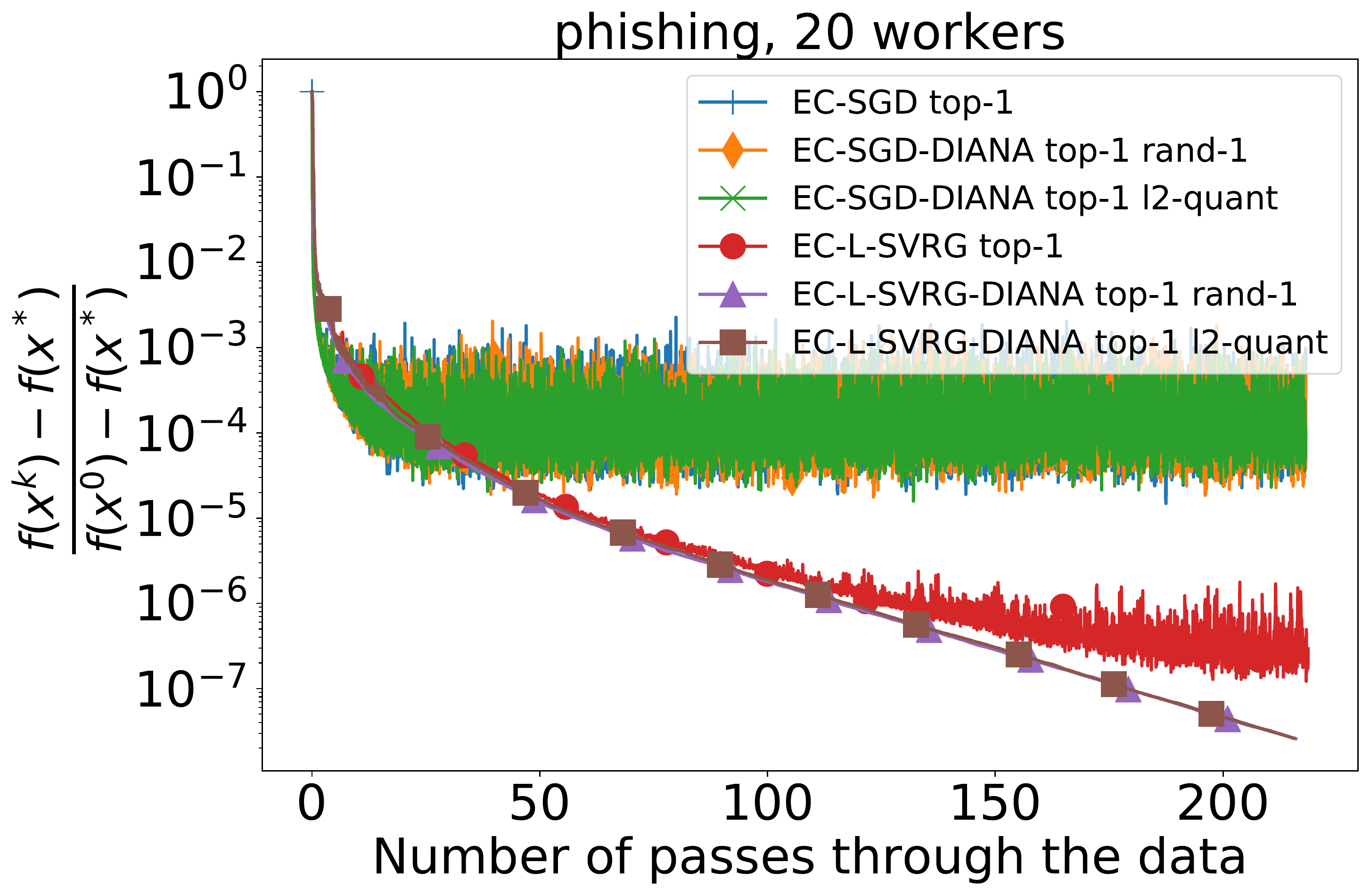} 
     
    \caption{Trajectories of {\tt EC-SGD}, {\tt EC-SGD-DIANA}, {\tt EC-LSVRG} and {\tt EC-LSVRG-DIANA} applied to solve logistic regression problem with $20$ workers.}
    \label{fig:sgd_logreg}
\end{figure}

\textbf{Compressing full gradients.}
\begin{figure}[b]
    \centering
    \includegraphics[width=0.32\textwidth]{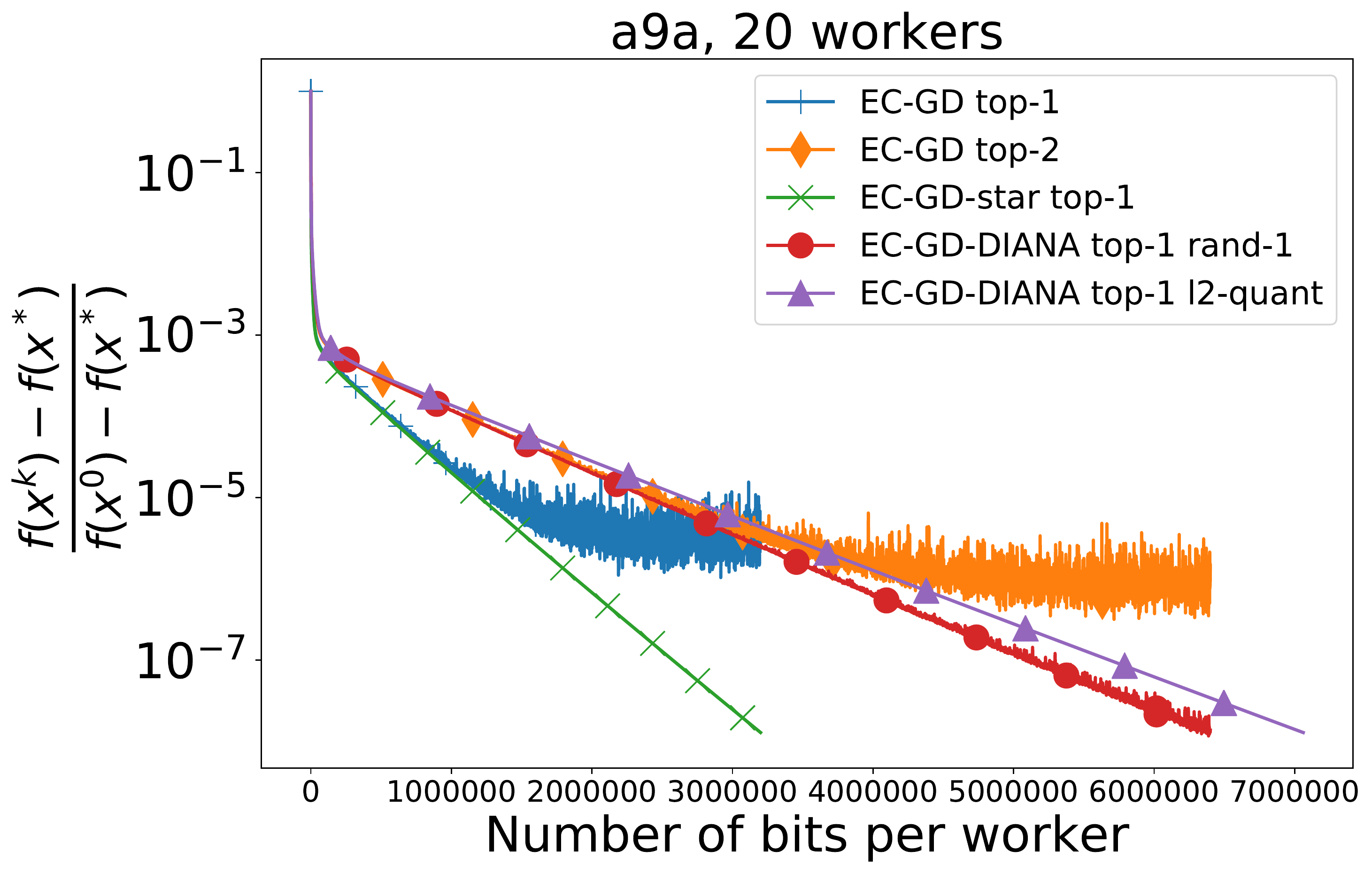}
	\includegraphics[width=0.32\textwidth]{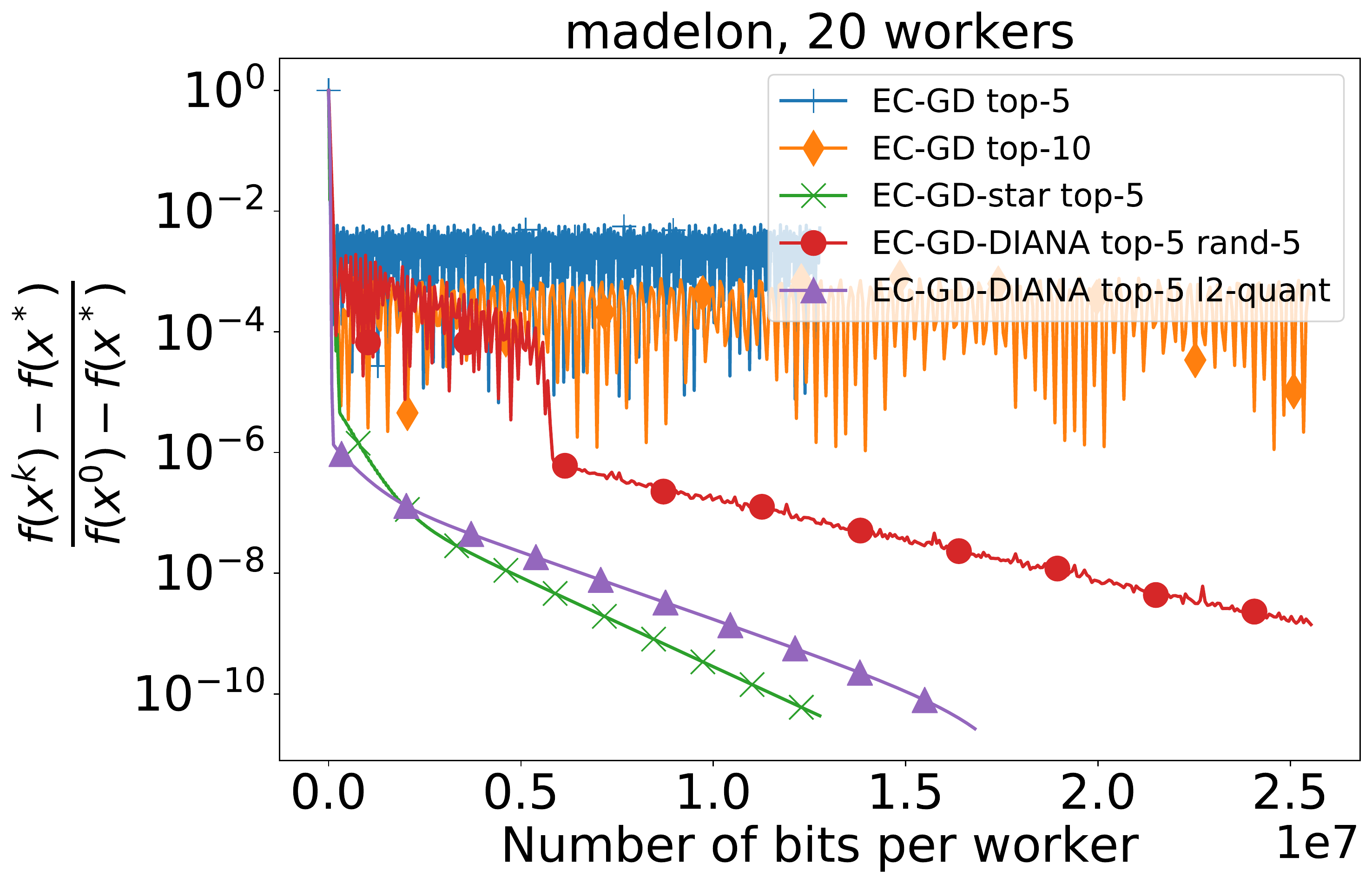}    
	\includegraphics[width=0.32\textwidth]{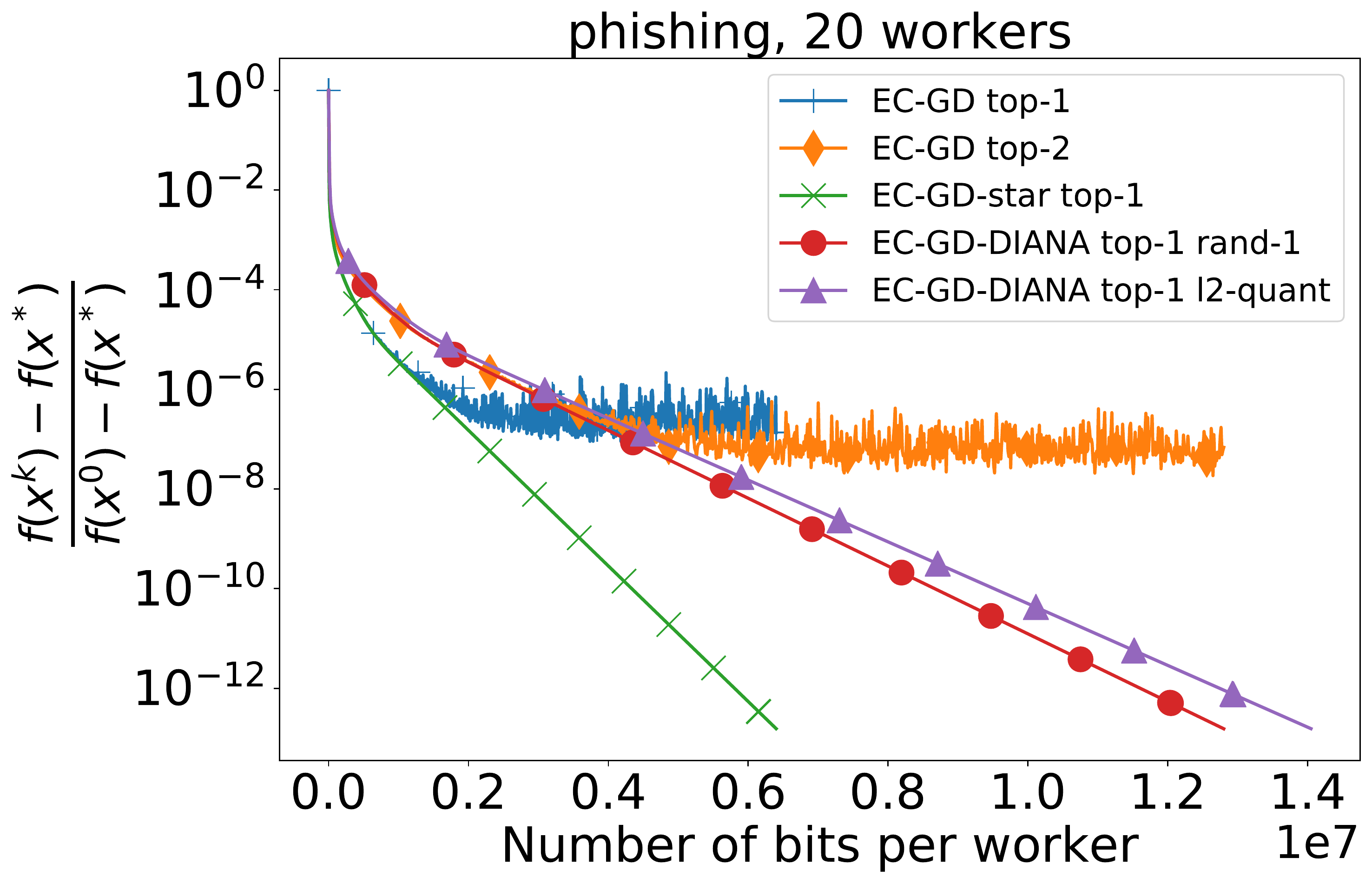}    
    \\
    \includegraphics[width=0.32\textwidth]{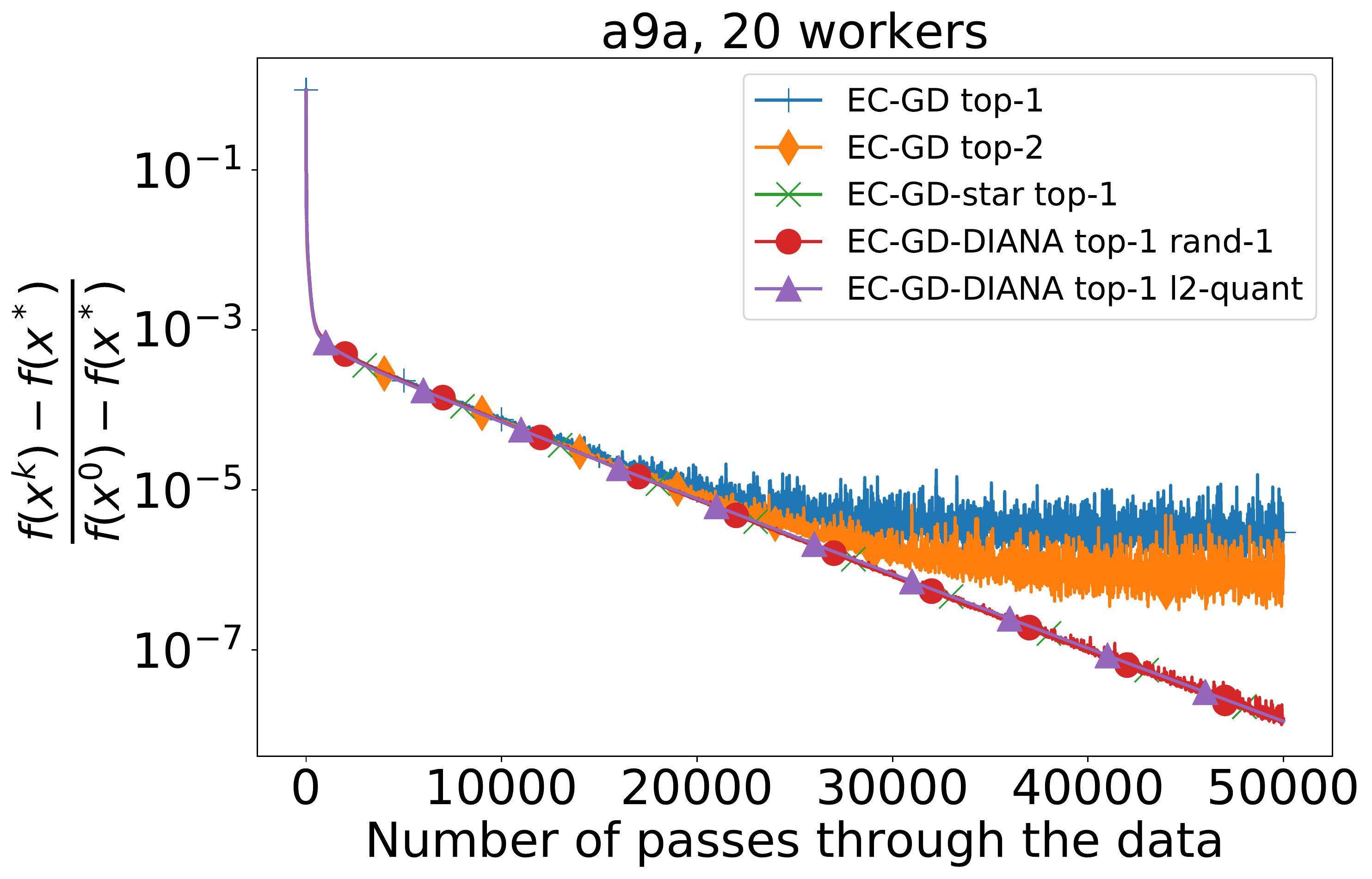}    
	\includegraphics[width=0.32\textwidth]{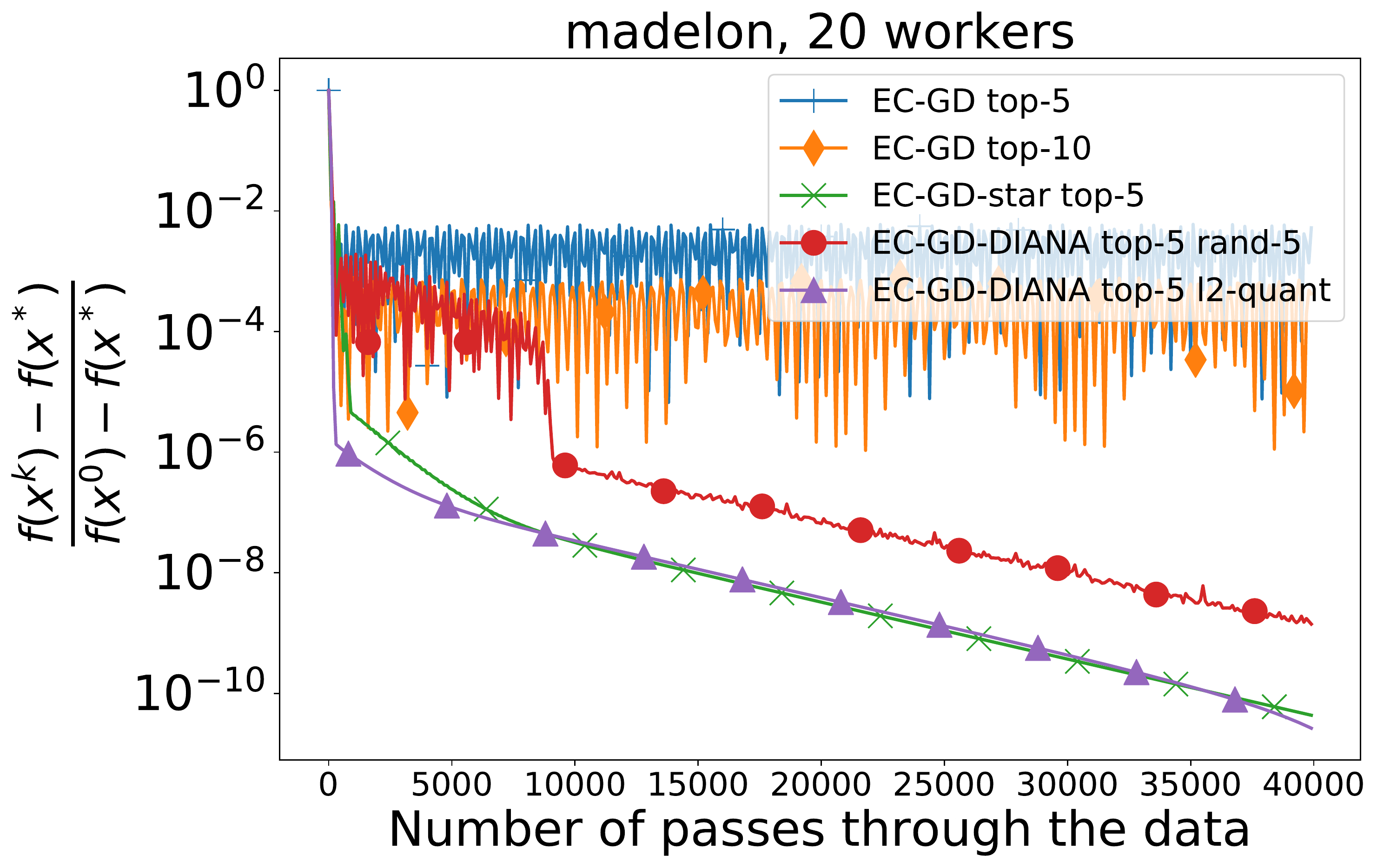}    
	\includegraphics[width=0.32\textwidth]{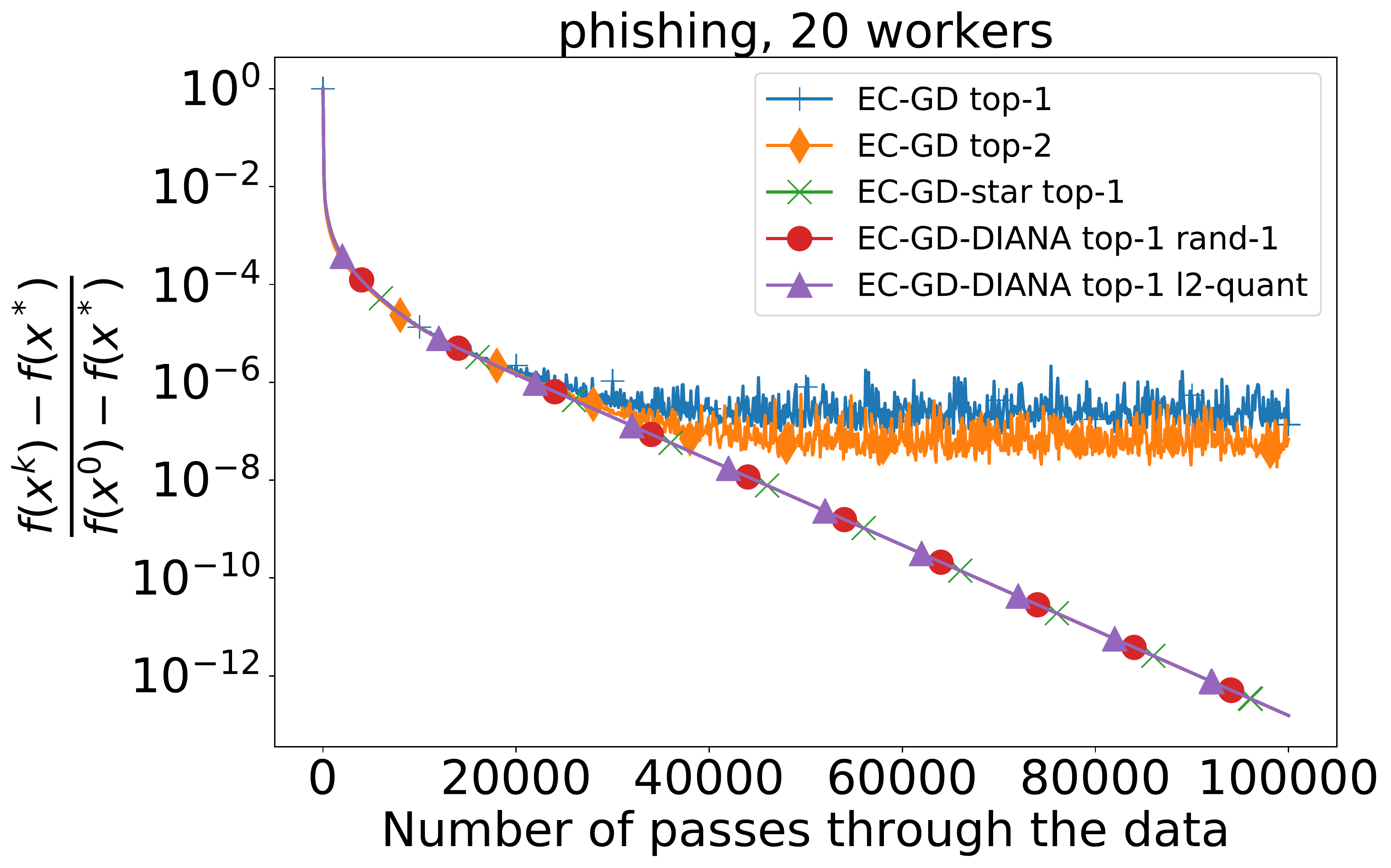}    	
	\caption{Trajectories of {\tt EC-GD}, {\tt EC-GD-star} and {\tt EC-DIANA} applied to solve logistic regression problem with $20$ workers.}
    \label{fig:gd_logreg_20_workers}
\end{figure}
In order to show the effect of {\tt DIANA}-type variance reduction itself, we consider the case when all workers compute the full gradients of their functions, see Figure~\ref{fig:gd_logreg_20_workers} (included here) and Figures~\ref{fig:gd_logreg_20_workers_appendix}--\ref{fig:gd_logreg_100_workers_id} (in the Appendix). Clearly, for all datasets except {\tt mushrooms}, {\tt EC-GD} with constant stepsize converges  to a neighborhood of the solution only, while {\tt EC-GDstar} and {\tt EC-GD-DIANA} converge with linear rate asymptotically to the exact solution. {\tt EC-GDstar} always show the best performance, however, it is impractical: we used a very good approximation of the solution to apply this method. In contrast, {\tt EC-DIANA} converges slightly slower and requires more bits for communication; but it is practical and shows better performance than {\tt EC-GD}. On the {\tt mushrooms} datasets, {\tt EC-GD} does not reach the oscillation region after the given number of epochs, therefore, it is preferable there.

\clearpage

\section*{Broader Impact}
Our contribution is primarily theoretical. Therefore, a broader impact discussion is not applicable.

\begin{ack}
The work of Peter Richt\'{a}rik, Eduard Gorbunov and Dmitry Kovalev was supported by KAUST Baseline Research Fund. Part of this work was done while E.~Gorbunov was a research intern at KAUST. The research of E.~Gorbunov was also partially supported by the Ministry of Science and Higher Education of the Russian Federation (Goszadaniye) 075-00337-20-03 and RFBR, project number 19-31-51001. 
\end{ack}

\bibliography{ef_sigma_k}

\clearpage

\part*{\Huge Appendix: Linearly Converging Error Compensated SGD}
\appendix

{
\tableofcontents
}

\clearpage

\section{Missing Plots} \label{sec:moreexperiments}

\subsection{Compressing stochastic gradients}

\begin{figure}[h]
    \centering
    \includegraphics[width=0.32\textwidth]{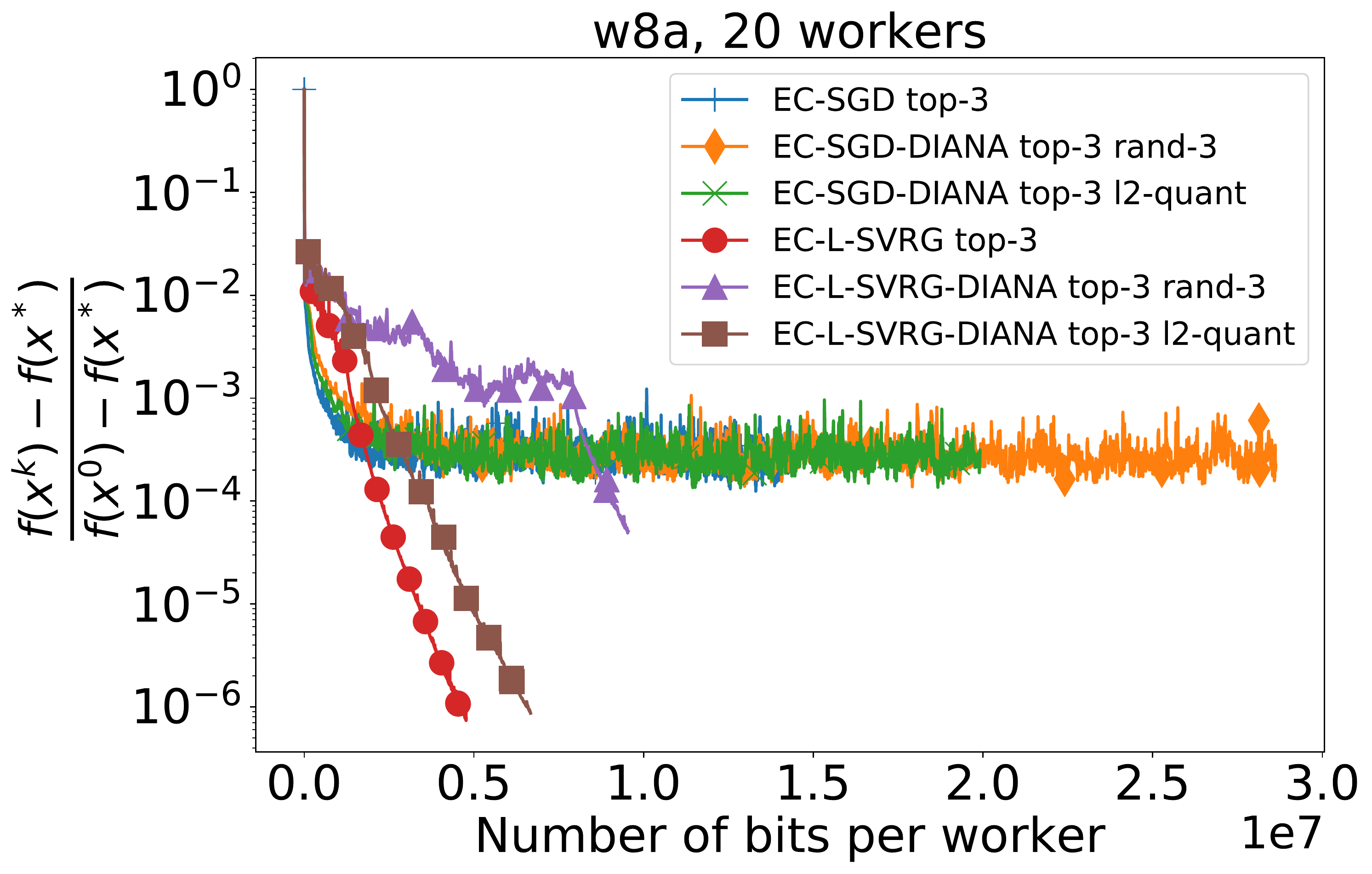}
	\includegraphics[width=0.32\textwidth]{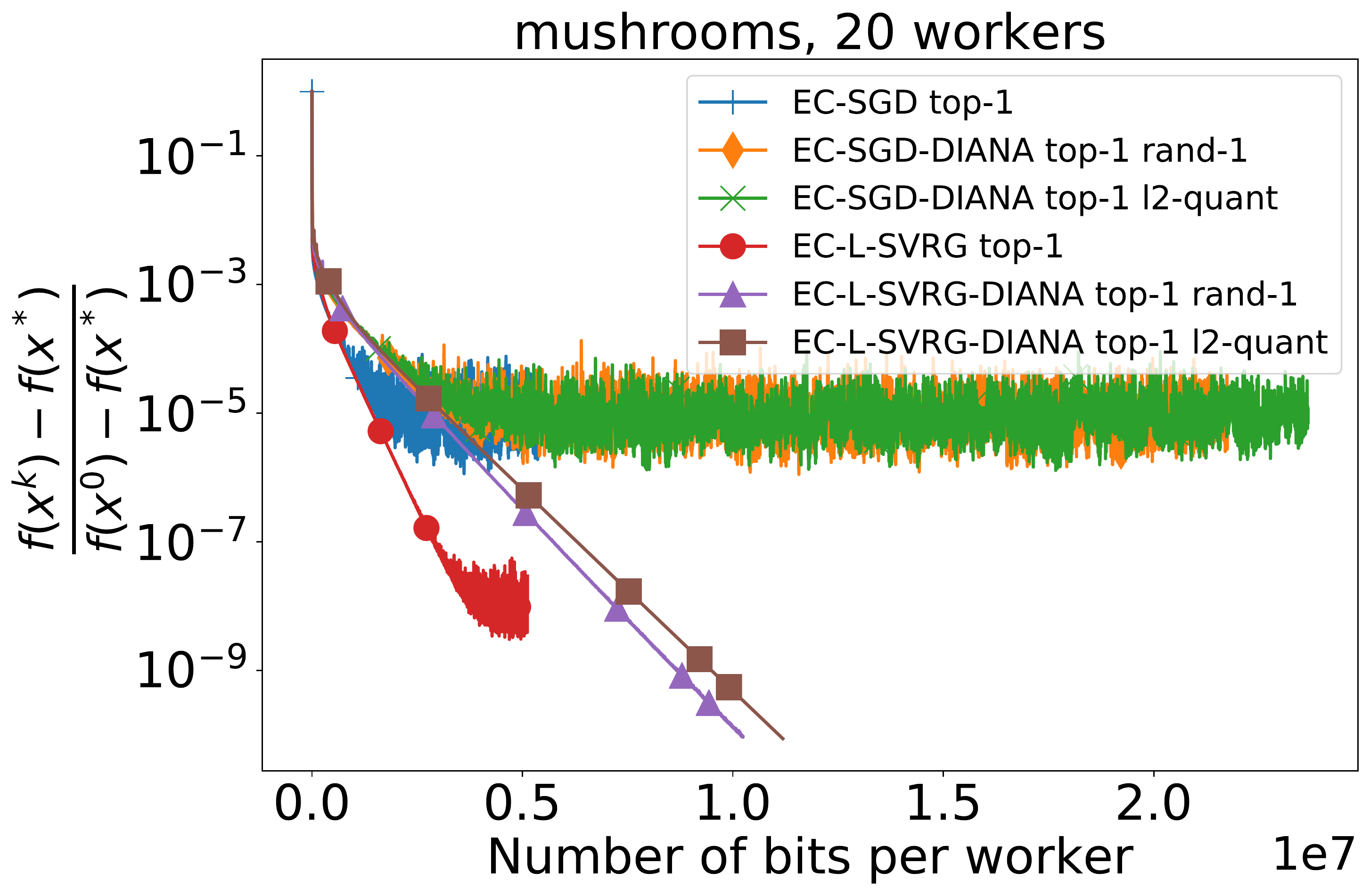}    
	\includegraphics[width=0.32\textwidth]{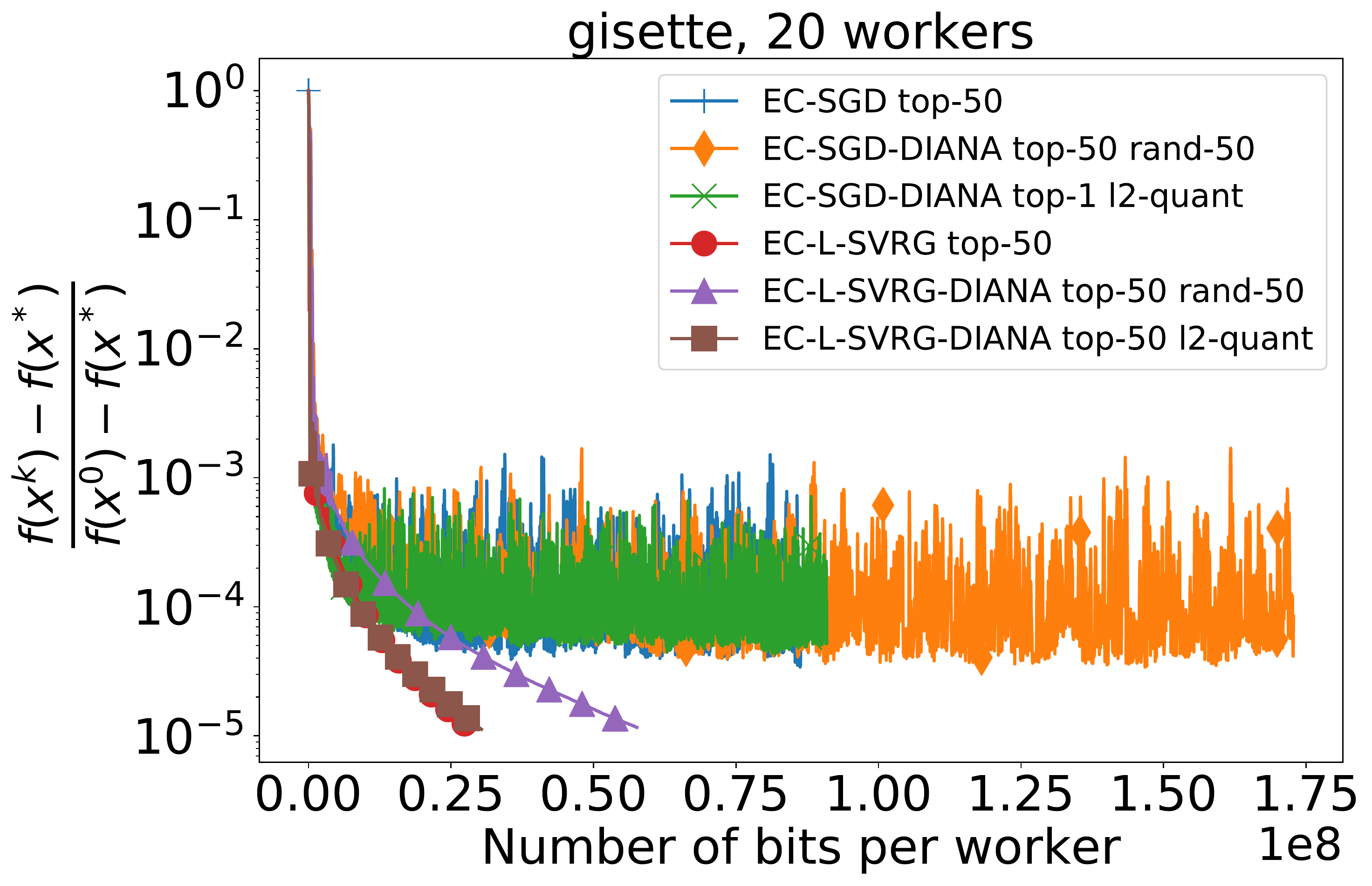}    
    \\
    \includegraphics[width=0.32\textwidth]{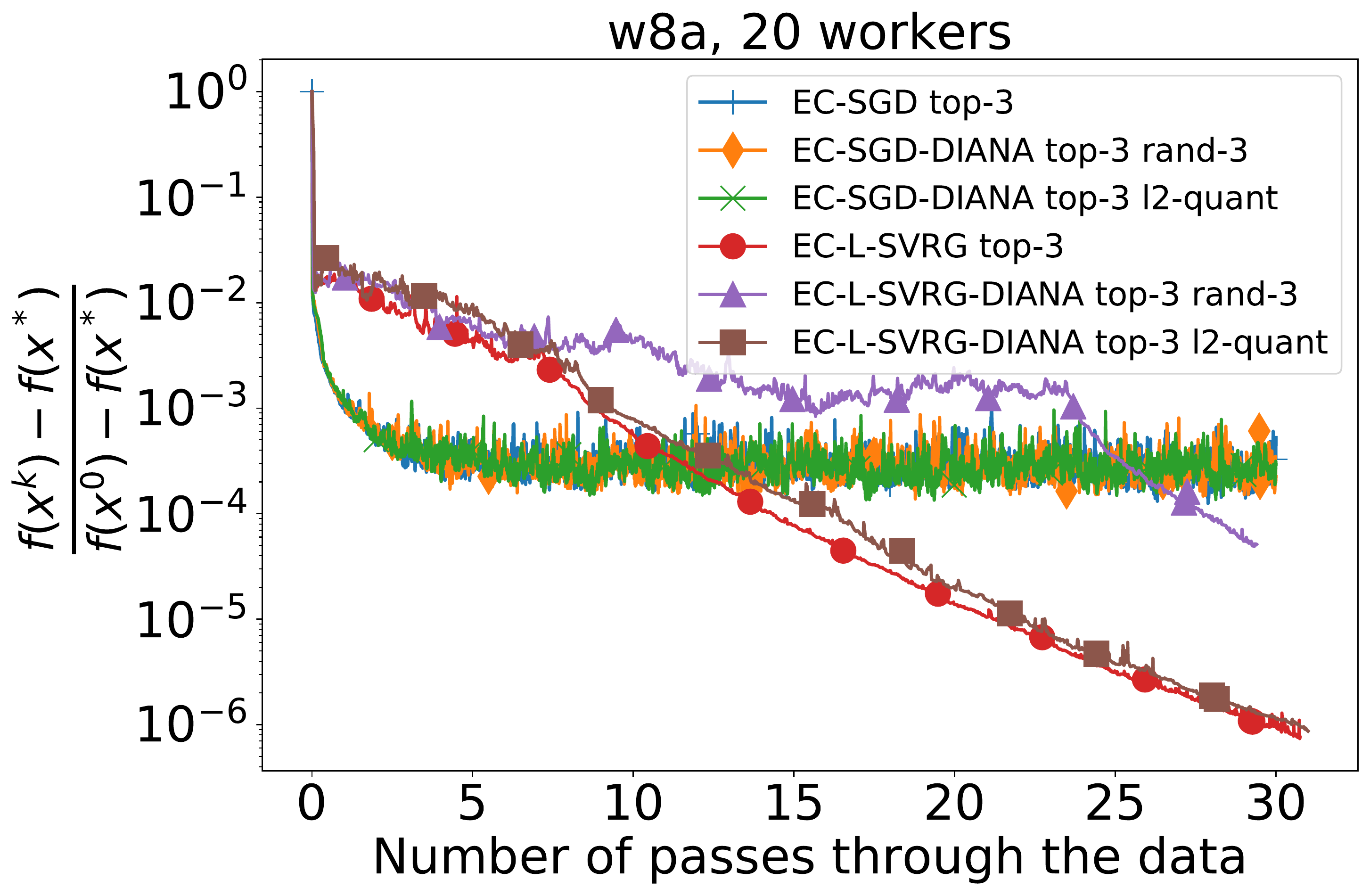}
	\includegraphics[width=0.32\textwidth]{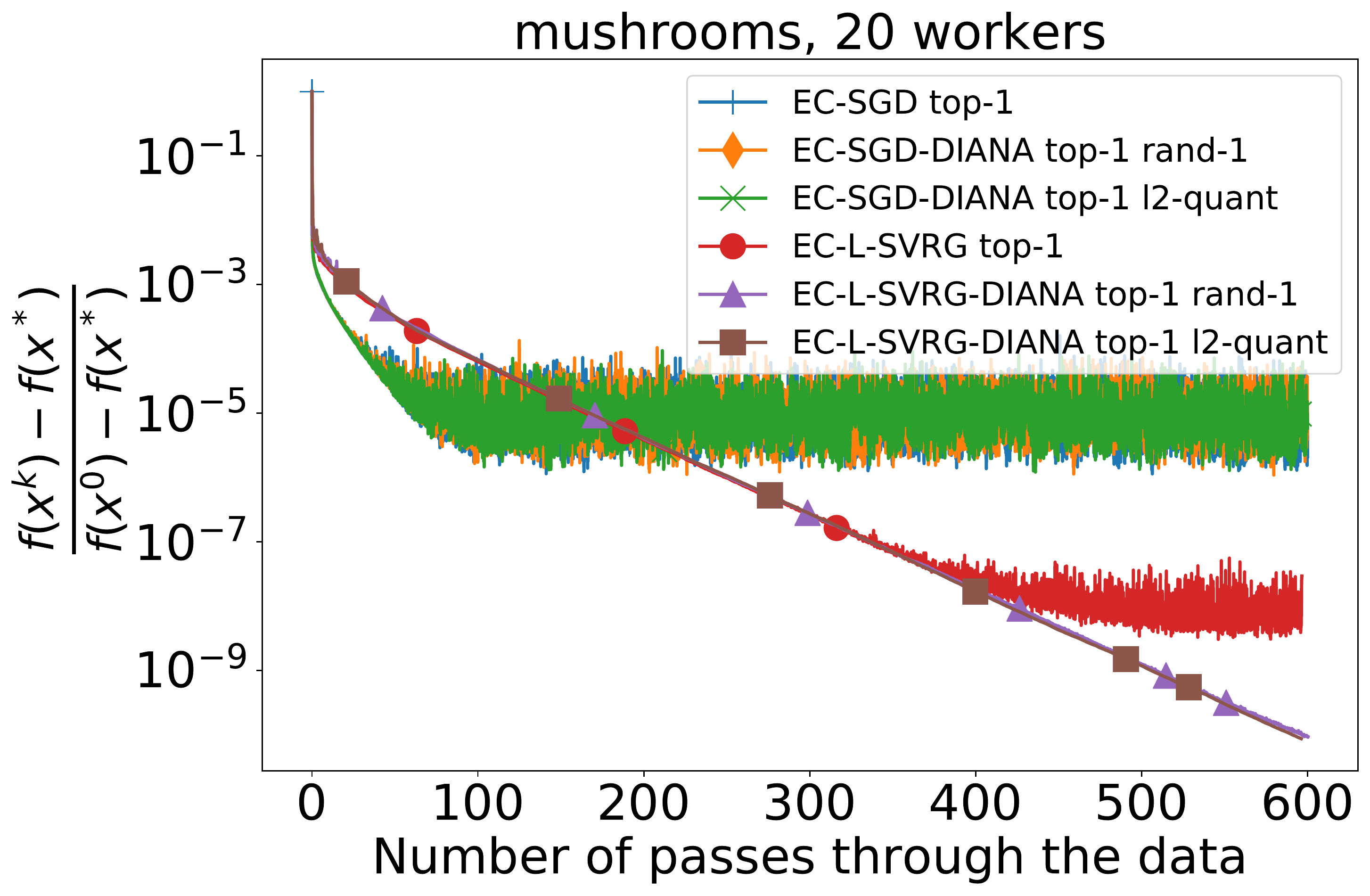}        
	\includegraphics[width=0.32\textwidth]{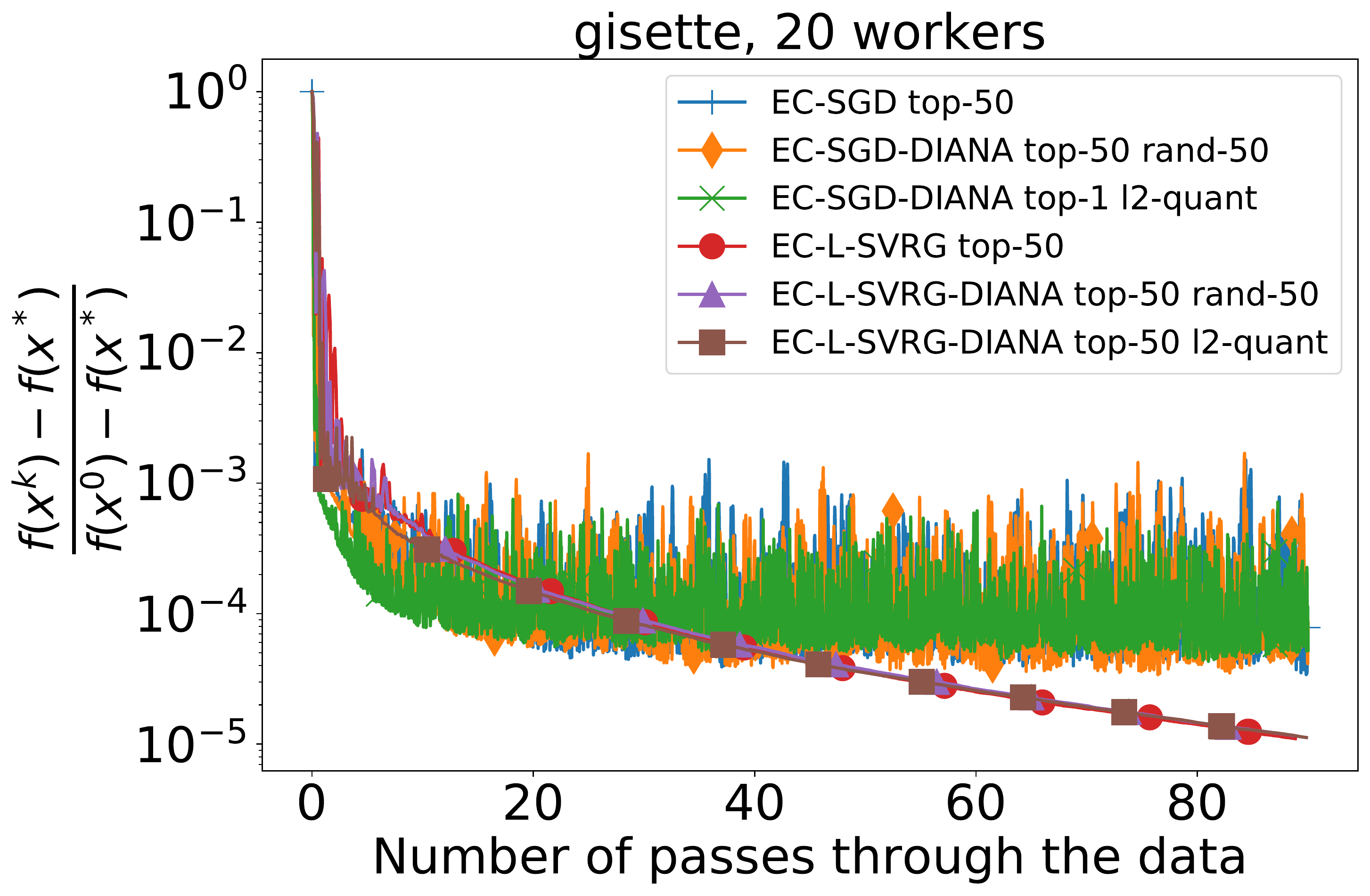}          
    \\
	\includegraphics[width=0.32\textwidth]{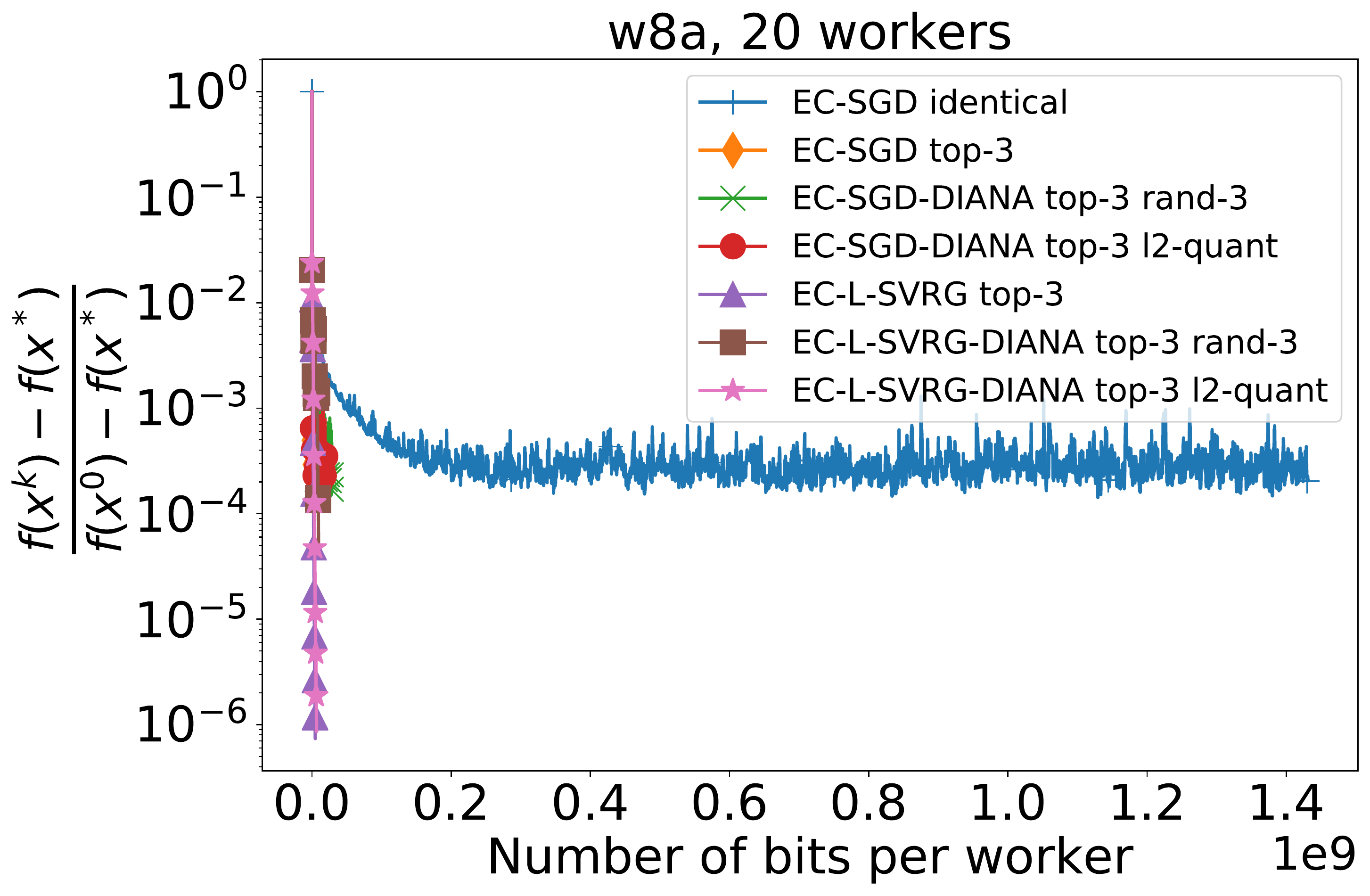}
	\includegraphics[width=0.32\textwidth]{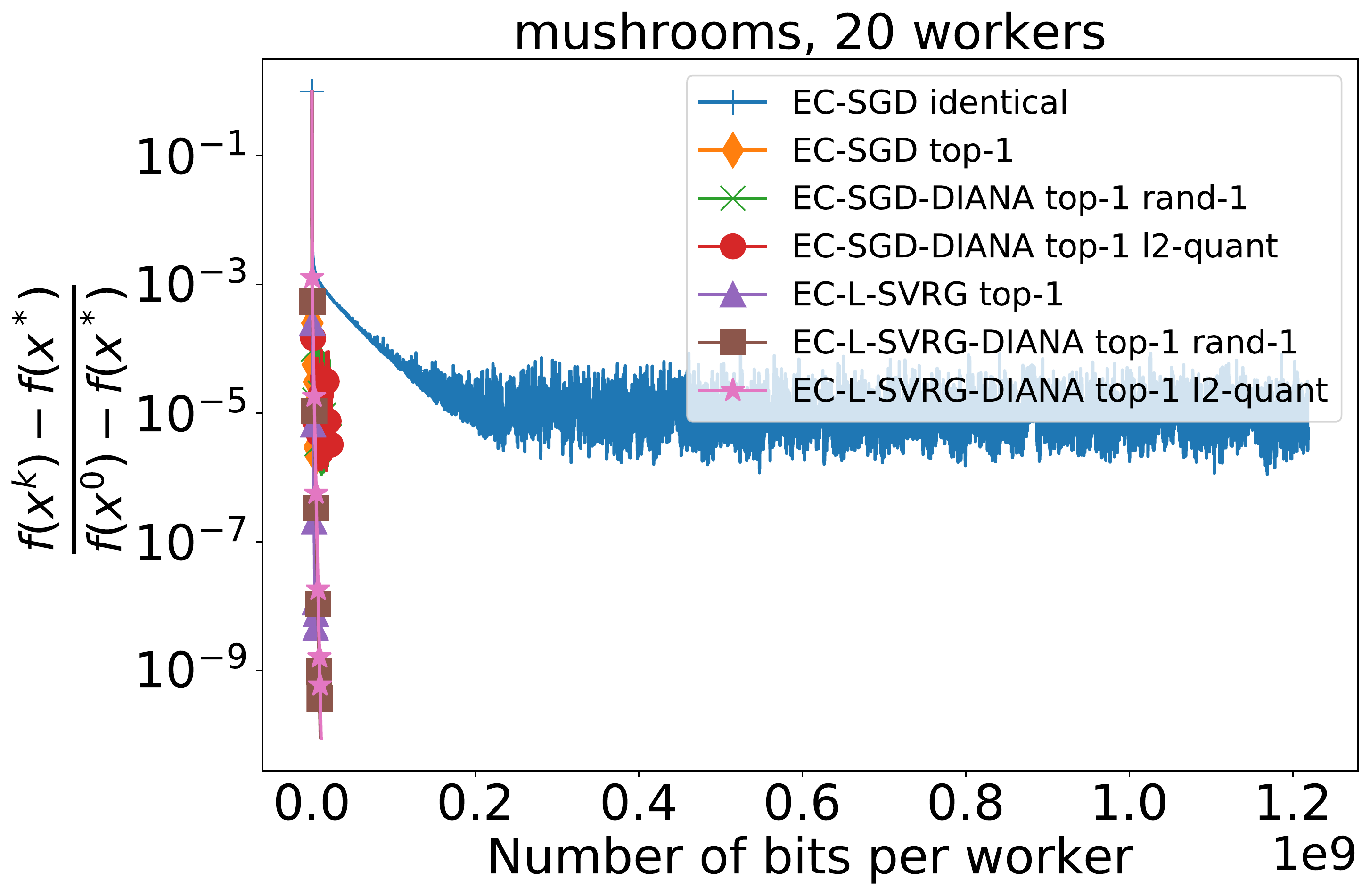}    
	\includegraphics[width=0.32\textwidth]{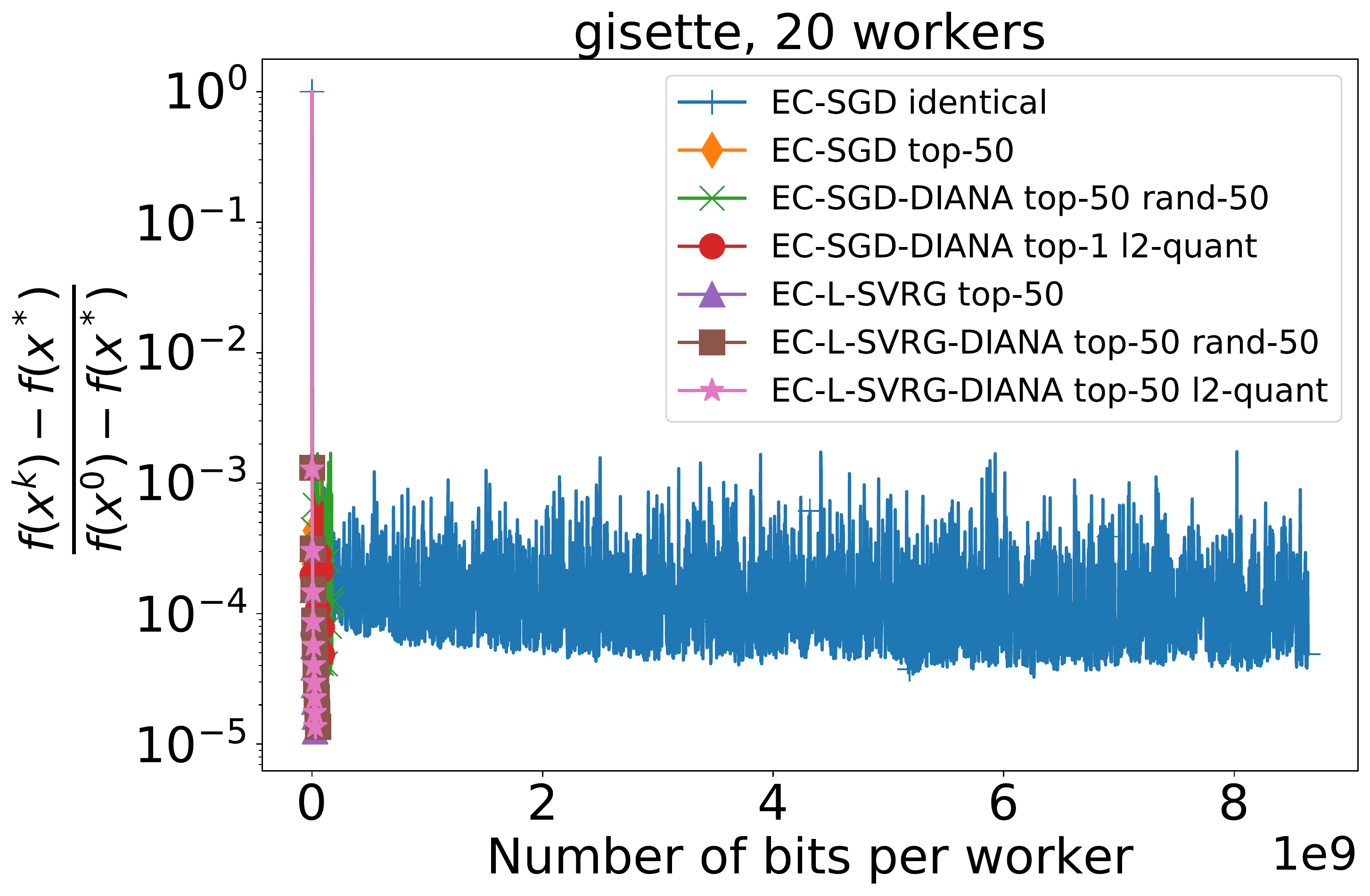}           
    \\
    \includegraphics[width=0.32\textwidth]{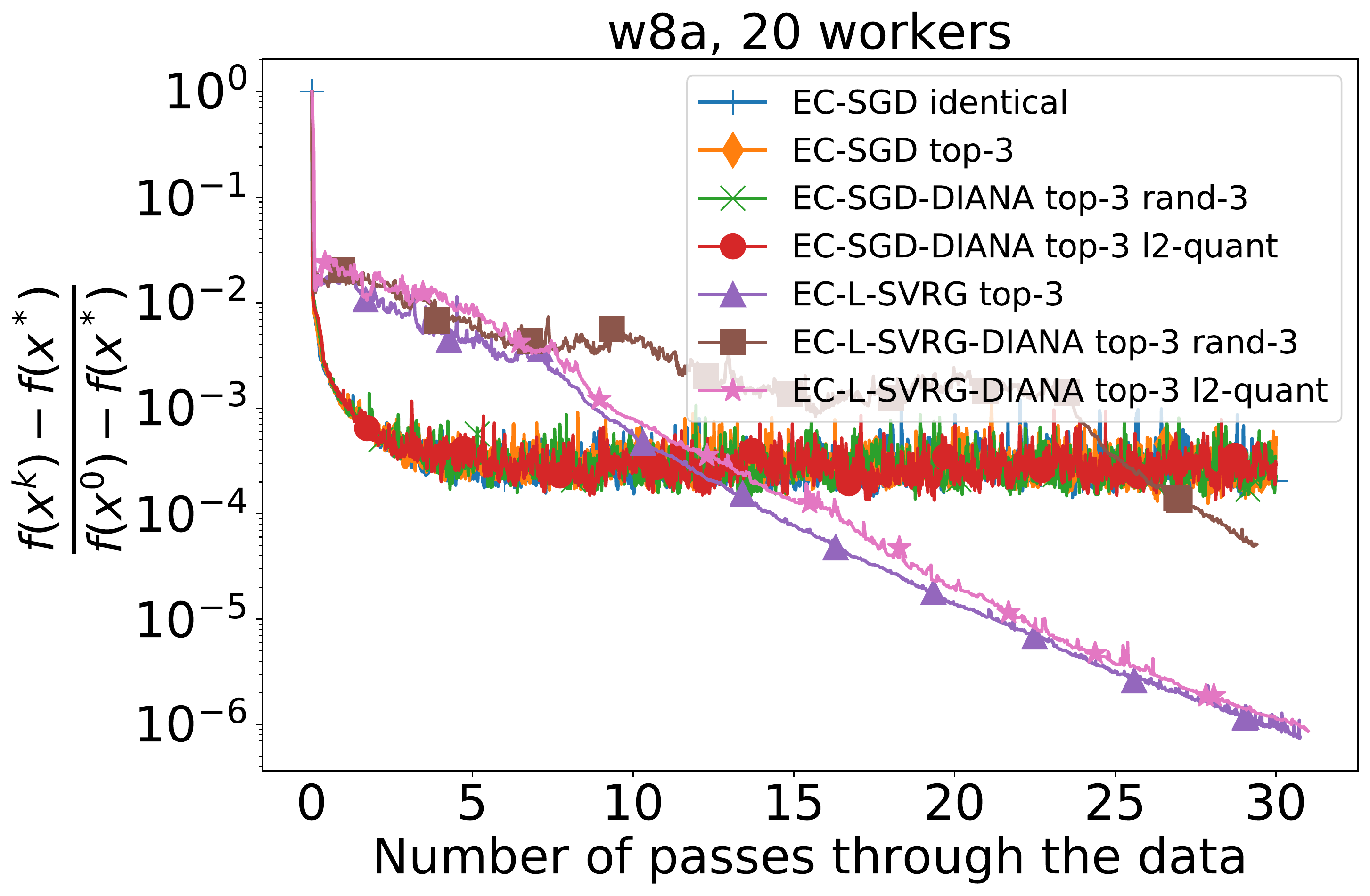}
    \includegraphics[width=0.32\textwidth]{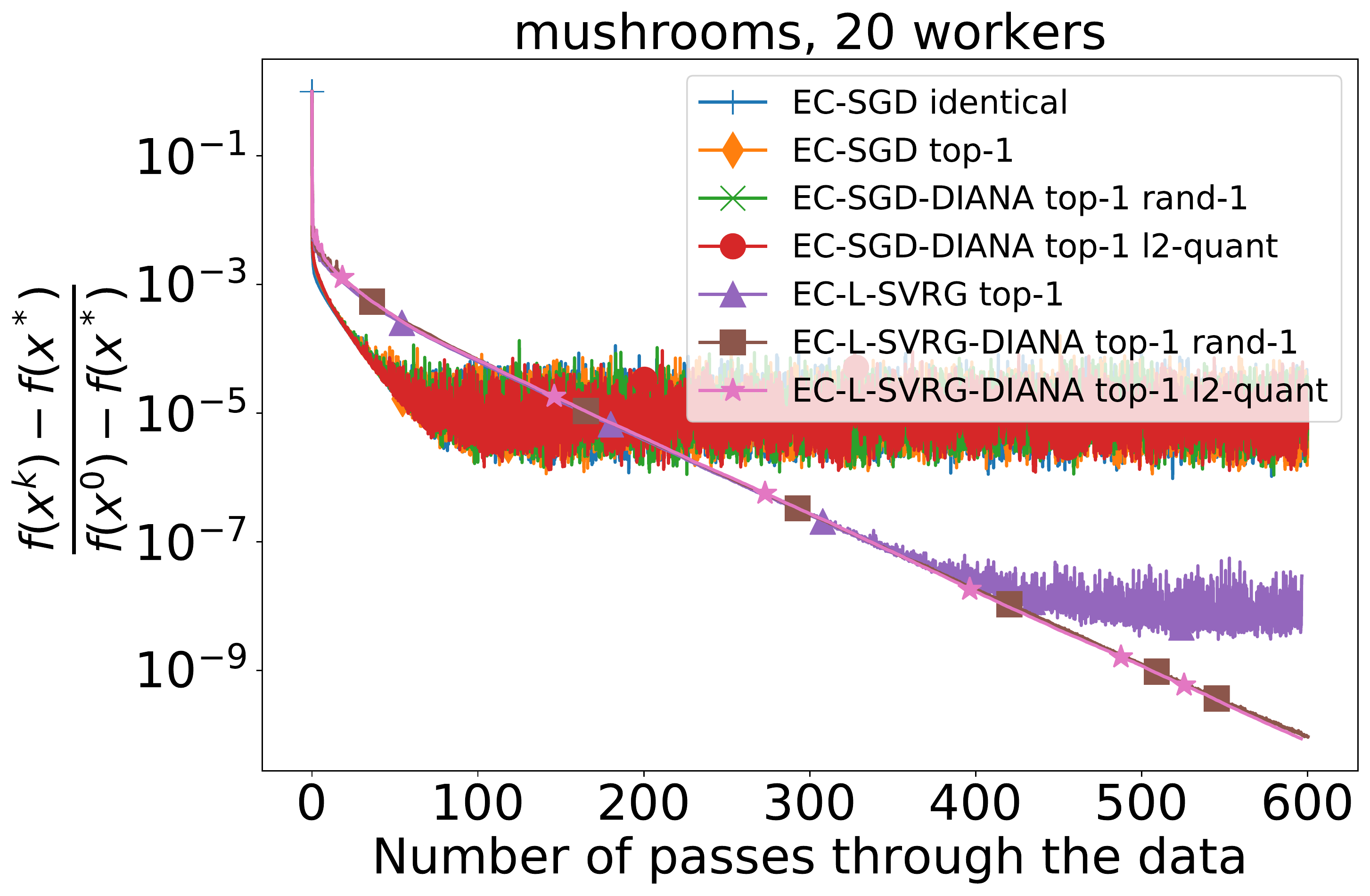}        
	\includegraphics[width=0.32\textwidth]{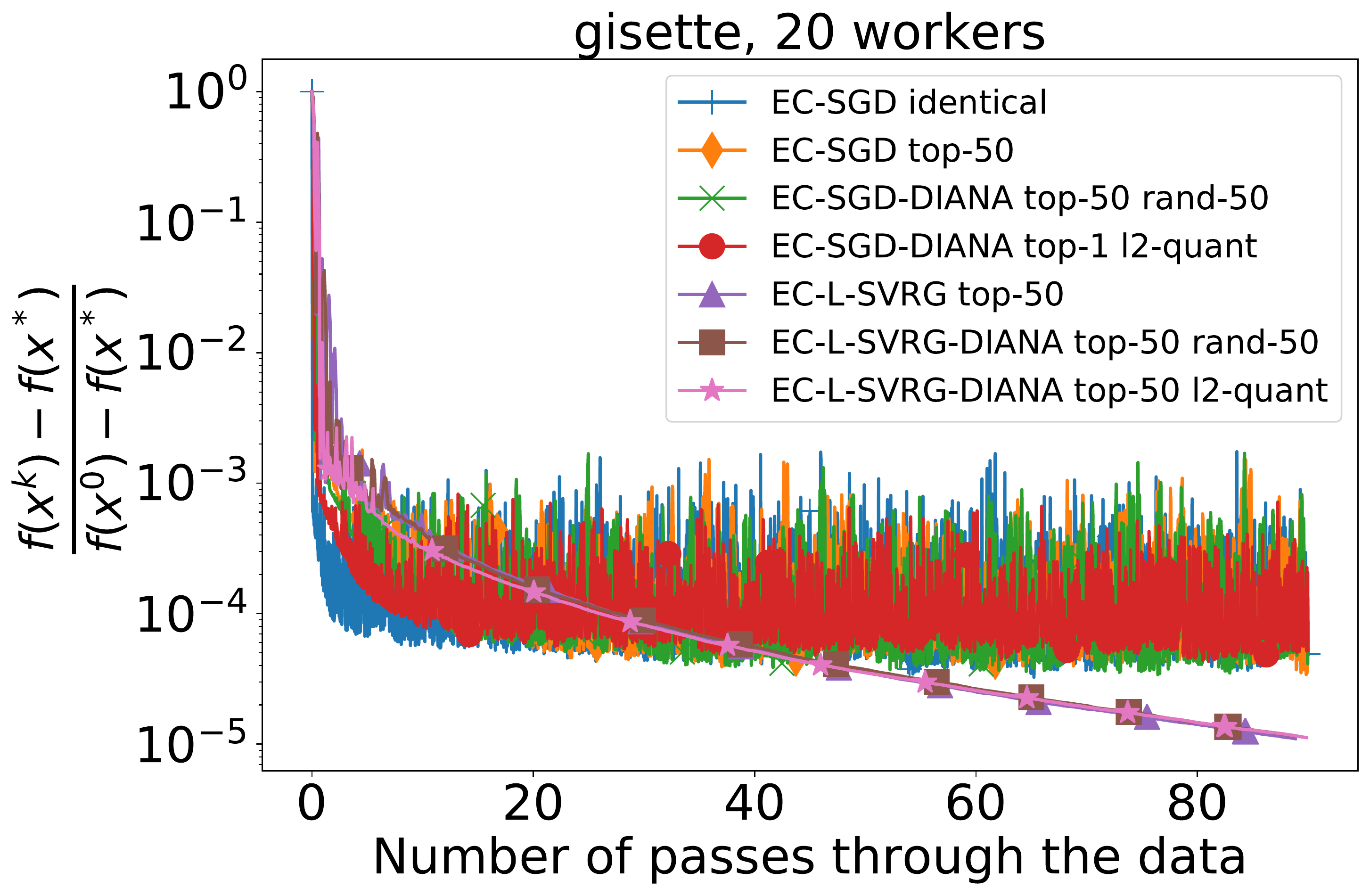}        
    \caption{Trajectories of {\tt EC-SGD}, {\tt EC-SGD-DIANA}, {\tt EC-LSVRG} and {\tt EC-LSVRG-DIANA} applied to solve logistic regression problem with $20$ workers. {\tt EC-SGD identical} corresponds to {\tt SGD} with error compensation with trivial compression operator $\cC(x) = x$, i.e., it is just parallel {\tt SGD}.}
    \label{fig:sgd_logreg_extra}
\end{figure}

\clearpage
\subsection{Compressing full gradients}

\begin{figure}[h]
    \centering
	\includegraphics[width=0.32\textwidth]{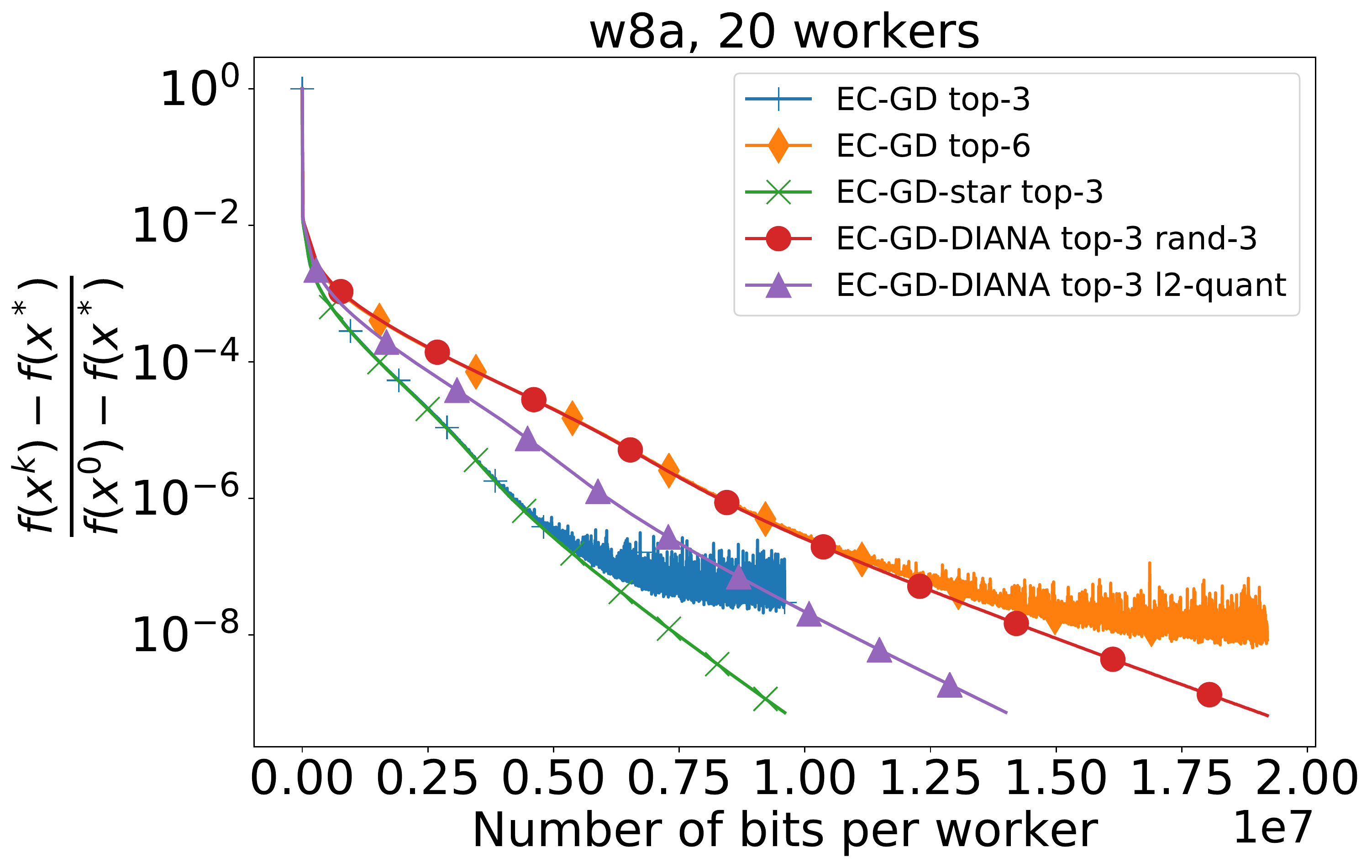}	\includegraphics[width=0.32\textwidth]{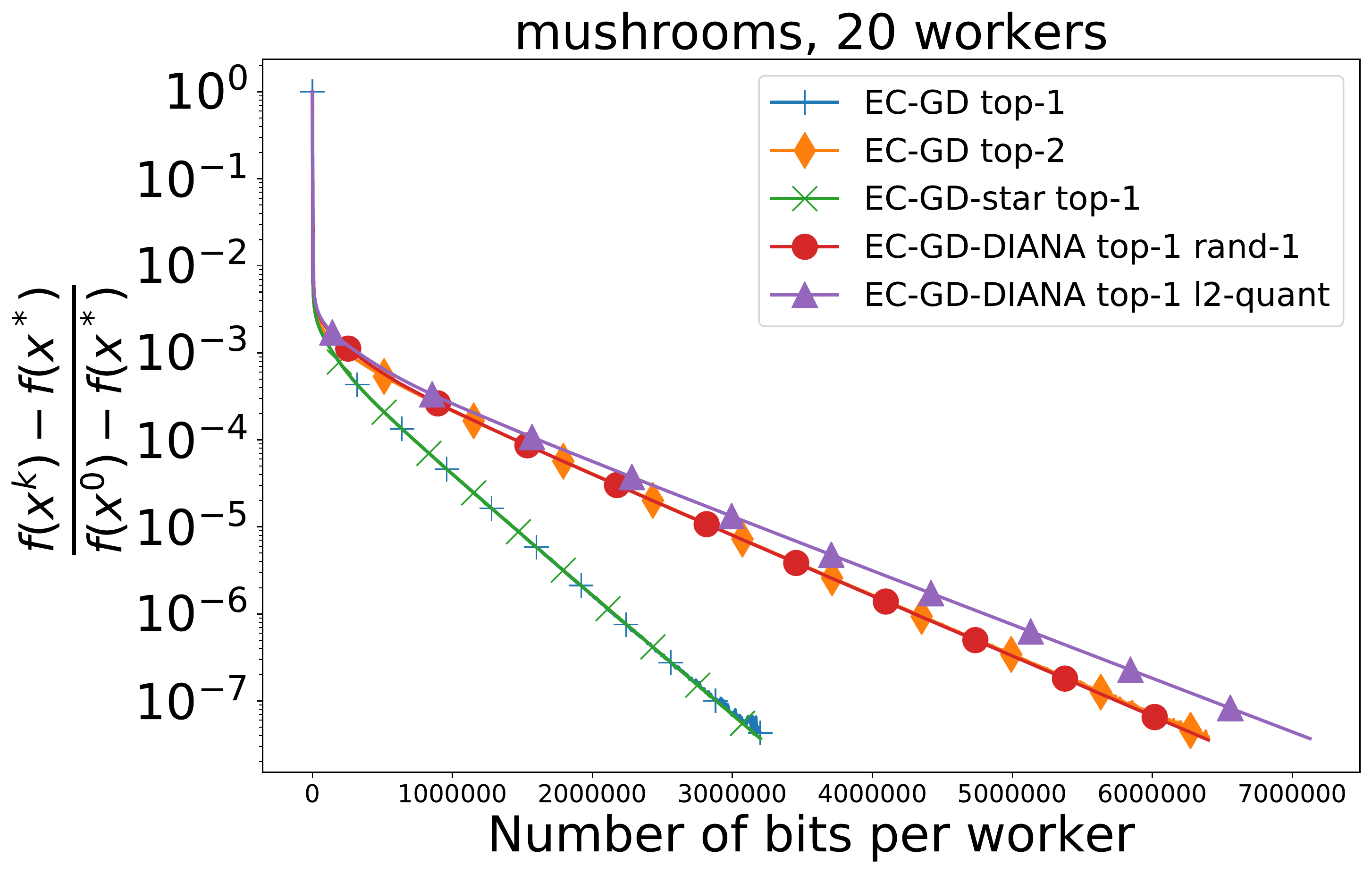}
	\includegraphics[width=0.32\textwidth]{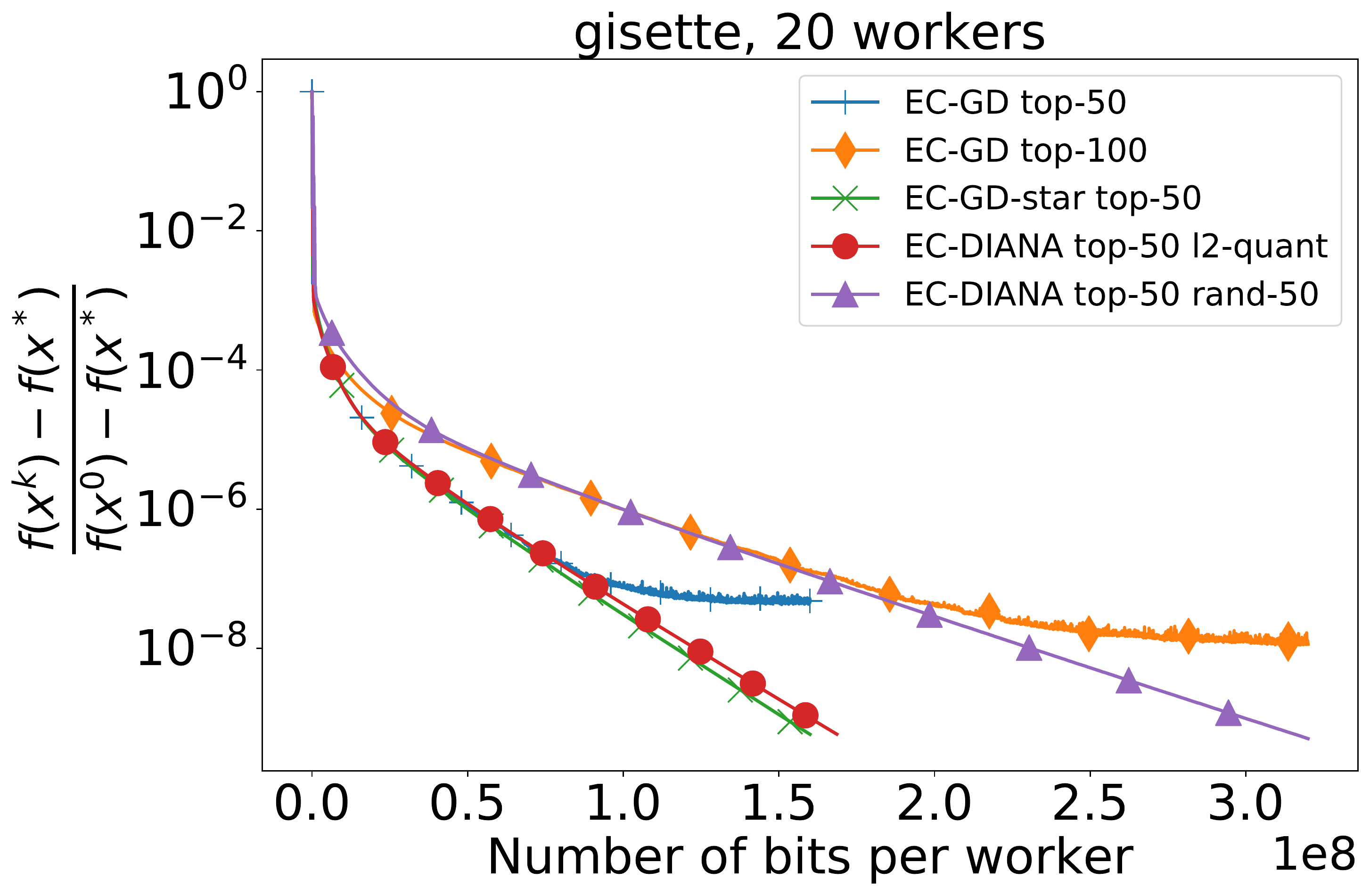}
	\\
	\includegraphics[width=0.32\textwidth]{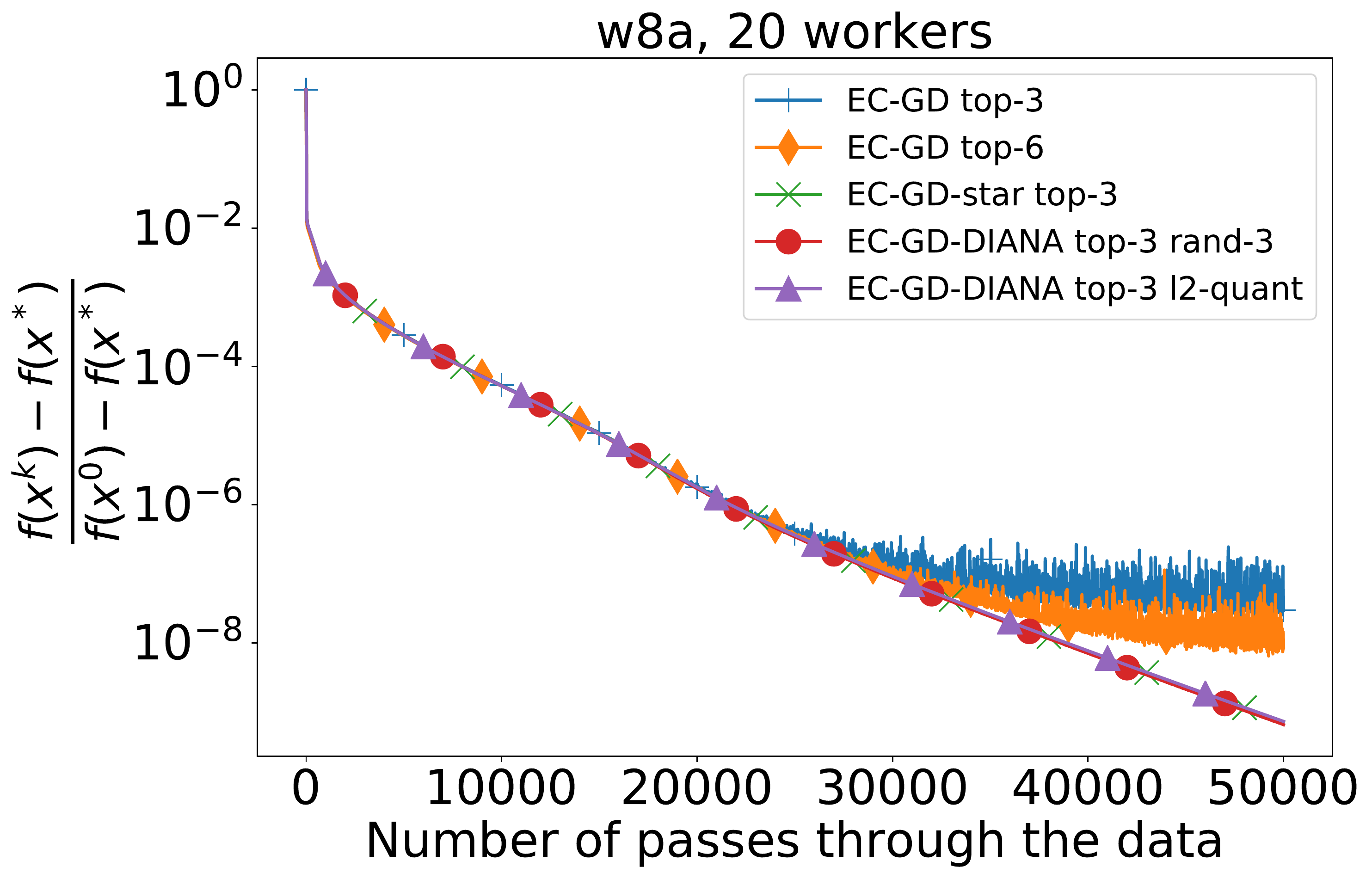}
	\includegraphics[width=0.32\textwidth]{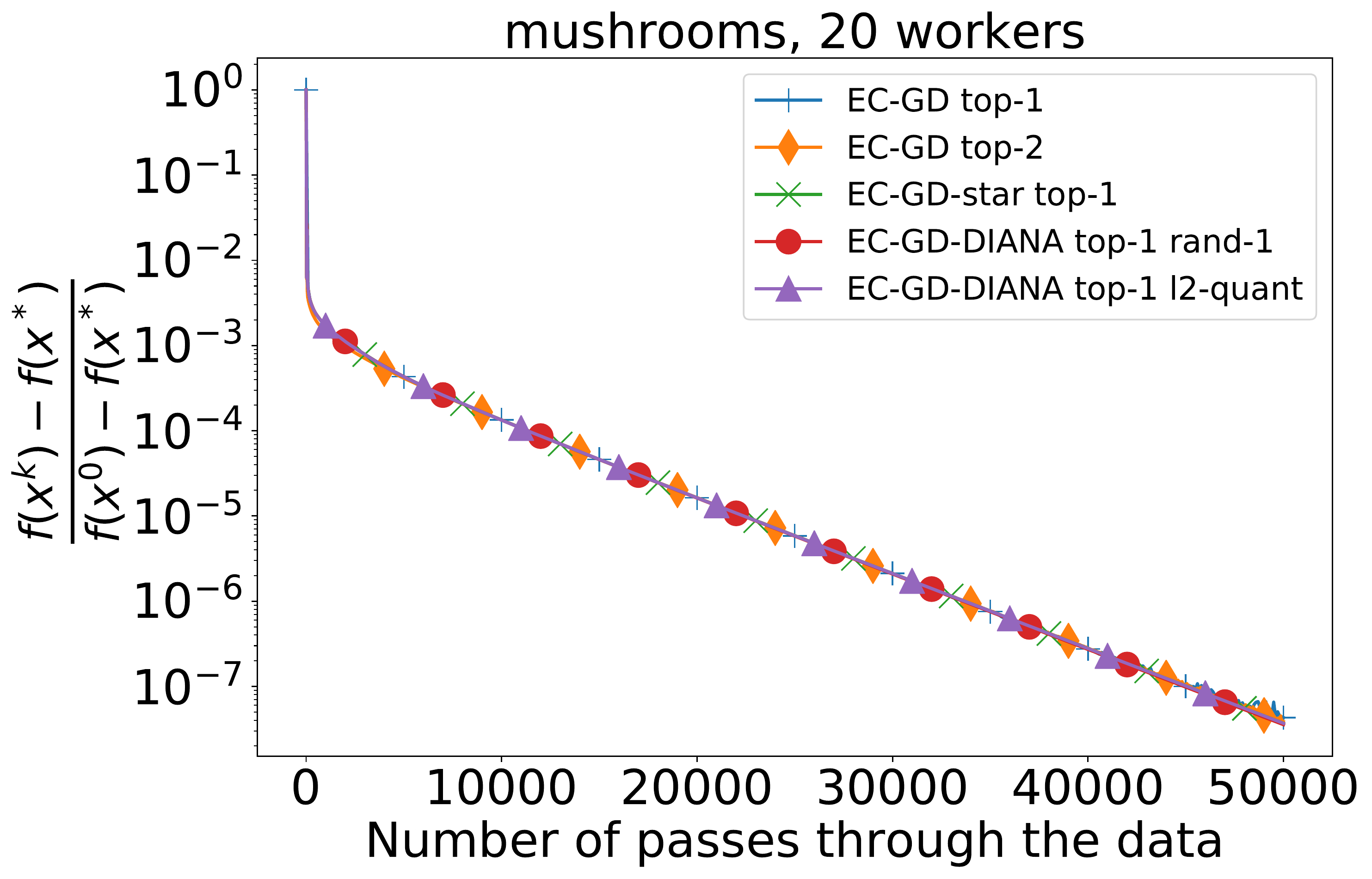}	\includegraphics[width=0.32\textwidth]{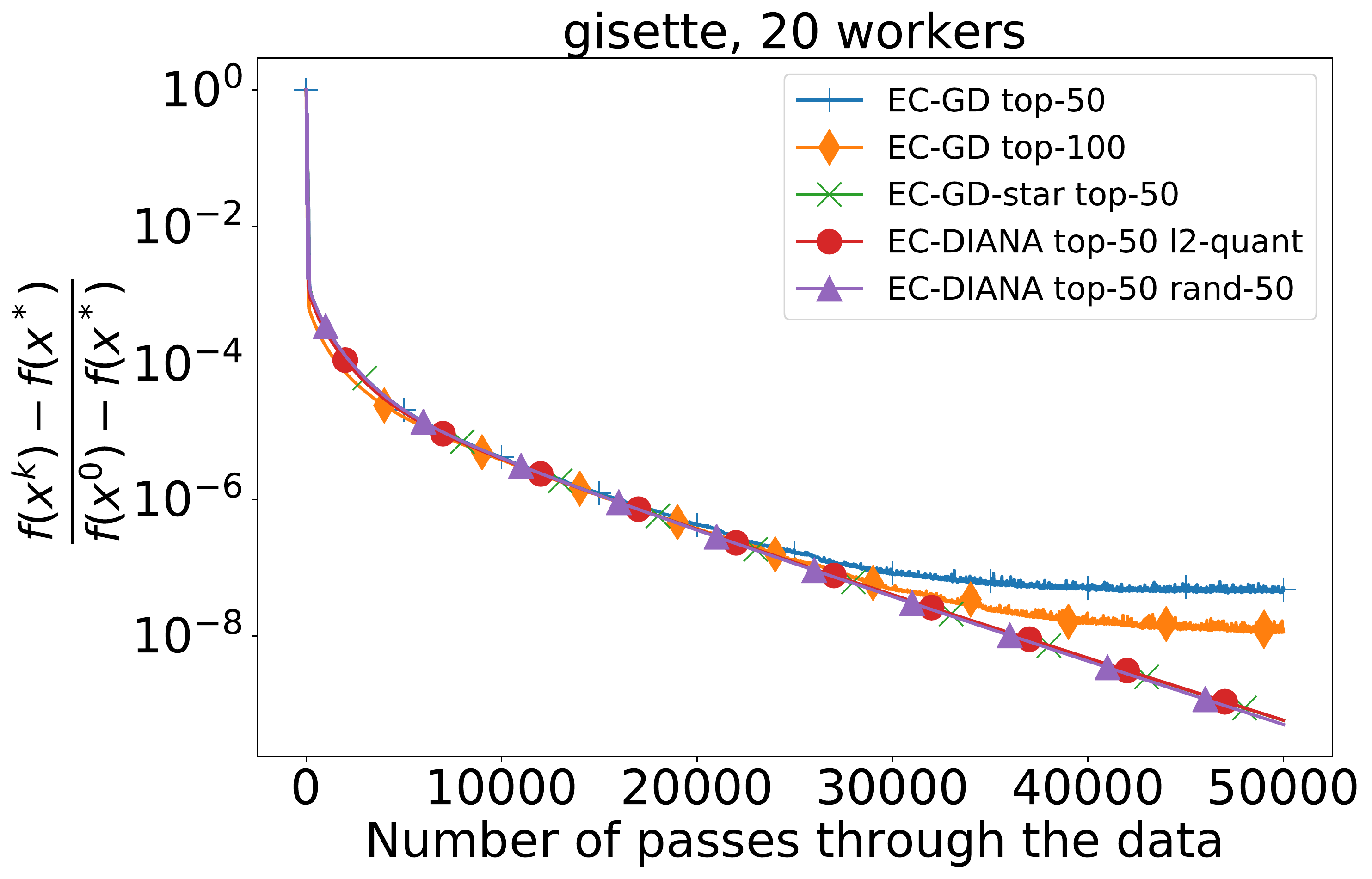}
	\caption{Trajectories of {\tt EC-GD}, {\tt EC-GD-star} and {\tt EC-DIANA} applied to solve logistic regression problem with $20$ workers.}
    \label{fig:gd_logreg_20_workers_appendix}
\end{figure}
\begin{figure}[h]
    \centering
    \includegraphics[width=0.32\textwidth]{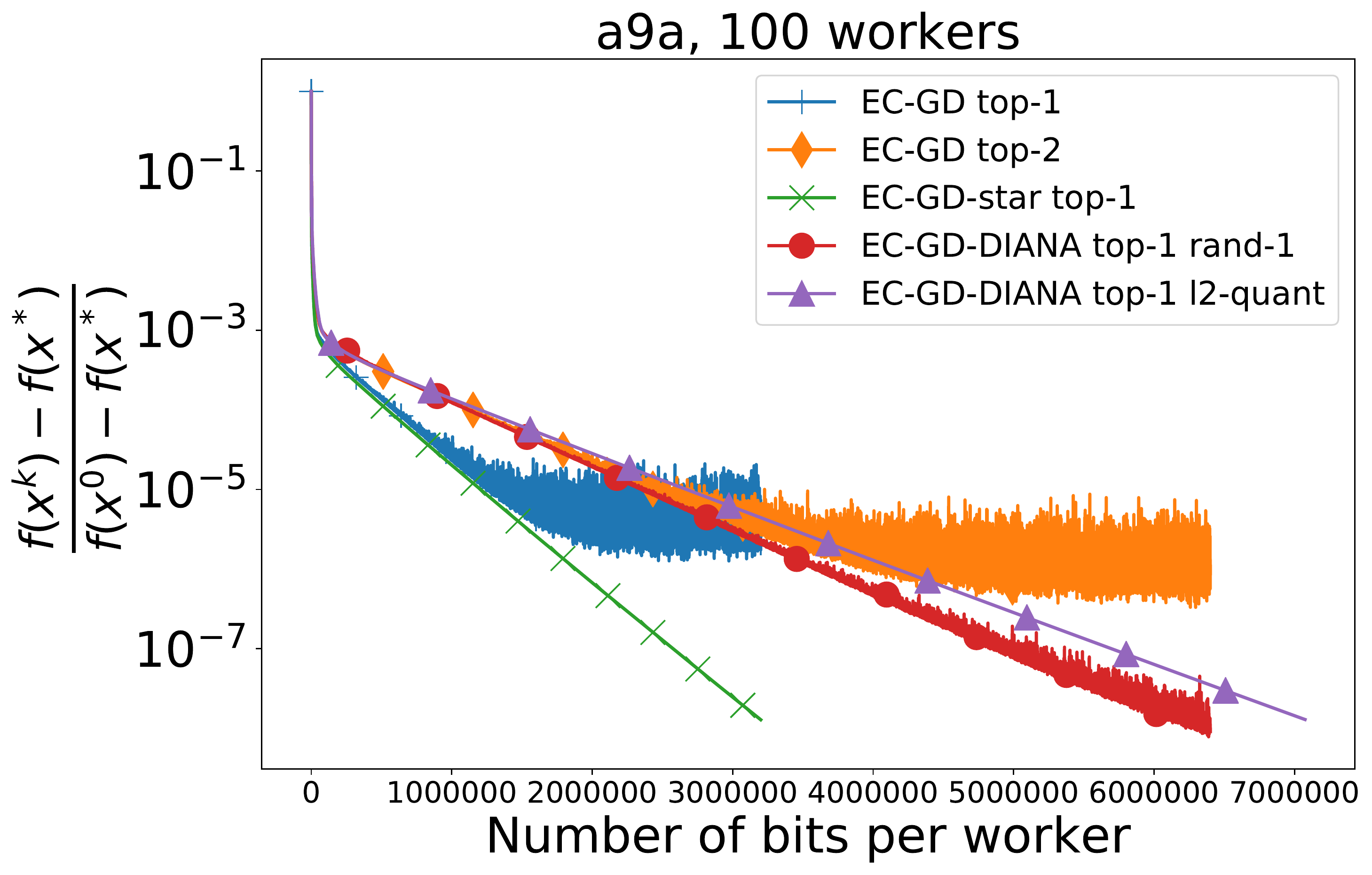}
	\includegraphics[width=0.32\textwidth]{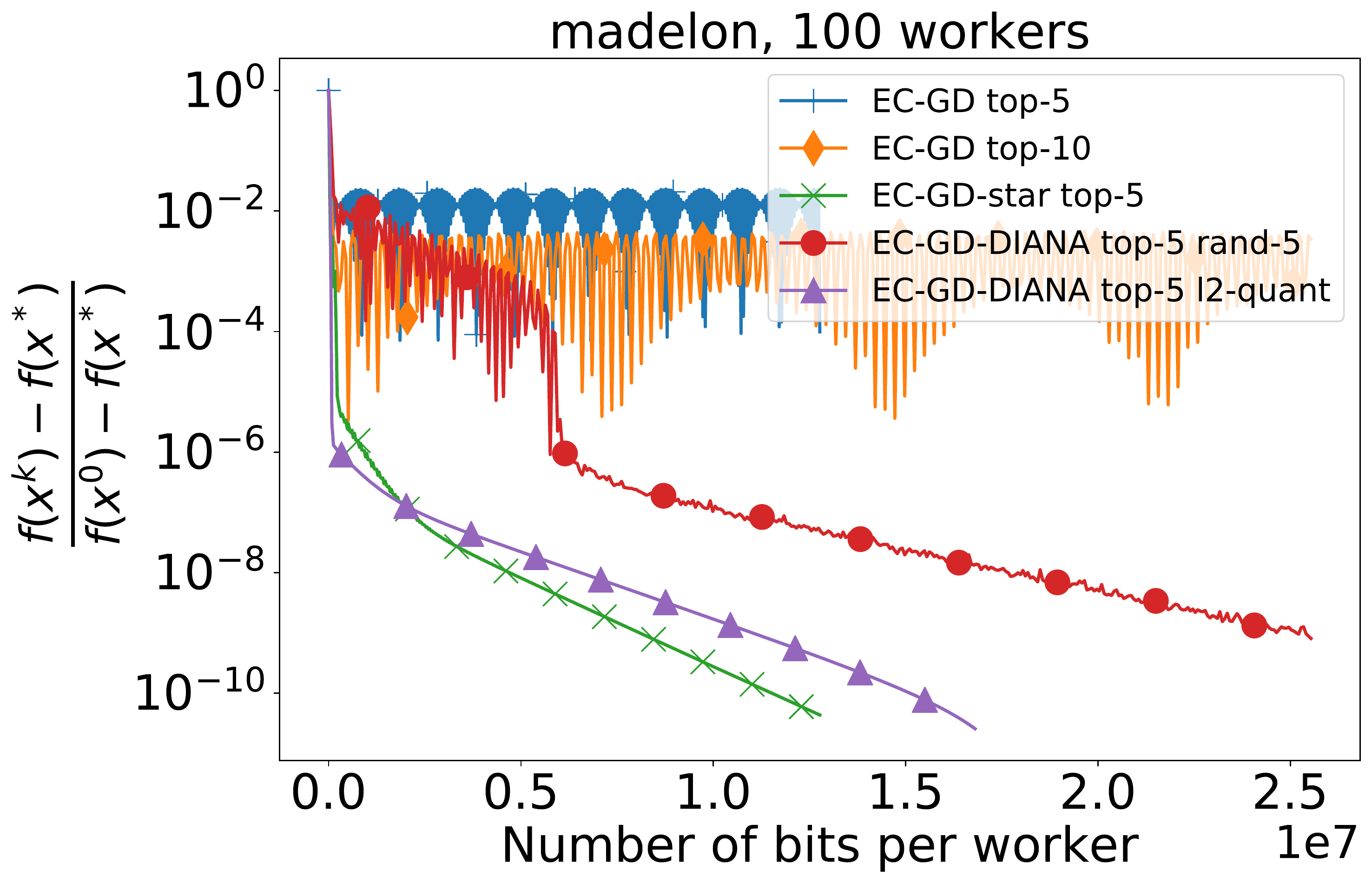}    
	\includegraphics[width=0.32\textwidth]{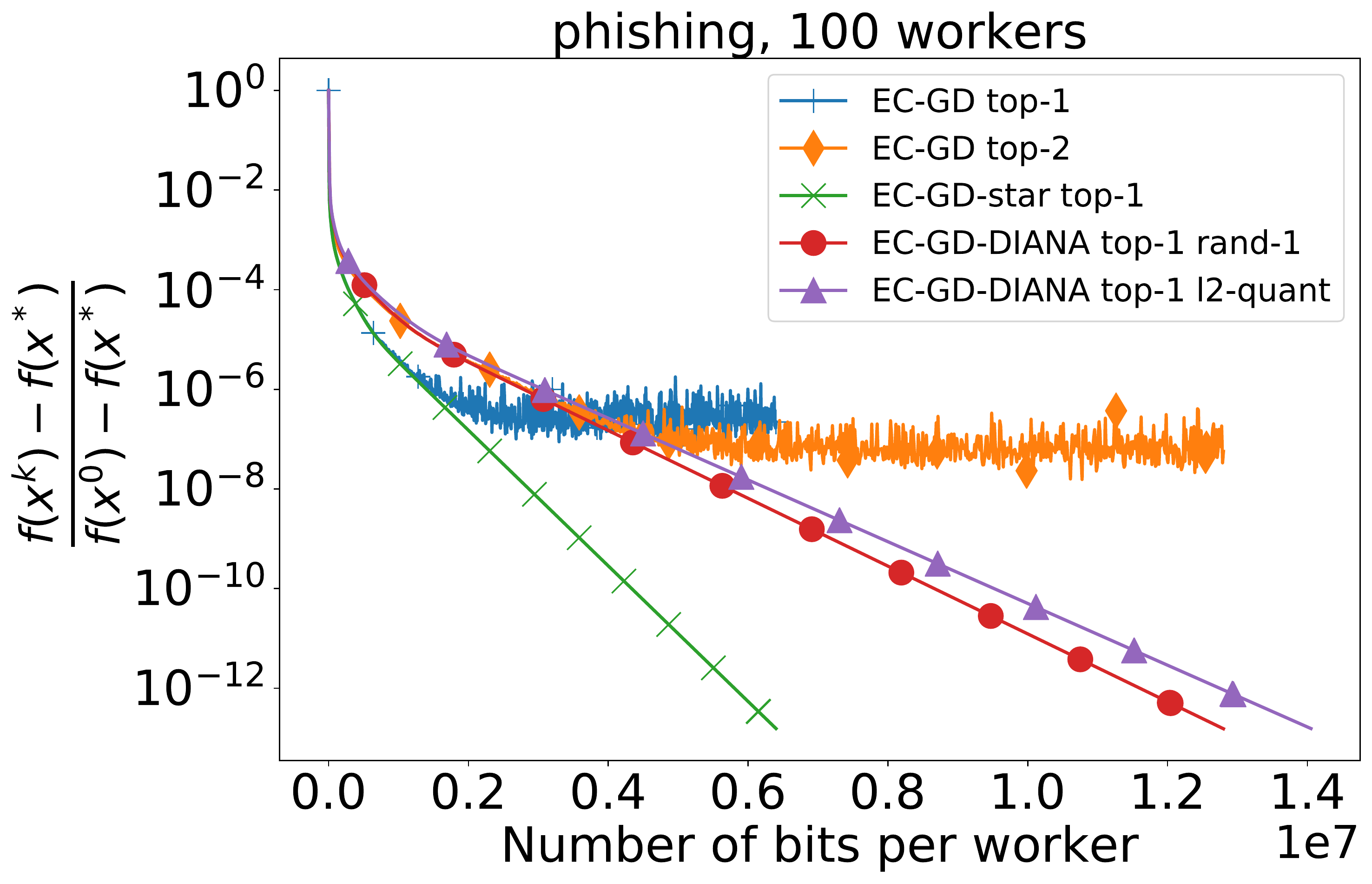}    
    \\
    \includegraphics[width=0.32\textwidth]{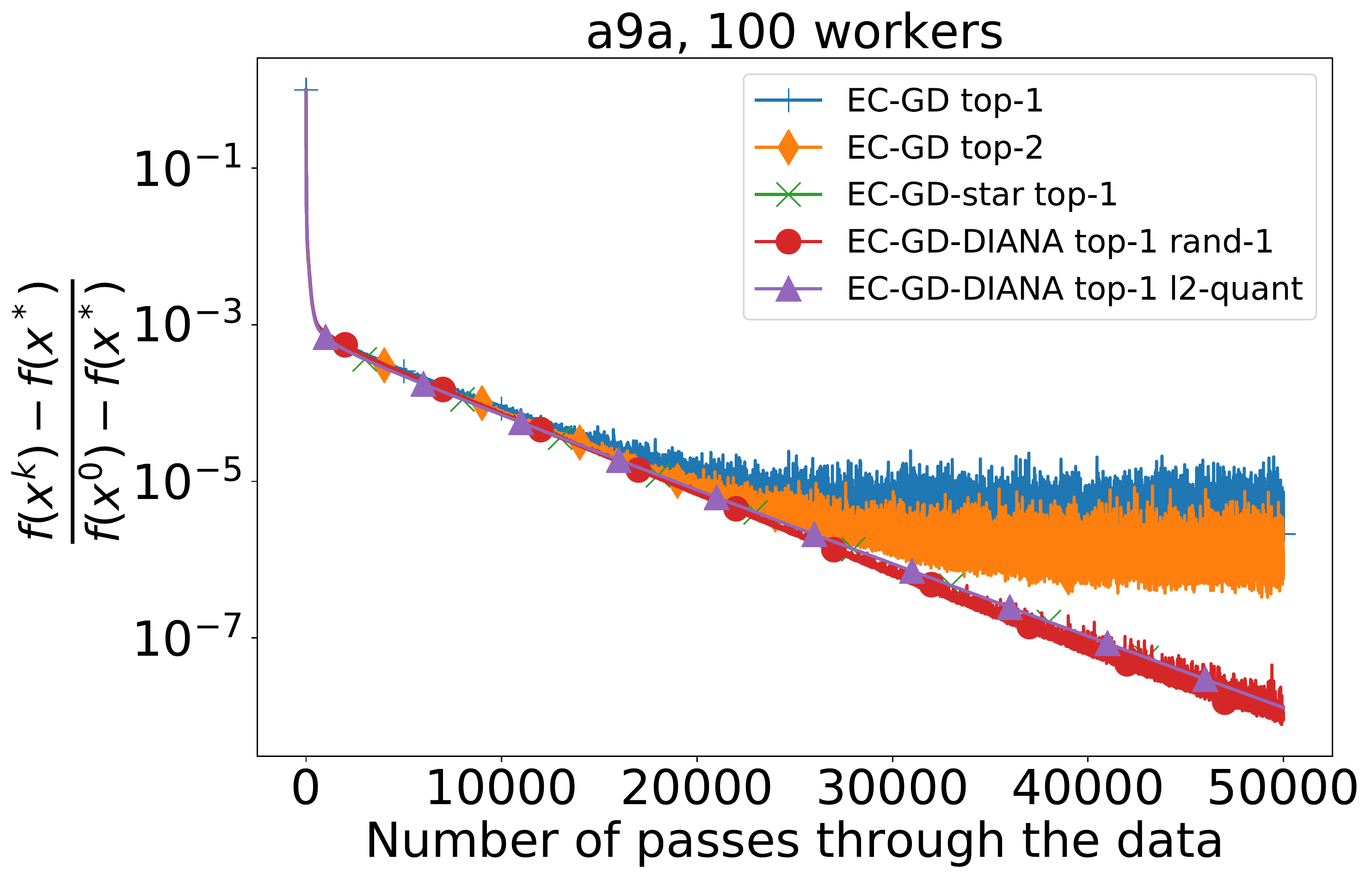}    
	\includegraphics[width=0.32\textwidth]{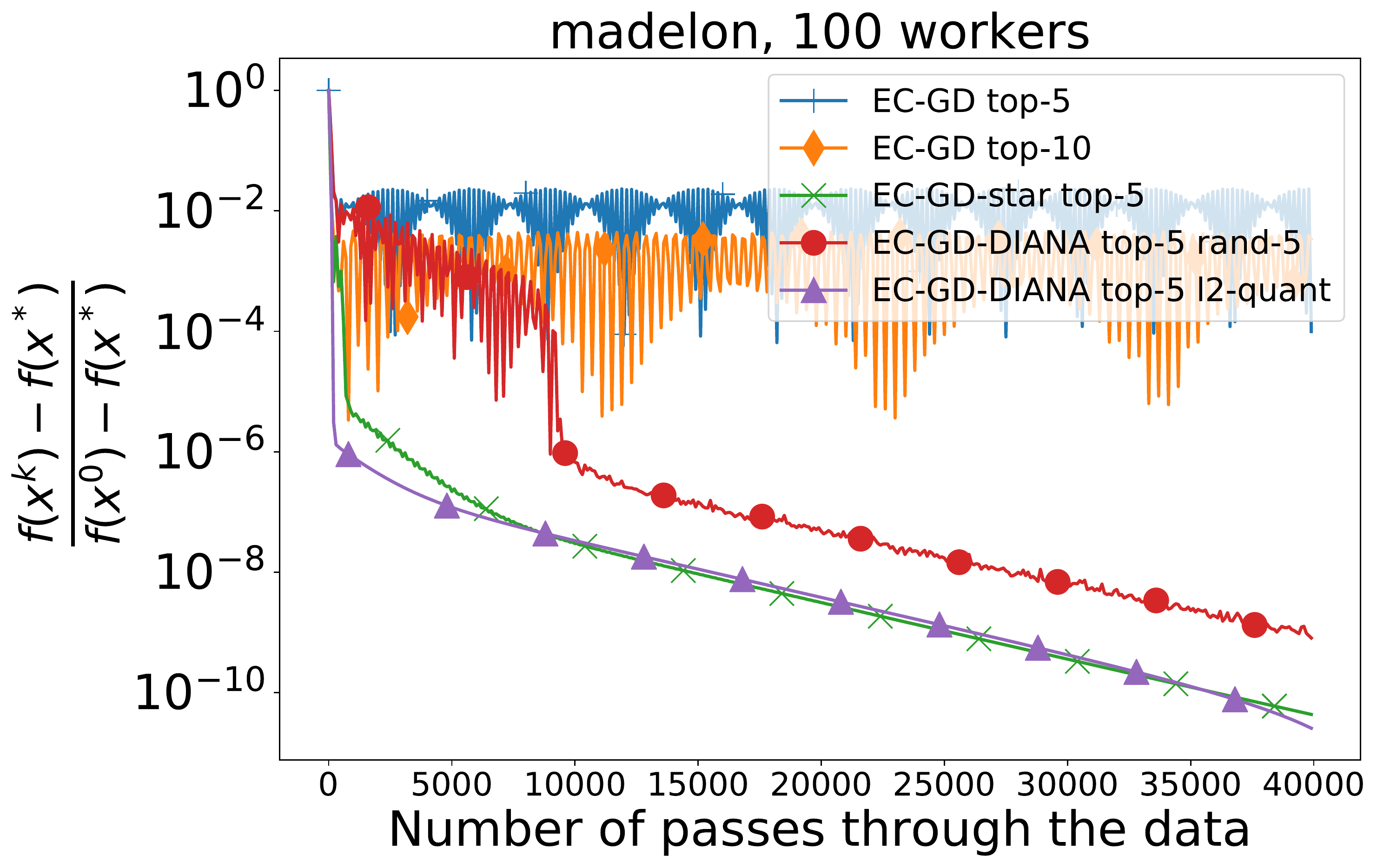}    
	\includegraphics[width=0.32\textwidth]{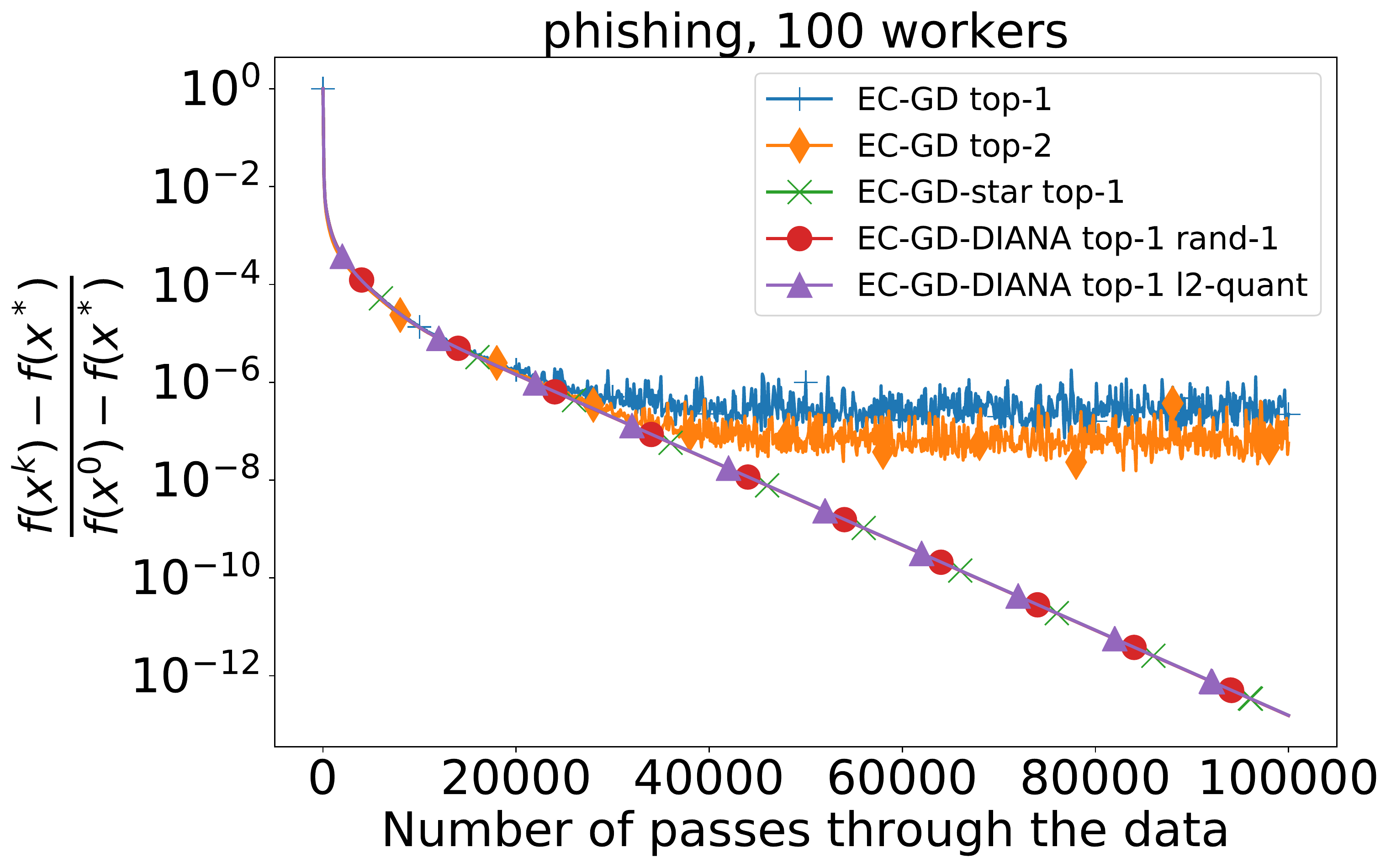}    	
	\\
	\includegraphics[width=0.32\textwidth]{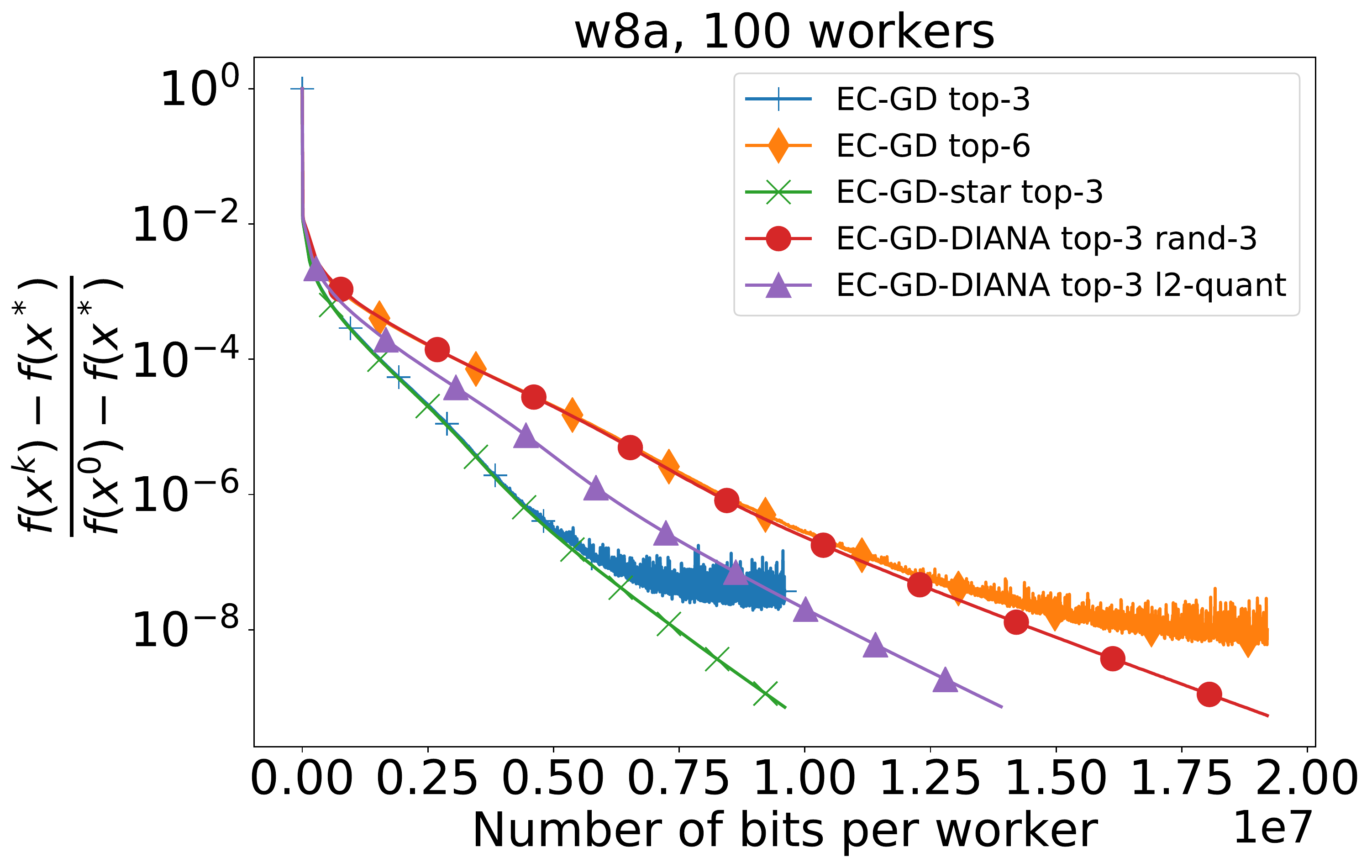}	\includegraphics[width=0.32\textwidth]{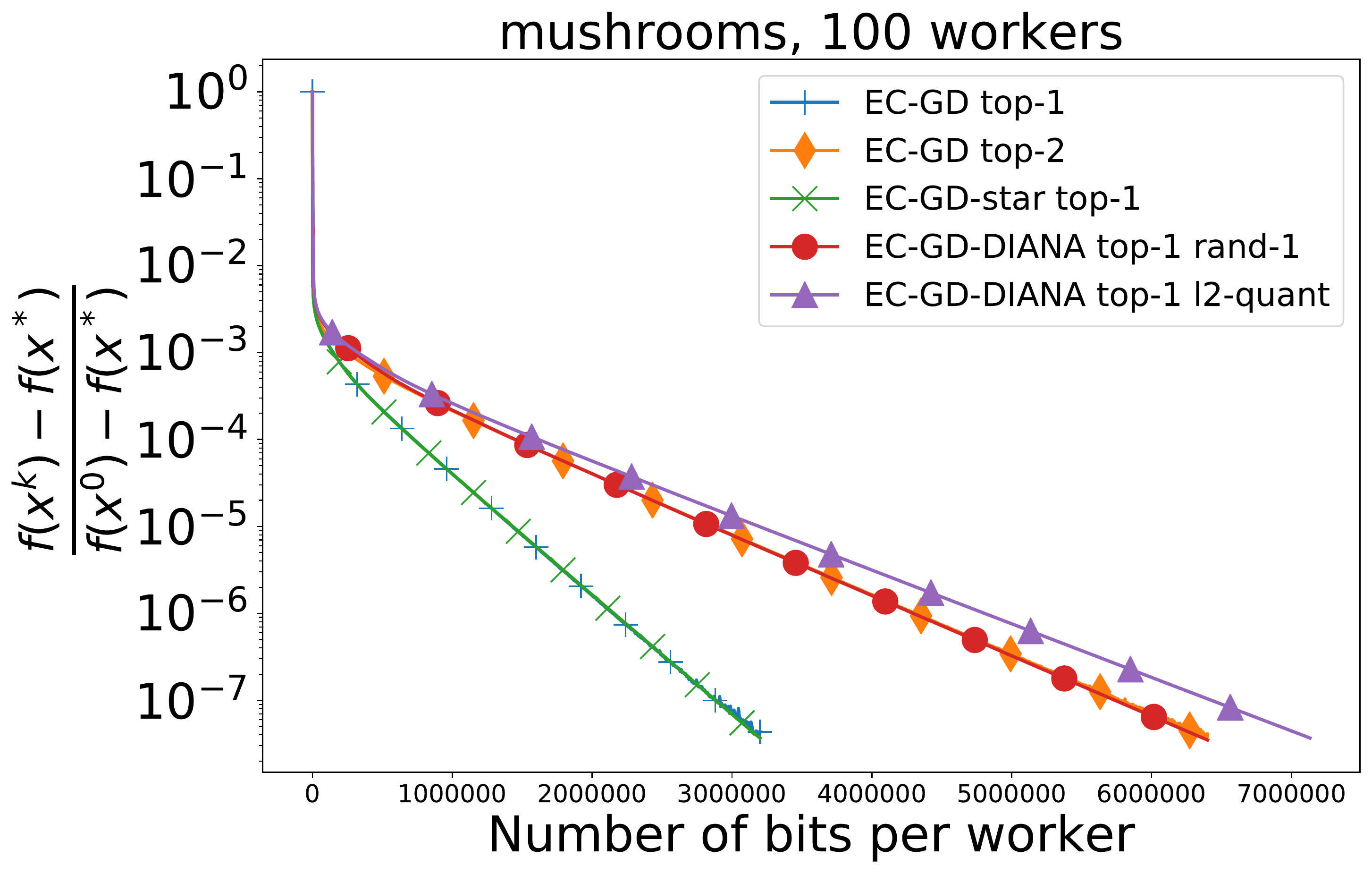}
	\includegraphics[width=0.32\textwidth]{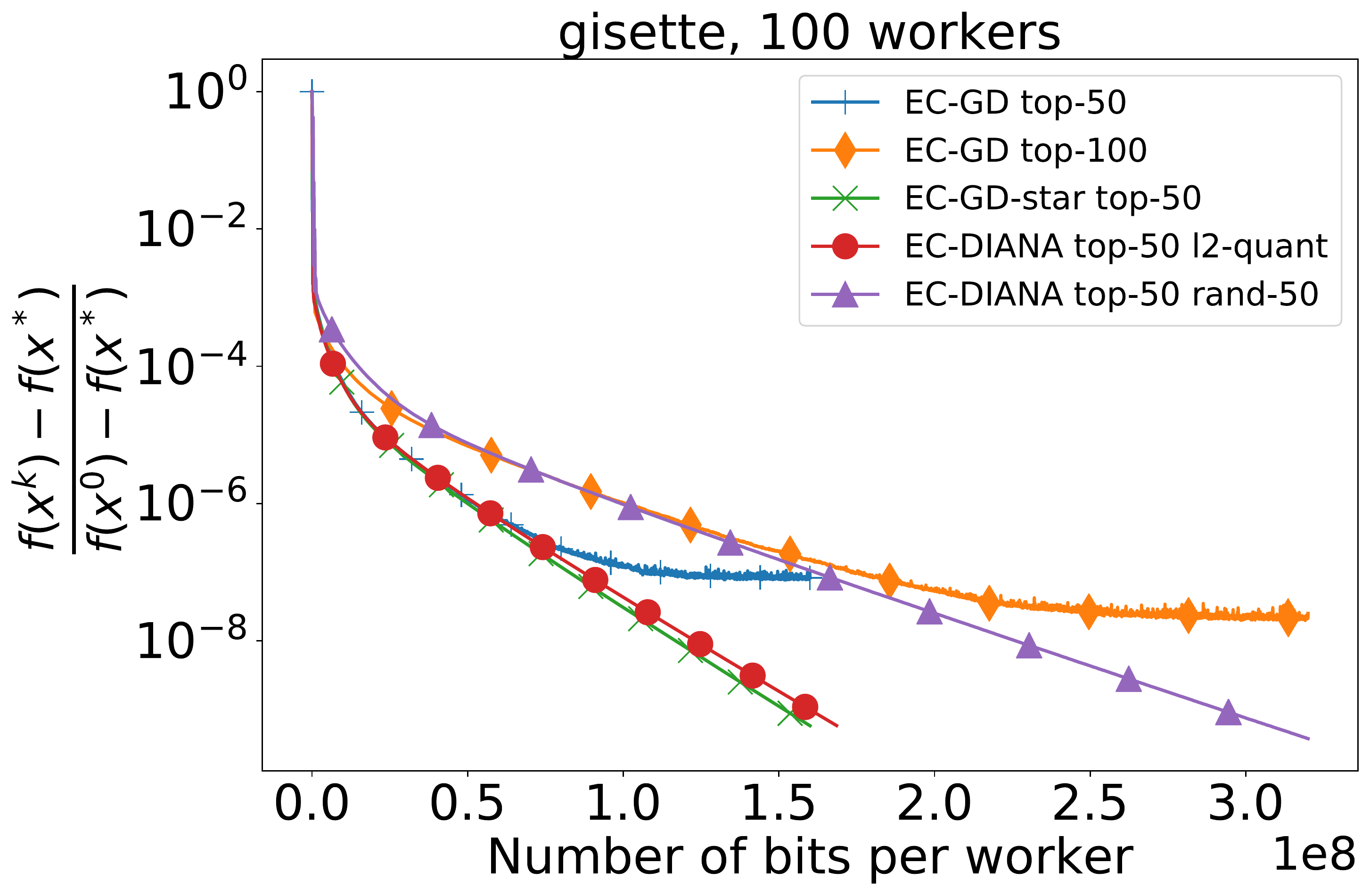}
	\\
	\includegraphics[width=0.32\textwidth]{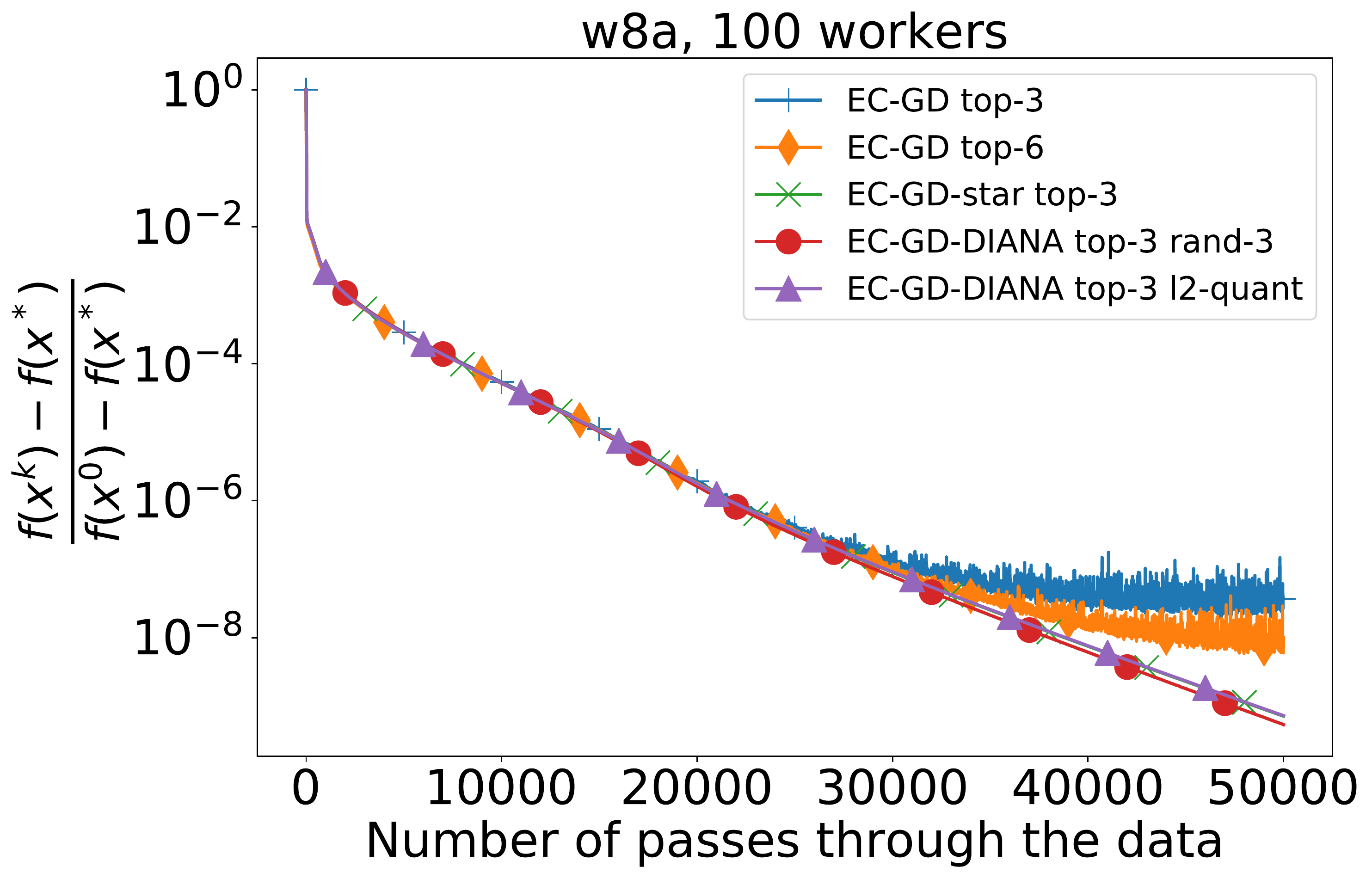}
	\includegraphics[width=0.32\textwidth]{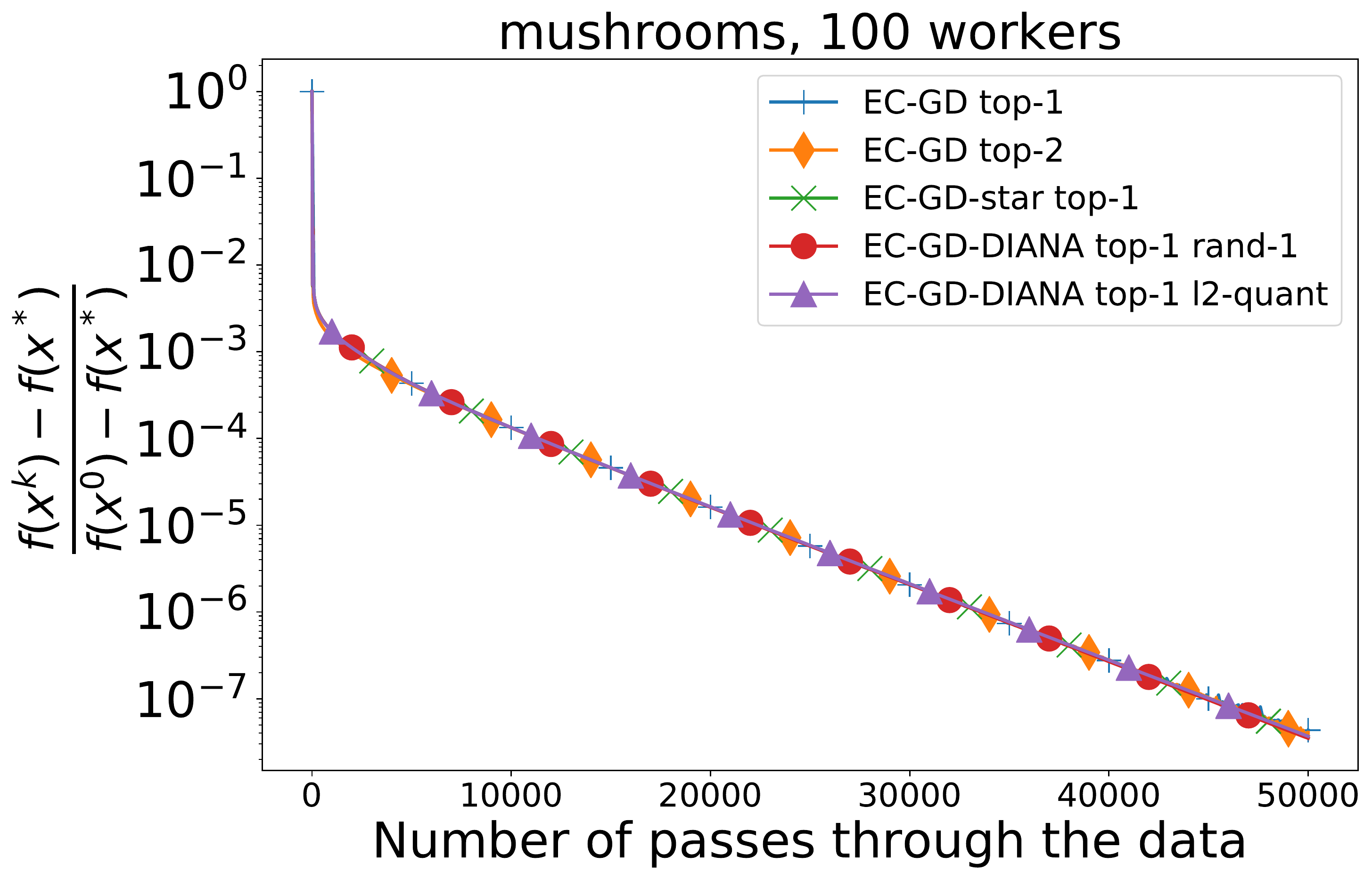}
	\includegraphics[width=0.32\textwidth]{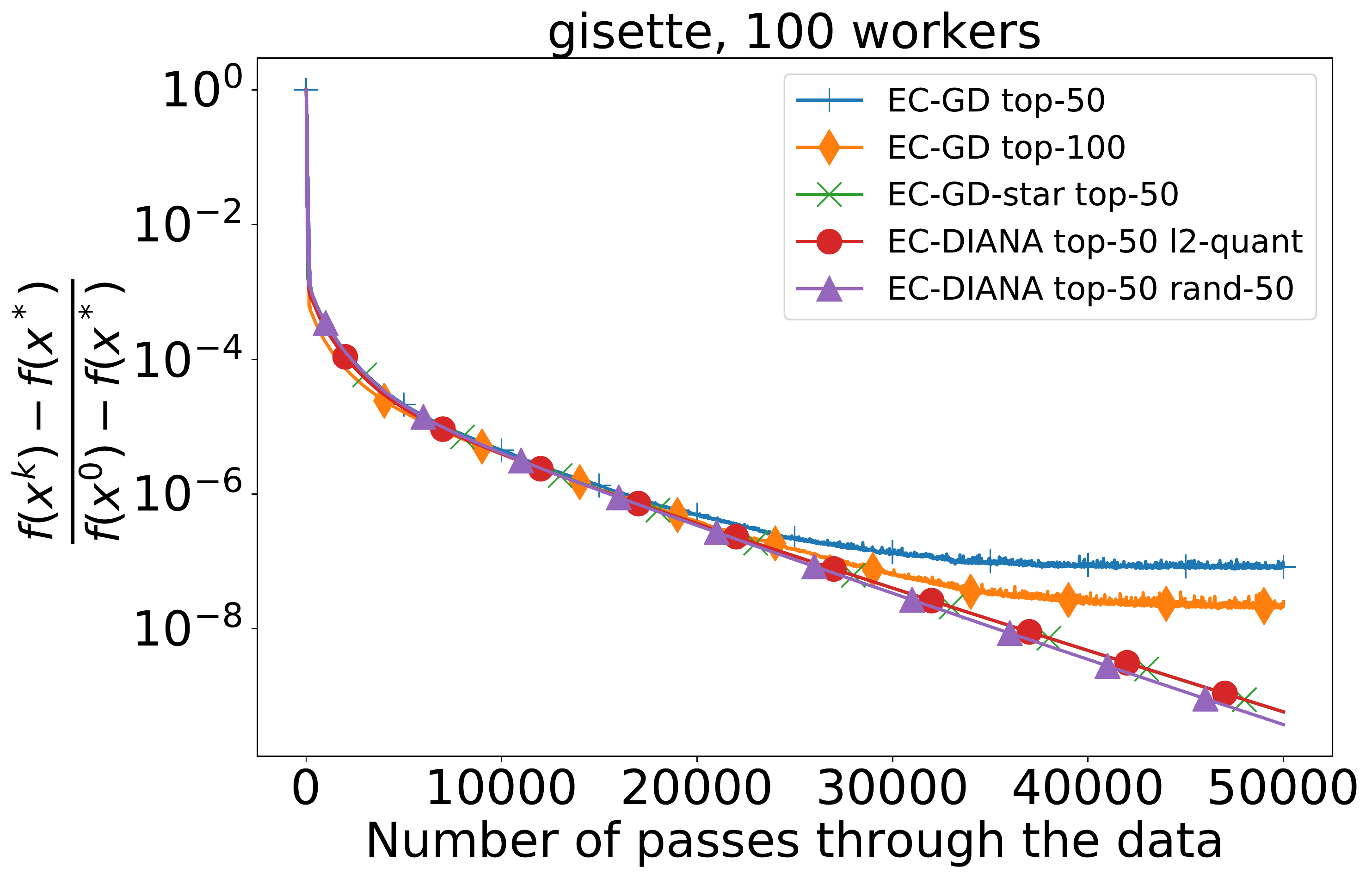}	    
	    \caption{Trajectories of {\tt EC-GD}, {\tt EC-GD-star} and {\tt EC-DIANA} applied to solve logistic regression problem with $100$ workers.}
    \label{fig:gd_logreg_100_workers}
\end{figure}

\begin{figure}[h]
    \centering
    \includegraphics[width=0.32\textwidth]{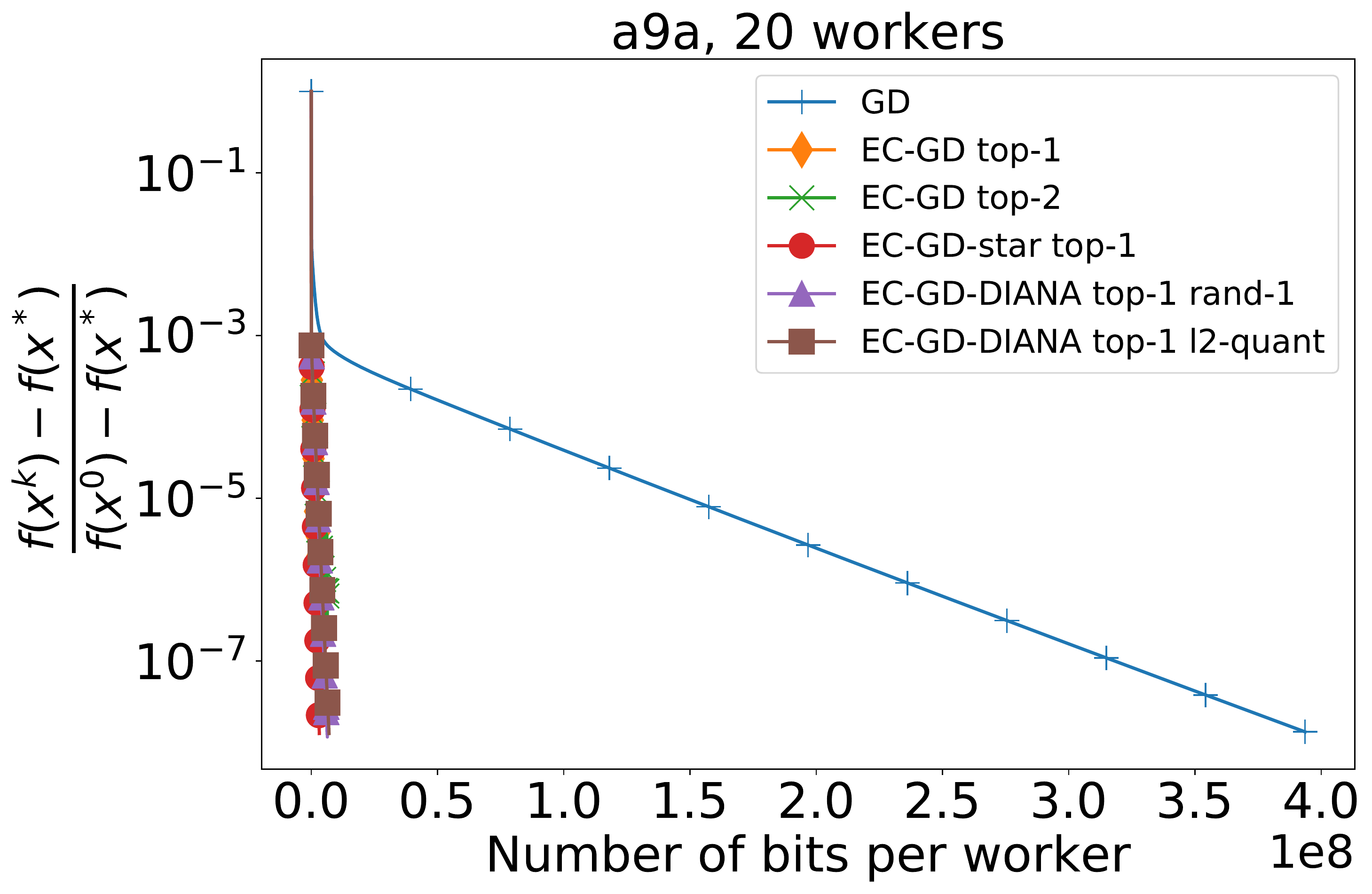}
	\includegraphics[width=0.32\textwidth]{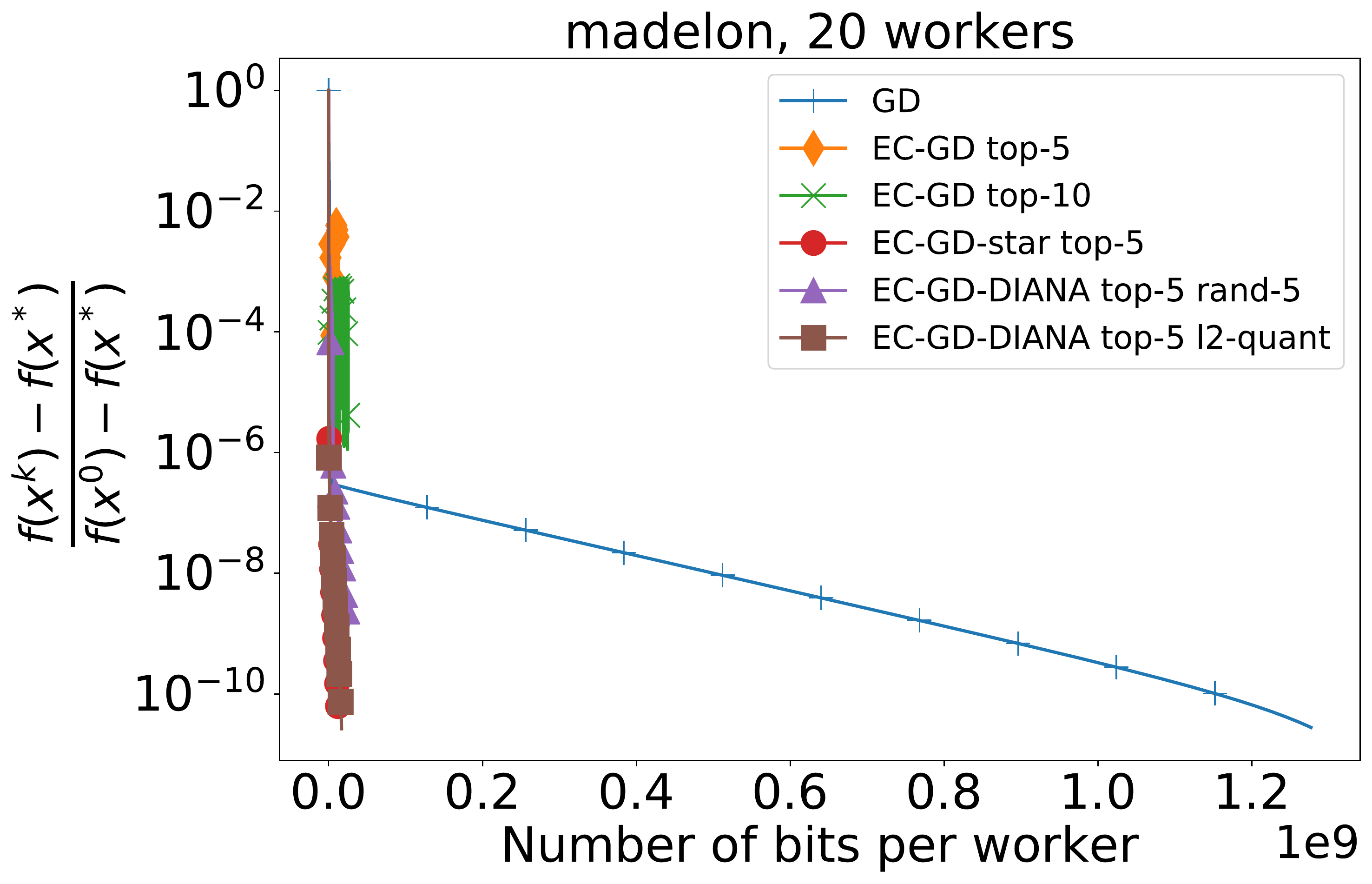}    
	\includegraphics[width=0.32\textwidth]{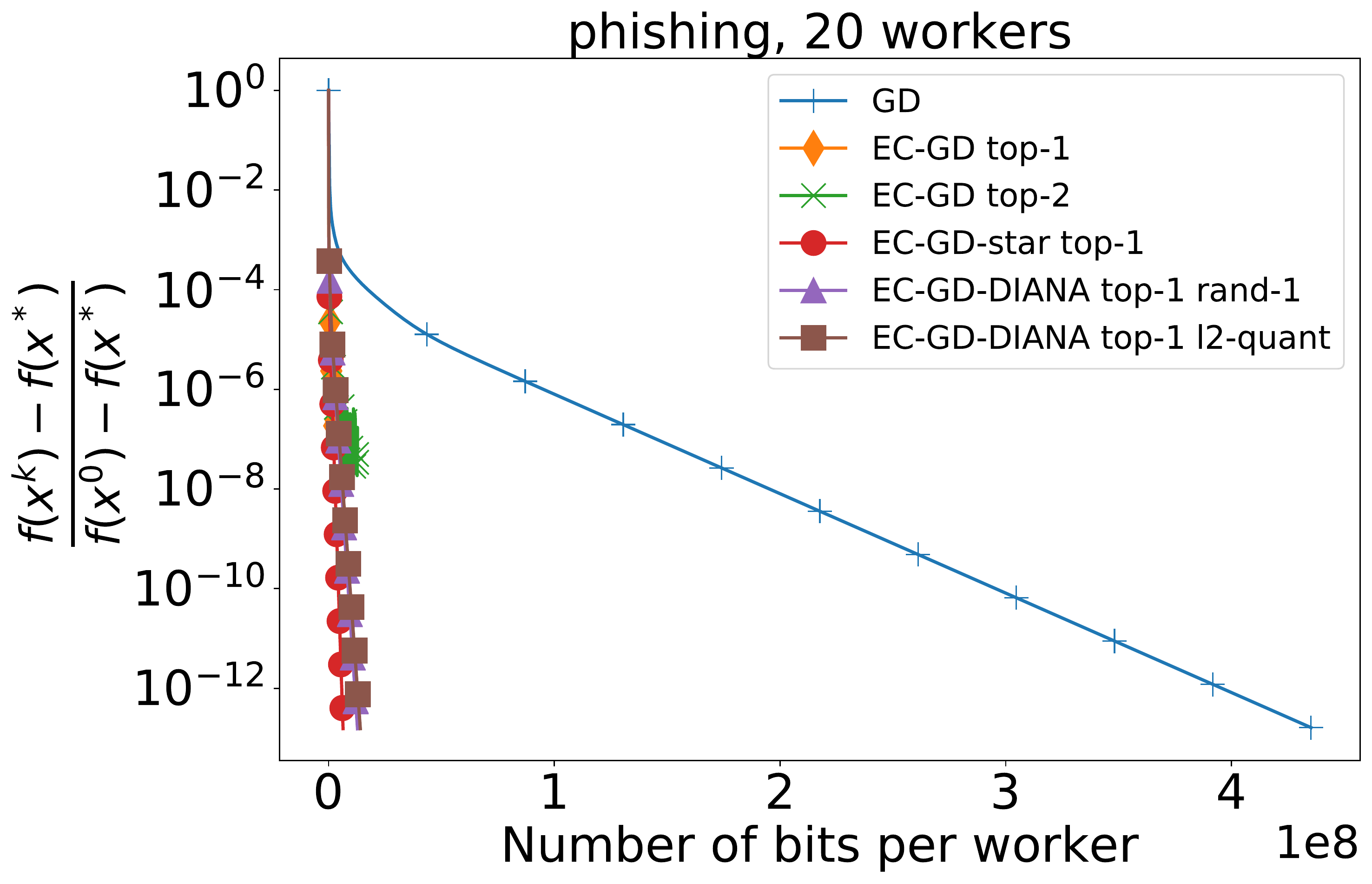}    
    \\
    \includegraphics[width=0.32\textwidth]{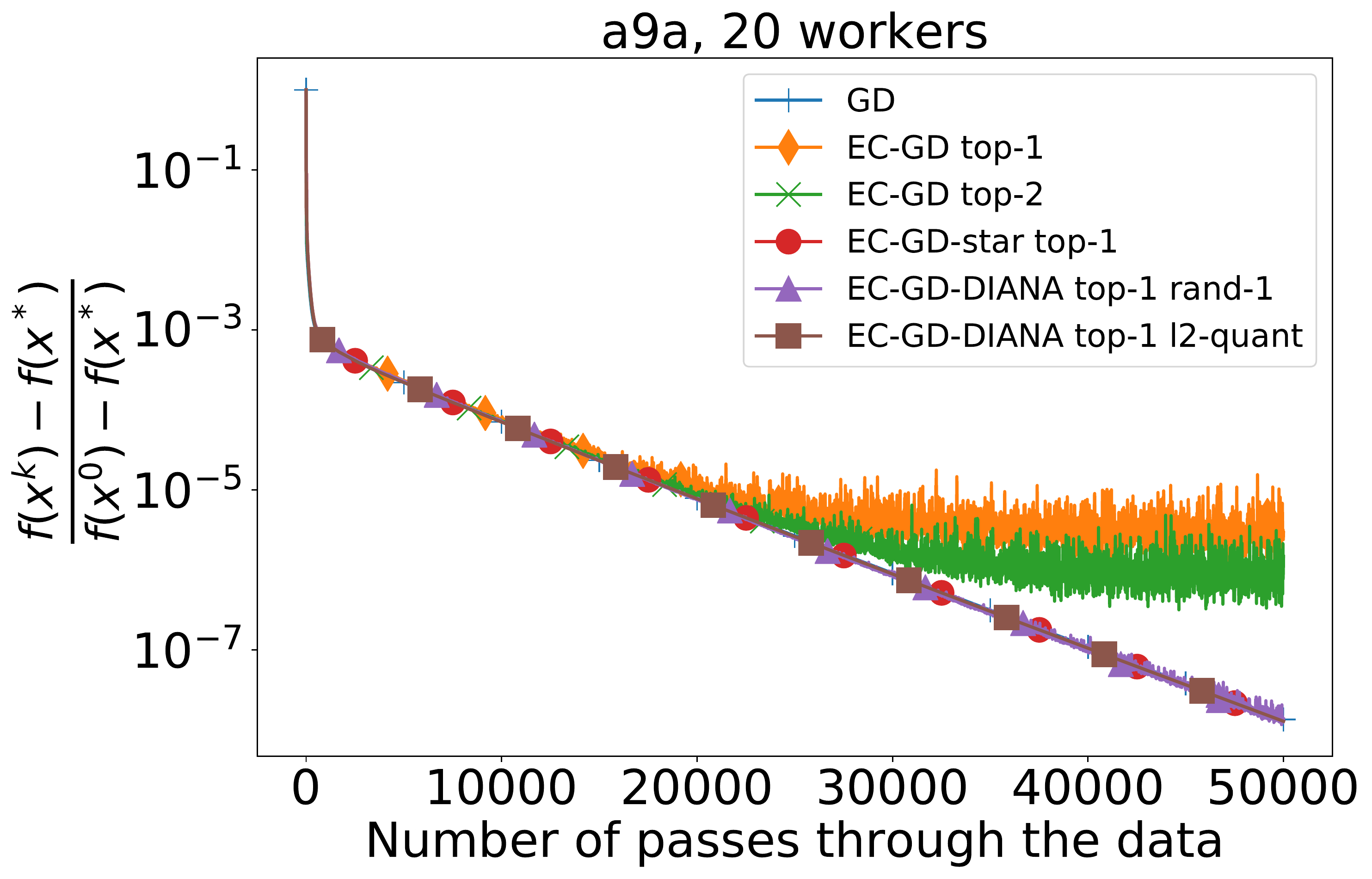}    
	\includegraphics[width=0.32\textwidth]{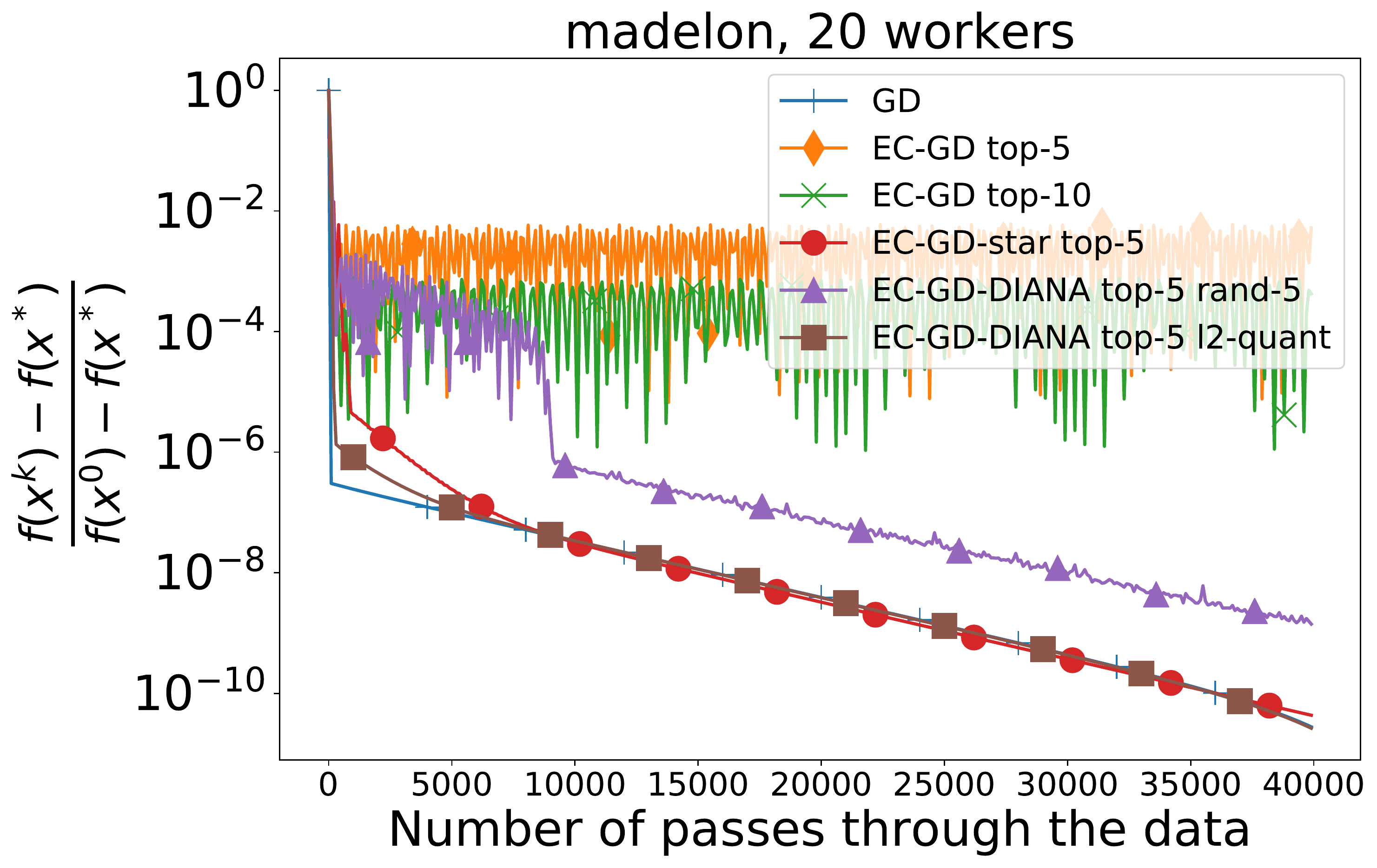}    
	\includegraphics[width=0.32\textwidth]{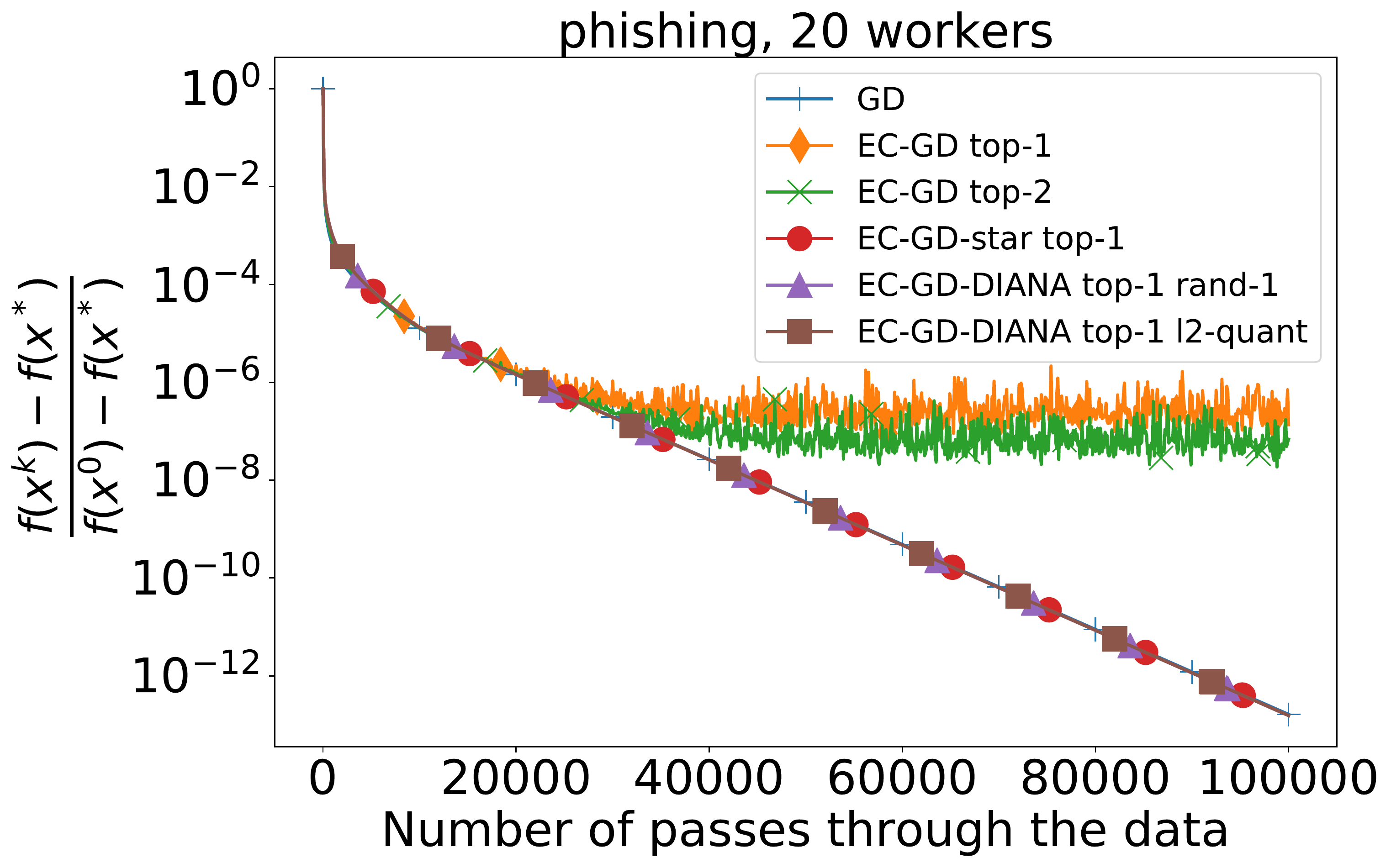}    	
	\\
	\includegraphics[width=0.32\textwidth]{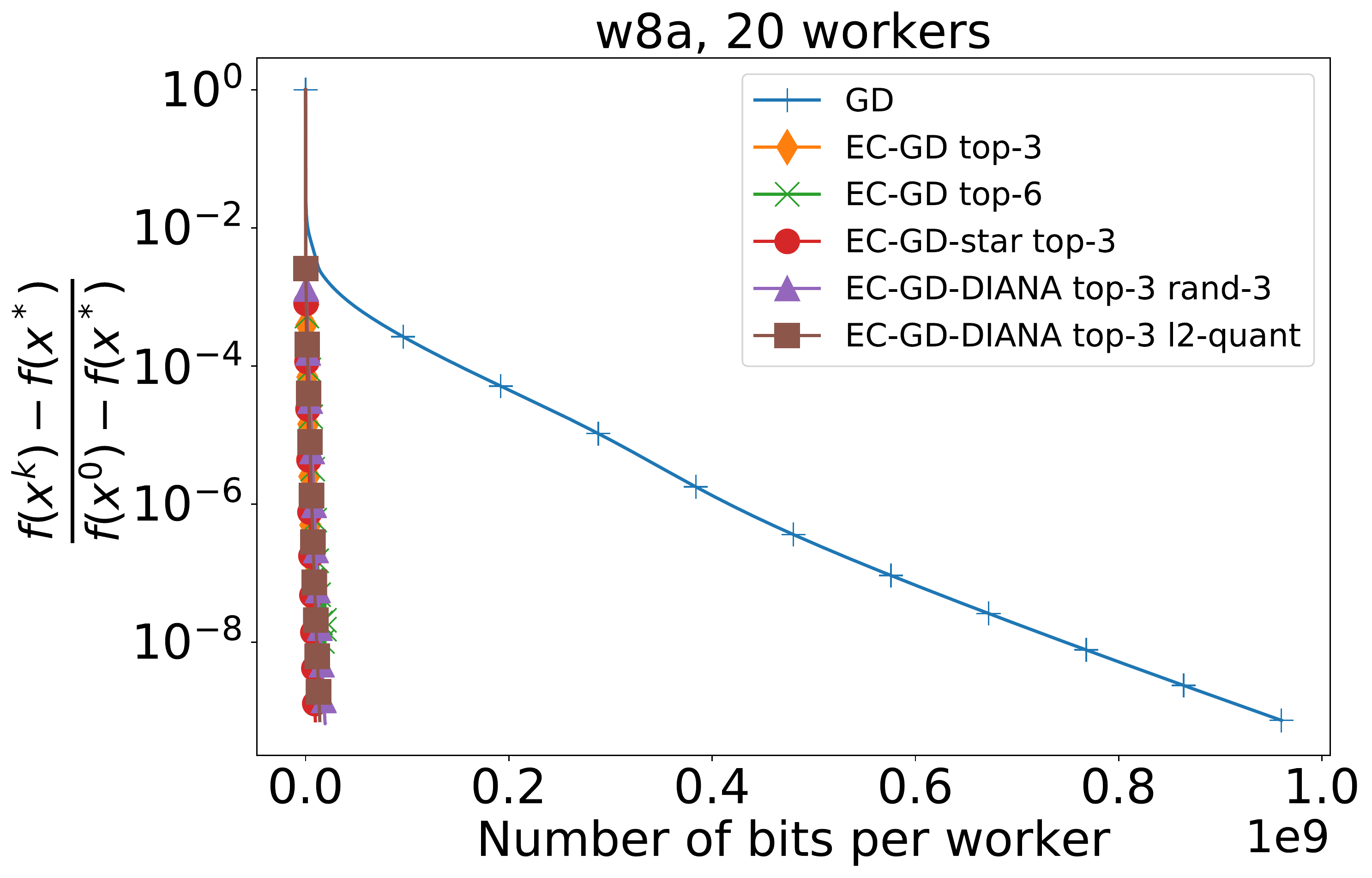}	\includegraphics[width=0.32\textwidth]{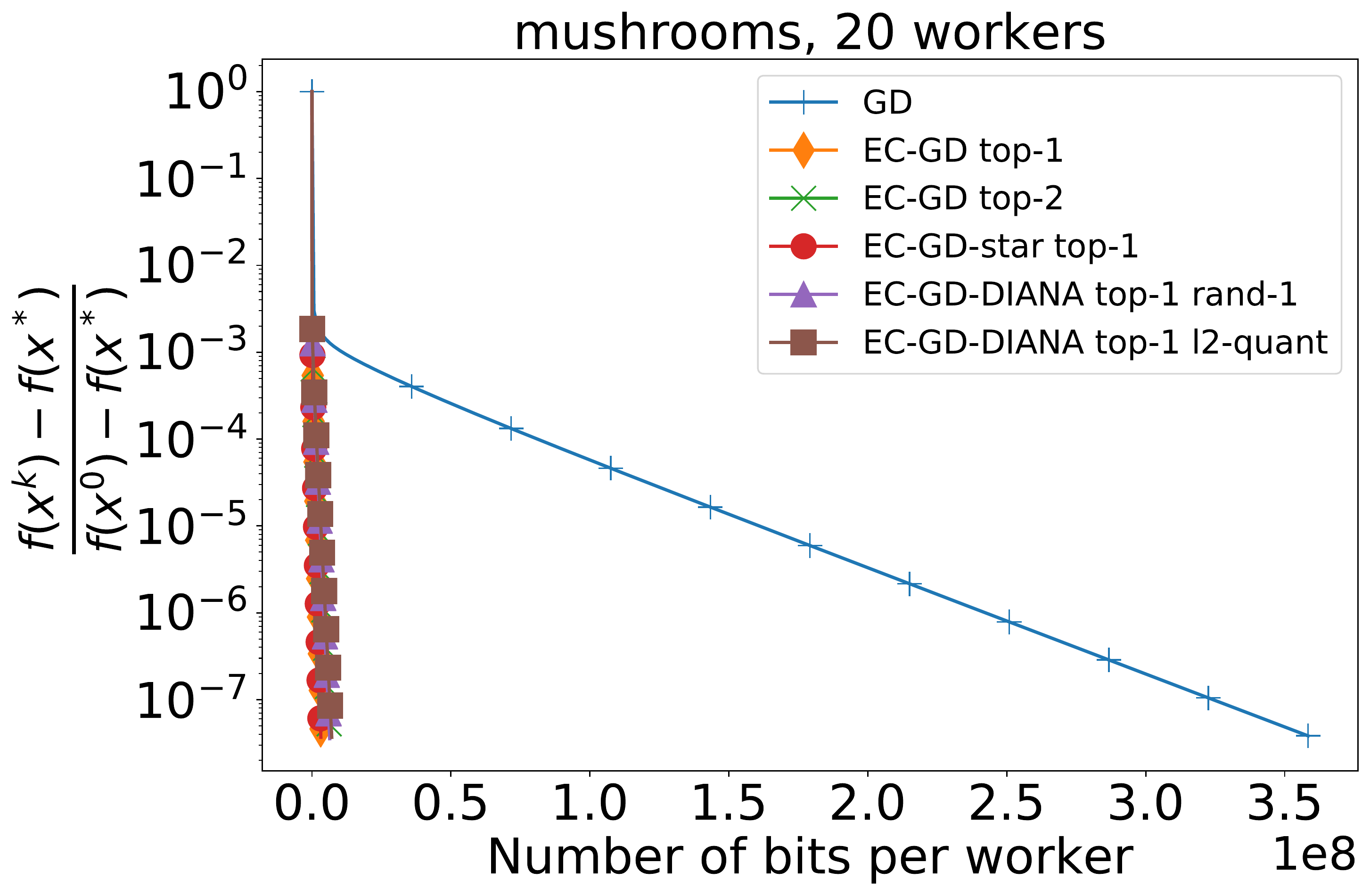}
	\includegraphics[width=0.32\textwidth]{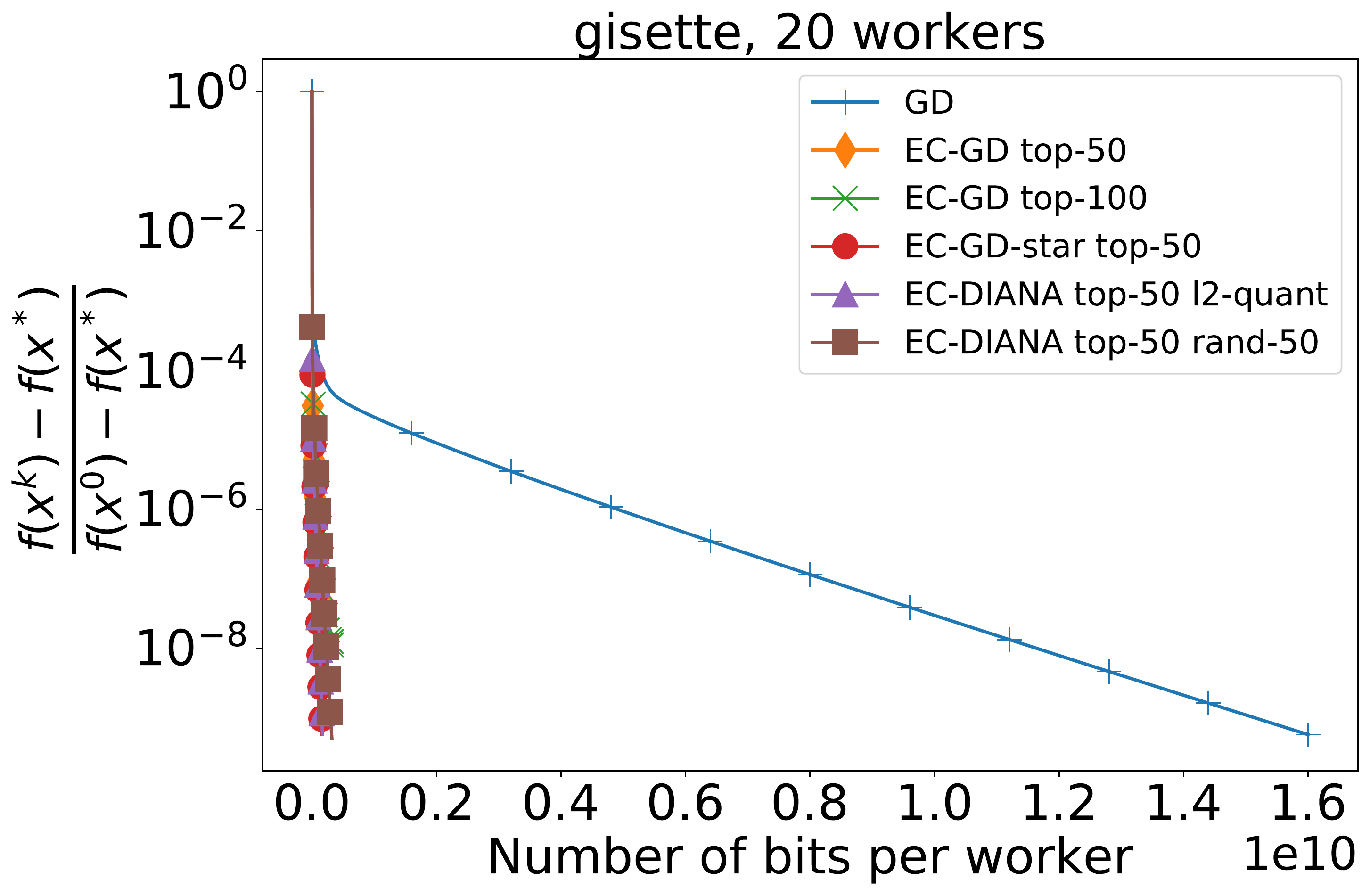}
	\\
	\includegraphics[width=0.32\textwidth]{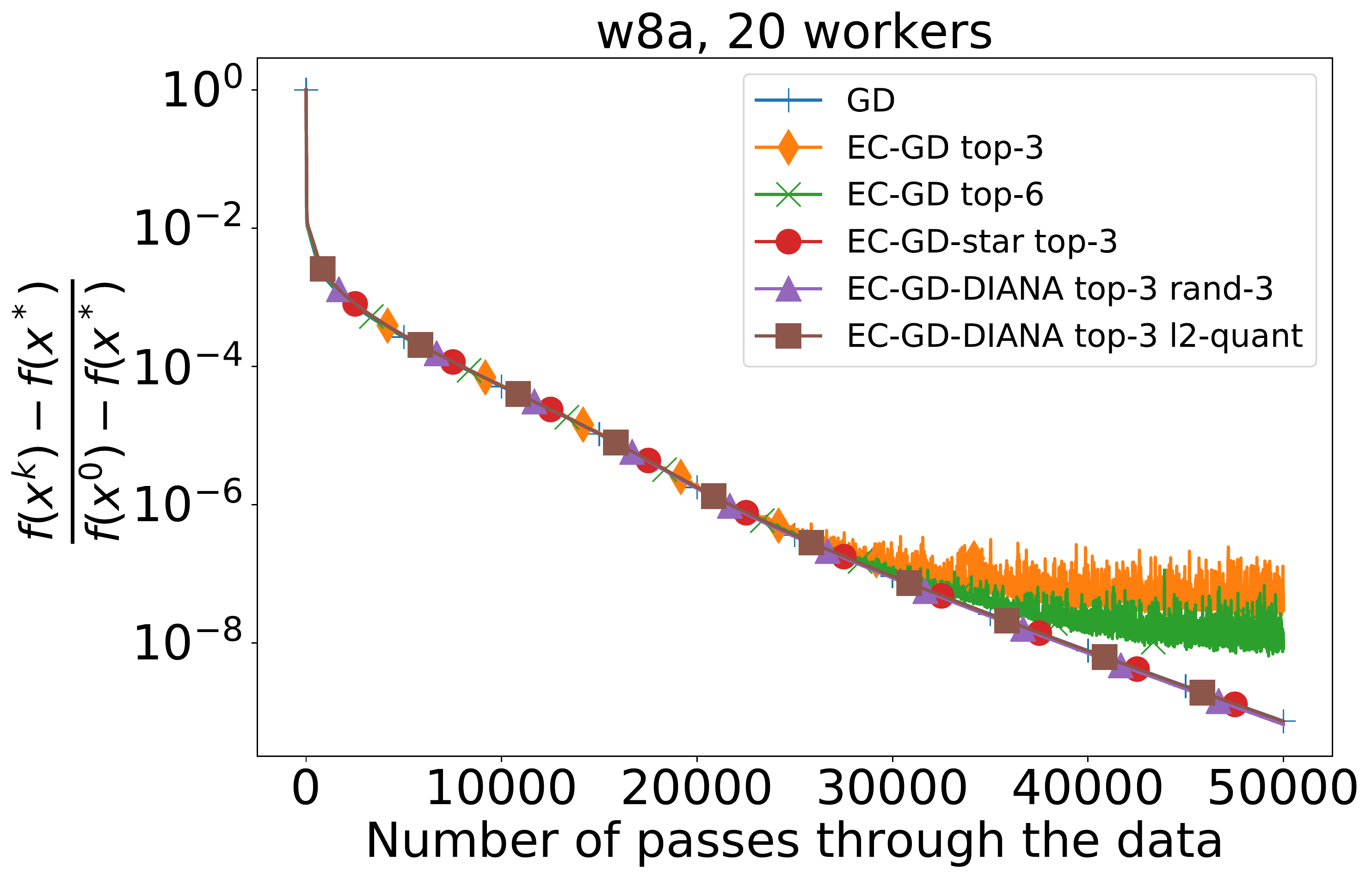}
	\includegraphics[width=0.32\textwidth]{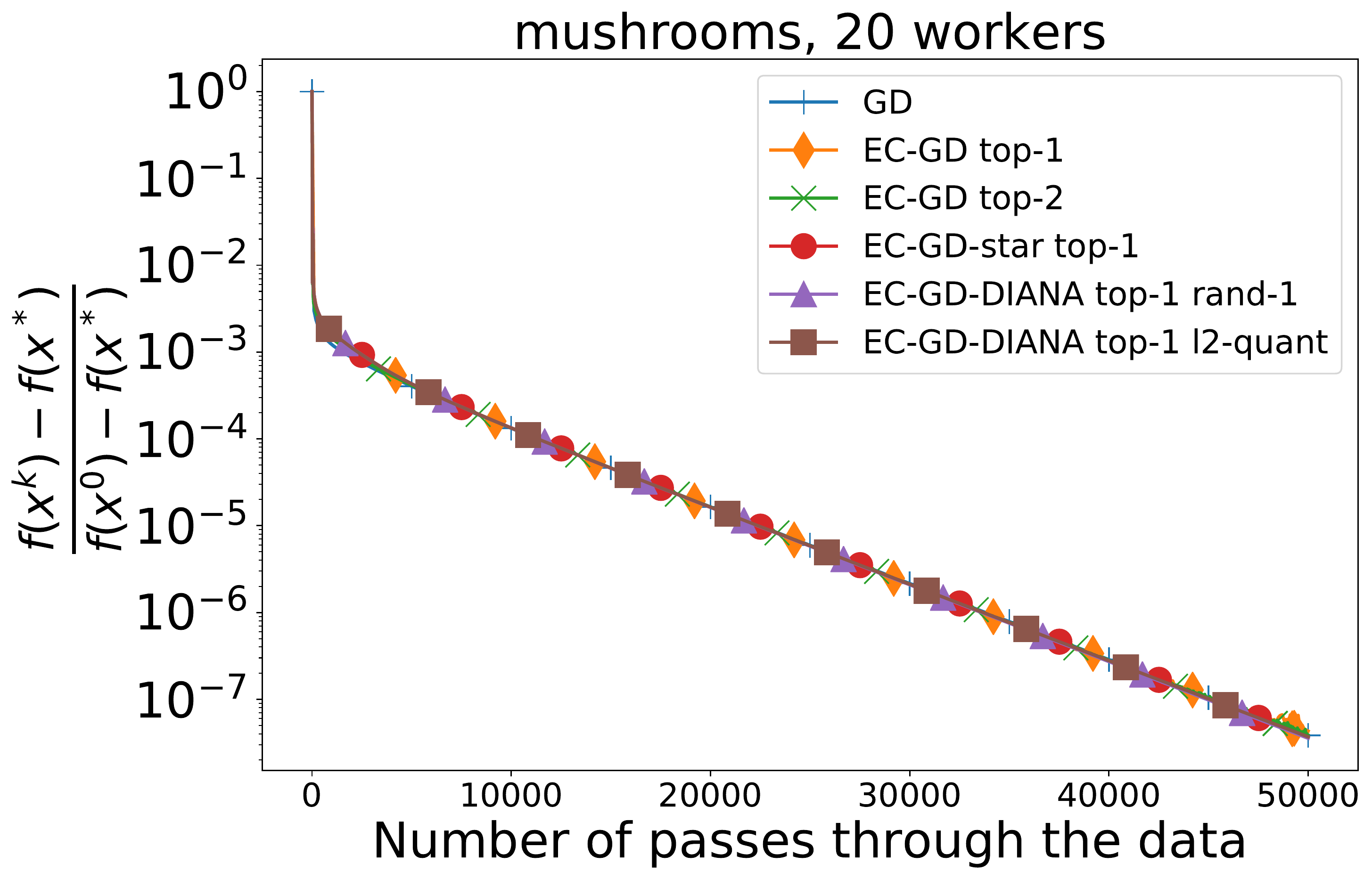}	\includegraphics[width=0.32\textwidth]{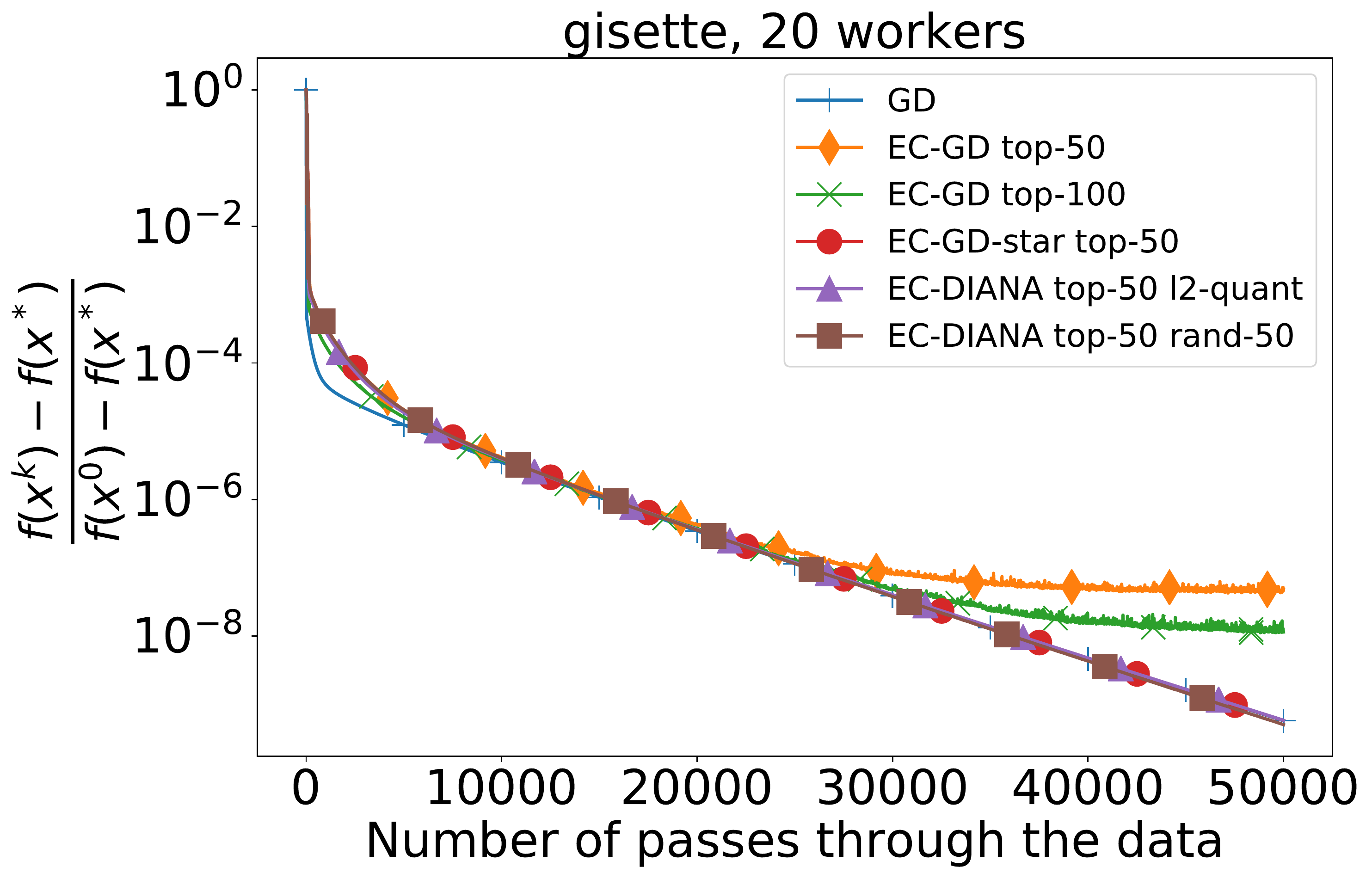}
	\caption{Trajectories of {\tt EC-GD}, {\tt EC-GD-star}, {\tt EC-DIANA} and {\tt GD} applied to solve logistic regression problem with $20$ workers.}
    \label{fig:gd_logreg_20_workers_id}
\end{figure}

\begin{figure}[h]
    \centering
    \includegraphics[width=0.32\textwidth]{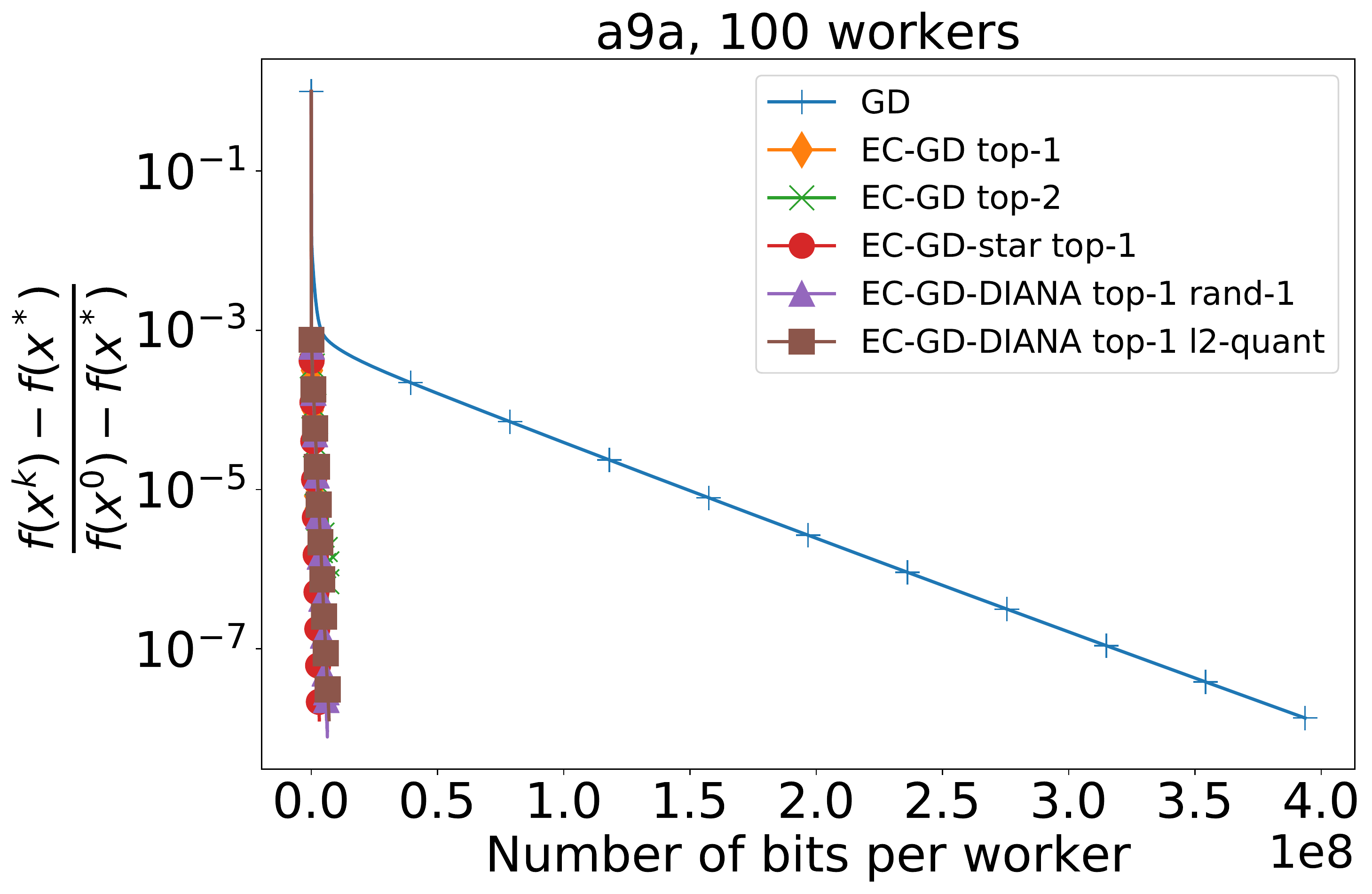}
	\includegraphics[width=0.32\textwidth]{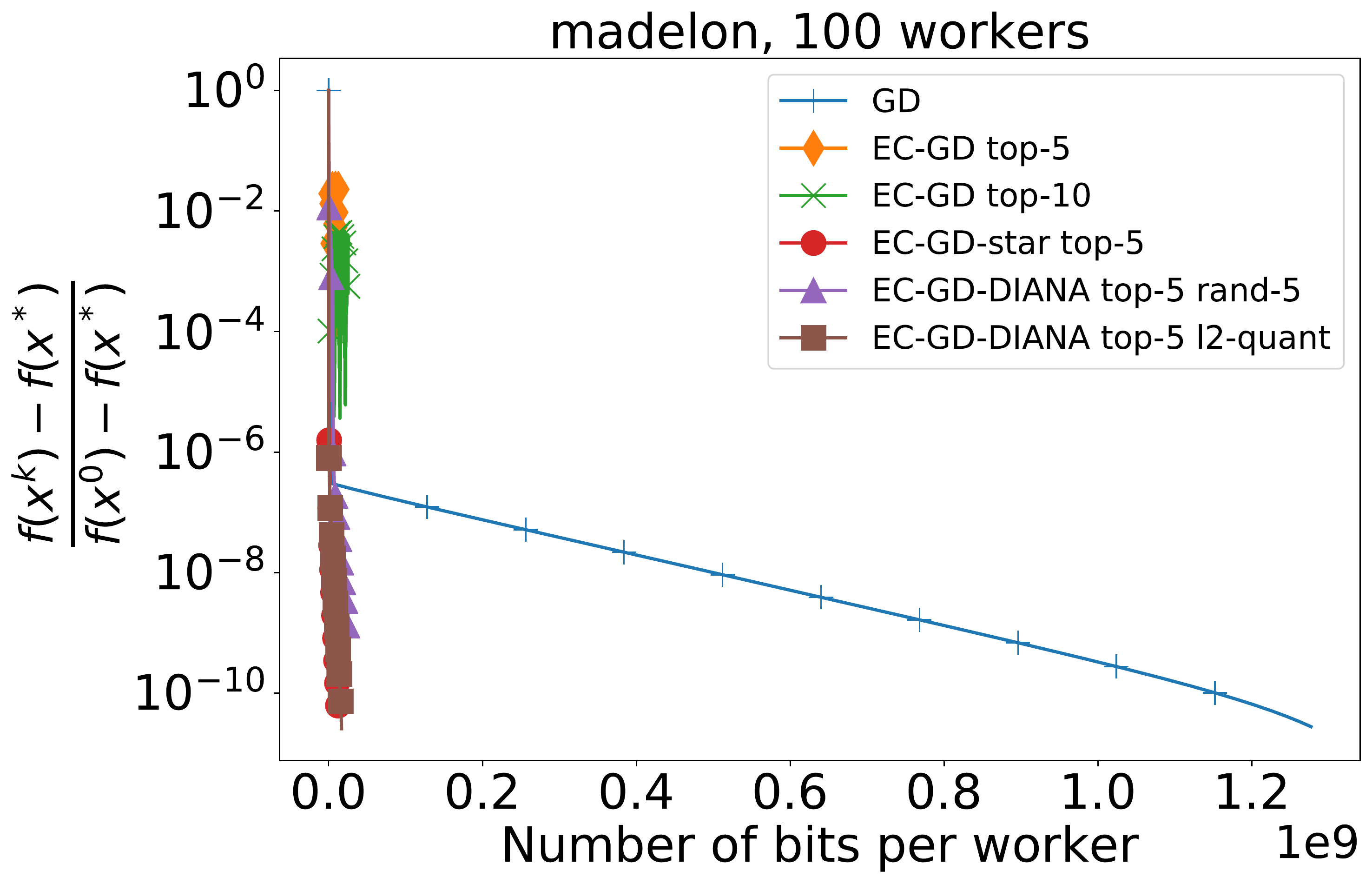}    
	\includegraphics[width=0.32\textwidth]{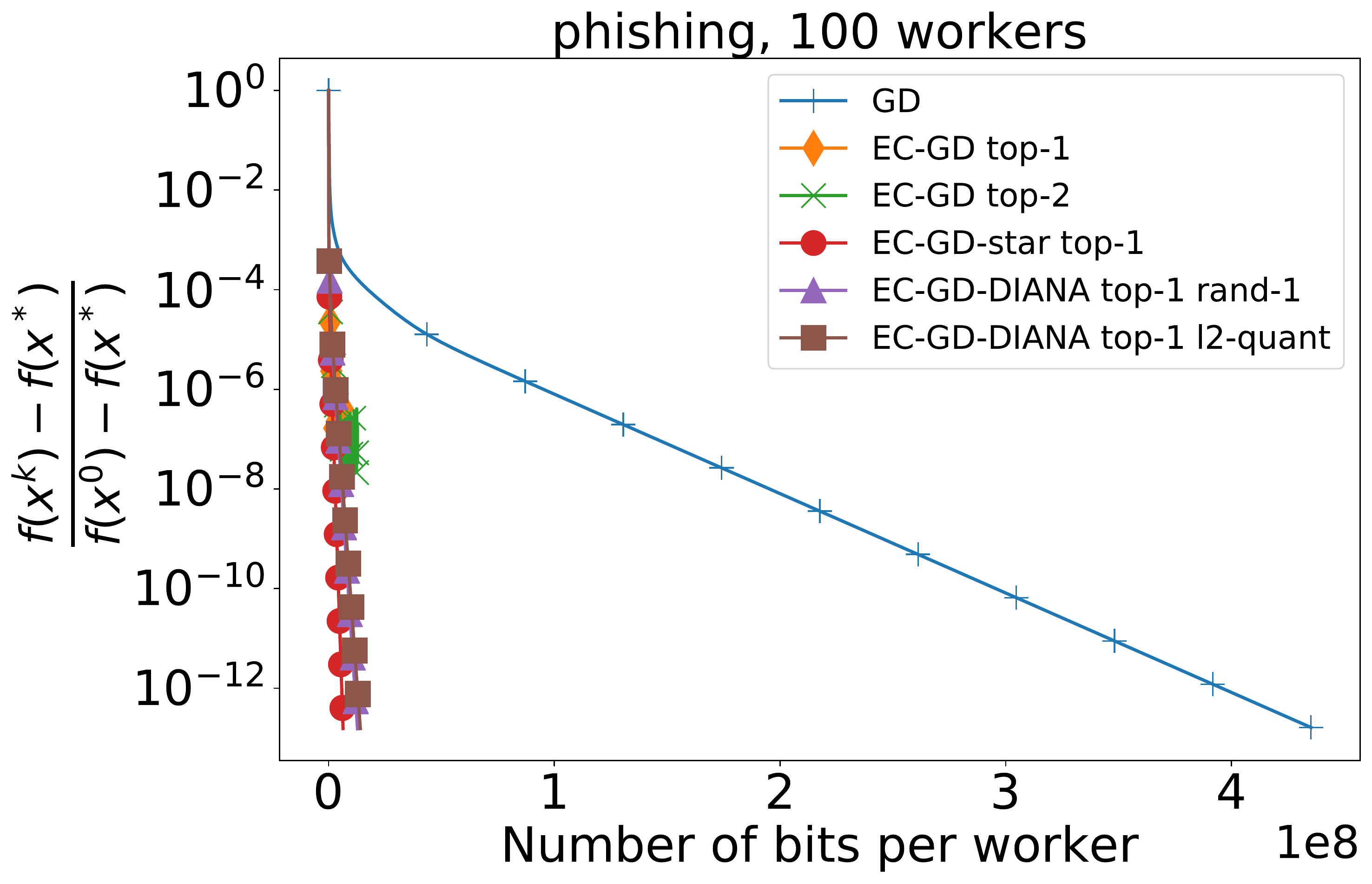}    
    \\
    \includegraphics[width=0.32\textwidth]{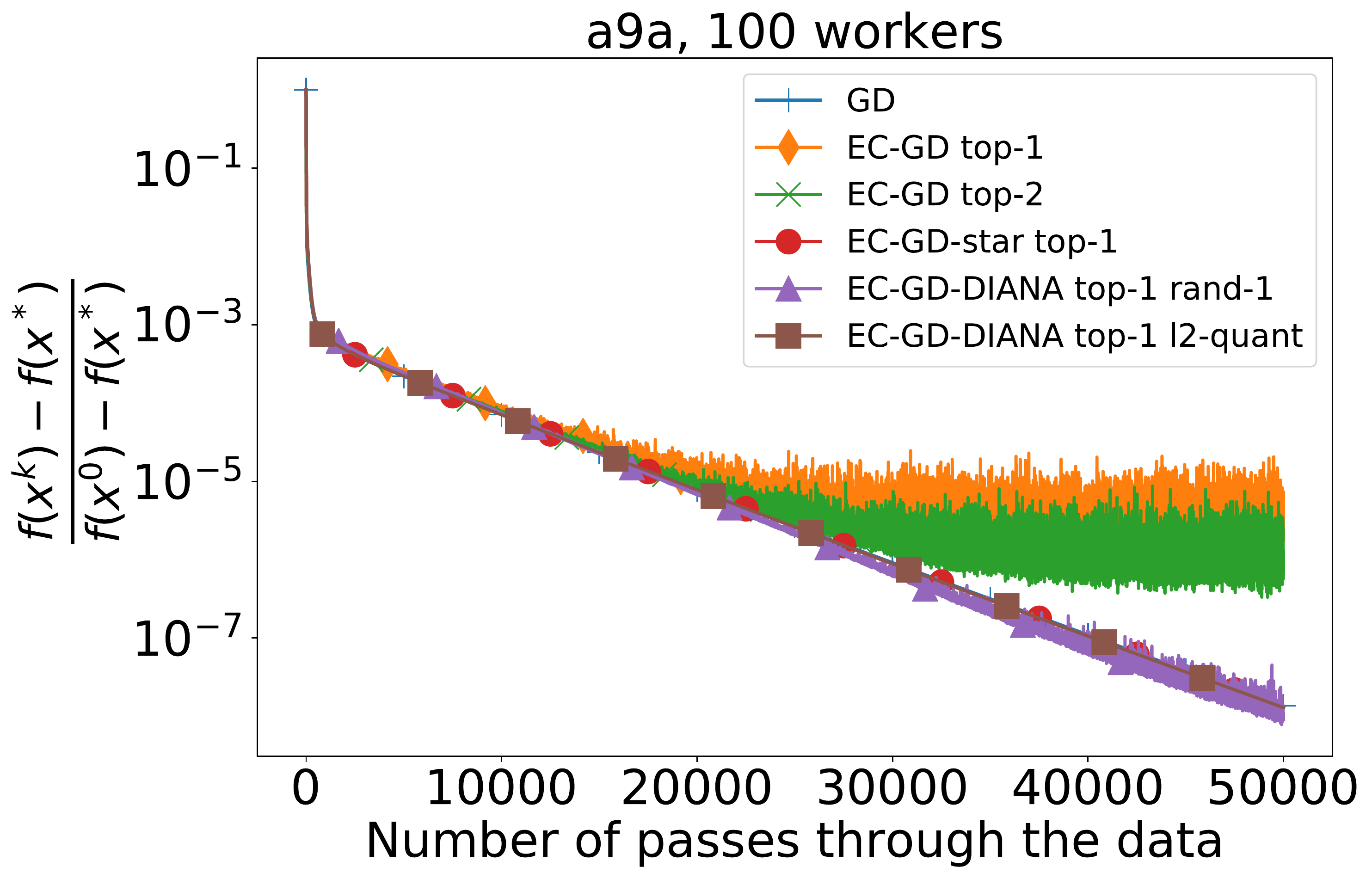}    
	\includegraphics[width=0.32\textwidth]{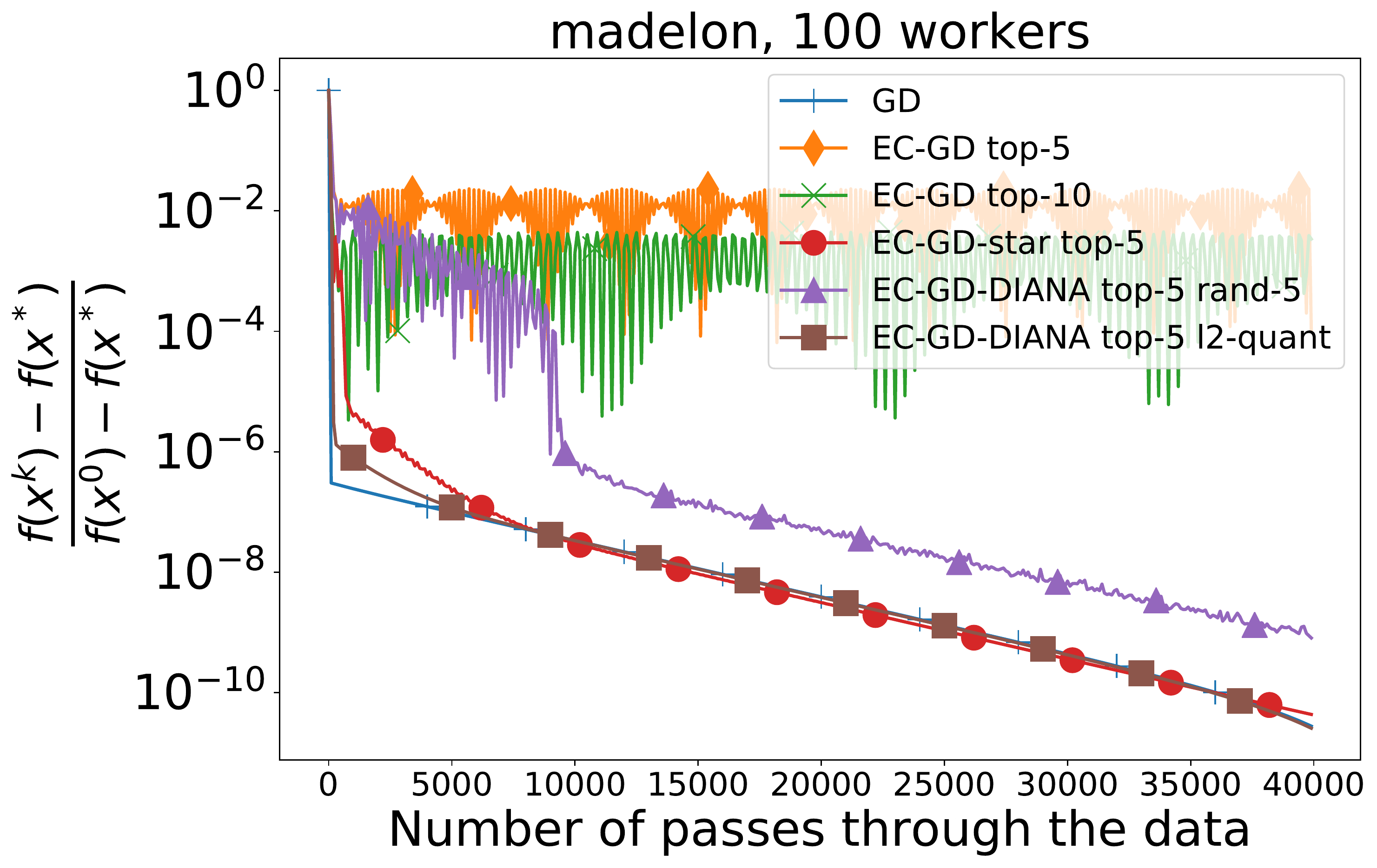}    
	\includegraphics[width=0.32\textwidth]{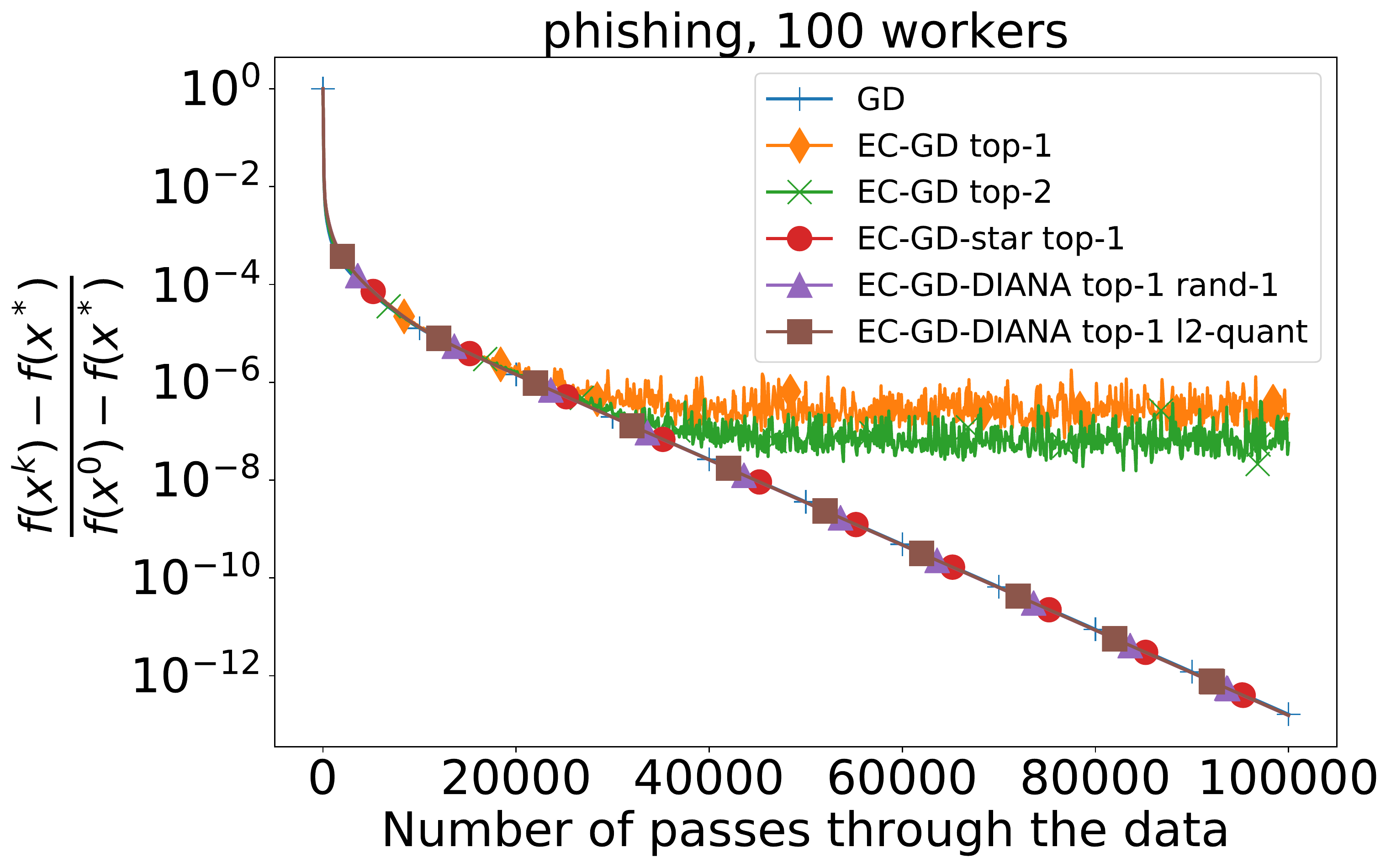}    	
	\\
	\includegraphics[width=0.32\textwidth]{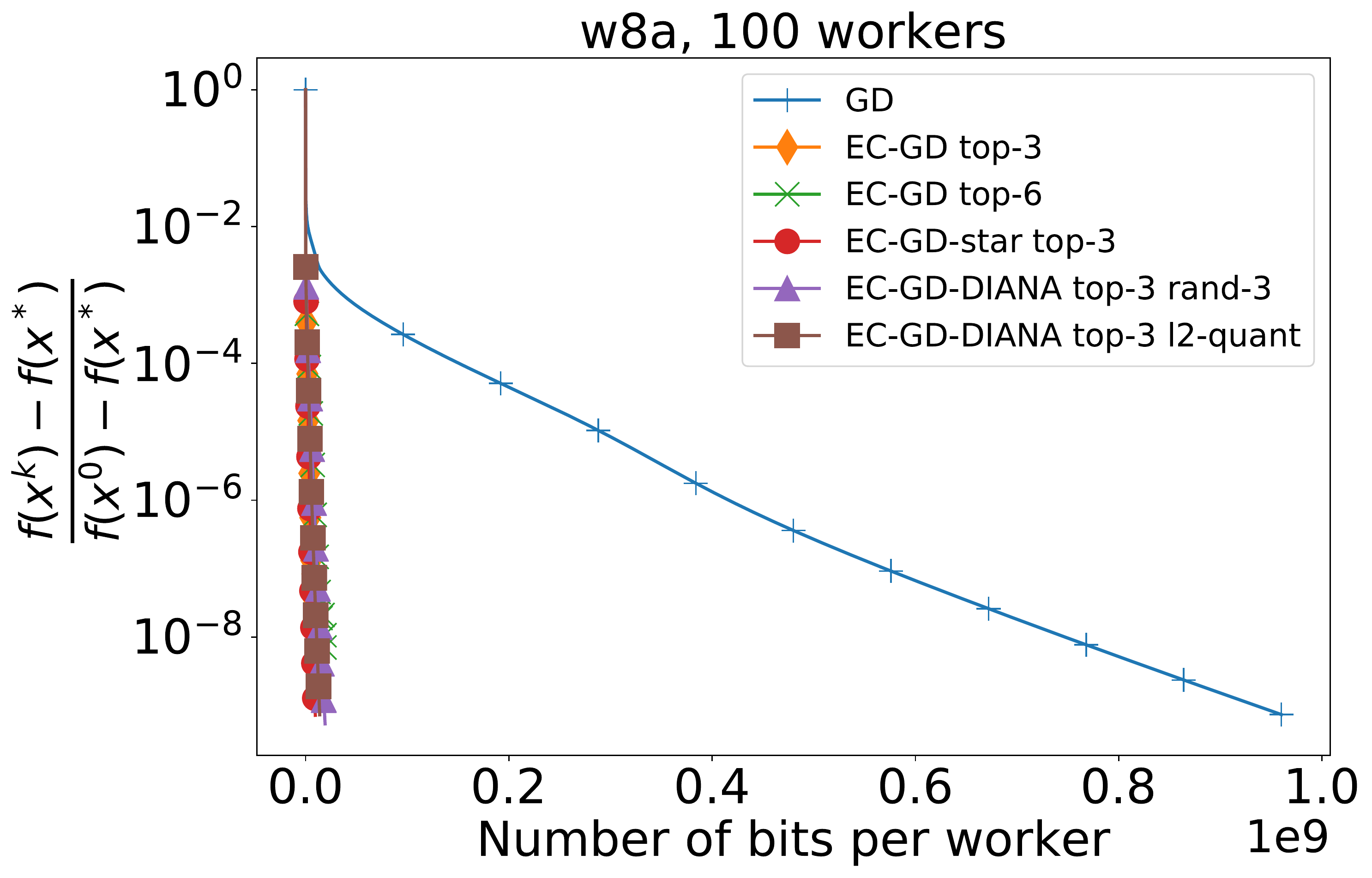}	\includegraphics[width=0.32\textwidth]{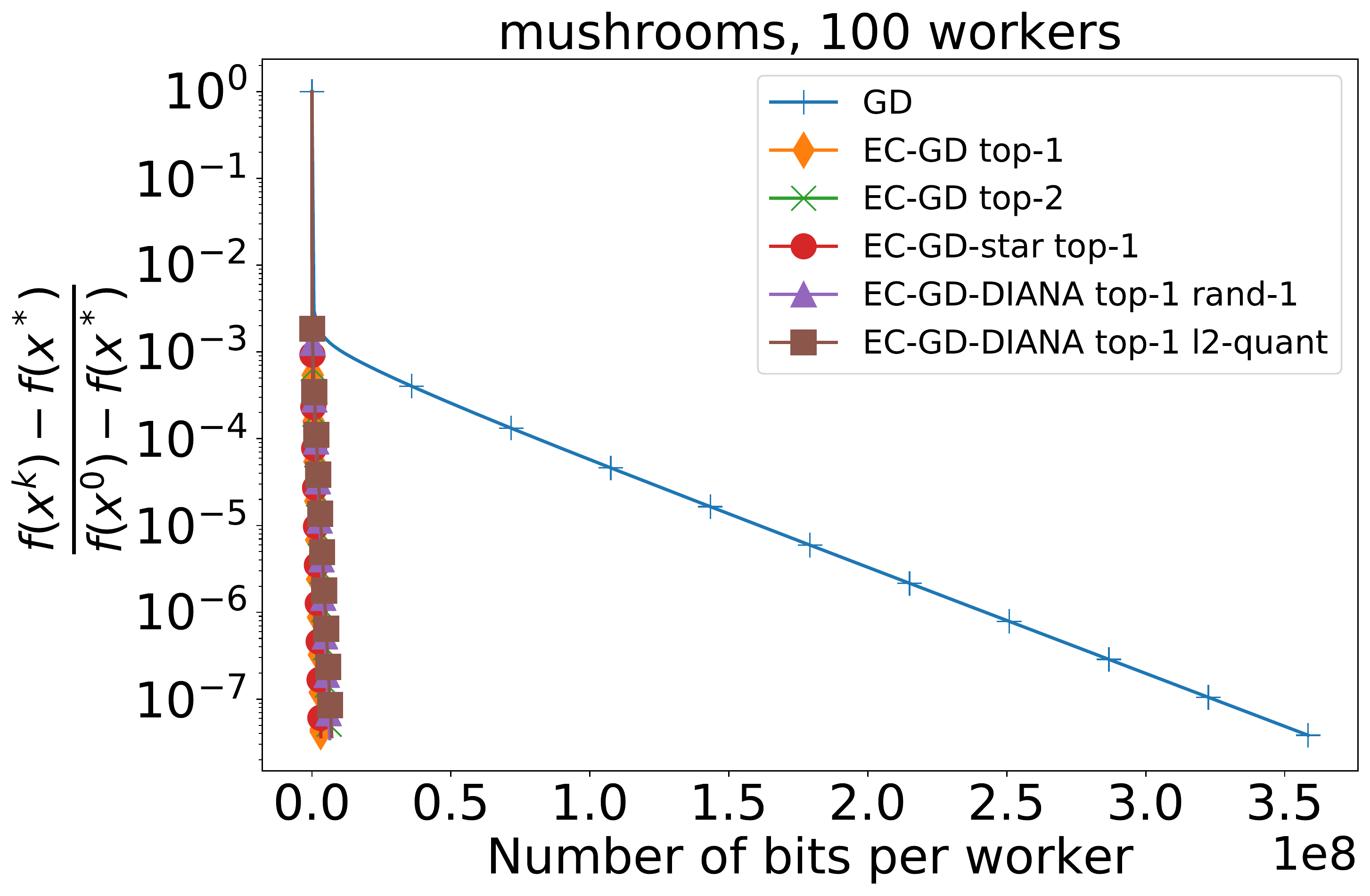}
	\includegraphics[width=0.32\textwidth]{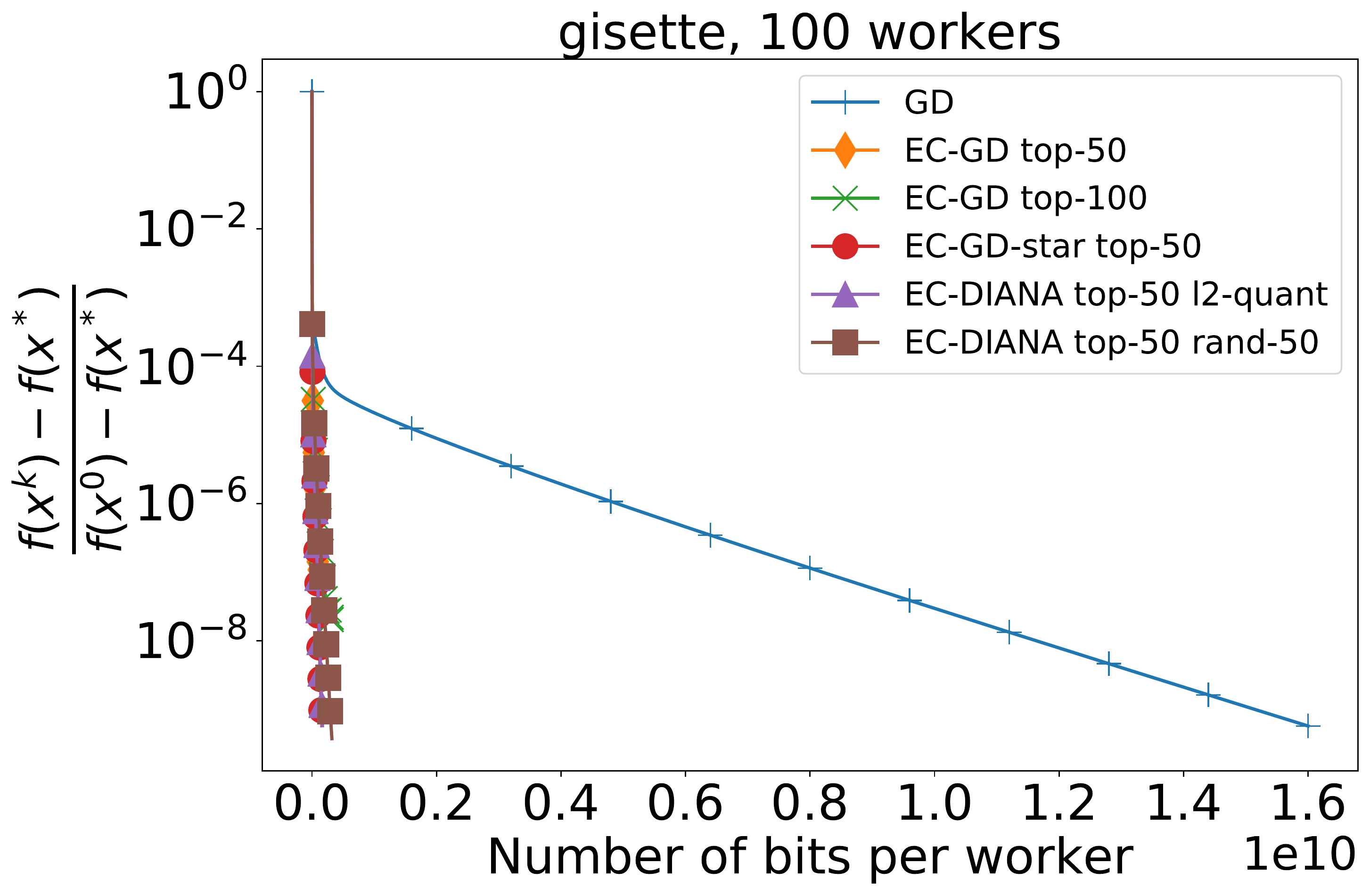}
	\\
	\includegraphics[width=0.32\textwidth]{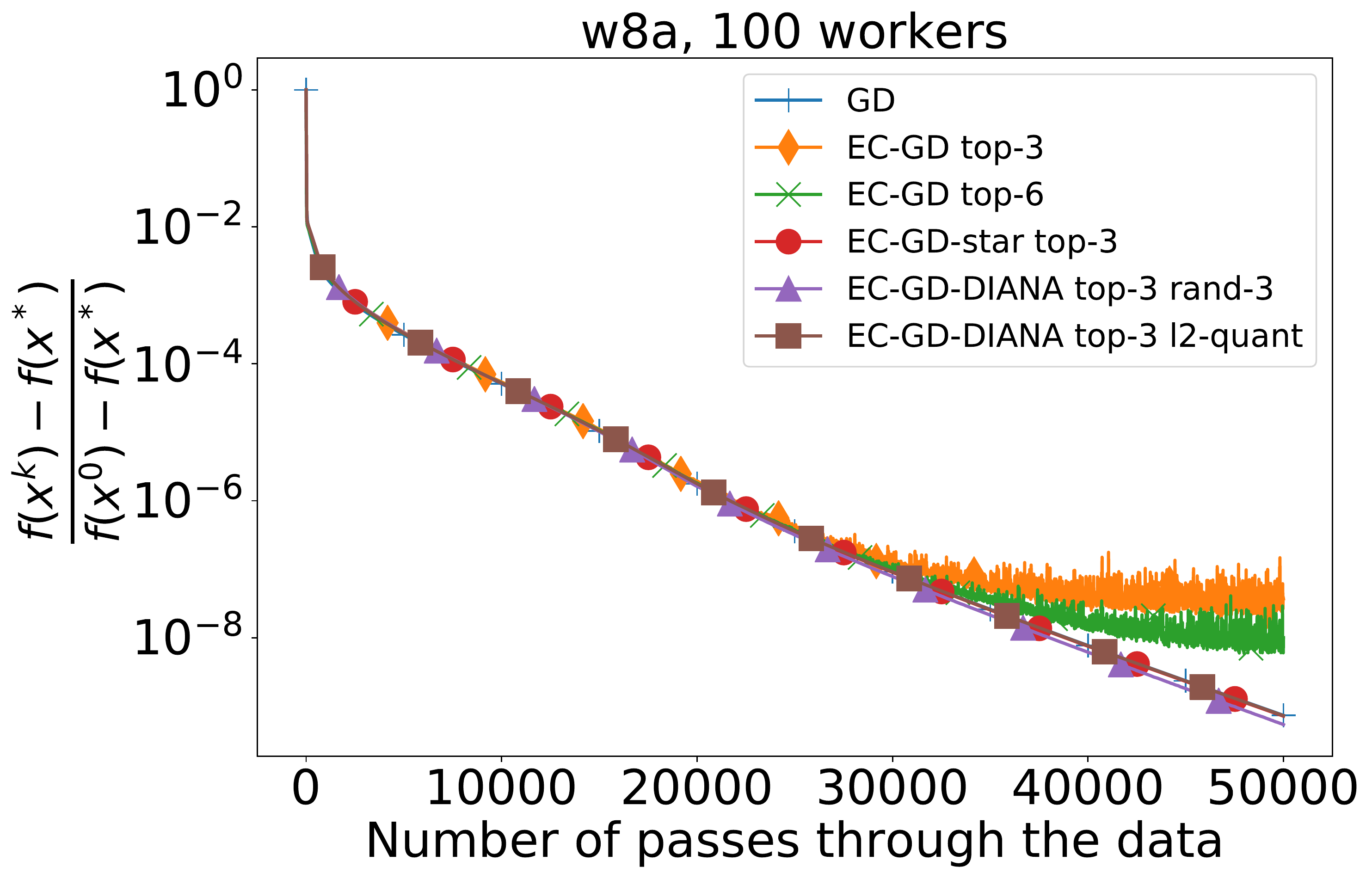}
	\includegraphics[width=0.32\textwidth]{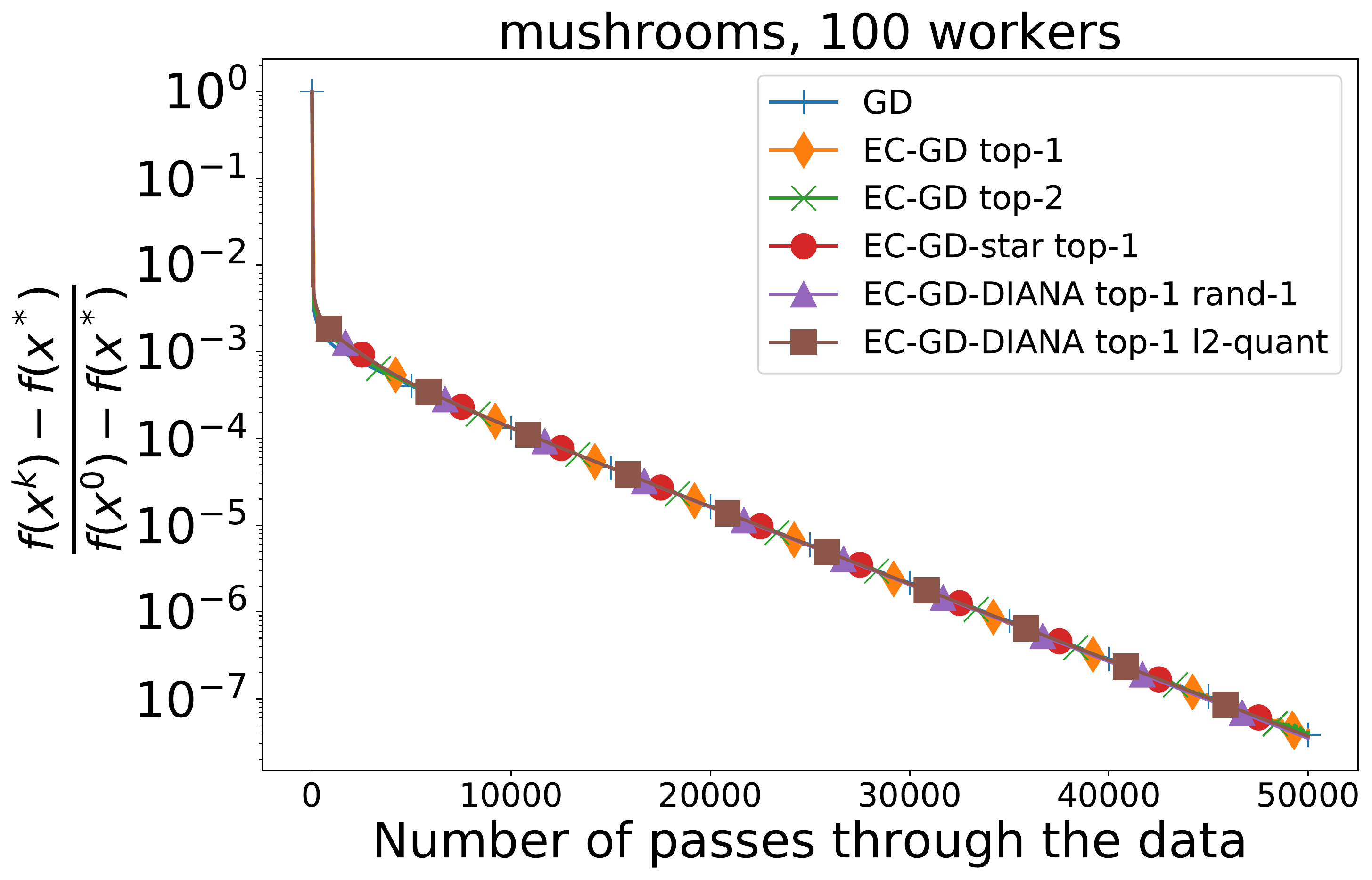}	\includegraphics[width=0.32\textwidth]{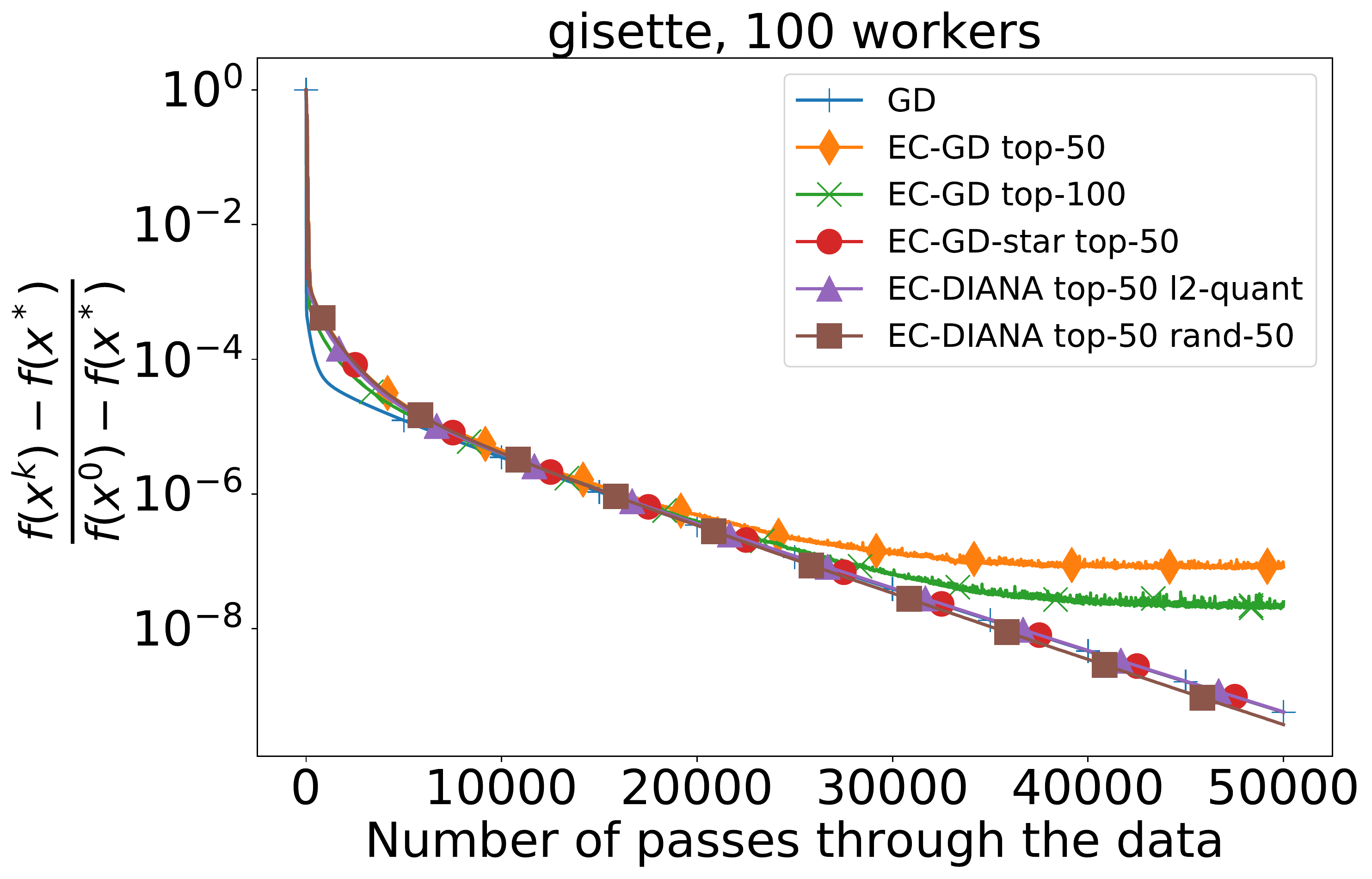}
	\caption{Trajectories of {\tt EC-GD}, {\tt EC-GD-star}, {\tt EC-DIANA} and {\tt GD} applied to solve logistic regression problem with $100$ workers.}
    \label{fig:gd_logreg_100_workers_id}
\end{figure}

\clearpage

\section{Compression Operators: Extra Commentary} \label{sec:compressions}

Communication efficient distributed {\tt SGD} methods based on the idea of communication compression exists in two distinct varieties: i) methods based on unbiased compression operators, and ii) methods based on biased compression operators. The first class of methods is much mire developed than the latter since it is easier to theoretically analyze unbiased operators. The subject of this paper is the study of the latter and dramatically less developed and understood class.

\subsection{Unbiased compressors} % \label{sec:unb_compr}

By unbiased compression operators we mean randomized mappings $\cQ:\R^d\to \R$ satisfying the relations
\[\Exp{\cQ(x)} = x \qquad \text{and} \qquad \Exp \|\cQ(x)-x\|^2 \leq \omega \|x\|^2, \qquad \forall x\in \R^d\]
for some $\omega \geq 0$.
While operators satisfying the above relations are often in the literature called \textit{quantization operators}, this class includes compressors which perform sparsification as well.

 Among the first methods using unbiased compressors developed in this field are {\tt QSGD} \cite{alistarh2017qsgd}, {\tt TernGrad} \cite{wen2017terngrad} and {\tt DQGD} \cite{khirirat2018distributed}. The first analysis of {\tt QSGD} and {\tt TernGrad} without bounded gradients assumptions was proposed in \cite{mishchenko2019distributed}, which contains the best known results for {\tt QSGD} and {\tt TernGrad}. However, existing guarantees in the strongly convex case for {\tt QGSD}, {\tt TernGrad}, and {\tt DQGD} establish linear convergence  to some neighborhood of the solution only, even if the workers quantize the full gradients of their functions. This problem was resolved by \citet{mishchenko2019distributed}, who proposed the first method, called {\tt DIANA}, which uses quantization for communication and enjoys the linear rate of convergence to the exact optimum asymptotically in the strongly convex case when workers compute the full gradients of their functions in each iteration. Unlike all previous approaches, {\tt DIANA} is based on the quantization of gradient differences rather than iterates or gradients. In essence, DIANA is a technique for reducing the variance  introduced by quantization.  \citet{horvath2019stochastic} generalized the {\tt DIANA} method to the case of more general quantization operators. Moreover, the same authors developed a new method called {\tt VR-DIANA} specially designed to solve problems  \eqref{eq:main_problem} with the individual functions having the finite sum structure \eqref{eq:f_i_sum}.

\subsection{Biased compressors}

By biased compressors we mean (possibly) randomized mappings $\cC:\R^d\to \R$ satisfying the average contraction relation
\[		\EE\left[\|\cC(x) - x\|^2\right] \le (1 - \delta)\|x\|^2, \qquad \forall x\in \R^d \]
for some $\delta>0$.
	
Perhaps the most popular biased  compression operator is  TopK, which takes vector $x$ as input and substitutes all coordinates of $x$ by zero except the $k$ components with the largest absolute values. However, such a greedy approach applied to simple distributed {\tt SGD} and even distributed {\tt GD} can break the convergence of the method even when applied to simple functions in small dimensions, and may even lead to exponential divergence \cite{beznosikov2020biased}. The \textit{error-feedback} framework described in \cite{karimireddy2019error,stich2019error,stich2018sparsified} and studies in this paper can fix this problem, and it remains the only known mechanism that does so for all compressors described in \eqref{eq:compression_def}.
%\footnote{The situation is different for the subclass of compressors which also satisfy $\Exp \left[\cC(x) \right] = t x$ for some $t>0$. In this case, convergence distributed } 
This is one of the main motivations for the study of the error-feedback mechanism. For instance, error feedback can fix convergence issues with   methods like {\tt sign-SGD} \citep{Bernstein2019signSGDWM}. The analysis of error feedback by \citet{karimireddy2019error,stich2019error,stich2018sparsified} works either under the assumption that the second moment of the stochastic gradient is uniformly bounded or only for the single-worker case. Recently Beznosikov et al. \cite{beznosikov2020biased} proposed the first analysis of {\tt SGD} with error feedback for the general case of multiple workers without bounded second moment assumption. There is another line of works \cite{koloskova2019decentralized, KoloskovaLSJ19decentralized} where authors apply arbitrary compressions in the decentralized setup. This approach has better potential than a centralized one in terms of reducing the communication cost. However, in this paper, we study only centralized architecture.

\clearpage

\section{Further Notation and Definitions}

In what follows it will be useful to denote
$$\mytextstyle v^k \eqdef \frac{1}{n}\sum_i  v_i^k, \quad  g^k \eqdef \frac{1}{n}\sum_i g_i^k, \quad e^k \eqdef \frac{1}{n}\sum_i e_i^k.$$ By aggregating identities  \eqref{eq:error_update} across all $i$, we get $e^{k+1} = e^k + \gamma g^k - v^k.$ In our proofs we also use the perturbed iterates technique \cite{leblond2018improved,mania2017perturbed} based on the analysis of the following sequence
\begin{equation}
	\tx^k = x^k - e^k. \label{eq:perturbed_uterate}
\end{equation}
This sequence satisfies very useful for the analysis relation:
\begin{equation}
	\tx^{k+1} \overset{\eqref{eq:perturbed_uterate}}{=} x^{k+1} - e^{k+1} \overset{\eqref{eq:x^k+1_update},\eqref{eq:error_update}}{=} x^k - v^k - (e^k + \gamma g^k - v^k) = x^k - e^k - \gamma g^k \overset{\eqref{eq:perturbed_uterate}}{=} \tx^k - \gamma g^k. \label{eq:perturbed_iterates_key_relation}
\end{equation}

\subsection{Quantization  operators}

\begin{definition}\label{def:quantization}
	We say that stochastic mapping $Q(x):\R^d \to \R^d$ is a quantization operator if there exists such $\omega > 0$ that for any $x\in\R^d$ 
	\begin{equation}
		\EE\left[\cQ(x)\right] = x,\quad \EE\left[\|\cQ(x) - x\|^2\right] \le \omega\|x\|^2. \label{eq:quantization_def}
	\end{equation}
\end{definition}

Below we enumerate some classical compression and quantization operators (see more in \cite{beznosikov2020biased}).
\begin{enumerate}
\item \textbf{TopK sparsification.} This compression operator is defined as follows:
\begin{equation*}
	\cC(x) = \sum\limits_{i=1}^K x_{(i)}e_{(i)}
\end{equation*}
where $|x_{(1)}| \ge |x_{(2)}| \ge \ldots \ge |x_{(d)}|$ are components of $x$ sorted in the decreasing order of their absolute values, $e_1,\ldots,e_d$ is the standard basis in $\R^d$ and $K$ is some number from $[d]$. Clearly, TopK is a biased compression operator. One can show that TopK satisfies \eqref{eq:compression_def} with $\delta = \frac{K}{d}$ \cite{beznosikov2020biased}.
\item \textbf{RandK sparsification} operator is defined as 
\begin{equation*}
	\cQ(x) = \frac{d}{K}\sum\limits_{i\in S} x_{i}e_{i}
\end{equation*}
where $S$ is a random subset of $[d]$ sampled from the uniform distribution on the all subset of $[d]$ with cardinality $K$. RandK is an unbiased compression operator satisfying \eqref{eq:quantization_def} with $\omega = \frac{d}{K}$.
\item \textbf{$\ell_p$-quantization.} By $\ell_2$-quantization we mean the following random operator:
\begin{equation*}
	\cQ(x) = \|x\|_p\text{sign}(x)\circ\xi
\end{equation*}
where $\|x\|_p = \left(\sum_{i=1}^d|x_i|^p\right)^{\nicefrac{1}{p}}$ is an $\ell_p$-norm of vector $x$, $\text{sign}(x)$ is a component-wise sign of vector $x$, $a\circ b$ defines a component-wise product of vectors $a$ and $b$ and $\xi = (\xi_1,\ldots,\xi_d)^\top$ is a random vector such that
\begin{equation*}
\xi_i = \begin{cases}1,&\text{with probability } \frac{|x_i|}{\|x\|_p},\\ 0,&\text{with probability } 1-\frac{|x_i|}{\|x\|_p}. \end{cases}
\end{equation*}
One can show that this operator satisfies \eqref{eq:quantization_def}. In particular, if $p = 2$ it satisfies \eqref{eq:quantization_def} with $\omega = \sqrt{d}-1$ and if $p = \infty$, then $\omega = \frac{1+\sqrt{d}}{2}-1$ (see \cite{mishchenko2019distributed}).
\end{enumerate}

We assume that $\cC$ is any operator which enjoys the following contractive property:  there exists a constant $0< \delta \leq 1$ such that \[\EE\left[\|x-\cC(x)\|^2\right] \leq (1-\delta) \|x\|^2, \qquad \forall x\in \R^d .\]

\clearpage
\section{Basic Inequalities, Identities and Technical Lemmas}\label{sec:basic_facts}

\subsection{Basic inequalities}

For all $a,b,x_1,\ldots,x_n\in\R^d$, $\beta > 0$ and $p\in(0,1]$ the following inequalities hold
\begin{equation}\label{eq:fenchel_young}
	\langle a,b\rangle \le \frac{\|a\|^2}{2\beta} + \frac{\beta\|b\|^2}{2},
\end{equation}
\begin{equation}\label{eq:a-b_a+b}
	\langle a-b,a+b\rangle = \|a\|^2 - \|b\|^2,
\end{equation}
\begin{equation}\label{eq:1/2a_minus_b}
    \frac{1}{2}\|a\|^2 - \|b\|^2 \le \|a+b\|^2,
\end{equation}
\begin{equation}\label{eq:a+b_norm_beta}
    \|a+b\|^2 \le (1+\beta)\|a\|^2 + (1+\nicefrac{1}{\beta})\|b\|^2,
\end{equation}
\begin{equation}\label{eq:a_b_norm_squared}
	\left\|\sum\limits_{i=1}^n x_n\right\|^2 \le n\sum\limits_{i=1}^n\|x_i\|^2,
\end{equation}
\begin{equation}
	\left(1 - \frac{p}{2}\right)^{-1} \le 1 + p, \label{eq:1-p/2_inequality}
\end{equation}
\begin{equation}
	\left(1 + \frac{p}{2}\right)(1 - p) \le 1 - \frac{p}{2}. \label{eq:1+p/2_inequality}
\end{equation}

\subsection{Identities and inequalities involving random variables}

\textbf{Variance decomposition.} For a random vector $\xi \in \R^d$ and any deterministic vector $x \in \R^d$ the variance can be decomposed as
\begin{equation}\label{eq:variance_decomposition}
	\EE\left[\left\|\xi - \EE\xi\right\|^2\right] = \EE\left[\|\xi-x\|^2\right] - \left\|\EE\xi - x\right\|^2
\end{equation}

\textbf{Tower property of mathematical expectation.} For random variables $\xi,\eta\in \R^d$ we have
\begin{equation}
	\EE\left[\xi\right] = \EE\left[\EE\left[\xi\mid \eta\right]\right]\label{eq:tower_property}
\end{equation}
under assumption that all expectations in the expression above are well-defined.

\begin{lemma}[Lemma 14 from \cite{stich2019error}]\label{lem:lemma14_stich}
	For any $\tau$ vectors $a_1,\ldots,a_\tau\in\R^d$ and $\xi_1, \ldots, \xi_\tau$ zero-mean random vectors in $\R^d$, each $\xi_t$ conditionally independent of $\{\xi_i\}_{i=1}^{t-1}$ for all $1\le t \le \tau$ the following inequality holds
	\begin{equation}
		\EE\left[\left\|\sum\limits_{t=1}^\tau (a_t + \xi_t)\right\|^2\right] \le \tau\sum\limits_{t=1}^\tau\|a_t\|^2 + \sum\limits_{t=1}^\tau\EE\|\xi_t\|^2. \label{eq:lemma14_stich}
	\end{equation}
\end{lemma}

\subsection{Technical lemmas}\label{sec:tech_lemmas}

\begin{lemma}[see also Lemma 2 from \cite{stich2019unified}]\label{lem:lemma2_stich}
	Let $\{r_k\}_{k\ge 0}$ satisfy
	\begin{equation}
		r_K \le \frac{a}{\gamma W_K} + c_1\gamma + c_2\gamma^2 \label{eq:lemma2_stich_tech_1}
	\end{equation}
	for all $K\ge 0$ with some constants $a, c_2> 0$, $c_1 \ge 0$ where $\{w_k\}_{k\ge 0}$ and $\{W_K\}_{K\ge 0}$ are defined in \eqref{eq:w_k_definition_new}, $\gamma \le \frac{1}{d}$. Then for all $K$ such that $\frac{\ln\left(\max\{2,\min\{\nicefrac{a\mu^2K^2}{c_1},\nicefrac{a\mu^3K^3}{c_2}\}\}\right)}{K}\le \min\{\rho_1,\rho_2\}$ and
	\begin{equation}
		\gamma = \min\left\{\frac{1}{d}, \frac{\ln\left(\max\{2,\min\{\nicefrac{a\mu^2K^2}{c_1},\nicefrac{a\mu^3K^3}{c_2}\}\}\right)}{\mu K}\right\} \label{eq:lemma2_stich_gamma}
	\end{equation}
	we have that
	\begin{equation}
		r_K = \widetilde\cO\left(da\exp\left(-\min\left\{\frac{\mu}{d}, \rho_1, \rho_2\right\}K\right) + \frac{c_1}{\mu K} + \frac{c_2}{\mu^2 K^2}\right). \label{eq:lemma2_stich}
	\end{equation}
\end{lemma}
\begin{proof}
	Since $W_K \ge w_K = (1-\eta)^{-(K+1)}$ we have
	\begin{eqnarray}
		r_K &\le& (1-\eta)^{K+1}\frac{a}{\gamma} + c_1\gamma + c_2\gamma^2 \le \frac{a}{\gamma}\exp\left(-\eta(K+1)\right) + c_1\gamma + c_2\gamma^2.\label{eq:lemma2_stich_tech_2}
	\end{eqnarray}
	Next we consider two possible situations.
	\begin{enumerate}
		\item If $\frac{1}{d} \ge \frac{\ln\left(\max\{2,\min\{\nicefrac{a\mu^2K^2}{c_1},\nicefrac{a\mu^3K^3}{c_2}\}\}\right)}{\mu K}$ then we choose $\gamma = \frac{\ln\left(\max\{2,\min\{\nicefrac{a\mu^2K^2}{c_1},\nicefrac{a\mu^3K^3}{c_2}\}\}\right)}{\mu K}$ and get that
		\begin{eqnarray*}
			r_K &\overset{\eqref{eq:lemma2_stich_tech_2}}{\le}& \frac{a}{\gamma}\exp\left(-\eta(K+1)\right) + c_1\gamma + c_2\gamma^2 \\
			&=& \widetilde\cO\left(a\mu K\exp\left(-\min\left\{\rho_1,\rho_2, \frac{\ln\left(\max\{2,\min\{\nicefrac{a\mu^2K^2}{c_1},\nicefrac{a\mu^3K^3}{c_2}\}\}\right)}{K}\right\}K\right)\right) \\
			&&\quad + \widetilde\cO\left(\frac{c_1}{\mu K} + \frac{c_2}{\mu^2 K^2}\right).
		\end{eqnarray*}
		Since $\frac{\ln\left(\max\{2,\min\{\nicefrac{a\mu^2K^2}{c_1},\nicefrac{a\mu^3K^3}{c_2}\}\}\right)}{K}\le \min\{\rho_1,\rho_2\}$ we have
		\begin{eqnarray*}
			r_K &=& \widetilde\cO\left(a\mu K\exp\left(-\ln\left(\max\left\{2,\min\left\{\frac{a\mu^2K^2}{c_1},\frac{a\mu^3K^3}{c_2}\right\}\right\}\right)\right)\right)\\
			&&\quad + \widetilde\cO\left(\frac{c_1}{\mu K} + \frac{c_2}{\mu^2 K^2}\right)\\
			&=& \widetilde\cO\left(\frac{c_1}{\mu K} + \frac{c_2}{\mu^2 K^2}\right).
		\end{eqnarray*}
		\item If $\frac{1}{d} \le \frac{\ln\left(\max\{2,\min\{\nicefrac{a\mu^2K^2}{c_1},\nicefrac{a\mu^3K^3}{c_2}\}\}\right)}{\mu K}$ then we choose $\gamma = \frac{1}{d}$ which implies that
		\begin{eqnarray*}
			r_K &\overset{\eqref{eq:lemma2_stich_tech_2}}{\le}& da\exp\left(-\min\left\{\frac{\mu}{d},\frac{\rho_1}{4},\frac{\rho_2}{4}\right\}(K+1)\right) + \frac{c_1}{d} + \frac{c_2}{d^2} \\
			&=& \widetilde\cO\left(da\exp\left(-\min\left\{\frac{\mu}{d}, \rho_1, \rho_2\right\}K\right) + \frac{c_1}{\mu K} + \frac{c_2}{\mu^2K^2}\right). 
		\end{eqnarray*}
	\end{enumerate}
	Combining the obtained bounds we get the result. 
\end{proof}

\begin{lemma}\label{lem:lemma_technical_cvx}
	Let $\{r_k\}_{k\ge 0}$ satisfy
	\begin{equation}
		r_K \le \frac{a}{\gamma K} + \frac{b_1\gamma}{K} + \frac{b_2\gamma^2}{K} + c_1\gamma + c_2\gamma^2 \label{eq:lemma_technical_cvx_1}
	\end{equation}
	for all $K\ge 0$ with some constants $a> 0$, $b_1, b_2, c_1, c_2 \ge 0$ where $\gamma \le \gamma_0$. Then for all $K$ and
	\begin{equation*}
		\gamma = \min\left\{\gamma_0, \sqrt{\frac{a}{b_1}}, \sqrt[3]{\frac{a}{b_2}}, \sqrt{\frac{a}{c_1 K}}, \sqrt[3]{\frac{a}{c_2 K}}\right\}
	\end{equation*}
	we have that
	\begin{equation}
		r_K = \cO\left(\frac{a}{\gamma_0 K} + \frac{\sqrt{ab_1}}{K} + \frac{\sqrt[3]{a^2b_2}}{K} + \sqrt{\frac{ac_1}{K}} + \frac{\sqrt[3]{a^2c_2}}{K^{\nicefrac{2}{3}}} \right). \label{eq:lemma_technical_cvx_2}
	\end{equation}
\end{lemma}
\begin{proof}
	We have
	\begin{eqnarray*}
		r_K &\le& \frac{a}{\gamma K} + \frac{b_1\gamma}{K} + \frac{b_2\gamma^2}{K} + c_1\gamma + c_2\gamma^2\\
		&\le& \frac{a}{\min\left\{\gamma_0, \sqrt{\frac{a}{b_1}}, \sqrt[3]{\frac{a}{b_2}}, \sqrt{\frac{a}{c_1 K}}, \sqrt[3]{\frac{a}{c_2 K}}\right\}K} + \frac{b_1}{K}\cdot\sqrt{\frac{a}{b_1}} + \frac{b_2}{K}\cdot\sqrt[3]{\frac{a}{b_2}}\\
		&&\quad + c_1\cdot\sqrt{\frac{a}{c_1 K}} + c_2 \left(\sqrt[3]{\frac{a}{c_2 K}}\right)^2\\
		 &=& \cO\left(\frac{a}{\gamma_0 K} + \frac{\sqrt{ab_1}}{K} + \frac{\sqrt[3]{a^2b_2}}{K} + \sqrt{\frac{ac_1}{K}} + \frac{\sqrt[3]{a^2c_2}}{K^{\nicefrac{2}{3}}} \right).
	\end{eqnarray*}
\end{proof}

\clearpage
\section{Proofs for Section~\ref{sec:main_res}}

\subsection{A lemma}
\begin{lemma}[See also Lemma~8 from \cite{stich2019error}]\label{lem:main_lemma_new}
	Let Assumptions~\ref{ass:quasi_strong_convexity},~\ref{ass:key_assumption_new}~and~\ref{ass:L_smoothness} be satisfied and $\gamma \le \nicefrac{1}{4(A'+C_1M_1+C_2M_2)}$. Then for all $k\ge 0$ we have
	\begin{equation}
		\frac{\gamma}{2}\EE\left[f(x^k) - f(x^*)\right] \le (1-\eta)\EE T^{k} - \EE T^{k+1} + \gamma^2(D_1' + M_1D_2) + 3L\gamma \EE\|e^k\|^2, \label{eq:main_lemma_new}
	\end{equation}	
	where $T^k \eqdef \|\tx^k - x^*\|^2 + M_1\gamma^2 \sigma_{1,k}^2 + M_2\gamma^2 \sigma_{2,k}^2$ and $M_1 = \frac{4B_1'}{3\rho_1}$, $M_2 = \frac{4\left(B_2' + \frac{4}{3}G\right)}{3\rho_2}$. 
\end{lemma}
\begin{proof}
	We start with the upper bound for $\EE\|\tx^{k+1} - x^*\|^2$. First of all, by definition of $\tx^k$ we have
	\begin{eqnarray}
		\|\tx^{k+1} - x^*\|^2 &\overset{\eqref{eq:perturbed_iterates_key_relation}}{=}& \|\tx^k - x^* - \gamma g^k\|^2\notag\\
		&=& \|\tx^k - x^*\|^2 -2\gamma\langle \tx^k - x^*, g^k\rangle + \gamma^2\|g^k\|^2\notag\\
		&=& \|\tx^k - x^*\|^2 -2\gamma\langle x^k - x^*, g^k\rangle + \gamma^2\|g^k\|^2 + 2\gamma\langle x^k - \tx^k, g^k\rangle.\notag
	\end{eqnarray}
	Taking conditional expectation $\EE\left[\cdot\mid x^k\right]$ from the both sides of the previous inequality we get
	\begin{eqnarray}
		\EE\left[\|\tx^{k+1} - x^*\|^2\mid x^k\right] &\overset{\eqref{eq:unbiasedness_new},\eqref{eq:second_moment_bound_new}}{\le}& \|\tx^k - x^*\|^2 -2\gamma\langle x^k - x^*, \nabla f(x^k)\rangle \notag\\
		&&\quad+ \gamma^2\left(2A'(f(x^k) - f(x^*)) + B_1'\sigma_{1,k}^2 + B_2'\sigma_{2,k}^2 + D_1'\right)\notag\\
		&&\quad + 2\gamma\langle x^k - \tx^k, \nabla f(x^k)\rangle\notag\\
		&\overset{\eqref{eq:str_quasi_cvx}}{\le}& \|\tx^k - x^*\|^2 - \gamma\mu\|x^k - x^*\|^2 - \gamma (2 - 2A'\gamma)(f(x^k) - f(x^*)) \notag\\
		&&\quad + \gamma^2 B_1'\sigma_{1,k}^2 + \gamma^2 B_2'\sigma_{2,k}^2 + \gamma^2 D_1' \notag\\
		&&\quad + 2\gamma\langle x^k - \tx^k, \nabla f(x^k)\rangle. \label{eq:main_lemma_technical_1_new}
	\end{eqnarray}
	Next,
	\begin{equation}
		-\|x^k - x^*\|^2 = -\| \tx^k - x^* + x^k - \tx^k\|^2 \overset{\eqref{eq:1/2a_minus_b}}{\le} -\frac{1}{2}\|\tx^k - x^*\|^2 + \|x^k - \tx^k\|^2. \label{eq:main_lemma_technical_2_new}
	\end{equation}
	Using Fenchel-Young inequality we derive an upper bound for the inner product from \eqref{eq:main_lemma_technical_1_new}:
	\begin{equation}
		\langle x^k - \tx^k, \nabla f(x^k)\rangle \overset{\eqref{eq:fenchel_young}}{\le} L\|x^k - \tx^k\|^2 + \frac{1}{4L}\|\nabla f(x^k)\|^2 \overset{\eqref{eq:L_smoothness_cor}}{\le} L\|x^k - \tx^k\|^2 + \frac{1}{2}(f(x^k) - f(x^*)).\label{eq:main_lemma_technical_3_new}
	\end{equation}
	Combining previous three inequalities we get
	\begin{eqnarray}
		\EE\left[\|\tx^{k+1} - x^*\|^2\mid x^k\right] &\overset{\eqref{eq:main_lemma_technical_1_new}-\eqref{eq:main_lemma_technical_3_new}}{\le}& \left(1 - \frac{\gamma\mu}{2}\right)\|\tx^k - x^*\|^2 - \gamma\left(1 - 2A'\gamma\right)(f(x^k) - f(x^*))\notag\\
		&&\quad + \gamma^2 B_1'\sigma_{1,k}^2 + \gamma^2 B_2'\sigma_{2,k}^2 + \gamma^2 D_1'\notag\\
		&&\quad + \gamma (2L + \mu)\|x^k - \tx^k\|^2. \label{eq:main_lemma_norm_squared_ub_new}
	\end{eqnarray}
	Taking into account that $T^k = \|\tx^k - x^*\|^2 + M_1\gamma^2 \sigma_{1,k}^2 + M_2\gamma^2 \sigma_{2,k}^2$ with $M_1 = \frac{4B_1'}{3\rho_1}$ and $M_2 = \frac{4\left(B_2' + \frac{4}{3}G\right)}{3\rho_2}$, using the tower property \eqref{eq:tower_property} of mathematical expectation together with $\gamma \le \frac{1}{4(A'+C_1M_1 + C_2M_2)}$, we conclude
	\begin{eqnarray*}
		\EE\left[T^{k+1}\right] &\overset{\eqref{eq:main_lemma_norm_squared_ub_new}}{\le}& \left(1 - \frac{\gamma\mu}{2}\right)\EE\|\tx^k - x^*\|^2 - \gamma\left(1 - 2A'\gamma\right)\EE\left[f(x^k) - f(x^*)\right] + M_1\gamma^2\EE\left[\sigma_{1,k+1}^2\right]\\
		&&\quad M_2\gamma^2\EE\left[\sigma_{2,k+1}^2\right] + \gamma^2 B_1'\sigma_{1,k}^2 + \gamma^2 B_2'\sigma_{2,k}^2 + \gamma^2 D_1' + \gamma (2L + \mu)\EE\|x^k - \tx^k\|^2\\
		&\overset{\eqref{eq:sigma_k+1_bound_1},\eqref{eq:sigma_k+1_bound_2}}{\le}& \left(1 - \frac{\gamma\mu}{2}\right)\EE\|\tx^k - x^*\|^2 + \left(1 + \frac{B_1'}{M_1} - \rho_1\right)M_1\gamma^2\EE\left[\sigma_{1,k}^2\right] \\
		&&\quad + \left(1 + \frac{B_2'+M_1G\rho_1}{M_2} - \rho_2\right)M_2\gamma^2\EE\left[\sigma_{2,k}^2\right]+ \gamma^2(D_1' + M_1D_2)\notag\\
		&&\quad - \gamma\left(1 - 2(A'+C_1M_1+C_2M_2)\gamma\right)\EE\left[f(x^k) - f(x^*)\right] + \gamma(2L+\mu)\EE\|x^k - \tx^k\|^2\\
		&\le& \left(1 - \frac{\gamma\mu}{2}\right)\EE\|\tx^k - x^*\|^2 + \left(1 - \frac{\rho_1}{4}\right)M_1\gamma^2\EE\left[\sigma_{1,k}^2\right] + \left(1 - \frac{\rho_2}{4}\right)M_2\gamma^2\EE\left[\sigma_{2,k}^2\right] \notag\\
		&&\quad - \frac{\gamma}{2}\EE\left[f(x^k) - f(x^*)\right] + \gamma(2L+\mu)\EE\|x^k - \tx^k\|^2 + \gamma^2(D_1' + M_1D_2).
	\end{eqnarray*}
	Since $L\ge \mu$, $\tx^k = x^k - e^k$ and $\eta \eqdef \min\{\frac{\gamma\mu}{2},\frac{\rho_1}{4},\frac{\rho_2}{4}\}$ the last inequality implies
	\begin{equation*}
		\frac{\gamma}{2}\EE\left[f(x^k) - f(x^*)\right] \le (1-\eta)\EE T^{k} - \EE T^{k+1} + \gamma^2(D_1' + M_1D_2) + 3L\gamma \EE\|e^k\|^2,
	\end{equation*}
	which concludes the proof.
\end{proof}

\subsection{Proof of Theorem~\ref{thm:main_result_new}}
\begin{proof}
	Form Lemma~\ref{lem:main_lemma_new} we have
	\begin{equation*}
		\frac{\gamma}{2}\EE\left[f(x^k) - f(x^*)\right] \le (1-\eta)\EE T^{k} - \EE T^{k+1} + \gamma^2(D_1' + M_1D_2) + 3L\gamma \EE\|e^k\|^2.
	\end{equation*}
	Summing up these inequalities for $k=0,\ldots,K$ with weights $w_k = (1-\eta)^{-(k+1)}$ we get
	\begin{eqnarray*}
		\frac{1}{2}\sum\limits_{k=0}^K w_k\EE\left[f(x^k) - f(x^*)\right] &\le& \sum\limits_{k=0}^K\left(\frac{w_k(1-\eta)}{\gamma}\EE T^k - \frac{w_k}{\gamma}\EE T^{k+1}\right) + \gamma(D_1' + M_1D_2)\sum\limits_{k=0}^K w_k\\
		&&\quad +3L\sum\limits_{k=0}^Kw_k\EE\|e^k\|^2 \\
		&\overset{\eqref{eq:sum_of_errors_bound_new},\eqref{eq:w_k_definition_new}}{\le}& \sum\limits_{k=0}^K\left(\frac{w_{k-1}}{\gamma}\EE T^k - \frac{w_k}{\gamma}\EE T^{k+1}\right) + F_1\sigma_{1,0}^2 + F_2\sigma_{2,0}^2  \\
		&&\quad+ \gamma^2(D_1' + M_1D_2 + D_3)W_K + \frac{1}{4}\sum\limits_{k=0}^K w_k\EE\left[f(x^k) - f(x^*)\right].
	\end{eqnarray*}
	Rearranging the terms and using $\bar{x}^K = \frac{1}{W_K}\sum_{k=0}^K w_k x^k$ together with Jensen's inequality we obtain
	\begin{eqnarray*}
		\EE\left[f(\bar x^K) - f(x^*)\right] &\le& \frac{4(T^0 + \gamma F_1 \sigma_{1,0}^2+ \gamma F_2 \sigma_{2,0}^2)}{\gamma W_K} + 4\gamma\left(D_1' + M_1D_2 + D_3\right).
	\end{eqnarray*}
	Finally, using the definition of the sequences $\{W_K\}_{K\ge 0}$ and $\{w_k\}_{k\ge 0}$ we derive that if $\mu > $, then $W_K \ge w_K \ge (1-\eta)^{-K}$ and we get \eqref{eq:main_result_new}. In the case when $\mu = 0$ we have $w_k = 1$ and $W_K = K$ which implies \eqref{eq:main_result_new_cvx}.
\end{proof}

\clearpage

\section{{\tt SGD} as a Special Case}\label{sec:sgd}
In this section we want to show that our approach is general enough to cover many existing methods of {\tt SGD} type. Consider the following situation:
\begin{equation}
	v^k = \gamma g^k,\quad e^0 = 0.\label{eq:vk_ek_plain_sgd}
\end{equation}
It implies that $e^k = 0$ for all $k\ge 0$ and the updates rules \eqref{eq:x^k+1_update}-\eqref{eq:error_update} gives us a simple {\tt SGD}:
\begin{equation}
	x^{k+1} = x^k - \gamma g^k.\label{eq:plain_sgd}
\end{equation}
The following lemma formally shows that {\tt SGD} under general enough assumptions satisfies Assumption~\ref{ass:key_assumption_new}.
\begin{lemma}\label{lem:sgd_as_a_special_case}
	Let Assumptions~\ref{ass:quasi_strong_convexity}~and~\ref{ass:L_smoothness} be satisfies and inequalities \eqref{eq:unbiasedness_new}, \eqref{eq:second_moment_bound_new}, \eqref{eq:sigma_k+1_bound_1} and \eqref{eq:sigma_k+1_bound_2} hold. Then for the method \eqref{eq:plain_sgd} inequality \eqref{eq:sum_of_errors_bound_new} holds with $F_1 = F_2 = 0$ and $D_3 = 0$ for all $k\ge 0$.
\end{lemma}
\begin{proof}
	Since $e^k = 0$ and $f(x^k)\ge f(x^*)$ for all $k\ge 0$ we get
	\begin{equation*}
		3L\sum\limits_{k=0}^K w_k\EE\|e^k\|^2 = 0 \le \frac{1}{4}\sum\limits_{k=0}^K w_k\EE\left[f(x^k) - f(x^*)\right]
	\end{equation*}
	which concludes the proof.	 
\end{proof}

It implies that all methods considered in \cite{gorbunov2019unified} fit our framework. Moreover, using Theorem~\ref{thm:main_result_new} we derive the following result.
\begin{theorem}\label{thm:main_result_sgd}
	Let Assumptions~\ref{ass:quasi_strong_convexity}~and~\ref{ass:L_smoothness} be satisfied, inequalities \eqref{eq:unbiasedness_new}, \eqref{eq:second_moment_bound_new}, \eqref{eq:sigma_k+1_bound_1}, \eqref{eq:sigma_k+1_bound_2} hold and $\gamma \le \nicefrac{1}{4(A'+C_1M_1 + C_2M_2)}$. Then for the method \eqref{eq:plain_sgd} for all $K\ge 0$ we have
	\begin{equation*}
		\EE\left[f(\bar x^K) - f(x^*)\right] \le \left(1 - \min\left\{\frac{\gamma\mu}{2},\frac{\rho_1}{4},\frac{\rho_2}{4}\right\}\right)^K\frac{4T^0}{\gamma} + 4\gamma\left(D_1' + M_1D_2\right),
	\end{equation*}	
	when $\mu > 0$ and
	\begin{equation}
		\EE\left[f(\bar x^K) - f(x^*)\right] \le \frac{4T^0}{\gamma K} + 4\gamma\left(D_1' + M_1D_2\right) \notag
	\end{equation}
	when $\mu = 0$, where $T^k \eqdef \|x^k - x^*\|^2 + M_1\gamma^2 \sigma_{1,k}^2 + M_2\gamma^2 \sigma_{2,k}^2$ and $M_1 = \frac{4B_1'}{3\rho_1}$, $M_2 = \frac{4\left(B_2' + \frac{4}{3}G\right)}{3\rho_2}$.
\end{theorem}
In particular, if $\sigma_{2,k}^2 \equiv 0$, then our assumption coincides with the key assumption from \cite{gorbunov2019unified} and our theorem recovers the same rates as in \cite{gorbunov2019unified} when $\mu > 0$. The case when $\mu = 0$ was not considered in \cite{gorbunov2019unified}, while in our analysis we get it for free.

\clearpage
\section{{Distributed \tt SGD} with Compression and Error Compensation}\label{sec:ec_sgd}
In this section we consider the scenario when compression and error-feedback is applied in order to reduce the communication cost of the method, i.e., we consider {\tt SGD} with error compensation and compression ({\tt EC-SGD}) which has updates of the form \eqref{eq:x^k+1_update}-\eqref{eq:error_update} with
\begin{eqnarray}
	g^k &=& \frac{1}{n}\sum\limits_{i=1}^ng_i^k\notag\\
	v^k &=& \frac{1}{n}\sum\limits_{i=1}^n v_i^k,\quad v_i^k = C(e_i^k + \gamma g_i^k)\label{eq:v^k_def_ec_sgd}\\
	e^k &=& \frac{1}{n}\sum\limits_{i=1}^n e_i^k,\quad e_i^{k+1} = e_i^k + \gamma g_i^k - v_i^k = e_i^k + \gamma g_i^k - C(e_i^k + \gamma g_i^k).\label{eq:e^k_def_ec_sgd}
\end{eqnarray}
Moreover, we assume that $e_i^0 = 0$ for $i = 1,\ldots,n$.

\begin{lemma}\label{lem:ec_sgd_key_lemma_new}
	Let Assumptions~\ref{ass:quasi_strong_convexity}~and~\ref{ass:L_smoothness} be satisfied,  Assumption~\ref{ass:key_assumption_finite_sums_new} holds and\footnote{When $\rho_1 = 1$ and $\rho_2=1$ one can always set the parameters in such a way that $B_1 = \widetilde{B}_1 = B_2 = \widetilde{B}_2 = C_1 = C_2 = 0$, $D_2 = 0$. In this case we assume that $\frac{2}{1-\rho_1}\left(\frac{C_1}{\rho_1}+\frac{2GC_2}{\rho_2(1-\rho_2)}\right)\left(\frac{2B_1}{\delta}+\widetilde{B}_1\right) + \frac{2C_2\left(\frac{2B_2}{\delta}+\widetilde{B}_2\right)}{\rho_2(1-\rho_2)} = 0$.}
	\begin{equation}
		\gamma \le \min\left\{\frac{\delta}{4\mu}, \sqrt{\frac{\delta}{96L\left(\frac{2A}{\delta} + \widetilde{A} + \frac{2}{1-\rho_1}\left(\frac{C_1}{\rho_1}+\frac{2GC_2}{\rho_2(1-\rho_2)}\right)\left(\frac{2B_1}{\delta}+\widetilde{B}_1\right) + \frac{2C_2\left(\frac{2B_2}{\delta}+\widetilde{B}_2\right)}{\rho_2(1-\rho_2)}\right)}}\right\},\label{eq:gamma_condition_ec_sgd_new}
	\end{equation}
	where $ M_1 = \frac{4B_1'}{3\rho_1}$ and $M_2 = \frac{4\left(B_2' + \frac{4}{3}G\right)}{3\rho_2}$. Then {\tt EC-SGD} satisfies Assumption~\ref{ass:key_assumption_finite_sums_new}, i.e., inequality \eqref{eq:sum_of_errors_bound_new} holds with the following parameters:
	\begin{equation}
		F_1 = \frac{24L\gamma^2}{\delta\rho_1(1-\eta)}\left(\frac{2B_1}{\delta}+\widetilde{B}_1\right),\quad F_2 = \frac{24L\gamma^2}{\delta\rho_2(1-\eta)}\left(\frac{2G}{1-\rho_1}\left(\frac{2B_1}{\delta}+\tilde{B}_1\right)+\frac{2B_2}{\delta} + \widetilde{B}_2\right), \label{eq:ec_sgd_parameters_new}
	\end{equation}
	\begin{equation}
		D_3 = \frac{6L\gamma}{\delta}\left(\frac{D_2}{\rho_1}\left(\frac{2B_1}{\delta}+\widetilde{B}_1\right) + \frac{2D_1}{\delta} + \widetilde{D}_1\right).\label{eq:ec_sgd_parameters_new_2}
	\end{equation}
\end{lemma}
\begin{proof}
	First of all, we derive an upper bound for the second moment of $e_i^{k+1}$:
	\begin{eqnarray}
		\EE\|e_i^{k+1}\|^2 &\overset{\eqref{eq:e^k_def_ec_sgd},\eqref{eq:tower_property}}{=}& \EE\left[\EE\left[\|e_i^k + \gamma g_i^k - C(e_i^k + \gamma g_i^k)\|^2\mid e_i^k,g_i^k\right]\right] \notag\\
		&\overset{\eqref{eq:compression_def}}{\le}& (1-\delta)\EE\|e_i^k + \gamma g_i^k\|^2 \notag\\
		&\overset{\eqref{eq:tower_property},\eqref{eq:variance_decomposition}}{=}& (1-\delta)\EE\|e_i^k + \gamma \bar{g}_i^k\|^2 + (1-\delta)\gamma^2\EE\|g_i^k-\bar{g}_i^k\|^2 \notag\\
		&\overset{\eqref{eq:a+b_norm_beta}}{\le}& (1-\delta)(1+\beta)\EE\|e_i^k\|^2 + (1-\delta)\left(1+\frac{1}{\beta}\right)\gamma^2\EE\|\bar{g}_i^k\|^2\notag\\
		&&\quad + (1-\delta)\gamma^2\EE\|g_i^k-\bar{g}_i^k\|^2. \notag
	\end{eqnarray}
	Summing up these inequalities for $i=1,\ldots, n$ we get
	\begin{eqnarray}
		\frac{1}{n}\sum\limits_{i=1}^n\EE\|e_i^{k+1}\|^2 &\le& (1-\delta)(1+\beta)\frac{1}{n}\sum\limits_{i=1}^n\EE\|e_i^{k}\|^2\notag\\
		&&\quad + (1-\delta)\left(1+\frac{1}{\beta}\right)\gamma^2\frac{1}{n}\sum\limits_{i=1}^n\EE\|\bar{g}_i^k\|^2 + (1-\delta)\gamma^2\frac{1}{n}\sum\limits_{i=1}^n\EE\|g_i^k-\bar{g}_i^k\|^2.\label{eq:ec_sgd_technical_1_new}
	\end{eqnarray}
	Consider $\beta = \frac{\delta}{2(1-\delta)}$. For this choice of $\beta$ we have
	\begin{eqnarray*}
		(1-\delta)(1+\beta) &=& (1-\delta)\left(1 + \frac{\delta}{2(1-\delta)}\right) = 1 - \frac{\delta}{2}\\
		(1-\delta)\left(1+\frac{1}{\beta}\right) &=& (1-\delta)\left(1 + \frac{2(1-\delta)}{\delta}\right) = \frac{(1-\delta)(2-\delta)}{\delta} \le \frac{2(1-\delta)}{\delta}.
	\end{eqnarray*}
	Using this we continue our derivations:
	\begin{eqnarray}
		\frac{1}{n}\sum\limits_{i=1}^n\EE\|e_i^{k+1}\|^2 &\le& \left(1 - \frac{\delta}{2}\right)\frac{1}{n}\sum\limits_{i=1}^n\EE\|e_i^{k}\|^2 + \frac{2\gamma^2(1-\delta)}{\delta}\frac{1}{n}\sum\limits_{i=1}^n\EE\|\bar{g}_i^k\|^2\notag\\
		&&\quad + (1-\delta)\gamma^2\frac{1}{n}\sum\limits_{i=1}^n\EE\|g_i^k-\bar{g}_i^k\|^2\notag\\
		&\overset{\eqref{eq:second_moment_bound_g_i^k_new},\eqref{eq:variance_bound_g_i^k_new}}{\le}& \left(1 - \frac{\delta}{2}\right)\frac{1}{n}\sum\limits_{i=1}^n\EE\|e_i^{k}\|^2 + 2\gamma^2(1-\delta)\left(\frac{2A}{\delta}+\widetilde{A}\right)\EE\left[f(x^k) - f(x^*)\right]\notag\\
		&&\quad + \gamma^2(1-\delta)\left(\frac{2B_1}{\delta}+\widetilde{B}_1\right)\EE\sigma_{1,k}^2 + \gamma^2(1-\delta)\left(\frac{2B_2}{\delta}+\widetilde{B}_2\right)\EE\sigma_{2,k}^2\notag\\
		&&\quad + \gamma^2(1-\delta)\left(\frac{2D_1}{\delta}+\widetilde{D}_1\right). \label{eq:ec_sgd_technical_2_new} 
	\end{eqnarray}
	Unrolling the recurrence above we get
	\begin{eqnarray}
		\frac{1}{n}\sum\limits_{i=1}^n\EE\|e_i^{k+1}\|^2 &\overset{\eqref{eq:ec_sgd_technical_2_new}}{\le}& 2\gamma^2(1-\delta)\left(\frac{2A}{\delta}+\widetilde{A}\right)\sum\limits_{l=0}^k\left(1 - \frac{\delta}{2}\right)^{k-l}\EE\left[f(x^l) - f(x^*)\right]\notag\\
		&&\quad + \gamma^2(1-\delta)\left(\frac{2B_1}{\delta}+\widetilde{B}_1\right)\sum\limits_{l=0}^k\left(1 - \frac{\delta}{2}\right)^{k-l}\EE\sigma_{1,l}^2 \notag\\
		&&\quad + \gamma^2(1-\delta)\left(\frac{2B_2}{\delta}+\widetilde{B}_2\right)\sum\limits_{l=0}^k\left(1 - \frac{\delta}{2}\right)^{k-l}\EE\sigma_{2,l}^2 \notag\\
		&&\quad + \gamma^2(1-\delta)\left(\frac{2D_1}{\delta}+\widetilde{D}_1\right)\sum\limits_{l=0}^k\left(1 - \frac{\delta}{2}\right)^{k-l}\label{eq:ec_sgd_technical_3_new}
	\end{eqnarray}
	which implies
	\begin{eqnarray}
		3L\sum\limits_{k=0}^Kw_k\EE\|e^k\|^2 &\overset{\eqref{eq:e^k_def_ec_sgd}}{=}& 3L\sum\limits_{k=0}^Kw_k\EE\left\|\frac{1}{n}\sum\limits_{i=1}^n e_i^k\right\|^2 \overset{\eqref{eq:a_b_norm_squared}}{\le} 3L\sum\limits_{k=0}^K w_k\frac{1}{n}\sum\limits_{i=1}^n\EE\left\|e_i^k\right\|^2\notag\\
		&\overset{\eqref{eq:ec_sgd_technical_3_new}}{\le}& \frac{6L\gamma^2(1-\delta)}{1-\frac{\delta}{2}}\left(\frac{2A}{\delta}+\widetilde{A}\right)\sum\limits_{k=0}^K\sum\limits_{l=0}^k w_k\left(1 - \frac{\delta}{2}\right)^{k-l}\EE\left[f(x^l) - f(x^*)\right]\notag\\
		&&\quad + \frac{3L\gamma^2(1-\delta)}{1-\frac{\delta}{2}}\left(\frac{2B_1}{\delta}+\widetilde{B}_1\right)\sum\limits_{k=0}^K\sum\limits_{l=0}^k w_k\left(1 - \frac{\delta}{2}\right)^{k-l}\EE\sigma_{1,l}^2 \notag\\
		&&\quad + \frac{3L\gamma^2(1-\delta)}{1-\frac{\delta}{2}}\left(\frac{2B_2}{\delta}+\widetilde{B}_2\right)\sum\limits_{k=0}^K\sum\limits_{l=0}^k w_k\left(1 - \frac{\delta}{2}\right)^{k-l}\EE\sigma_{2,l}^2 \notag\\
		&&\quad +\frac{3L\gamma^2(1-\delta)}{1-\frac{\delta}{2}}\left(\frac{2D_1}{\delta}+\widetilde{D}_1\right)\sum\limits_{k=0}^K\sum\limits_{l=0}^k w_k\left(1 - \frac{\delta}{2}\right)^{k-l}.\label{eq:ec_sgd_technical_4_new}
	\end{eqnarray}
	In the remaining part of the proof we derive upper bounds for three terms in the right-hand side of the previous inequality. First of all, recall that $w_k = (1 - \eta)^{-(k+1)}$ and $\eta = \min\left\{\frac{\gamma\mu}{2}, \frac{\rho_1}{4}, \frac{\rho_2}{4}\right\}$. It implies that for all $0 \le i < k$ we have
	\begin{eqnarray}
		w_k &=& (1 - \eta)^{-(k-j+1)}\left(1 - \eta\right)^{-j} \overset{\eqref{eq:1-p/2_inequality}}{\le} w_{k-j}\left(1 + 2\eta\right)^{j} \notag\\
		&\le& w_{k-j}\left(1 + \gamma\mu\right)^{j} \overset{\eqref{eq:gamma_condition_ec_sgd_new}}{\le} w_{k-j}\left(1 + \frac{\delta}{4}\right)^j, \label{eq:ec_sgd_technical_5_new}\\
		w_k &=& \left(1 - \eta\right)^{-(k-j+1)}\left(1 - \eta\right)^{-j} \overset{\eqref{eq:1-p/2_inequality}}{\le} w_{k-j}\left(1 + 2\eta\right)^j \notag\\
		&\le& w_{k-j}\left(1 + \frac{\min\{\rho_1,\rho_2\}}{2}\right)^j. \label{eq:ec_sgd_technical_6_new}
	\end{eqnarray}
	For simplicity, we introduce new notation: $r_k \eqdef \EE\left[f(x^k) - f(x^*)\right]$. Using this we get
	\begin{eqnarray}
		\sum\limits_{k=0}^K\sum\limits_{l=0}^k w_k\left(1 - \frac{\delta}{2}\right)^{k-l}r_l &\overset{\eqref{eq:ec_sgd_technical_5_new}}{\le}& \sum\limits_{k=0}^K\sum\limits_{l=0}^k w_l r_l\left(1 + \frac{\delta}{4}\right)^{k-l}\left(1 - \frac{\delta}{2}\right)^{k-l}\notag\\
		&\overset{\eqref{eq:1+p/2_inequality}}{\le}& \sum\limits_{k=0}^K\sum\limits_{l=0}^k w_l r_l\left(1 - \frac{\delta}{4}\right)^{k-l}\notag\\
		&\le& \left(\sum\limits_{k=0}^K w_k r_k\right)\left(\sum\limits_{k=0}^\infty \left(1 - \frac{\delta}{4}\right)^{k}\right) = \frac{4}{\delta}\sum\limits_{k=0}^K w_k r_k. \label{eq:ec_sgd_technical_7_new}
	\end{eqnarray}
	Next, we apply our assumption on $\sigma_{2,k}^2$ and derive that
	\begin{eqnarray}
		\EE\sigma_{2,k+1}^2 &\overset{\eqref{eq:sigma_k+1_bound_2}}{\le}& (1 - \rho_2)\EE\sigma_{2,k}^2 + 2C_2 \underbrace{\EE\left[f(x^k) - f(x^*)\right]}_{r_k}\notag\\
		&\le& (1-\rho_2)^{k+1}\sigma_{2,0}^2 + 2C_2\sum\limits_{l=0}^{k}(1-\rho_2)^{k-l}r_l,\label{eq:sigma_2_k_useful_recurrence_new}
	\end{eqnarray}
	hence
	\begin{eqnarray}
		\sum\limits_{k=0}^K\sum\limits_{l=0}^k w_k\left(1 - \frac{\delta}{2}\right)^{k-l}\EE\sigma_{2,l}^2 &\le& \sum\limits_{k=0}^K\sum\limits_{l=0}^k w_k\left(1 - \frac{\delta}{2}\right)^{k-l}(1-\rho_2)^l\sigma_{2,0}^2\notag\\
		&&\quad + \frac{2C_2}{1-\rho_2}\sum\limits_{k=0}^K\sum\limits_{l=0}^k\sum\limits_{t=0}^lw_k\left(1 - \frac{\delta}{2}\right)^{k-l}(1-\rho_2)^{l-t}r_t\notag.
	\end{eqnarray}
	Using this and 
	\begin{eqnarray*}
		w_k\left(1 - \frac{\delta}{2}\right)^{k-l}(1-\rho_2)^{l-t} &\overset{\eqref{eq:ec_sgd_technical_5_new}}{\le}& w_{l}\left(1+\frac{\delta}{4}\right)^{k-l}\left(1 - \frac{\delta}{2}\right)^{k-l}(1-\rho_2)^{l-t}\\
		&\overset{\eqref{eq:1+p/2_inequality},\eqref{eq:ec_sgd_technical_6_new}}{\le}& \left(1-\frac{\delta}{4}\right)^{k-l}\left(1+\frac{\rho_2}{2}\right)^{l-t}(1-\rho_2)^{l-t}w_t\\
		&\overset{\eqref{eq:1+p/2_inequality}}{\le}& \left(1-\frac{\delta}{4}\right)^{k-l}\left(1-\frac{\rho_2}{2}\right)^{l-t}w_t
	\end{eqnarray*}
	we derive
	\begin{eqnarray}
		\sum\limits_{k=0}^K\sum\limits_{l=0}^k w_k\left(1 - \frac{\delta}{2}\right)^{k-l}\EE\sigma_{2,l}^2 &\le& \sum\limits_{k=0}^K\sum\limits_{l=0}^k w_k\left(1 - \frac{\delta}{4}\right)^{k-l}\left(1-\frac{\rho_2}{2}\right)^lw_0\sigma_{2,0}^2\notag\\
		&&\quad + \frac{2C_2}{1-\rho_2}\sum\limits_{k=0}^K\sum\limits_{l=0}^k\sum\limits_{t=0}^l\left(1 - \frac{\delta}{4}\right)^{k-l}\left(1-\frac{\rho_2}{2}\right)^{l-t}w_tr_t\notag\\
		&\le& w_0\sigma_{2,0}^2 \left(\sum\limits_{k=0}^{\infty}\left(1-\frac{\delta}{4}\right)^k\right)\left(\sum\limits_{k=0}^{\infty}\left(1-\frac{\rho_2}{2}\right)^k\right)\notag\\
		&&\quad \frac{2C_2}{1-\rho_2}\left(\sum\limits_{k=0}^K w_kr_k\right)\left(\sum\limits_{k=0}^{\infty}\left(1-\frac{\delta}{4}\right)^k\right)\left(\sum\limits_{k=0}^{\infty}\left(1-\frac{\rho_2}{2}\right)^k\right)\notag\\
		&=& \frac{8\sigma_{2,0}^2}{\delta\rho_2(1-\eta)} + \frac{16C_2}{\delta\rho_2(1-\rho_2)}\sum\limits_{k=0}^K w_kr_k.\label{eq:ec_sgd_technical_8_new}
	\end{eqnarray}
	Similarly, we estimate $\sigma_{1,k}^2$:
	\begin{eqnarray}
		\EE\sigma_{1,k+1}^2 &\overset{\eqref{eq:sigma_k+1_bound_1}}{\le}& (1 - \rho_1)\EE\sigma_{1,k}^2 + 2C_1 \underbrace{\EE\left[f(x^k) - f(x^*)\right]}_{r_k} + G\rho_1\EE\sigma_{2,k}^2 + D_2\notag\\
		&\le& (1-\rho_1)^{k+1}\sigma_{1,0}^2 + 2C_1\sum\limits_{l=0}^{k}(1-\rho_1)^{k-l}r_l + G\rho_1\sum\limits_{l=0}^k(1-\rho_1)^{k-l}\EE\sigma_{2,k}^2 \notag\\
		&&+ D_2\sum\limits_{l=0}^{k}(1-\rho_1)^l\notag\\
		&\le& (1-\rho_1)^{k+1}\sigma_{1,0}^2 + 2C_1\sum\limits_{l=0}^{k}(1-\rho_1)^{k-l}r_l + G\rho_1\sum\limits_{l=0}^k(1-\rho_1)^{k-l}\EE\sigma_{2,k}^2 \notag\\
		&&\quad + D_2\sum\limits_{l=0}^{\infty}(1-\rho_1)^l\notag\\
		&=& (1-\rho_1)^{k+1}\sigma_{1,0}^2 + 2C_1\sum\limits_{l=0}^{k}(1-\rho_1)^{k-l}r_l + G\rho_1\sum\limits_{l=0}^k(1-\rho_1)^{k-l}\EE\sigma_{2,k}^2 \notag\\
		&&\quad + \frac{D_2}{\rho_1}.\label{eq:sigma_1_k_useful_recurrence_new}
	\end{eqnarray}
	Using this we get
	\begin{eqnarray}
		\sum\limits_{k=0}^K\sum\limits_{l=0}^k w_k\left(1 - \frac{\delta}{2}\right)^{k-l}\EE\sigma_{1,l}^2 &\le& \sigma_{1,0}^2\sum\limits_{k=0}^K\sum\limits_{l=0}^kw_k\left(1-\frac{\delta}{2}\right)^{k-l}(1-\rho_1)^{l}\notag\\
		&&\quad + \frac{2C_1}{1-\rho_1}\sum\limits_{k=0}^K\sum\limits_{l=0}^k\sum\limits_{t=0}^{l} w_k\left(1 - \frac{\delta}{2}\right)^{k-l}(1-\rho_1)^{l-t}r_t\notag\\
		&&\quad + \frac{G\rho_1}{1-\rho_1}\sum\limits_{k=0}^K\sum\limits_{l=0}^k\sum\limits_{t=0}^{l} w_k\left(1 - \frac{\delta}{2}\right)^{k-l}(1-\rho_1)^{l-t}\EE\sigma_{2,t}^2\notag\\
		&&\quad + \frac{D_2}{\rho_1}\sum\limits_{k=0}^K\sum\limits_{l=0}^k\sum\limits_{t=0}^{l} w_k\left(1 - \frac{\delta}{2}\right)^{k-l}(1-\rho_1)^{l-t}. \label{eq:ec_sgd_technical_9_new}
	\end{eqnarray}
	Moreover,
	\begin{eqnarray*}
		w_k\left(1 - \frac{\delta}{2}\right)^{k-l}(1-\rho_1)^{l-t} &\overset{\eqref{eq:ec_sgd_technical_5_new}}{\le}& w_{l}\left(1+\frac{\delta}{4}\right)^{k-l}\left(1 - \frac{\delta}{2}\right)^{k-l}(1-\rho_1)^{l-t}\\
		&\overset{\eqref{eq:1+p/2_inequality},\eqref{eq:ec_sgd_technical_6_new}}{\le}& \left(1-\frac{\delta}{4}\right)^{k-l}\left(1+\frac{\rho_1}{2}\right)^{l-t}(1-\rho_1)^{l-t}w_t\\
		&\overset{\eqref{eq:1+p/2_inequality}}{\le}& \left(1-\frac{\delta}{4}\right)^{k-l}\left(1-\frac{\rho_1}{2}\right)^{l-t}w_t,
	\end{eqnarray*}
	hence
	\begin{eqnarray}
		\sum\limits_{k=0}^K\sum\limits_{l=0}^k w_k\left(1 - \frac{\delta}{2}\right)^{k-l}\EE\sigma_{1,l}^2 &\overset{\eqref{eq:ec_sgd_technical_9_new}}{\le}& w_0\sigma_{1,0}^2\sum\limits_{k=0}^K\sum\limits_{l=0}^k\left(1-\frac{\delta}{4}\right)^{k-l}\left(1-\frac{\rho_1}{2}\right)^{l}\notag\\
		&&\quad + \frac{2C_1}{1-\rho_1}\sum\limits_{k=0}^K\sum\limits_{l=0}^k\sum\limits_{t=0}^{l} \left(1 - \frac{\delta}{4}\right)^{k-l}\left(1-\frac{\rho_1}{2}\right)^{l-t}w_tr_t\notag\\
		&&\quad + \frac{G\rho_1}{1-\rho_1}\sum\limits_{k=0}^K\sum\limits_{l=0}^k\sum\limits_{t=0}^{l} \left(1 - \frac{\delta}{4}\right)^{k-l}\left(1-\frac{\rho_1}{2}\right)^{l-t}w_t\EE\sigma_{2,t}^2\notag\\
		&&\quad + \frac{D_2}{\rho_1}\left(\sum\limits_{k=0}^Kw_k\right)\left(\sum\limits_{k=0}^\infty\left(1 - \frac{\delta}{2}\right)^{k}\right)\left(\sum\limits_{k=0}^\infty(1-\rho_1)^{k}\right)\notag\\
		&\le& w_0\sigma_{1,0}^2\left(\sum\limits_{k=0}^\infty\left(1 - \frac{\delta}{4}\right)^{k}\right)\left(\sum\limits_{k=0}^\infty\left(1-\frac{\rho_1}{2}\right)^{k}\right)\notag\\
		&&\quad + \frac{2C_1}{1-\rho_1}\left(\sum\limits_{k=0}^Kw_kr_k\right)\left(\sum\limits_{k=0}^\infty\left(1 - \frac{\delta}{4}\right)^{k}\right)\left(\sum\limits_{k=0}^\infty\left(1-\frac{\rho_1}{2}\right)^{k}\right)\notag\\
		&&\quad+ \frac{G\rho_1}{1-\rho_1}\left(\sum\limits_{k=0}^Kw_k\EE\sigma_{2,k}^2\right)\left(\sum\limits_{k=0}^\infty\left(1 - \frac{\delta}{4}\right)^{k}\right)\left(\sum\limits_{k=0}^\infty\left(1-\frac{\rho_1}{2}\right)^{k}\right)\notag\\
		&&\quad + \frac{2D_2}{\delta\rho_1}W_K\notag\\
		&=& \frac{8\sigma_{1,0}^2}{\delta\rho_1(1-\eta)} + \frac{16C_1}{\delta\rho_1(1-\rho_1)}\sum\limits_{k=0}^K w_kr_k + \frac{8G}{\delta(1-\rho_1)}\sum\limits_{k=0}^Kw_k\EE\sigma_{2,k}^2\notag\\
		&&\quad + \frac{2D_2}{\delta\rho_1}W_K.\label{eq:ec_sgd_technical_10_new}
	\end{eqnarray}
	For the third term in the right-hand side of previous inequality we have
	\begin{eqnarray}
		\frac{8G}{\delta(1-\rho_1)}\sum\limits_{k=0}^Kw_k\EE\sigma_{2,k}^2 &\overset{\eqref{eq:sigma_2_k_useful_recurrence_new}}{\le}& \frac{8G\sigma_{2,0}^2}{\delta(1-\rho_1)}\sum\limits_{k=0}^Kw_k(1-\rho_2)^k\notag\\
		&&\quad + \frac{16GC_2}{\delta(1-\rho_1)(1-\rho_2)}\sum\limits_{k=0}^K\sum\limits_{l=0}^kw_k(1-\rho_2)^{k-l}r_l\notag\\
		&\overset{\eqref{eq:ec_sgd_technical_6_new}}{\le}&\frac{8G\sigma_{2,0}^2w_0}{\delta(1-\rho_1)}\sum\limits_{k=0}^K\left(1+\frac{\rho_2}{2}\right)^k(1-\rho_2)^k\notag\\
		&&\quad + \frac{16GC_2}{\delta(1-\rho_1)(1-\rho_2)}\sum\limits_{k=0}^K\sum\limits_{l=0}^k\left(1+\frac{\rho_2}{2}\right)^{k-l}(1-\rho_2)^{k-l}w_lr_l\notag\\
		&\overset{\eqref{eq:1+p/2_inequality}}{\le}&\frac{8G\sigma_{2,0}^2w_0}{\delta(1-\rho_1)}\sum\limits_{k=0}^\infty\left(1-\frac{\rho_2}{2}\right)^k\notag\\
		&&\quad + \frac{16GC_2}{\delta(1-\rho_1)(1-\rho_2)}\sum\limits_{k=0}^K\sum\limits_{l=0}^k\left(1-\frac{\rho_2}{2}\right)^{k-l}w_lr_l\notag\\
		&\le& \frac{16G\sigma_{2,0}^2w_0}{\delta\rho_2(1-\rho_1)}+ \frac{16GC_2}{\delta(1-\rho_1)(1-\rho_2)}\left(\sum\limits_{k=0}^Kw_kr_k\right)\left(\sum\limits_{k=0}^\infty\left(1-\frac{\rho_2}{2}\right)^k\right)\notag\\
		&=&\frac{16G\sigma_{2,0}^2}{\delta\rho_2(1-\rho_1)(1-\eta)}+ \frac{32GC_2}{\delta\rho_2(1-\rho_1)(1-\rho_2)}\sum\limits_{k=0}^Kw_kr_k\label{eq:ec_sgd_technical_11_new}
	\end{eqnarray}
	Combining inequalities \eqref{eq:ec_sgd_technical_10_new} and \eqref{eq:ec_sgd_technical_11_new} we get
	\begin{eqnarray}
		\sum\limits_{k=0}^K\sum\limits_{l=0}^k w_k\left(1 - \frac{\delta}{2}\right)^{k-l}\EE\sigma_{1,l}^2 &\le& \frac{8\sigma_{1,0}^2}{\delta\rho_1(1-\eta)} + \frac{16}{\delta(1-\rho_1)}\left(\frac{C_1}{\rho_1} + \frac{2GC_2}{\rho_2(1-\rho_2)}\right) \sum\limits_{k=0}^K w_kr_k \notag\\
		&&\quad + \frac{16G\sigma_{2,0}^2}{\delta\rho_2(1-\rho_1)(1-\eta)}+ \frac{2D_2}{\delta\rho_1}W_K\label{eq:ec_sgd_technical_12_new}
	\end{eqnarray}
	Finally, we estimate the last term in the right-hand side of \eqref{eq:ec_sgd_technical_4_new}:
	\begin{eqnarray}
		\sum\limits_{k=0}^K\sum\limits_{l=0}^k w_k\left(1 - \frac{\delta}{2}\right)^{k-l} &\le& \left(\sum\limits_{k=0}^K w_k\right)\left(\sum\limits_{k=0}^\infty \left(1 - \frac{\delta}{2}\right)^k\right) = \frac{2}{\delta}W_K. \label{eq:ec_sgd_technical_13_new}  
	\end{eqnarray}
	Plugging inequalities \eqref{eq:ec_sgd_technical_7_new}, \eqref{eq:ec_sgd_technical_8_new}, \eqref{eq:ec_sgd_technical_12_new}, \eqref{eq:ec_sgd_technical_13_new} and $\frac{1-\delta}{1-\frac{\delta}{2}} \le 1$ in \eqref{eq:ec_sgd_technical_4_new} we obtain
	\begin{eqnarray}
		3L\sum\limits_{k=0}^Kw_k\EE\|e^k\|^2 &\le& \frac{24L\left(\frac{2A}{\delta} + \widetilde{A} + \frac{2}{1-\rho_1}\left(\frac{C_1}{\rho_1}+\frac{2GC_2}{\rho_2(1-\rho_2)}\right)\left(\frac{2B_1}{\delta}+\widetilde{B}_1\right) + \frac{2C_2\left(\frac{2B_2}{\delta}+\widetilde{B}_2\right)}{\rho_2(1-\rho_2)}\right)\gamma^2}{\delta}\sum\limits_{k=0}^K w_k r_k\notag\\
		&&\quad + \frac{24L\gamma^2}{\delta\rho_1(1-\eta)}\left(\frac{2B_1}{\delta}+\widetilde{B}_1\right)\sigma_{1,0}^2\notag\\
		&&\quad + \frac{24L\gamma^2}{\delta\rho_2(1-\eta)}\left(\frac{2G}{1-\rho_1}\left(\frac{2B_1}{\delta}+\tilde{B}_1\right)+\frac{2B_2}{\delta} + \widetilde{B}_2\right)\sigma_{2,0}^2\notag\\
		&&\quad + \frac{6L\gamma^2}{\delta}\left(\frac{D_2}{\rho_1}\left(\frac{2B_1}{\delta}+\widetilde{B}_1\right) + \frac{2D_1}{\delta} + \widetilde{D}_1\right)W_K. \notag
	\end{eqnarray}
	Taking into account that $\gamma \le \sqrt{\frac{\delta}{96L\left(\frac{2A}{\delta} + \widetilde{A} + \frac{2}{1-\rho_1}\left(\frac{C_1}{\rho_1}+\frac{2GC_2}{\rho_2(1-\rho_2)}\right)\left(\frac{2B_1}{\delta}+\widetilde{B}_1\right) + \frac{2C_2\left(\frac{2B_2}{\delta}+\widetilde{B}_2\right)}{\rho_2(1-\rho_2)}\right)}}$, $F_1 = \frac{24L\gamma^2}{\delta\rho_1(1-\eta)}\left(\frac{2B_1}{\delta}+\widetilde{B}_1\right)$, $F_2 = \frac{24L\gamma^2}{\delta\rho_2(1-\eta)}\left(\frac{2G}{1-\rho_1}\left(\frac{2B_1}{\delta}+\tilde{B}_1\right)+\frac{2B_2}{\delta} + \widetilde{B}_2\right)$ and $D_3 = \frac{6L\gamma}{\delta}\left(\frac{D_2}{\rho_1}\left(\frac{2B_1}{\delta}+\widetilde{B}_1\right) + \frac{2D_1}{\delta} + \widetilde{D}_1\right)$ we get
	\begin{eqnarray*}
		3L\sum\limits_{k=0}^Kw_k\EE\|e^k\|^2 &\le& \frac{1}{4}\sum\limits_{k=0}^K w_k r_k + F_1\sigma_{1,0}^2 + F_2\sigma_{2,0}^2 + \gamma D_3.
	\end{eqnarray*}
\end{proof}

As a direct application of Lemma~\ref{lem:ec_sgd_key_lemma_new} and Theorem~\ref{thm:main_result_new} we get the following result.
\begin{theorem}\label{thm:ec_sgd_main_result_new}
	Let Assumptions~\ref{ass:quasi_strong_convexity}~and~\ref{ass:L_smoothness} be satisfied, Assumption~\ref{ass:key_assumption_finite_sums_new} holds and
	\begin{eqnarray*}
		\gamma &\le& \frac{1}{4(A'+C_1M_1+C_2M_2)},\\
		\gamma &\le& \min\left\{\frac{\delta}{4\mu}, \sqrt{\frac{\delta}{96L\left(\frac{2A}{\delta} + \widetilde{A} + \frac{2}{1-\rho_1}\left(\frac{C_1}{\rho_1}+\frac{2GC_2}{\rho_2(1-\rho_2)}\right)\left(\frac{2B_1}{\delta}+\widetilde{B}_1\right) + \frac{2C_2\left(\frac{2B_2}{\delta}+\widetilde{B}_2\right)}{\rho_2(1-\rho_2)}\right)}}\right\},
	\end{eqnarray*}
	where $ M_1 = \frac{4B_1'}{3\rho_1}$ and $M_2 = \frac{4\left(B_2' + \frac{4}{3}G\right)}{3\rho_2}$. Then for all $K\ge 0$ we have
	\begin{equation*}
		\EE\left[f(\bar x^K) - f(x^*)\right] \le \left(1 - \eta\right)^K\frac{4(T^0 + \gamma F_1 \sigma_{1,0}^2+ \gamma F_2 \sigma_{2,0}^2)}{\gamma} + 4\gamma\left(D_1' + M_1D_2 + D_3\right), 
	\end{equation*}
	when $\mu > 0$ and
	\begin{equation*}
		\EE\left[f(\bar x^K) - f(x^*)\right] \le \frac{4(T^0 + \gamma F_1 \sigma_{1,0}^2+ \gamma F_2 \sigma_{2,0}^2)}{\gamma K} + 4\gamma\left(D_1' + M_1D_2 + D_3\right) 
	\end{equation*} 
	when $\mu = 0$, where $\eta = \min\left\{\nicefrac{\gamma\mu}{2},\nicefrac{\rho_1}{4},\nicefrac{\rho_2}{4}\right\}$, $T^k \eqdef \|\tx^k - x^*\|^2 + M_1\gamma^2 \sigma_{1,k}^2 + M_2\gamma^2 \sigma_{2,k}^2$ and 
	\begin{equation*}
		F_1 = \frac{24L\gamma^2}{\delta\rho_1(1-\eta)}\left(\frac{2B_1}{\delta}+\widetilde{B}_1\right),\quad F_2 = \frac{24L\gamma^2}{\delta\rho_2(1-\eta)}\left(\frac{2G}{1-\rho_1}\left(\frac{2B_1}{\delta}+\tilde{B}_1\right)+\frac{2B_2}{\delta} + \widetilde{B}_2\right),
	\end{equation*}
	\begin{equation*}
		D_3 = \frac{6L\gamma}{\delta}\left(\frac{D_2}{\rho_1}\left(\frac{2B_1}{\delta}+\widetilde{B}_1\right) + \frac{2D_1}{\delta} + \widetilde{D}_1\right).
	\end{equation*}
\end{theorem}

\clearpage
\section{{\tt SGD} with Delayed Updates}\label{sec:d_sgd}
In this section we consider the {\tt SGD} with delayed updates ({\tt D-SGD}) \cite{agarwal2011distributed,lian2015asynchronous,feyzmahdavian2016asynchronous,arjevani2018tight,stich2019error}. This method has updates of the form \eqref{eq:x^k+1_update}-\eqref{eq:error_update} with
\begin{eqnarray}
	g^k &=& \frac{1}{n}\sum\limits_{i=1}^ng_i^k\notag\\
	v^k &=& \frac{1}{n}\sum\limits_{i=1}^n v_i^k,\quad v_i^k = \begin{cases}\gamma g_i^{k-\tau},&\text{if } t\ge\tau,\\ 0,&\text{if } t < \tau \end{cases}\label{eq:v^k_def_d_sgd}\\
	e^k &=& \frac{1}{n}\sum\limits_{i=1}^n e_i^k,\quad e_i^{k+1} = e_i^k + \gamma g_i^k - v_i^k = \gamma\sum\limits_{t=1}^\tau g_i^{k+1-t},\label{eq:e^k_def_d_sgd}
\end{eqnarray}
where the summation is performed only for non-negative indices. Moreover, we assume that $e_i^0 = 0$ for $i = 1,\ldots,n$.

For convenience we also introduce new constant:
\begin{equation}
	\hat A = A' + L\tau. \label{eq:d_sgd_hat_A}
\end{equation}
\begin{lemma}\label{lem:d_sgd_key_lemma_new}
	Let Assumptions~\ref{ass:quasi_strong_convexity}~and~\ref{ass:L_smoothness} be satisfied, inequalities \eqref{eq:second_moment_bound_new}, \eqref{eq:sigma_k+1_bound_1} and \eqref{eq:sigma_k+1_bound_2} hold and\footnote{When $\rho_1 = 1$ and $\rho_2=1$ one can always set the parameters in such a way that $B_1 = B_1' = B_2 = B_2' = C_1 = C_2 = 0$, $D_2 = 0$. In this case we assume that $\frac{2B_1'C_1}{\rho_1(1-\rho_1)} = \frac{2B_2'C_2}{\rho_2(1-\rho_2)} = 0$.}
	\begin{equation}
		\gamma \le \min\left\{\frac{1}{2\tau\mu}, \frac{1}{8\sqrt{L\tau\left(\hat A + \frac{2 B_1'C_1}{\rho_1(1-\rho_1)} + \frac{2 B_2'C_2}{\rho_2(1-\rho_2)} + \frac{4B_1'GC_2}{\rho_2(1-\rho_1)(1-\rho_2)}\right)}}\right\}, \label{eq:gamma_condition_d_sgd_new}
	\end{equation}
	where $M_1 = \frac{4B_1'}{3\rho_1}$ and  $M_2 = \frac{4\left(B_2' + \frac{4}{3}G\right)}{3\rho_2}$. Then {\tt D-SGD} satisfies Assumption~\ref{ass:key_assumption_new}, i.e., inequality \eqref{eq:sum_of_errors_bound_new} holds with the following parameters:
	\begin{equation}
		F_1 = \frac{6\gamma^2L B_1' \tau(2+\rho_1)}{\rho_1},\quad F_2=\frac{6\gamma^2\tau L(2+\rho_2)}{\rho_2}\left(\frac{2B_1' G}{1-\rho_1} + B_2'\right), \label{eq:d_sgd_parameters_new_1}
	\end{equation}
	\begin{equation}
		D_3 = 3\gamma\tau L\left( D_1' + \frac{2 B_1' D_2}{\rho_1}\right). \label{eq:d_sgd_parameters_new_2}
	\end{equation}
\end{lemma}
\begin{proof}
	First of all, we derive an upper bound for the second moment of $e_i^{k}$:
	\begin{eqnarray}
		\EE\|e^{k}\|^2 &\overset{\eqref{eq:e^k_def_d_sgd}}{=}& \gamma^2\EE\left[\left\|\sum\limits_{t=1}^\tau g^{k-t}\right\|^2\right] \notag\\
		&\overset{\eqref{eq:lemma14_stich}}{\le}& \gamma^2\tau\sum\limits_{t=1}^\tau\EE\left[\left\|\nabla f(x^{k-t})\right\|^2\right] + \gamma^2\sum\limits_{t=1}^\tau\EE\left[\left\|g^{k-t} - \nabla f(x^{k-t})\right\|^2\right]\notag\\
		&\overset{\eqref{eq:variance_decomposition}}{\le}& \gamma^2\tau\sum\limits_{t=1}^\tau\EE\left[\left\|\nabla f(x^{k-t})\right\|^2\right] + \gamma^2\sum\limits_{t=1}^\tau\EE\left[\left\|g^{k-t}\right\|^2\right]\notag\\
		&\overset{\eqref{eq:second_moment_bound_new},\eqref{eq:L_smoothness_cor}}{\le}& 2\gamma^2\underbrace{(A'+L\tau)}_{\hat A}\sum\limits_{t=1}^\tau\EE\left[f(x^{k-t}) - f(x^*)\right] + \gamma^2 B_1'\sum\limits_{t=1}^\tau \EE\sigma_{1,k-t}^2\notag\\
		&&\quad + \gamma^2 B_2'\sum\limits_{t=1}^\tau \EE\sigma_{2,k-t}^2 + \gamma^2\tau D_1' \label{eq:d_sgd_technical_1_new}
	\end{eqnarray}
	which implies
	\begin{eqnarray}
		3L\sum\limits_{k=0}^Kw_k\EE\|e^k\|^2 &\overset{\eqref{eq:d_sgd_technical_1_new}}{\le}& 6\gamma^2L\hat A\sum\limits_{k=0}^K\sum\limits_{t=1}^\tau w_k\EE\left[f(x^{k-t}) - f(x^*)\right]\notag\\
		&&\quad + 3\gamma^2L B_1'\sum\limits_{k=0}^K\sum\limits_{t=1}^\tau w_k\EE\sigma_{1,k-t}^2\notag\\
		&&\quad + 3\gamma^2L B_2'\sum\limits_{k=0}^K\sum\limits_{t=1}^\tau w_k\EE\sigma_{2,k-t}^2 + 3\gamma^2\tau L D_1' W_K \label{eq:d_sgd_technical_2_new}
	\end{eqnarray}
	In the remaining part of the proof we derive upper bounds for four terms in the right-hand side of the previous inequality. First of all, recall that $w_k = (1 - \eta)^{-(k+1)}$ and $\eta = \min\left\{\frac{\gamma\mu}{2}, \frac{\rho_1}{4}, \frac{\rho_2}{4}\right\}$. It implies that for all $0 \le i < k$ and $0\le t \le \tau$ we have
	\begin{eqnarray}
		w_k &=& (1 - \eta)^{-(k-t+1)}\left(1 - \eta\right)^{-t} \overset{\eqref{eq:1-p/2_inequality}}{\le} w_{k-t}\left(1 + 2\eta\right)^{t} \notag\\
		&\le& w_{k-t}\left(1 + \gamma\mu\right)^{t} \overset{\eqref{eq:gamma_condition_d_sgd_new}}{\le} w_{k-t}\left(1 + \frac{1}{2\tau}\right)^t \le w_{k-t}\exp\left(\frac{t}{2\tau}\right) \le 2w_{k-t}, \label{eq:d_sgd_technical_3_new}\\
		w_k &=& \left(1 - \eta\right)^{-(k-j+1)}\left(1 - \eta\right)^{-j} \overset{\eqref{eq:1-p/2_inequality}}{\le} w_{k-j}\left(1 + 2\eta\right)^j \le w_{k-j}\left(1 + \frac{\min\{\rho_1,\rho_2\}}{2}\right)^j. \label{eq:d_sgd_technical_4_new}
	\end{eqnarray}
	For simplicity, we introduce new notation: $r_k \eqdef \EE\left[f(x^k) - f(x^*)\right]$. Using this we get
	\begin{eqnarray}
		\sum\limits_{k=0}^K\sum\limits_{t=1}^\tau w_kr_{k-t} &\overset{\eqref{eq:d_sgd_technical_3_new}}{\le}& \sum\limits_{k=0}^K\sum\limits_{t=1}^\tau 2w_{k-t} r_{k-t} \le 2\tau\sum\limits_{k=0}^K w_kr_k \label{eq:d_sgd_technical_6_new}
	\end{eqnarray}
	Similarly, we estimate the second term in the right-hand side of \eqref{eq:d_sgd_technical_4_new}:
	\begin{eqnarray}
		\sum\limits_{k=0}^K\sum\limits_{t=1}^\tau w_k\EE\sigma_{1,k-t}^2 &\le& \sum\limits_{k=0}^K\sum\limits_{t=1}^\tau 2w_{k-t}\EE\sigma_{1,k-t}^2 \le 2\tau\sum\limits_{k=0}^K w_k \EE\sigma_{1,k}^2\notag\\
		&\overset{\eqref{eq:sigma_1_k_useful_recurrence_new}}{\le}& 2\tau\sigma_{1,0}^2\sum\limits_{k=0}^K w_k(1-\rho_1)^k + \frac{4C_1\tau}{1-\rho_1}\sum\limits_{k=0}^K\sum\limits_{l=0}^k w_k (1-\rho_1)^{k-l}r_l \notag\\
		&&\quad + \frac{2G\rho_1\tau}{1-\rho_1}\sum\limits_{k=0}^K\sum\limits_{l=0}^k w_k (1-\rho_1)^{k-l}\EE\sigma_{2,l}^2 + \frac{2\tau D_2}{\rho}W_K. \label{eq:d_sgd_technical_7_new}
	\end{eqnarray}
	For the first term in the right-hand side of previous inequality we have
	\begin{eqnarray}
		2\tau\sigma_{1,0}^2\sum\limits_{k=0}^K w_k(1-\rho_1)^{k} &\overset{\eqref{eq:d_sgd_technical_4_new}}{\le}& 2\tau\sigma_{1,0}^2\sum\limits_{k=0}^K \left(1 + \frac{\rho_1}{2}\right)^{k+1}(1-\rho_1)^k \notag\\
		&\overset{\eqref{eq:1+p/2_inequality}}{\le}& 2\tau\left(1+\frac{\rho_1}{2}\right)\sigma_{1,0}^2\sum\limits_{k=0}^K\left(1 - \frac{\rho_1}{2}\right)^k \notag\\
		&\le& \tau\left(2+\rho_1\right)\sigma_{1,0}^2\sum\limits_{k=0}^\infty\left(1 - \frac{\rho_1}{2}\right)^k \le \frac{2\tau\left(2 + \rho_1\right)\sigma_{1,0}^2}{\rho_1}.\label{eq:d_sgd_technical_8_new}
	\end{eqnarray}
	The second term in the right-hand side of \eqref{eq:d_sgd_technical_7_new} can be upper bounded in the following way:
	\begin{eqnarray}
		\frac{4C_1\tau}{1-\rho_1}\sum\limits_{k=0}^K\sum\limits_{l=0}^k w_k (1-\rho_1)^{k-l}r_l &\overset{\eqref{eq:d_sgd_technical_4_new}}{\le}& \frac{4C_1\tau}{1-\rho_1}\sum\limits_{k=0}^K\sum\limits_{l=0}^k w_l r_l \left(1 + \frac{\rho_1}{2}\right)^{k-l}(1-\rho_1)^{k-l}\notag\\
		&\overset{\eqref{eq:1+p/2_inequality}}{\le}& \frac{4C_1\tau}{1-\rho_1}\sum\limits_{k=0}^K\sum\limits_{l=0}^k w_l r_l \left(1 - \frac{\rho_1}{2}\right)^{k-l}\notag\\
		&\le& \frac{4C_1\tau}{1-\rho_1}\left(\sum\limits_{k=0}^K w_k r_k\right)\left(\sum\limits_{k=0}^\infty\left(1 - \frac{\rho_1}{2}\right)^k\right)\notag\\
		&\le& \frac{8C_1\tau}{\rho_1(1-\rho_1)}\sum\limits_{k=0}^K w_k r_k.\label{eq:d_sgd_technical_9_new}
	\end{eqnarray}
	Repeating similar steps we estimate the third term in the right-hand side of \eqref{eq:d_sgd_technical_7_new}:
	\begin{eqnarray}
		\frac{2G\rho_1\tau}{1-\rho_1}\sum\limits_{k=0}^K\sum\limits_{l=0}^k w_k (1-\rho_1)^{k-l}\EE\sigma_{2,l}^2 &\le& \frac{4G\tau}{1-\rho_1}\sum\limits_{k=0}^Kw_k\EE\sigma_{2,k}^2\notag\\
		&\overset{\eqref{eq:sigma_2_k_useful_recurrence_new}}{\le}& \frac{4G\tau\sigma_{2,0}^2}{1-\rho_1}\sum\limits_{k=0}^Kw_k(1-\rho_2)^k\notag\\
		&& + \frac{8GC_2}{(1-\rho_1)(1-\rho_2)}\sum\limits_{k=0}^K\sum\limits_{l=0}^kw_k(1-\rho_2)^{k-l}r_l\notag\\
		&\overset{\eqref{eq:d_sgd_technical_4_new}}{\le}& \frac{4G\tau\sigma_{2,0}^2}{1-\rho_1}\sum\limits_{k=0}^K\left(1+\frac{\rho_2}{2}\right)^{k+1}(1-\rho_2)^k\notag\\
		&&\hspace{-2cm}+ \frac{8GC_2\tau}{(1-\rho_1)(1-\rho_2)}\sum\limits_{k=0}^K\sum\limits_{l=0}^k\left(1+\frac{\rho_2}{2}\right)^{k-l}(1-\rho_2)^{k-l}w_lr_l\notag\\
		&\overset{\eqref{eq:1+p/2_inequality}}{\le}& \frac{2G\tau(2+\rho_2)\sigma_{2,0}^2}{1-\rho_1}\sum\limits_{k=0}^\infty\left(1-\frac{\rho_2}{2}\right)^{k}\notag\\
		&&+ \frac{8GC_2\tau}{(1-\rho_1)(1-\rho_2)}\sum\limits_{k=0}^K\sum\limits_{l=0}^k\left(1-\frac{\rho_2}{2}\right)^{k-l}w_lr_l\notag\\
		&\le& \frac{4G\tau(2+\rho_2)\sigma_{2,0}^2}{\rho_2(1-\rho_1)}\notag\\
		&&+\frac{8GC_2\tau}{(1-\rho_1)(1-\rho_2)}\left(\sum\limits_{k=0}^Kw_kr_k\right)\left(\sum\limits_{k=0}^\infty\left(1-\frac{\rho_2}{2}\right)^k\right)\notag\\
		&=& \frac{4G\tau(2+\rho_2)\sigma_{2,0}^2}{\rho_2(1-\rho_1)} \notag\\
		&&\quad+ \frac{16GC_2\tau}{\rho_2(1-\rho_1)(1-\rho_2)}\sum\limits_{k=0}^Kw_kr_k\label{eq:d_sgd_technical_9_1_new}
	\end{eqnarray}
	Combining inequalities \eqref{eq:d_sgd_technical_7_new}, \eqref{eq:d_sgd_technical_8_new}, \eqref{eq:d_sgd_technical_9_new} and \eqref{eq:d_sgd_technical_9_1_new} we get
	\begin{eqnarray}
		\sum\limits_{k=0}^K\sum\limits_{t=1}^\tau w_k\EE\sigma_{1,k-t}^2 &\le& \frac{2\tau\left(2 + \rho_1\right)\sigma_{1,0}^2}{\rho_1} + \frac{8\tau}{1-\rho_1}\left(\frac{C_1}{\rho_1}+\frac{2GC_2}{\rho_2(1-\rho_2)}\right)\sum\limits_{k=0}^K w_k r_k \notag\\
		&&\quad + \frac{4G\tau(2+\rho_2)\sigma_{2,0}^2}{\rho_2(1-\rho_1)} + \frac{2\tau D_2}{\rho} W_K.\label{eq:d_sgd_technical_10_new}
	\end{eqnarray}
	Next, we derive
	\begin{eqnarray}
		\sum\limits_{k=0}^K\sum\limits_{t=1}^\tau w_k\EE\sigma_{2,k-t}^2 &\le& \sum\limits_{k=0}^K\sum\limits_{t=1}^\tau 2w_{k-t}\EE\sigma_{2,k-t}^2 \le 2\tau\sum\limits_{k=0}^K w_k \EE\sigma_{2,k}^2\notag\\
		&\overset{\eqref{eq:sigma_2_k_useful_recurrence_new}}{\le}& 2\tau\sigma_{2,0}^2\sum\limits_{k=0}^K w_k(1-\rho_1)^k \notag\\
		&&\quad+ \frac{4C_2\tau}{1-\rho_2}\sum\limits_{k=0}^K\sum\limits_{l=0}^k w_k (1-\rho_2)^{k-l}r_l.\label{eq:d_sgd_technical_11_new}
	\end{eqnarray}
	For the first term in the right-hand side of previous inequality we have
	\begin{eqnarray}
		2\tau\sigma_{2,0}^2\sum\limits_{k=0}^K w_k(1-\rho_2)^{k} &\overset{\eqref{eq:d_sgd_technical_4_new}}{\le}& 2\tau\sigma_{2,0}^2\sum\limits_{k=0}^K \left(1 + \frac{\rho_2}{2}\right)^{k+1}(1-\rho_2)^k \notag\\
		&\overset{\eqref{eq:1+p/2_inequality}}{\le}& 2\tau\left(1+\frac{\rho_2}{2}\right)\sigma_{2,0}^2\sum\limits_{k=0}^K\left(1 - \frac{\rho_2}{2}\right)^k \notag\\
		&\le& \tau\left(2+\rho_2\right)\sigma_{2,0}^2\sum\limits_{k=0}^\infty\left(1 - \frac{\rho_2}{2}\right)^k \le \frac{2\tau\left(2 + \rho_2\right)\sigma_{2,0}^2}{\rho_2}.\notag
	\end{eqnarray}
	The second term in the right-hand side of \eqref{eq:d_sgd_technical_11_new} can be upper bounded in the following way:
	\begin{eqnarray}
		\frac{4C_2\tau}{1-\rho_2}\sum\limits_{k=0}^K\sum\limits_{l=0}^k w_k (1-\rho_2)^{k-l}r_l &\overset{\eqref{eq:d_sgd_technical_4_new}}{\le}& \frac{4C_2\tau}{1-\rho_2}\sum\limits_{k=0}^K\sum\limits_{l=0}^k w_l r_l \left(1 + \frac{\rho_2}{2}\right)^{k-l}(1-\rho_2)^{k-l}\notag\\
		&\overset{\eqref{eq:1+p/2_inequality}}{\le}& \frac{4C_2\tau}{1-\rho_2}\sum\limits_{k=0}^K\sum\limits_{l=0}^k w_l r_l \left(1 - \frac{\rho_2}{2}\right)^{k-l}\notag\\
		&\le& \frac{4C_2\tau}{1-\rho_2}\left(\sum\limits_{k=0}^K w_k r_k\right)\left(\sum\limits_{k=0}^\infty\left(1 - \frac{\rho_2}{2}\right)^k\right)\notag\\
		&\le& \frac{8C_2\tau}{\rho_2(1-\rho_2)}\sum\limits_{k=0}^K w_k r_k,\notag
	\end{eqnarray}	
	hence
	\begin{eqnarray}
		\sum\limits_{k=0}^K\sum\limits_{t=1}^\tau w_k\EE\sigma_{2,k-t}^2 &\overset{\eqref{eq:d_sgd_technical_11_new}}{\le}& \frac{2\tau\left(2 + \rho_2\right)\sigma_{2,0}^2}{\rho_2} + \frac{8C_2\tau}{\rho_2(1-\rho_2)}\sum\limits_{k=0}^K w_k r_k. \label{eq:d_sgd_technical_12_new}
	\end{eqnarray}
	Plugging inequalities \eqref{eq:d_sgd_technical_6_new}, \eqref{eq:d_sgd_technical_10_new} and \eqref{eq:d_sgd_technical_12_new} in \eqref{eq:d_sgd_technical_2_new} we obtain
	\begin{eqnarray}
		3L\sum\limits_{k=0}^Kw_k\EE\|e^k\|^2 &\le& 12\gamma^2L\tau\left(\hat A + \frac{2 B_1'C_1}{\rho_1(1-\rho_1)} + \frac{2 B_2'C_2}{\rho_2(1-\rho_2)} + \frac{4B_1'GC_2}{\rho_2(1-\rho_1)(1-\rho_2)}\right)\sum\limits_{k=0}^K w_k r_k\notag\\
		&&\quad + \frac{6\gamma^2L B_1' \tau(2+\rho_1)}{\rho_1}\sigma_0^2 + \frac{6\gamma^2\tau L(2+\rho_2)}{\rho_2}\left(\frac{2B_1' G}{1-\rho_1} + B_2'\right)\sigma_{2,0}^2\notag\\
		&&\quad + 3\gamma^2\tau L\left(D_1' + \frac{2 B_1' D_2}{\rho}\right)W_K. \notag
	\end{eqnarray}
	Taking into account that $\gamma \le \frac{1}{4\sqrt{4L\tau\left(\hat A + \frac{2 B_1'C_1}{\rho_1(1-\rho_1)} + \frac{2 B_2'C_2}{\rho_2(1-\rho_2)} + \frac{4B_1'GC_2}{\rho_2(1-\rho_1)(1-\rho_2)}\right)}}$, $F_1 = \frac{6\gamma^2L B_1' \tau(2+\rho_1)}{\rho_1}$, $F_2=\frac{6\gamma^2\tau L}{\rho_2}\left(\frac{2B_1' G(2+\rho_2)}{1-\rho_1} + B_2'\right)$ and $D_3 = 3\gamma\tau L\left( D_1' + \frac{2 B_1' D_2}{\rho}\right)$ we get
	\begin{eqnarray*}
		3L\sum\limits_{k=0}^Kw_k\EE\|e^k\|^2 &\le& \frac{1}{4}\sum\limits_{k=0}^K w_k r_k + F_1\sigma_{1,0}^2 + F_2\sigma_{2,0}^2 + \gamma D_3.
	\end{eqnarray*}
\end{proof}

As a direct application of Lemma~\ref{lem:d_sgd_key_lemma_new} and Theorem~\ref{thm:main_result_new} we get the following result.
\begin{theorem}\label{thm:d_sgd_main_result_new}
	Let Assumptions~\ref{ass:quasi_strong_convexity}~and~\ref{ass:L_smoothness} be satisfied, inequalities \eqref{eq:second_moment_bound_new}, \eqref{eq:sigma_k+1_bound_1} and \eqref{eq:sigma_k+1_bound_2} hold and
	\begin{equation*}
		\gamma \le \min\left\{\frac{1}{4(A'+C_1M_1 + C_2M_2)},\frac{1}{2\tau\mu}, \frac{1}{8\sqrt{L\tau\left(\hat A + \frac{2 B_1'C_1}{\rho_1(1-\rho_1)} + \frac{2 B_2'C_2}{\rho_2(1-\rho_2)} + \frac{4B_1'GC_2}{\rho_2(1-\rho_1)(1-\rho_2)}\right)}}\right\},
	\end{equation*}
	where $M_1 = \frac{4B_1'}{3\rho_1}$ and  $M_2 = \frac{4\left(B_2' + \frac{4}{3}G\right)}{3\rho_2}$. Then for all $K\ge 0$ we have
	\begin{equation*}
		\EE\left[f(\bar x^K) - f(x^*)\right] \le \left(1 - \eta\right)^K\frac{4(T^0 + \gamma F_1 \sigma_{1,0}^2+ \gamma F_2 \sigma_{2,0}^2)}{\gamma} + 4\gamma\left(D_1' + MD_2 + D_3\right) 
	\end{equation*}
	when $\mu > 0$ and
	\begin{equation*}
		\EE\left[f(\bar x^K) - f(x^*)\right] \le \frac{4(T^0 + \gamma F_1 \sigma_{1,0}^2+ \gamma F_2 \sigma_{2,0}^2)}{\gamma K} + 4\gamma\left(D_1' + MD_2 + D_3\right) 
	\end{equation*}
	when $\mu = 0$, where $\eta = \min\left\{\nicefrac{\gamma\mu}{2},\nicefrac{\rho_1}{4},\nicefrac{\rho_2}{4}\right\}$, $T^k \eqdef \|\tx^k - x^*\|^2 + M_1\gamma^2 \sigma_{1,k}^2 + M_2\gamma^2 \sigma_{2,k}^2$ and 
	\begin{equation*}
		F_1 = \frac{6\gamma^2L B_1' \tau(2+\rho_1)}{\rho_1},\quad F_2=\frac{6\gamma^2\tau L(2+\rho_2)}{\rho_2}\left(\frac{2B_1' G}{1-\rho_1} + B_2'\right),
	\end{equation*}
	\begin{equation*}
		D_3 = 3\gamma\tau L\left( D_1' + \frac{2 B_1' D_2}{\rho_1}\right).
	\end{equation*}
\end{theorem}

\clearpage

\section{Special Cases: {\tt SGD}}\label{sec:special_cases_sgd}
To illustrate the generality of our approach, we develop and analyse a new special case of {\tt SGD} without error-feedback and show that in some cases, our framework recovers tighter rates than the framework from \cite{gorbunov2019unified}.

\subsection{{\tt DIANA} with Arbitrary Sampling and Double Quantization}\label{sec:diana_arbitrary_sampling}
In this section we consider problem \eqref{eq:main_problem} with $f(x)$ being $\mu$-quasi strongly convex and $f_i(x)$ satisfying \eqref{eq:f_i_sum} where functions $f_{ij}(x)$ are differentiable, but not necessary convex. Following \cite{gower2019sgd} we construct a stochastic reformulation of this problem:
\begin{equation}
	f(x) = \EE_{\cD}\left[f_\xi(x)\right],\quad f_\xi(x) = \frac{1}{n}\sum\limits_{i=1}^n f_{\xi_i}(x),\quad f_{\xi_i}(x) = \frac{1}{m}\sum\limits_{j=1}^m \xi_{ij}f_{ij}(x), \label{eq:sr_def}
\end{equation}
where $\xi = (\xi_1^\top,\ldots, \xi_n^\top), \xi_i = (\xi_{i1},\ldots, \xi_{im})^\top$ is a random vector with distribution $\cD_i$ such that $\EE_{\cD_i}[\xi_{ij}] = 1$ for all $i\in[n], j\in[m]$ and the following assumption holds.
\begin{assumption}[Expected smoothness]\label{ass:exp_smoothness}
	We assume that functions $f_1,\ldots, f_n$ are $\cL$-smooth in expectation w.r.t.\ distributions $\cD_1,\ldots,\cD_n$, i.e., there exists constant $\cL = \cL(f,\cD_1,\ldots,\cD_n)$ such that
	\begin{equation}
		\EE_{\cD_i}\left[\|\nabla f_{\xi_i}(x) - \nabla f_{\xi_i}(x^*)\|^2\right] \le 2\cL D_{f_i}(x,x^*)\label{eq:exp_smoothness}
	\end{equation}
	for all $i\in [n]$ and $x\in\R^d$.
\end{assumption}

To solve this problem, we consider {\tt DIANA} \cite{mishchenko2019distributed, horvath2019stochastic}~--- a distributed stochastic method using unbiased compressions or \textit{quantizations} for communication between workers and master. We start with the formal definition of quantization. In \cite{mishchenko2019distributed, horvath2019stochastic} {\tt DIANA} was analyzed under the assumption that stochastic gradients $g_i^k$ have uniformly bounded variances which is not very practical.

Therefore, we consider a slightly different method called {\tt DIANAsr-DQ} which works with the stochastic reformulation \eqref{eq:sr_def} of problem \eqref{eq:main_problem}+\eqref{eq:f_i_sum}, see Algorithm~\ref{alg:DIANAsr-DQ}.
\begin{algorithm}[t]
   \caption{{\tt DIANAsr} with Double Compression ({\tt DIANAsr-DQ})}\label{alg:DIANAsr-DQ}
\begin{algorithmic}[1]
   \Require learning rates $\gamma>0$, $\alpha \in (0,1]$, initial vectors $x^0, h_1^0,\ldots, h_n^0 \in \R^d$
   \State Set $h^0 = \frac{1}{n}\sum_{i=1}^n h_i^0$   
   \For{$k=0,1,\dotsc$}
       \State Broadcast $g^{k-1}$ to all workers \Comment{If $k=0$, then broadcast $x^0$}
        \For{$i=1,\dotsc,n$ in parallel}
        	\State $x^{k} = x^{k-1} - \gamma g^{k-1}$ \Comment{Ignore this line if $k=0$}
			\State Sample $g_i^{k,1} = \nabla f_{\xi_i^k}(x^k)$ satisfying Assumption~\ref{ass:exp_smoothness} independtently from other workers
            \State $\hat\Delta_i^k = g_i^{k,1} - h_i^k$
            \State Sample $\Delta_i^k \sim Q_1(\hat\Delta_i^k)$ indepently from other workers
            \State $g_i^{k,2} = h_i^k + \Delta_i^k$
            \State $h_i^{k+1} = h_i^k + \alpha \Delta_i^k$
        \EndFor
        \State $g^{k,2} = \frac{1}{n}\sum_{i=1}^ng_i^{k,2} = h^k + \frac{1}{n}\sum_{i=1}^n\Delta_i^k$
        \State $h^{k+1} = \frac{1}{n}\sum\limits_{i=1}^n h_i^{k+1} = h^k + \alpha\frac{1}{n}\sum\limits_{i=1}^n\Delta_i^k$
       \State Sample $g^k \sim Q_2(g^{k,2})$
       \State $x^{k+1} = x^{k} - \gamma g^{k-1}$
   \EndFor
\end{algorithmic}
\end{algorithm}
Moreover, to illustrate the flexibility of our approach, we consider compression not only on the workers' side but also on the master side. To perform an update of {\tt DIANAsr-DQ} master needs to gather quantized gradient differences $\Delta_i^k$ and the to broadcast quantized stochastic gradient $g^k$ to all workers. Clearly, in this case, only compressed vectors participate in communication.  

In the concurrent work \cite{philippenko2020artemis} the same method was independently proposed under the name of {\tt Artemis}. However, our analysis is slightly more general: it is based on Assumption~\ref{ass:exp_smoothness} while in \cite{philippenko2020artemis} authors assume $L$-cocoercivity of stochastic gradients almost surely. Next, a very similar approach was considered in \cite{tang2019doublesqueeze}, where authors present a method with error compensation on master and worker sides. Moreover, recently another method called {\tt DORE} was developed in \cite{liu2019double}, which uses {\tt DIANA}-trick on the worker side and error compensation on the master side. However, in these methods, compression operators are the same on both sides, despite the fact that gathering the information often costs much more than broadcasting. Therefore, the natural idea is in using different quantization for gathering and broadcasting, and it is what {\tt DIANAsr-DQ} does. Moreover, we do not assume uniform boundedness of the second moment of the stochastic gradient like in \cite{tang2019doublesqueeze}, and we also do not assume uniform boundedness of the variance of the stochastic gradient like in \cite{liu2019double}. Assumption~\ref{ass:exp_smoothness} is more natural and always holds for the problems \eqref{eq:main_problem}+\eqref{eq:f_i_sum} when $f_{ij}$ are convex and $L$-smooth for each $i\in[n]$, $j\in[m]$. In contrast, in the same setup, there exist such problems that the variance of the stochastic gradients is not uniformly upper bounded by any finite constant.

We assume that $Q_1$ and $Q_2$ satisfy \eqref{eq:quantization_def} with parameters $\omega_1$ and $\omega_2$ respectively.
\begin{lemma}\label{lem:diana_second_moment_bound}
	Let Assumption~\ref{ass:exp_smoothness} be satisfied. Then, for all $k\ge 0$ we have
	\begin{eqnarray}
		\EE\left[g^k\mid x^k\right] &=& \nabla f(x^k), \label{eq:diana_unbiasedness}\\
		\EE\left[\|g^k\|^2\mid x^k\right] &\le& 2\cL(1+\omega_2)\left(2+\frac{3\omega_1}{n}\right)\left(f(x^k) - f(x^*)\right) + \frac{3\omega_1(1+\omega_2)}{n}\sigma_k^2 + D_1', \label{eq:diana_second_moment_bound}
	\end{eqnarray}
	where $\sigma_k^2 = \frac{1}{n}\sum_{i=1}^n\|h_i^k - \nabla f(x^*)\|^2$ and $D_1' = \frac{(2+3\omega_1)(1+\omega_2)}{n^2}\sum\limits_{i=1}^n\EE_{\cD_i}\left[\|\nabla f_{\xi_i}(x^*) - \nabla f_i(x^*)\|^2\right]$.
\end{lemma}
\begin{proof}
	First of all, we show inbiasedness of $g^k$:
	\begin{eqnarray*}
		\EE\left[g^k\mid x^k\right] &\overset{\eqref{eq:tower_property},\eqref{eq:quantization_def}}{=}& \EE\left[g^{k,2}\mid x^k\right] = h^k + \frac{1}{n}\sum\limits_{i=1}^n\EE\left[\Delta_i^k\mid x^k\right]\\
		&\overset{\eqref{eq:tower_property},\eqref{eq:quantization_def}}{=}& h^k + \frac{1}{n}\sum\limits_{i=1}^n\EE\left[\hat \Delta_i^k\mid x^k\right]\\
		&=& h^k + \frac{1}{n}\sum\limits_{i=1}^n\left(\nabla f_i(x^k) - h_i^k\right) = \nabla f(x^k).
	\end{eqnarray*}
	Next, to denote mathematical expectation w.r.t.\ the randomness coming from quantizations $Q_1$ and $Q_2$ at iteration $k$ we use $\EE_{Q_1^k}[\cdot]$ and $\EE_{Q_2^k}[\cdot]$ respectively. Using these notations and the definition of quantization we derive
	\begin{eqnarray*}
		\EE_{Q_2^k}[\|g^k\|^2] &\overset{\eqref{eq:variance_decomposition},\eqref{eq:quantization_def}}{=}& \|g^{k,1}\|^2 + \EE_{Q_2^k}\left[\|g^{k,2}-g^{k,1}\|^2\right]\\
		&\overset{\eqref{eq:quantization_def}}{\le}& (1+\omega_2)\|g^{k,1}\|^2. 
	\end{eqnarray*}
	Taking the conditopnal mathematical expectation $\EE_{Q_1^k}[\cdot]$ from the both sides of previous inequality and using the  independence of $\Delta_i^1,\ldots,\Delta_i^n$ we get
	\begin{eqnarray*}
		\EE_{Q_1^k,Q_2^k}\left[\|g^k\|^2\right] &\overset{\eqref{eq:tower_property}}{=}& (1+\omega_2)\EE_{Q_1^k}\left[\|g^{k,1}\|^2\right]  = (1+\omega_2)\EE_{Q_1^k}\left[\left\|\frac{1}{n}\sum\limits_{i=1}^n (h_i^k + \Delta_i^k)\right\|^2\right]\\
		&\overset{\eqref{eq:variance_decomposition}}{=}& (1+\omega_2)\left\|\frac{1}{n}\sum\limits_{i=1}^n\left(h_i^k + \hat \Delta_i^k\right)\right\|^2 + (1+\omega_2)\EE_{Q_1^k}\left[\left\|\frac{1}{n}\sum\limits_{i=1}^n(\Delta_i^k - \hat\Delta_i^k)\right\|^2\right]\\
		&=& (1+\omega_2)\left\|\frac{1}{n}\sum\limits_{i=1}^n\left(\nabla f_{\xi_i^k}(x^k) - \nabla f_{\xi_i^k}(x^*) + \nabla f_{\xi_i^k}(x^*) - \nabla f_i(x^*)\right)\right\|^2\\
		&&\quad + \frac{(1+\omega_2)}{n^2}\sum\limits_{i=1}^n\EE_{Q_1^k}\left[\|\Delta_i^k - \hat\Delta_i^k\|^2\right]\\
		&\overset{\eqref{eq:a_b_norm_squared},\eqref{eq:quantization_def}}{\le}& \frac{2(1+\omega_2)}{n}\sum\limits_{i=1}^n\|\nabla f_{\xi_i^k}(x^k) - \nabla f_{\xi_i^k}(x^*)\|^2\\
		&&\quad + 2(1+\omega_2)\left\|\frac{1}{n}\sum\limits_{i=1}^n\left(\nabla f_{\xi_i^k}(x^*) - \nabla f_i(x^*)\right)\right\|^2\\
		&&\quad + \frac{\omega_1(1+\omega_2)}{n^2}\sum\limits_{i=1}^n\|\nabla f_{\xi_i^k}(x^k) - h_i^k\|^2\\
		&\overset{\eqref{eq:a_b_norm_squared}}{\le}& \frac{2(1+\omega_2)}{n}\sum\limits_{i=1}^n\|\nabla f_{\xi_i^k}(x^k) - \nabla f_{\xi_i^k}(x^*)\|^2\\
		&&\quad  + 2(1+\omega_2)\left\|\frac{1}{n}\sum\limits_{i=1}^n\left(\nabla f_{\xi_i^k}(x^*) - \nabla f_i(x^*)\right)\right\|^2\\
		&&\quad + \frac{3\omega_1(1+\omega_2)}{n^2}\sum\limits_{i=1}^n\|\nabla f_{\xi_i^k}(x^k) - \nabla f_{\xi_i^k}(x^*)\|^2\\
		&&\quad + \frac{3\omega_1(1+\omega_2)}{n^2}\sum\limits_{i=1}^n\|\nabla f_{\xi_i^k}(x^*) - \nabla f_i(x^*)\|^2\\
		&&\quad + \frac{3\omega_1(1+\omega_2)}{n^2}\sum\limits_{i=1}^n\|h_i^k - \nabla f_i(x^*)\|^2.
	\end{eqnarray*}
	Finally, we take conditional mathematical expectation $\EE[\cdot\mid x^k]$ from the both sides of the inequality above and use the independece of $\xi_1^k,\ldots,\xi_n^k$:
	\begin{eqnarray*}
		\EE\left[\|g^k\|^2\mid x^k\right] &\overset{\eqref{eq:exp_smoothness}}{\le}& 2\cL(1+\omega_2)\left(2+\frac{3\omega_1}{n}\right)(f(x^k) - f(x^*))+\frac{3\omega_1(1+\omega_2)}{n}\sigma_{k}^2\\
		&&\quad + 2(1+\omega_2)\EE\left[\left\|\frac{1}{n}\sum\limits_{i=1}^n\left(\nabla f_{\xi_i^k}(x^*) - \nabla f_i(x^*)\right)\right\|^2\mid x^k\right]\\
		&&\quad + \frac{3\omega_1(1+\omega_2)}{n^2}\sum\limits_{i=1}^n\EE_{\cD_i}\left[\|\nabla f_{\xi_i}(x^*) - \nabla f_i(x^*)\|^2\right]\\
		&=& 2\cL(1+\omega_2)\left(2+\frac{3\omega_1}{n}\right)(f(x^k) - f(x^*))+\frac{3\omega_1(1+\omega_2)}{n}\sigma_{k}^2\\
		&&\quad + \frac{(1+\omega_2)(2+3\omega_1)}{n^2}\sum\limits_{i=1}^n\EE_{\cD_i}\left[\|\nabla f_{\xi_i}(x^*) - \nabla f_i(x^*)\|^2\right].
	\end{eqnarray*}
\end{proof}

\begin{lemma}\label{lem:diana_sigma_k+1_bound}
	Let $f_i$ be convex and $L$-smooth, Assumption~\ref{ass:exp_smoothness} holds and $\alpha \le \nicefrac{1}{(\omega_1+1)}$. Then, for all $k\ge 0$ we have
	\begin{equation}
		\EE\left[\sigma_{k+1}^2\mid x^k\right] \le (1 - \alpha)\sigma_k^2 + 2\alpha(3\cL+4L)(f(x^k) - f(x^*)) + D_2, \label{eq:diana_sigma_k+1_bound}
	\end{equation}
	where $\sigma_k^2 = \frac{1}{n}\sum_{i=1}^n\|h_i^k - \nabla f_i(x^*)\|^2$ and $D_2 = \frac{3\alpha}{n}\sum_{i=1}^n \EE_{\cD_i}\left[\|\nabla f_{\xi_i}(x^*) - \nabla f_i(x^*)\|^2\right]$.
\end{lemma}
\begin{proof}
	For simplicity, we introduce new notation: $h_i^* \eqdef \nabla f_i(x^*)$. Using this we derive an upper bound for the second moment of $h_i^{k+1} - h_i^*$:
	\begin{eqnarray*}
		\EE\left[\|h_i^{k+1} - h_i^*\|^2\mid x^k\right] &=& \EE\left[\left\|h_i^k - h_i^* + \alpha \Delta_i^k \right\|^2\mid x^k\right]\\
		&\overset{\eqref{eq:quantization_def}}{=}& \|h_i^k - h_i^*\|^2 +2\alpha\langle h_i^k - h_i^*, \nabla f_i(x^k) - h_i^k \rangle + \alpha^2\EE\left[\|\Delta_i^k\|^2\mid x^k\right]\\
		&\overset{\eqref{eq:quantization_def},\eqref{eq:tower_property}}{\le}& \|h_i^k - h_i^*\|^2 +2\alpha\langle h_i^k - h_i^*, \nabla f_i(x^k) - h_i^k \rangle\\
		&&\quad + \alpha^2(\omega_1+1)\EE\left[\|\nabla f_{\xi_i^k}(x^k) - h_i^k\|^2\mid x^k\right].
	\end{eqnarray*}
	Using variance decomposition \eqref{eq:variance_decomposition} and $\alpha \le \nicefrac{1}{(\omega_1+1)}$ we get
	\begin{eqnarray*}
		\alpha^2(\omega_1+1)\EE_{\cD_i}\left[\|\nabla f_{\xi_i^k}(x^k) - h_i^k\|^2\right] &\overset{\eqref{eq:variance_decomposition}}{=}& \alpha^2(\omega_1+1)\EE_{\cD_i}\left[\|\nabla f_{\xi_i^k}(x^k) - \nabla f_i(x^k)\|^2\right]\\
		&&\quad + \alpha^2(\omega_1+1)\|\nabla f_i(x^k) - h_i^k\|^2\\
		&\overset{\eqref{eq:a_b_norm_squared}}{\le}& 3\alpha\EE_{\cD_i}\left[\|\nabla f_{\xi_i^k}(x^k) - \nabla f_{\xi_i^k}(x^*)\|^2\right]\\
		&&\quad +3\alpha\EE_{\cD_i}\left[\|\nabla f_{\xi_i^k}(x^*) - \nabla f_i(x^*)\|^2\right]\\
		&&\quad +3\alpha\|\nabla f_i(x^k) - \nabla f_i(x^*)\|^2\\
		&&\quad + \alpha\|\nabla f_i(x^k) - h_i^k\|^2\\
		&\overset{\eqref{eq:L_smoothness_cor},\eqref{eq:exp_smoothness}}{\le}& 6\alpha(\cL + L)D_{f_i}(x^k,x^*) + \alpha\|\nabla f_i(x^k) - h_i^k\|^2\\
		&&\quad +3\alpha\EE_{\cD_i}\left[\|\nabla f_{\xi_i^k}(x^*) - \nabla f_i(x^*)\|^2\right]
	\end{eqnarray*}
	Putting all together we obtain
	\begin{eqnarray*}
		\EE\left[\|h_i^{k+1} - h_i^*\|^2\mid x^k\right] &\le& \|h_i^k - h_i^*\|^2 + \alpha\left\langle \nabla f_i(x^k) - h_i^k, f_i(x^k) + h_i^k - 2h_i^* \right\rangle\\
		&&\quad + 6\alpha(\cL + L)D_{f_i}(x^k,x^*) +3\alpha\EE_{\cD_i}\left[\|\nabla f_{\xi_i^k}(x^*) - \nabla f_i(x^*)\|^2\right]\\
		&\overset{\eqref{eq:a-b_a+b}}{=}& \|h_i^k - h_i^*\|^2 + \alpha\|\nabla f_i(x^k) - h_i^*\|^2 - \alpha\|h_i^k - h_i^*\|^2\\
		&&\quad + 6\alpha(\cL + L)D_{f_i}(x^k,x^*) +3\alpha\EE_{\cD_i}\left[\|\nabla f_{\xi_i^k}(x^*) - \nabla f_i(x^*)\|^2\right]\\
		&\overset{\eqref{eq:L_smoothness_cor}}{\le}& (1-\alpha)\|h_i^k - h_i^*\|^2 + \alpha(6\cL + 8L)D_{f_i}(x^k,x^*)\\
		&&\quad +3\alpha\EE_{\cD_i}\left[\|\nabla f_{\xi_i^k}(x^*) - \nabla f_i(x^*)\|^2\right].
	\end{eqnarray*}
	Summing up the above inequality for $i=1,\ldots, n$ we derive
	\begin{eqnarray*}
		\frac{1}{n}\sum\limits_{i=1}^n\EE\left[\|h_i^{k+1} - h_i^*\|^2\mid x^k\right] &\le& \frac{1-\alpha}{n}\sum\limits_{i=1}^n\|h_i^k - h_i^*\|^2 + \alpha(6\cL + 8L)(f(x^k) - f(x^*))\\
		&&\quad + \frac{3\alpha}{n}\sum\limits_{i=1}^n \EE_{\cD_i}\left[\|\nabla f_{\xi_i^k}(x^*) - \nabla f_i(x^*)\|^2\right].
	\end{eqnarray*}
\end{proof}

\begin{theorem}\label{thm:diana}
	Assume that $f_i(x)$ is convex and $L$-smooth for all $i=1,\ldots, n$, $f(x)$ is $\mu$-quasi strongly convex and Assumption~\ref{ass:exp_smoothness} holds. Then {\tt DIANAsr-DQ} satisfies Assumption~\ref{ass:key_assumption_new} with
	\begin{gather*}
		A' = \cL(1+\omega_2)\left(2+\frac{3\omega_1}{n}\right),\quad B_1' = \frac{3\omega_1(1+\omega_2)}{n},\\
		D_1' = \frac{(2+3\omega_1)(1+\omega_2)}{n^2}\sum\limits_{i=1}^n\EE_{\cD_i}\left[\|\nabla f_{\xi_i}(x^*) - \nabla f_i(x^*)\|^2\right],\\
		\sigma_{1,k}^2 = \sigma_k^2 = \frac{1}{n}\sum\limits_{i=1}^n\|h_i^k - \nabla f_i(x^*)\|^2,\quad B_2' = 0,\quad \sigma_{2,k}^2\equiv 0,\quad \rho_1 = \alpha,\quad \rho_2 = 1,\\
		C_1 = \alpha(3\cL+4L),\quad C_2 = 0,\quad D_2 = \frac{3\alpha}{n}\sum_{i=1}^n \EE_{\cD_i}\left[\|\nabla f_{\xi_i}(x^*) - \nabla f_i(x^*)\|^2\right],\\
		 G = 0,\quad F_1 = F_2 = 0,\quad D_3 = 0,
	\end{gather*}
	with $\gamma$ and $\alpha$ satisfying
	\begin{equation*}
		\gamma \le \frac{1}{4(1+\omega_2)\left(\cL\left(2+\frac{15\omega_1}{n}\right)+\frac{16L\omega_1}{n}\right)},\quad \alpha \le \frac{1}{\omega+1},\quad M_1 = \frac{4\omega_1(1+\omega_2)}{n\alpha},\quad M_2 = 0
	\end{equation*}
	and for all $K \ge 0$
	\begin{equation*}
		\EE\left[f(\bar x^K) - f(x^*)\right] \le \left(1 - \min\left\{\frac{\gamma\mu}{2},\frac{\alpha}{4}\right\}\right)^K\frac{4T^0}{\gamma} + 4\gamma\left(D_1' + M_1D_2\right),
	\end{equation*}	
	when $\mu > 0$ and
	\begin{equation*}
		\EE\left[f(\bar{x}^K) - f(x^*)\right] \le \frac{4T^0}{\gamma K} + 4\gamma\left(D_1' + M_1D_2\right)
	\end{equation*}
	when $\mu=0$, where $T^k \eqdef \|x^k - x^*\|^2 + M_1\gamma^2 \sigma_{1,k}^2$.
\end{theorem}
In other words, if 
\begin{equation*}
		\gamma = \frac{1}{4(1+\omega_2)\left(\cL\left(2+\frac{15\omega_1}{n}\right)+\frac{16L\omega_1}{n}\right)},\quad \alpha = \frac{1}{\omega+1}
\end{equation*}
and $D_1 = 0$, i.e., $\nabla f_{\xi_i^k}(x^k) = \nabla f_i(x^k)$ almost surely, {\tt DIANAsr-DQ} converges with the linear rate
\begin{equation*}
	\cO\left(\left(\omega_1 + \frac{\cL}{\mu}(1+\omega_2)\left(1+\frac{\omega_1}{n}\right)\right)\ln\frac{1}{\varepsilon}\right)
\end{equation*}
to the exact solution. Applying Lemma~\ref{lem:lemma2_stich} we establish the rate of convergence to $\varepsilon$-solution.
\begin{corollary}\label{cor:diana_str_cvx_cor}
	Let the assumptions of Theorem~\ref{thm:diana} hold and $\mu > 0$. Then after $K$ iterations of {\tt DIANAsq-DQ} with the stepsize
	\begin{eqnarray*}
		\gamma_0 &=& \frac{1}{4(1+\omega_2)\left(\cL\left(2+\frac{15\omega_1}{n}\right)+\frac{16L\omega_1}{n}\right)}\\
		\gamma &=& \min\left\{\gamma_0, \frac{\ln\left(\max\left\{2,\frac{\mu^2K^2(\|x^0-x^*\|^2+M_1\gamma_0^2\sigma_{1,0}^2)}{D_1'+M_1D_2}\right\}\right)}{\mu K}\right\},\quad M_1 = \frac{4\omega_1(1+\omega_2)}{n\alpha}
	\end{eqnarray*}		
	and $\alpha = \frac{1}{\omega+1}$ we have
	\begin{equation*}
		\EE\left[f(\bar{x}^K) - f(x^*)\right] = \widetilde\cO\left(A'\|x^0 - x^*\|^2\exp\left(-\min\left\{\frac{\mu}{A'},\frac{1}{\omega_1}\right\}K\right) + \frac{D_1'+M_1D_2}{\mu K}\right).
	\end{equation*}
	That is, to achive $\EE\left[f(\bar{x}^K) - f(x^*)\right] \le \varepsilon$ {\tt DIANAsq-DQ} requires
	\begin{equation*}
		\widetilde{\cO}\left(\omega_1 + \frac{\cL\left(1+\frac{\omega_1}{n}\right)(1+\omega_2)}{\mu} + \frac{(1+\omega_1)(1+\omega_2)}{n^2\mu\varepsilon}\sum\limits_{i=1}^n\EE_{\cD_i}\|\nabla f_{\xi_i}(x^*)-\nabla f_i(x^*)\|^2\right) \text{ iterations.}
	\end{equation*}
\end{corollary}

Applying Lemma~\ref{lem:lemma_technical_cvx} we get the complexity result in the case when $\mu = 0$.
\begin{corollary}\label{cor:diana_cvx_cor}
	Let the assumptions of Theorem~\ref{thm:diana} hold and $\mu = 0$. Then after $K$ iterations of {\tt DIANAsq-DQ} with the stepsize
	\begin{eqnarray*}
		\gamma_0 &=& \frac{1}{4(1+\omega_2)\left(\cL\left(2+\frac{15\omega_1}{n}\right)+\frac{16L\omega_1}{n}\right)}\\
		\gamma &=& \min\left\{\gamma_0, \sqrt{\frac{\|x^0 - x^*\|^2}{M_1\sigma_{1,0}^2}}, \sqrt{\frac{\|x^0 - x^*\|^2}{(D_1'+M_1D_2) K}}\right\},\quad M_1 = \frac{4\omega_1(1+\omega_2)}{n\alpha}
	\end{eqnarray*}		
	and $\alpha = \frac{1}{\omega+1}$ we have $\EE\left[f(\bar{x}^K) - f(x^*)\right]$ of order
	\begin{equation*}
		\cO\left(\frac{\cL R_0^2(1+\omega_2)\left(1+\frac{\omega_1}{n}\right)}{K} + \frac{R_0\sigma_{1,0}(1+\omega_1)\sqrt{1+\omega_2}}{\sqrt{n}K} +  \frac{R_0\sqrt{(1+\omega_1)(1+\omega_2)D_{\text{opt}}}}{\sqrt{nK}}\right)
	\end{equation*}
	where $R_0 = \|x^0 - x^*\|^2, D_{\text{opt}} = \frac{1}{n}\sum\limits_{i=1}^n\EE_{\cD_i}\|\nabla f_{\xi_i}(x^*)-\nabla f_i(x^*)\|^2$. That is, to achive $\EE\left[f(\bar{x}^K) - f(x^*)\right] \le \varepsilon$ {\tt DIANAsq-DQ} requires
	\begin{equation*}
		\cO\left(\frac{\cL R_0^2(1+\omega_2)\left(1+\frac{\omega_1}{n}\right)}{\varepsilon} + \frac{R_0\sigma_{1,0}(1+\omega_1)\sqrt{1+\omega_2}}{\sqrt{n}\varepsilon} +  \frac{R_0^2(1+\omega_1)(1+\omega_2)D_{\text{opt}}}{n\varepsilon^2}\right)
	\end{equation*}
	iterations.
\end{corollary}

\subsection{Recovering Tight Complexity Bounds for {\tt VR-DIANA}}
In this section we consider the same problem \eqref{eq:main_problem}+\eqref{eq:f_i_sum} and variance reduced version of {\tt DIANA} called {\tt VR-DIANA} \cite{horvath2019stochastic}, see Algorithm~\ref{alg:vr-diana}.
\begin{algorithm}[t]
   \caption{{\tt VR-DIANA} based on {\tt LSVRG} (Variant 1), {\tt SAGA} (Variant 2), \cite{horvath2019stochastic}}
   \label{alg:vr-diana}
\begin{algorithmic}[1]
        \Require{learning rates $\alpha > 0$ and $\gamma > 0$, initial vectors $x^0, h_{1}^0, \dots, h_{n}^0$, $h^0 = \frac{1}{n}\sum_{i=1}^n h_i^0$}
        \For{$k = 0,1,\ldots$}
        \State Sample random 
            $
                u^k = \begin{cases}
                    1,& \text{with probability } \frac{1}{m}\\
                    0,& \text{with probability } 1 - \frac{1}{m}\\
                \end{cases}
            $ \Comment{only for Variant 1}
        \State Broadcast $x^k$, $u^k$ to all workers\;
            \For{$i = 1, \ldots, n$ in parallel} \Comment{Worker side}
            \State Pick $j_i^k$ uniformly at random from $[m]$\;
            \State $\mu_i^k = \frac{1}{m} \sum\limits_{j=1}^{m} \nabla f_{ij}(w_{ij}^k)$\label{ln:mu} \;
            \State $g_i^k = \nabla f_{ij_i^k}(x^k) - \nabla f_{ij_i^k}(w_{ij_i^k}^k) + \mu_i^k$\;
            \State $\hat{\Delta}_i^k = Q(g_i^k - h_i^k)$\;
            \State $h_i^{k+1} = h_i^k + \alpha \hat{\Delta}_i^k$\;
                \For{$j = 1, \ldots, m$}
                    %\tcp{Variant 1 (L-SVRG):}  
                    \State
                    $
                    w_{ij}^{k+1} =
                    \begin{cases}
                        x^k, & \text{if } u^k = 1 \\
                        w_{ij}^k, &\text{if } u^k = 0\\
                    \end{cases}
                    $ \Comment{Variant 1 (L-SVRG): update epoch gradient if $u^k = 1$}
                    \State
                    $
                    w_{ij}^{k+1} =
                    \begin{cases}
                    x^k, & j = j_i^k\\
                    w_{ij}^k, & j \neq j_i^k\\
                    \end{cases}
                    $ \Comment{Variant 2 (SAGA): update gradient table}
                \EndFor
            \EndFor
            \State $h^{k+1} \! = \! h^k \!+\! \frac{\alpha}{n} \displaystyle \sum_{i=1}^n \hat{\Delta}_i^k$ \Comment{Gather quantized updates} 
            \State $g^k = \frac{1}{n}\sum\limits_{i=1}^{n} (\hat{\Delta}_i^k + h_i^k)$\;
            \State $x^{k+1} = x^k - \gamma g^k$\;
        \EndFor
\end{algorithmic}  
\end{algorithm}
For simplicity we assume that each $f_{ij}$ is convex and $L$-smooth and $f_i$ is additionally $\mu$-strongly convex.
\begin{lemma}[Lemmas 3, 5, 6 and 7 from \cite{horvath2019stochastic}]\label{lemmas_vr_diana}
    Let $\alpha \le \frac{1}{\omega+1}$. Then for all iterates $k\ge 0$ of Algorithm~\ref{alg:vr-diana} the following inequalities hold:
    \begin{eqnarray}
        \EE\left[g^k\mid x^k\right] &=& \nabla f(x^k),\label{eq:unbiased_g_k_vr_diana}\\
        \EE\left[H^{k+1}\mid x^k\right] &\le& \left(1-\alpha\right)H^k + \frac{2\alpha}{m}D^k + 8\alpha Ln\left(f(x^k) - f(x^*)\right),\label{eq:H_k+1_bound_vr_diana}\\
        \EE\left[D^{k+1}\mid x^k\right] &\le& \left(1 - \frac{1}{m}\right)D^k + 2Ln\left(f(x^k) - f(x^*)\right),\label{eq:D_k+1_bound_vr_diana}\\
        \EE\left[\|g^k\|^2\mid x^k\right] &\le& 2L\left(1+\frac{4\omega + 2}{n}\right)\left(f(x^k)-f(x^*)\right) + \frac{2\omega}{n^2}\frac{D^k}{m} + \frac{2(\omega+1)}{n^2}H^k,\label{eq:second_moment_g_k_vr_diana}
    \end{eqnarray}
    where $H^k = \sum\limits_{i=1}^n\|h_i^k - \nabla f_i(x^*)\|^2$ and $D^k = \sum\limits_{i=1}^n\sum\limits_{j=1}^m\|\nabla f_{ij}(w_{ij}^k) - \nabla f_{ij}(x^*)\|^2$.
\end{lemma}
This lemma shows that {\tt VR-DIANA} satisfies \eqref{eq:second_moment_bound_new}, \eqref{eq:sigma_k+1_bound_1} and \eqref{eq:sigma_k+1_bound_2}. Applying Theorem~\ref{thm:main_result_sgd} we get the following result.
\begin{theorem}\label{thm:vr-diana}
	Assume that $f_{ij}(x)$ is convex and $L$-smooth for all $i=1,\ldots, n$ and $f_i(x)$ is $\mu$-strongly convex for all $i=1,\ldots,n$. Then {\tt VR-DIANA} satisfies Assumption~\ref{ass:key_assumption_new} with
	\begin{gather*}
		A' = L\left(1+\frac{4\omega+2}{n}\right),\quad B_1' = \frac{2(\omega+1)}{n},\quad D_1' = 0,\\
		\sigma_{1,k}^2 = H^{k} = \frac{1}{n}\sum\limits_{i=1}^n\|h_i^k - \nabla f_i(x^*)\|^2,\quad B_2' = \frac{2\omega}{n},\\
		\sigma_{2,k}^2 =  D^k = \frac{1}{nm}\sum\limits_{i=1}^n\sum\limits_{j=1}^m\|\nabla f_{ij}(w_{ij}^k) - \nabla f_{ij}(x^*)\|^2,\quad \rho_1 = \alpha,\quad \rho_2 = \frac{1}{m},\\
		C_1 = 4\alpha L,\quad C_2 = \frac{L}{m},\quad D_2 = 0,\quad G = 2,\quad F_1 = F_2 = 0,\quad D_3 = 0,
	\end{gather*}
	with $\gamma$ and $\alpha$ satisfying
	\begin{equation*}
		\gamma \le \frac{3}{L\left(\frac{41}{3}+\frac{52\omega+35}{n}\right)},\quad \alpha \le \frac{1}{\omega+1},\quad M_1 = \frac{8(\omega+1)}{3n\alpha},\quad M_2 = \frac{8\omega m}{3n} + \frac{32m}{9}
	\end{equation*}
	and for all $K \ge 0$
	\begin{equation*}
		\EE\left[f(\bar x^K) - f(x^*)\right] \le \left(1 - \min\left\{\frac{\gamma\mu}{2},\frac{\alpha}{4},\frac{1}{4m}\right\}\right)^K\frac{4T^0}{\gamma},
	\end{equation*}	
	when $\mu > 0$ and
	\begin{equation*}
		\EE\left[f(\bar{x}^K) - f(x^*)\right] \le \frac{4T^0}{\gamma K}
	\end{equation*}
	when $\mu=0$, where $T^k \eqdef \|x^k - x^*\|^2 + M_1\gamma^2 \sigma_{1,k}^2 + M_2\gamma^2\sigma_{2,k}^2$.
\end{theorem}
In other words, if $\mu > 0$ and
\begin{equation*}
		\gamma =\frac{3}{L\left(\frac{41}{3}+\frac{52\omega+35}{n}\right)},\quad \alpha = \frac{1}{\omega+1},
\end{equation*}
then {\tt VR-DIANA} converges with the linear rate
\begin{equation*}
	\cO\left(\left(\omega + m + \kappa\left(1+\frac{\omega}{n}\right)\right)\ln\frac{1}{\varepsilon}\right)
\end{equation*}
to the exact solution which coincides with the rate obtained in \cite{horvath2019stochastic}. We notice that the framework from \cite{gorbunov2019unified} establishes slightly worse guarantee:
\begin{equation*}
	\cO\left(\left(\omega + m + \kappa\left(1+\frac{\omega}{n}\right)\frac{\max\{m,\omega+1\}}{m}\right)\ln\frac{1}{\varepsilon}\right)
\end{equation*}
This guarantee is strictly worse than our bound when $m \le 1+\omega$. The key tool that helps us to improve the rate is two sequences of $\{\sigma_{1,k}^2\}_{k\ge 0}$, $\{\sigma_{2,k}^2\}_{k\ge 0}$ instead of one sequence $\{\sigma_{k}^2\}_{k\ge 0}$ as in \cite{gorbunov2019unified}.

Applying Lemma~\ref{lem:lemma_technical_cvx} we get the complexity result in the case when $\mu = 0$.
\begin{corollary}\label{cor:vr-diana_cvx_cor}
	Let the assumptions of Theorem~\ref{thm:vr-diana} hold and $\mu = 0$. Then after $K$ iterations of {\tt VR-DIANA} with the stepsize
	\begin{eqnarray*}
		\gamma_0 &=&\frac{3}{L\left(\frac{41}{3}+\frac{52\omega+35}{n}\right)}\\
		\gamma &=& \min\left\{\gamma_0, \sqrt{\frac{\|x^0 - x^*\|^2}{M_1\sigma_{1,0}^2+ M_2\sigma_{2,0}^2}}\right\},\quad  M_1 = \frac{8(\omega+1)}{3n\alpha},\quad M_2 = \frac{8\omega m}{3n} + \frac{32m}{9}
	\end{eqnarray*}		
	and $\alpha = \frac{1}{\omega+1}$ we have $\EE\left[f(\bar{x}^K) - f(x^*)\right]$ of order
	\begin{equation*}
		\cO\left(\frac{L R_0^2\left(1+\frac{\omega}{n}\right)}{K} + \frac{R_0\sqrt{\frac{(1+\omega)^2}{n}\sigma_{1,0}^2 + \left(1+\frac{\omega}{n}\right)m\sigma_{2,0}^2}}{K}\right)
	\end{equation*}
	where $R_0 = \|x^0 - x^*\|^2$. That is, to achive $\EE\left[f(\bar{x}^K) - f(x^*)\right] \le \varepsilon$ {\tt VR-DIANA} requires
	\begin{equation*}
		\cO\left(\frac{L R_0^2\left(1+\frac{\omega}{n}\right)}{\varepsilon} + \frac{R_0\sqrt{\frac{(1+\omega)^2}{n}\sigma_{1,0}^2 + \left(1+\frac{\omega}{n}\right)m\sigma_{2,0}^2}}{\varepsilon}\right)
	\end{equation*}
	iterations.
\end{corollary}

\clearpage

\section{Special Cases: Error Compensated Methods}\label{sec:special_cases}

\subsection{{\tt EC-SGDsr}}\label{sec:ec_SGDsr}
In this section we consider the same setup as in Section~\ref{sec:diana_arbitrary_sampling} and assume additionally that $f_1,\ldots,f_n$ are $L$-smooth.
%problem \eqref{eq:main_problem} with $f(x)$ being $\mu$-quasi strongly convex and $f_i(x)$ satisfying \eqref{eq:f_i_sum} where functions $f_{ij}(x)$ are differentiable, but not necessary convex. Following \cite{gower2019sgd} we construct a stochastic reformulation of this problem:
%\begin{equation}
%	f(x) = \EE_{\cD}\left[f_\xi(x)\right],\quad f_\xi(x) = \frac{1}{n}\sum\limits_{i=1}^n f_{\xi_i}(x),\quad f_{\xi_i}(x) = \frac{1}{m}\sum\limits_{j=1}^m \xi_{ij}f_{ij}(x), \label{eq:sr_def}
%\end{equation}
%where $\xi = (\xi_1^\top,\ldots, \xi_n^\top), \xi_i = (\xi_{i1},\ldots, \xi_{im})^\top$ is a random vector with distribution $\cD$ such that $\EE_{\cD}[\xi_{ij}] = 1$ for all $i\in[n], j\in[m]$ and the following assumption holds.
%\begin{assumption}[Expected smoothness]\label{ass:exp_smoothness}
%	We assume that functions $f_1,\ldots, f_n$ are $\cL$-smooth in expectation w.r.t.\ distribution $\cD$, i.e., there exists constant $\cL = \cL(f,\cD)$ such that
%	\begin{equation}
%		\EE_{\cD}\left[\|\nabla f_{\xi_i}(x) - \nabla f_{\xi_i}(x^*)\|^2\right] \le 2\cL D_{f_i}(x,x^*)\label{eq:exp_smoothness}
%	\end{equation}
%	for all $i\in [n]$ and $x\in\R^d$.
%\end{assumption}
\begin{algorithm}[t]
   \caption{{\tt EC-SGDsr}}\label{alg:ec-SGDsr}
\begin{algorithmic}[1]
   \Require learning rate $\gamma>0$, initial vector $x^0 \in \R^d$
	\State Set $e_i^0 = 0$ for all $i=1,\ldots, n$   
   \For{$k=0,1,\dotsc$}
       \State Broadcast $x^{k}$ to all workers
        \For{$i=1,\dotsc,n$ in parallel}
            \State Sample $g^{k}_i = \nabla f_{\xi_i}(x^k)$
            \State $v_i^k = C(e_i^k + \gamma g_i^k)$
            \State $e_i^{k+1} = e_i^k + \gamma g_i^k - v_i^k$
        \EndFor
        \State $e^k = \frac{1}{n}\sum_{i=1}^ne_i^k$, $g^k = \frac{1}{n}\sum_{i=1}^ng_i^k$, $v^k = \frac{1}{n}\sum_{i=1}^nv_i^k$
       \State $x^{k+1} = x^k - v^k$
   \EndFor
\end{algorithmic}
\end{algorithm}

\begin{lemma}\label{lem:key_lemma_ec-SGDsr}
	For all $k\ge 0$ we have
	\begin{eqnarray*}
		\frac{1}{n}\sum\limits_{i=1}^n\EE\left[\|g_i^k\|^2\mid x^k\right] &\le& 4L\left(f(x^k) - f(x^*)\right) + \frac{2}{n}\sum\limits_{i=1}^n\|\nabla f_{i}(x^*)\|^2, \notag\\
		\frac{1}{n}\sum\limits_{i=1}^n\EE\left[\|g_i^k-\bar{g}_i^k\|^2\mid x^k\right] &\le& 6(\cL + L)\left(f(x^k)-f(x^*)\right) + \frac{3}{n}\sum\limits_{i=1}^n\EE_{\cD}\left[\|\nabla f_{\xi_i}(x^*)-\nabla f_{i}(x^*)\|^2\right],\\
		\EE\left[\|g^k\|^2\mid x^k\right] &\le& 4\cL\left(f(x^k) - f(x^*)\right) + \frac{2}{n^2}\sum\limits_{i=1}^n\EE_{\cD}\left[\|\nabla f_{\xi_i}(x^*) - \nabla f_i(x^*)\|^2\right]. \notag
	\end{eqnarray*}
\end{lemma}
\begin{proof}
	Applying straightforward inequality $\|a+b\|^2 \le 2\|a\|^2 + 2\|b\|^2$ for $a,b\in\R^d$ we get
	\begin{eqnarray}
		\frac{1}{n}\sum\limits_{i=1}^n\|\bar{g}_i^k\|^2 &=& \frac{1}{n}\sum\limits_{i=1}^n\|\nabla f_i(x^k)-\nabla f_i(x^*)+\nabla f_i(x^*)\|^2\notag\\
		&\overset{\eqref{eq:a_b_norm_squared}}{\le}& \frac{1}{n}\sum\limits_{i=1}^n\|\nabla f_{i}(x^k) - \nabla f_{i}(x^*)\|^2 + \frac{2}{n}\sum\limits_{i=1}^n\|\nabla f_{i}(x^*)\|^2\notag\\
		&\overset{\eqref{eq:L_smoothness_cor}}{\le}& 4L\left(f(x^k) - f(x^*)\right) + \frac{2}{n}\sum\limits_{i=1}^n\|\nabla f_{i}(x^*)\|^2.\label{eq:ec_useful_technical_stuff}
	\end{eqnarray}
	Similarly we obtain
	\begin{eqnarray*}
		\frac{1}{n}\sum\limits_{i=1}^n\EE\left[\|g_i^k-\bar{g}_i^k\|^2\mid x^k\right] &=& \frac{1}{n}\sum\limits_{i=1}^n\EE_{\cD}\left[\|\nabla f_{\xi_i}(x^k)  - \nabla f_i(x^k)\|^2\right]\\
		&\overset{\eqref{eq:a_b_norm_squared}}{\le}& \frac{3}{n}\sum\limits_{i=1}^n\EE_{\cD}\left[\|\nabla f_{\xi_i}(x^k)-\nabla f_{\xi_i}(x^*)\|^2\right]\\
		&&\quad + \frac{3}{n}\sum\limits_{i=1}^n\EE_{\cD}\left[\|\nabla f_{\xi_i}(x^*)-\nabla f_{i}(x^*)\|^2\right] \\
		&&\quad + \frac{3}{n}\sum\limits_{i=1}^n\|\nabla f_i(x^*) - \nabla f_i(x^k)\|^2\\
		&\overset{\eqref{eq:L_smoothness_cor},\eqref{eq:exp_smoothness}}{\le}& 6(\cL + L)\left(f(x^k)-f(x^*)\right)\\
		&&\quad + \frac{3}{n}\sum\limits_{i=1}^n\EE_{\cD}\left[\|\nabla f_{\xi_i}(x^*)-\nabla f_{i}(x^*)\|^2\right].
	\end{eqnarray*}
	Next, using the independence of $\xi_1^k,\ldots, \xi_n^k$ we derive
	\begin{eqnarray*}
		\EE\left[\left\|g^k\right\|^2\mid x^k\right] &=& \EE\left[\left\|\frac{1}{n}\sum\limits_{i=1}^n\left(\nabla f_{\xi_i^k}(x^k) - \nabla f_{\xi_i^k}(x^*) + \nabla f_{\xi_i^k}(x^*) - \nabla f_i(x^*)\right)\right\|^2\mid x^k\right]\\
		&\overset{\eqref{eq:a_b_norm_squared}}{\le}& \frac{2}{n}\sum\limits_{i=1}^n\EE\left[\left\|\nabla f_{\xi_i^k}(x^k) - \nabla f_{\xi_i^k}(x^*)\right\|^2\mid x^k\right]\\
		&&\quad + 2\EE\left[\left\|\frac{1}{n}\sum\limits_{i=1}^n\left(\nabla f_{\xi_i^k}(x^*) - \nabla f_i(x^*)\right)\right\|^2\mid x^k\right]\\
		&\overset{\eqref{eq:exp_smoothness}}{\le}& 4\cL\left(f(x^k) - f(x^*)\right) + \frac{2}{n^2}\sum\limits_{i=1}^n\EE_{\cD_i}\left[\left\|\nabla f_{\xi_i}(x^*) - \nabla f_i(x^*)\right\|^2\right].
	\end{eqnarray*}
\end{proof}

Applying Theorem~\ref{thm:ec_sgd_main_result_new} we get the following result.
\begin{theorem}\label{thm:ec_SGDsr}
	Assume that $f(x)$ is $\mu$-quasi strongly convex, $f_1,\ldots,f_n$ are $L$-smooth and Assumption~\ref{ass:exp_smoothness} holds. Then {\tt EC-SGDsr} satisfies Assumption~\ref{ass:key_assumption_finite_sums_new} with
	\begin{gather*}
		A = 2L,\quad \widetilde{A} = 3(\cL+L),\quad A' = 2\cL,\quad B_1 = \widetilde{B}_1 = B_1' = B_2 = \widetilde{B}_2 = B_2' = 0,\\
		D_1 = \frac{2}{n}\sum\limits_{i=1}^n\|\nabla f_{i}(x^*)\|^2,\quad \widetilde{D}_1 = \frac{3}{n}\sum\limits_{i=1}^n\EE_{\cD}\left[\|\nabla f_{\xi_i}(x^*) - \nabla f_i(x^*)\|^2\right],\quad \sigma_{1,k}^2 \equiv \sigma_{2,k}^2 \equiv 0,\\
		D_1' = \frac{2}{n^2}\sum\limits_{i=1}^n\EE_{\cD}\left[\|\nabla f_{\xi_i}(x^*) - \nabla f_i(x^*)\|^2\right],\quad \rho_1 = \rho_2 = 1,\quad C_1 = C_2 = 0,\quad G = 0,\quad D_2 = 0,\\
		F_1 = F_2 = 0,\quad D_3 = \frac{6L\gamma}{\delta}\left(\frac{2D_1}{\delta}+\widetilde{D}_1\right),
	\end{gather*}
	with $\gamma$ satisfying
	\begin{equation*}
		\gamma \le \min\left\{\frac{1}{8\cL},\frac{\delta}{4\sqrt{6L\left(4L+3\delta(\cL+L)\right)}}\right\}
	\end{equation*}
	and for all $K \ge 0$
	\begin{equation*}
		\EE\left[f(\bar{x}^K) - f(x^*)\right] \le \left(1 - \frac{\gamma\mu}{2}\right)^K\frac{4\|x^0 - x^*\|^2}{\gamma} + 4\gamma\left(D_1' + \frac{12L\gamma}{\delta^2}D_1 + \frac{6L\gamma}{\delta}\widetilde{D}_1\right)
	\end{equation*}
	when $\mu > 0$ and
	\begin{equation*}
		\EE\left[f(\bar{x}^K) - f(x^*)\right] \le \frac{4\|x^0 - x^*\|^2}{K\gamma} + 4\gamma\left(D_1' + \frac{12L\gamma}{\delta^2}D_1 + \frac{6L\gamma}{\delta}\widetilde{D}_1\right)
	\end{equation*}
	when $\mu=0$.
\end{theorem}
In other words, {\tt EC-SGDsr} converges with linear rate $\cO\left(\left(\frac{\cL}{\mu} + \frac{L + \sqrt{\delta L\cL}}{\mu\delta}\right)\ln\frac{1}{\varepsilon}\right)$ to the neighbourhood of the solution when $\mu > 0$. Applying Lemma~\ref{lem:lemma2_stich} we establish the rate of convergence to $\varepsilon$-solution.
\begin{corollary}\label{cor:ec_SGDsr_str_cvx_cor}
	Let the assumptions of Theorem~\ref{thm:ec_SGDsr} hold and $\mu > 0$. Then after $K$ iterations of {\tt EC-SGDsr} with the stepsize
	\begin{equation*}
		\gamma = \min\left\{\frac{1}{8\cL},\frac{\delta}{4\sqrt{6L\left(4L+3\delta(\cL+L)\right)}}, \frac{\ln\left(\max\left\{2,\min\left\{\frac{\|x^0-x^*\|^2\mu^2K^2}{D_1'}, \frac{\delta\|x^0-x^*\|^2\mu^3K^3}{6L(\nicefrac{2D_1}{\delta}+\widetilde{D}_1)}\right\}\right\}\right)}{\mu K}\right\}
	\end{equation*}	 
	we have $\EE\left[f(\bar{x}^K) - f(x^*)\right]$ of order
	\begin{equation*}
		\widetilde\cO\left(\left(\cL + \frac{L+\sqrt{\delta L\cL}}{\delta}\right)\|x^0 - x^*\|^2\exp\left(-\frac{\mu}{\cL+\frac{L+\sqrt{\delta L\cL}}{\delta}}K\right) + \frac{D_1'}{\mu K} + \frac{L(\widetilde{D}_1 + \nicefrac{D_1}{\delta})}{\delta\mu^2 K^2}\right).
	\end{equation*}
	That is, to achive $\EE\left[f(\bar{x}^K) - f(x^*)\right] \le \varepsilon$ {\tt EC-SGDsr} requires
	\begin{equation*}
		\widetilde{\cO}\left(\frac{\cL}{\mu} + \frac{L+\sqrt{\delta L\cL}}{\delta\mu} + \frac{D_1'}{\mu\varepsilon} + \frac{\sqrt{L(\widetilde{D}_1 + \nicefrac{D_1}{\delta})}}{\mu\sqrt{\delta\varepsilon}}\right) \text{ iterations.}
	\end{equation*}
\end{corollary}

Applying Lemma~\ref{lem:lemma_technical_cvx} we get the complexity result in the case when $\mu = 0$.
\begin{corollary}\label{cor:ec_sgdsr_cvx_cor}
	Let the assumptions of Theorem~\ref{thm:ec_SGDsr} hold and $\mu = 0$. Then after $K$ iterations of {\tt EC-SGDsr} with the stepsize
	\begin{eqnarray*}
		\gamma_0 &=&\min\left\{\frac{1}{8\cL},\frac{\delta}{4\sqrt{6L\left(4L+3\delta(\cL+L)\right)}}\right\}\\
		\gamma &=& \min\left\{\gamma_0, \sqrt{\frac{\|x^0 - x^*\|^2}{D_1' K}}, \sqrt[3]{\frac{\|x^0 - x^*\|^2\delta}{6L(\nicefrac{2D_1}{\delta}+\widetilde{D}_1) K}}\right\}	\end{eqnarray*}		
	we have $\EE\left[f(\bar{x}^K) - f(x^*)\right]$ of order
	\begin{equation*}
		\cO\left(\frac{R_0^2\left(\cL+\frac{L+\sqrt{\delta L\cL}}{\delta}\right)}{K} + \sqrt{\frac{R_0^2 D_1'}{K}} + \frac{\sqrt[3]{LR_0^4(\nicefrac{2D_1}{\delta}+\widetilde{D}_1)}}{\left(\delta K^2\right)^{\nicefrac{1}{3}}}\right)
	\end{equation*}
	where $R_0 = \|x^0 - x^*\|^2$. That is, to achive $\EE\left[f(\bar{x}^K) - f(x^*)\right] \le \varepsilon$ {\tt EC-SGDsr} requires
	\begin{equation*}
		\cO\left(\frac{R_0^2\left(\cL+\frac{L+\sqrt{\delta L\cL}}{\delta}\right)}{\varepsilon} + \frac{R_0^2 D_1'}{\varepsilon^2} + \frac{R_0^2\sqrt{L(\nicefrac{2D_1}{\delta}+\widetilde{D}_1)}}{\sqrt{\delta \varepsilon^3}}\right)
	\end{equation*}
	iterations.
\end{corollary}

\subsection{{\tt EC-SGD}}\label{sec:ec_sgd_pure}
In this section we consider problem \eqref{eq:main_problem} with $f_i(x)$ satisfying \eqref{eq:f_i_expectation} where functions $f_{\xi_i}(x)$ are differentiable and $L$-smooth almost surely in $\xi_i$, $i=1,\ldots,n$.
\begin{algorithm}[t]
   \caption{{\tt EC-SGD}}\label{alg:ec-sgd}
\begin{algorithmic}[1]
   \Require learning rate $\gamma>0$, initial vector $x^0 \in \R^d$
	\State Set $e_i^0 = 0$ for all $i=1,\ldots, n$   
   \For{$k=0,1,\dotsc$}
       \State Broadcast $x^{k}$ to all workers
        \For{$i=1,\dotsc,n$ in parallel}
            \State Sample $g^{k}_i = \nabla f_{\xi_i}(x^k)$ independently from other workers
            \State $v_i^k = C(e_i^k + \gamma g_i^k)$
            \State $e_i^{k+1} = e_i^k + \gamma g_i^k - v_i^k$
        \EndFor
        \State $e^k = \frac{1}{n}\sum_{i=1}^ne_i^k$, $g^k = \frac{1}{n}\sum_{i=1}^ng_i^k$, $v^k = \frac{1}{n}\sum_{i=1}^nv_i^k$
       \State $x^{k+1} = x^k - v^k$
   \EndFor
\end{algorithmic}
\end{algorithm}

\begin{lemma}[See also Lemmas 1,2 from~\citep{nguyen2018sgd}]\label{lem:lemma_sgd}
    Assume that $f_{\xi_i}(x)$ are convex in $x$ for every $\xi_i$, $i=1,\ldots,n$. Then for every $x\in\R^d$ and $i=1,\ldots, n$
	\begin{eqnarray*}
		\frac{1}{n}\sum\limits_{i=1}^n\|\nabla f_{i}(x)\|^2 &\le& 4L\left(f(x) - f(x^*)\right) + \frac{2}{n}\sum\limits_{i=1}^n\|\nabla f_{i}(x^*)\|^2,\\
		\frac{1}{n}\sum\limits_{i=1}^n\EE_{\xi_i\sim\cD_i}{\|\nabla f_{\xi_i}(x)-\nabla f_i(x)\|^2} &\le& 12L\left(f(x) - f(x^*)\right) + \frac{3}{n}\sum\limits_{i=1}^n\EE\left[\|\nabla f_{\xi_i}(x^*)-\nabla f_{i}(x^*)\|^2\right],\\
		\EE_{\xi_1,\ldots,\xi_n}{\left\|\frac{1}{n}\sum\limits_{i=1}^n\nabla f_{\xi_i}(x)\right\|^2} &\le& 4L\left(f(x) - f(x^*)\right) + \frac{2}{n^2}\sum\EE\left[\|\nabla f_{\xi_i}(x^*)-\nabla f_i(x^*)\|^2\right].
	\end{eqnarray*}	    
    If further $f(x)$ is $\mu$-strongly convex with $\mu > 0$ and possibly non-convex $f_i,f_{\xi_i}$, then for every $x\in\R^d$ and $i = 1,\ldots, n$
    \begin{eqnarray*}
    	\frac{1}{n}\sum\limits_{i=1}^n\|\nabla f_{i}(x)\|^2 &\le& 4L\kappa\left(f(x) - f(x^*)\right) + \frac{2}{n}\sum\limits_{i=1}^n\|\nabla f_{i}(x^*)\|^2,\\
		\frac{1}{n}\sum\limits_{i=1}^n\EE_{\xi_i\sim\cD_i}{\|\nabla f_{\xi_i}(x)-\nabla f_i(x)\|^2} &\le& 12L\kappa\left(f(x) - f(x^*)\right) + \frac{3}{n}\sum\limits_{i=1}^n\EE\left[\|\nabla f_{\xi_i}(x^*)-\nabla f_{i}(x^*)\|^2\right],\\
    	\EE_{\xi_1,\ldots,\xi_n}{\left\|\frac{1}{n}\sum\limits_{i=1}^n\nabla f_{\xi_i}(x)\right\|^2} &\le& 4L\kappa\left(f(x) - f(x^*)\right) + \frac{2}{n^2}\sum\limits_{i=1}^n\EE\left[\|\nabla f_{\xi_i}(x^*)-\nabla f_i(x^*)\|^2\right].
    \end{eqnarray*}
    where $\kappa = \frac{L}{\mu}$.
\end{lemma}
\begin{proof}
	We start with the case when functions $f_{\xi_i}(x)$ are convex in $x$ for every $\xi_i$. The first inequality follows from \eqref{eq:ec_useful_technical_stuff}. Next, we derive
	\begin{eqnarray*}
		\frac{1}{n}\sum\limits_{i=1}^n\EE_{\xi_i\sim\cD_i}{\|\nabla f_{\xi_i}(x)-\nabla f_i(x)\|^2} &\overset{\eqref{eq:a_b_norm_squared}}{\le}& \frac{3}{n} \sum\limits_{i=1}^n\EE_{\xi_i\sim\cD_i}{\|\nabla f_{\xi_i}(x) - \nabla f_{\xi_i}(x^*)\|^2} \\
		&&\quad + \frac{3}{n}\sum\limits_{i=1}^n\EE_{\xi_i\sim\cD_i}{\|\nabla f_{\xi_i}(x^*) - \nabla f_i(x^*)\|^2}\\
		&&\quad + \frac{3}{n}\sum\limits_{i=1}^n{\|\nabla f_{i}(x^*) - \nabla f_i(x)\|^2}\\
		&\overset{\eqref{eq:L_smoothness_cor}}{\le}& 12L\left(f(x)-f(x^*)\right) + \frac{3}{n}\sum\limits_{i=1}^n\EE{\|\nabla f_{\xi_i}(x^*)\|^2}.
	\end{eqnarray*}
	Due to independence of $\xi_1^k,\ldots,\xi_n^k$ we get
	\begin{eqnarray*}
		\EE_{\xi_1,\ldots,\xi_n}{\left\|\frac{1}{n}\sum\limits_{i=1}^n\nabla f_{\xi_i}(x)\right\|^2} &=& \EE_{\xi_1,\ldots,\xi_n}{\left\|\frac{1}{n}\sum\limits_{i=1}^n\left(\nabla f_{\xi_i}(x)-\nabla f_{\xi_i}(x^*) + \nabla f_{\xi_i}(x^*) - \nabla f_i(x^*)\right)\right\|^2}\\
		&\overset{\eqref{eq:a_b_norm_squared}}{\le}& \frac{2}{n}\sum\limits_{i=1}^n\EE_{\xi_i\sim\cD_i}\left[\|\nabla f_{\xi_i}(x) - \nabla f_{\xi_i}(x^*)\|^2\right]\\
		&&\quad + 2\EE_{\xi_1,\ldots,\xi_n}{\left\|\frac{1}{n}\sum\limits_{i=1}^n\left(\nabla f_{\xi_i}(x^*) - \nabla f_i(x^*)\right)\right\|^2}\\
		&\overset{\eqref{eq:L_smoothness_cor}}{\le}& 4L\left(f(x)-f(x^*)\right) + \frac{2}{n^2}\sum\EE\left[\|\nabla f_{\xi_i}(x^*)-\nabla f_i(x^*)\|^2\right].
	\end{eqnarray*}
	
	Next, we consider the second case: $f(x)$ is $\mu$-strongly convex with possibly non-convex $f_i,f_{\xi_i}$. In this case
	\begin{eqnarray*}
		\frac{1}{n}\sum\limits_{i=1}^n{\|\nabla f_{i}(x)\|^2} &\overset{\eqref{eq:a_b_norm_squared}}{\le}& \frac{2}{n} \sum\limits_{i=1}^n{\|\nabla f_{i}(x) - \nabla f_{i}(x^*)\|^2} + \frac{2}{n}\sum\limits_{i=1}^n{\|\nabla f_{i}(x^*)\|^2}\\
		&\overset{\eqref{eq:L_smoothness}}{\le}& 2L^2\|x - x^*\|^2 + \frac{2}{n}\sum\limits_{i=1}^n{\|\nabla f_{i}(x^*)\|^2}\\
		&\le& \frac{4L^2}{\mu}\left(f(x)-f(x^*)\right) + \frac{2}{n}\sum\limits_{i=1}^n{\|\nabla f_{i}(x^*)\|^2}
	\end{eqnarray*}
	where the last inequality follows from $\mu$-strong convexity of $f$. Similarly, we get
	\begin{eqnarray*}
		\frac{1}{n}\sum\limits_{i=1}^n\EE_{\xi_i\sim\cD_i}\left[{\|\nabla f_{\xi_i}(x)-\nabla f_{i}(x)\|^2}\right] &\overset{\eqref{eq:a_b_norm_squared}}{\le}& \frac{3}{n} \sum\limits_{i=1}^n\EE_{\xi_i\sim\cD_i}\left[{\|\nabla f_{\xi_i}(x)-\nabla f_{\xi_i}(x^*)\|^2}\right]\\
		&&\quad + \frac{3}{n} \sum\limits_{i=1}^n\EE_{\xi_i\sim\cD_i}\left[{\|\nabla f_{\xi_i}(x^*)-\nabla f_{i}(x^*)\|^2}\right]\\
		&&\quad + \frac{3}{n} \sum\limits_{i=1}^n{\|\nabla f_{i}(x^*)-\nabla f_{i}(x)\|^2}\\
		&\overset{\eqref{eq:L_smoothness}}{\le}& 6L^2\|x - x^*\|^2 \\
		&&\quad + \frac{3}{n} \sum\limits_{i=1}^n\EE_{\xi_i\sim\cD_i}\left[{\|\nabla f_{\xi_i}(x^*)-\nabla f_{i}(x^*)\|^2}\right]\\
		&\le& \frac{12L^2}{\mu}\left(f(x)-f(x^*)\right)\\
		&&\quad + \frac{3}{n} \sum\limits_{i=1}^n\EE_{\xi_i\sim\cD_i}\left[{\|\nabla f_{\xi_i}(x^*)-\nabla f_{i}(x^*)\|^2}\right].
	\end{eqnarray*}
	Finally, using independence of $\xi_1^k,\ldots,\xi_n^k$ we derive
	\begin{eqnarray*}
		\EE_{\xi_1,\ldots,\xi_n}{\left\|\frac{1}{n}\sum\limits_{i=1}^n\nabla f_{\xi_i}(x)\right\|^2} &=& \EE_{\xi_1,\ldots,\xi_n}{\left\|\frac{1}{n}\sum\limits_{i=1}^n\left(\nabla f_{\xi_i}(x)-\nabla f_{\xi_i}(x^*) + \nabla f_{\xi_i}(x^*) - \nabla f_i(x^*)\right)\right\|^2}\\
		&\overset{\eqref{eq:a_b_norm_squared}}{\le}& \frac{2}{n}\sum\limits_{i=1}^n\EE_{\xi_i\sim\cD_i}\left[\|\nabla f_{\xi_i}(x) - \nabla f_{\xi_i}(x^*)\|^2\right]\\
		&&\quad + 2\EE_{\xi_1,\ldots,\xi_n}{\left\|\frac{1}{n}\sum\limits_{i=1}^n\left(\nabla f_{\xi_i}(x^*) - \nabla f_i(x^*)\right)\right\|^2}\\
		&\overset{\eqref{eq:L_smoothness}}{\le}& 2L^2\|x-x^*\|^2 + \frac{2}{n^2}\sum\EE\left[\|\nabla f_{\xi_i}(x^*)-\nabla f_i(x^*)\|^2\right]\\
		&\le& \frac{4L^2}{\mu}\left(f(x)-f(x^*)\right) + \frac{2}{n^2}\sum\EE\left[\|\nabla f_{\xi_i}(x^*)-\nabla f_i(x^*)\|^2\right].
	\end{eqnarray*}
\end{proof}

Applying Theorem~\ref{thm:ec_sgd_main_result_new} we get the following result.
\begin{theorem}\label{thm:ec_sgd_pure}
	Assume that $f_\xi(x)$ is convex and $L$-smooth in $x$ for every $\xi$ and $f(x)$ is $\mu$-quasi strongly convex. Then {\tt EC-SGD} satisfies Assumption~\ref{ass:key_assumption_finite_sums_new} with
	\begin{gather*}
		A = A' = 2L,\quad \widetilde{A} = 6L,\quad B_1 = \widetilde{B}_1 = B_1' = B_2 = \widetilde{B}_2 = B_2' = 0,\\
		D_1 = \frac{2}{n}\sum\limits_{i=1}^n\|\nabla f_{i}(x^*)\|^2,\quad \widetilde{D}_1 = \frac{2}{n}\sum\limits_{i=1}^n\EE\left[\|\nabla f_{\xi_i}(x^*) - \nabla f_i(x^*)\|^2\right],\quad \sigma_{1,k}^2 \equiv \sigma_{2,k}^2 \equiv 0,\\
		D_1' = \frac{2}{n^2}\sum\limits_{i=1}^n\EE\left[\|\nabla f_{\xi_i}(x^*) - \nabla f_i(x^*)\|^2\right],\quad \rho_1 = \rho_2 = 1,\quad C_1 = C_2 = 0,\quad G = 0,\quad D_2 = 0,\\
		F_1 = F_2 = 0,\quad D_3 = \frac{6L\gamma}{\delta}\left(\frac{2D_1}{\delta}+\widetilde{D}_1\right),
	\end{gather*}
	with $\gamma$ satisfying
	\begin{equation*}
		\gamma \le \frac{\delta}{8L\sqrt{6+9\delta}}
	\end{equation*}
	and for all $K \ge 0$
	\begin{equation*}
		\EE\left[f(\bar{x}^K) - f(x^*)\right] \le \left(1 - \frac{\gamma\mu}{2}\right)^K\frac{4\|x^0 - x^*\|^2}{\gamma} + 4\gamma\left(D_1' + \frac{12L\gamma}{\delta^2}D_1 + \frac{6L\gamma}{\delta}\widetilde{D}_1\right)
	\end{equation*}
	when $\mu > 0$ and
	\begin{equation*}
		\EE\left[f(\bar{x}^K) - f(x^*)\right] \le \frac{4\|x^0 - x^*\|^2}{K\gamma} + 4\gamma\left(D_1' + \frac{12L\gamma}{\delta^2}D_1 + \frac{6L\gamma}{\delta}\widetilde{D}_1\right)
	\end{equation*}
	when $\mu=0$. If further $f(x)$ is $\mu$-strongly convex with $\mu > 0$ and possibly non-convex $f_i,f_{\xi_i}$, then
	{\tt EC-SGD} satisfies Assumption~\ref{ass:key_assumption_finite_sums_new} with
	\begin{gather*}
		A = A' = 2L\kappa,\quad \widetilde{A} = 6L\kappa,\quad B_1 = \widetilde{B}_1 = B_1' = B_2 = \widetilde{B}_2 = B_2' = 0,\\
		D_1 = \frac{2}{n}\sum\limits_{i=1}^n\|\nabla f_{i}(x^*)\|^2,\quad \widetilde{D}_1 = \frac{2}{n}\sum\limits_{i=1}^n\EE\left[\|\nabla f_{\xi_i}(x^*) - \nabla f_i(x^*)\|^2\right],\quad \sigma_{1,k}^2 \equiv \sigma_{2,k}^2 \equiv 0,\\
		D_1' = \frac{2}{n^2}\sum\limits_{i=1}^n\EE\left[\|\nabla f_{\xi_i}(x^*) - \nabla f_i(x^*)\|^2\right],\quad \rho_1 = \rho_2 = 1,\quad C_1 = C_2 = 0,\quad G = 0,\quad D_2 = 0,\\
		F_1 = F_2 = 0,\quad D_3 = \frac{6L\gamma}{\delta}\left(\frac{2D_1}{\delta}+\widetilde{D}_1\right),
	\end{gather*}
	with $\gamma$ satisfying
	\begin{equation*}
		\gamma \le \min\left\{\frac{1}{8\kappa L},\frac{\delta}{8L\sqrt{3\kappa(2+3\delta)}}\right\}
	\end{equation*}
	and for all $K \ge 0$
	\begin{equation*}
		\EE\left[f(\bar{x}^K) - f(x^*)\right] \le \left(1 - \frac{\gamma\mu}{2}\right)^K\frac{4\|x^0 - x^*\|^2}{\gamma} + 4\gamma\left(D_1' + \frac{12L\gamma}{\delta^2}D_1 + \frac{6L\gamma}{\delta}\widetilde{D}_1\right).
	\end{equation*}
\end{theorem}

In other words, {\tt EC-SGD} converges with linear rate $\cO\left(\frac{\kappa}{\delta}\ln\frac{1}{\varepsilon}\right)$ to the neighbourhood of the solution when $f_{\xi}(x)$ are convex for each $\xi$ and $\mu > 0$. Applying Lemma~\ref{lem:lemma2_stich} we establish the rate of convergence to $\varepsilon$-solution.
\begin{corollary}\label{cor:ec_SGD_pure_str_cvx_cor}
	Let the assumptions of Theorem~\ref{thm:ec_sgd_pure} hold, $f_{\xi}(x)$ are convex for each $\xi$ and $\mu > 0$. Then after $K$ iterations of {\tt EC-SGD} with the stepsize
	\begin{equation*}
		\gamma = \min\left\{\frac{\delta}{8L\sqrt{6+9\delta}}, \frac{\ln\left(\max\left\{2,\min\left\{\frac{\|x^0-x^*\|^2\mu^2K^2}{D_1'}, \frac{\delta\|x^0-x^*\|^2\mu^3K^3}{6L(\nicefrac{2D_1}{\delta}+\widetilde{D}_1)}\right\}\right\}\right)}{\mu K}\right\}
	\end{equation*}	 
	we have
	\begin{equation*}
		\EE\left[f(\bar{x}^K) - f(x^*)\right] = \widetilde\cO\left(\frac{L}{\delta}\|x^0 - x^*\|^2\exp\left(-\frac{\delta\mu}{L}K\right) + \frac{D_1'}{\mu K} + \frac{L(\widetilde{D}_1 + \nicefrac{D_1}{\delta})}{\delta\mu^2 K^2}\right).
	\end{equation*}
	That is, to achive $\EE\left[f(\bar{x}^K) - f(x^*)\right] \le \varepsilon$ {\tt EC-SGD} requires
	\begin{equation*}
		\widetilde{\cO}\left(\frac{L}{\delta\mu} + \frac{D_1'}{\mu\varepsilon} + \frac{\sqrt{L(\widetilde{D}_1 + \nicefrac{D_1}{\delta})}}{\mu\sqrt{\delta\varepsilon}}\right) \text{ iterations.}
	\end{equation*}
\end{corollary}
\begin{corollary}\label{cor:ec_SGD_pure_str_cvx_cor_2}
	Let the assumptions of Theorem~\ref{thm:ec_sgd_pure} hold and $f(x)$ is $\mu$-strongly convex with $\mu > 0$ and possibly non-convex $f_i,f_{\xi_i}$. Then after $K$ iterations of {\tt EC-SGD} with the stepsize
	\begin{equation*}
		\gamma = \min\left\{\frac{1}{8\kappa L}, \frac{\delta}{8L\sqrt{3\kappa(2+3\delta)}}, \frac{\ln\left(\max\left\{2,\min\left\{\frac{\|x^0-x^*\|^2\mu^2K^2}{D_1'}, \frac{\delta\|x^0-x^*\|^2\mu^3K^3}{6L(\nicefrac{2D_1}{\delta}+\widetilde{D}_1)}\right\}\right\}\right)}{\mu K}\right\}
	\end{equation*}	 
	we have $\EE\left[f(\bar{x}^K) - f(x^*)\right]$ of order
	\begin{equation*}
		\widetilde\cO\left(\left(L\kappa+\frac{L\sqrt{\kappa}}{\delta}\right)\|x^0 - x^*\|^2\exp\left(-\min\left\{\frac{\delta\mu}{L\sqrt{\kappa}}, \frac{1}{\kappa^2}\right\}K\right) + \frac{D_1'}{\mu K} + \frac{L(\widetilde{D}_1 + \nicefrac{D_1}{\delta})}{\delta\mu^2 K^2}\right).
	\end{equation*}
	That is, to achive $\EE\left[f(\bar{x}^K) - f(x^*)\right] \le \varepsilon$ {\tt EC-SGD} requires
	\begin{equation*}
		\widetilde{\cO}\left(\kappa^2 + \frac{\kappa^{\nicefrac{3}{2}}}{\delta} + \frac{D_1'}{\mu\varepsilon} + \frac{\sqrt{L(\widetilde{D}_1 + \nicefrac{D_1}{\delta})}}{\mu\sqrt{\delta\varepsilon}}\right) \text{ iterations.}
	\end{equation*}
\end{corollary}

Applying Lemma~\ref{lem:lemma_technical_cvx} we get the complexity result in the case when $\mu = 0$.
\begin{corollary}\label{cor:ec_sgd_cvx_cor}
	Let the assumptions of Theorem~\ref{thm:ec_sgd_pure} hold, $f_{\xi}(x)$ are convex for each $\xi$ and $\mu = 0$. Then after $K$ iterations of {\tt EC-SGD} with the stepsize
	\begin{eqnarray*}
		\gamma &=& \min\left\{\frac{\delta}{8L\sqrt{6+9\delta}}, \sqrt{\frac{\|x^0 - x^*\|^2}{D_1' K}}, \sqrt[3]{\frac{\|x^0 - x^*\|^2\delta}{6L(\nicefrac{2D_1}{\delta}+\widetilde{D}_1) K}}\right\}	\end{eqnarray*}		
	we have $\EE\left[f(\bar{x}^K) - f(x^*)\right]$ of order
	\begin{equation*}
		\cO\left(\frac{LR_0^2}{\delta K} + \sqrt{\frac{R_0^2 D_1'}{K}} + \frac{\sqrt[3]{LR_0^4(\nicefrac{2D_1}{\delta}+\widetilde{D}_1)}}{\left(\delta K^2\right)^{\nicefrac{1}{3}}}\right)
	\end{equation*}
	where $R_0 = \|x^0 - x^*\|^2$. That is, to achive $\EE\left[f(\bar{x}^K) - f(x^*)\right] \le \varepsilon$ {\tt EC-SGD} requires
	\begin{equation*}
		\cO\left(\frac{LR_0^2}{\delta\varepsilon} + \frac{R_0^2 D_1'}{\varepsilon^2} + \frac{R_0^2\sqrt{L(\nicefrac{2D_1}{\delta}+\widetilde{D}_1)}}{\sqrt{\delta \varepsilon^3}}\right)
	\end{equation*}
	iterations.
\end{corollary}

\subsection{{\tt EC-GDstar}}\label{sec:ec_SGDstar}
We assume that $i$-th node has access to the gradient of $f_i$ at the optimality, i.e., to the $\nabla f_i(x^*)$. It is unrealistic scenario but it gives some insights that we will use next in order to design the method that converges asymptotically to \textit{the exact solution}.
\begin{algorithm}[t]
   \caption{{\tt EC-GDstar} (see also \cite{gorbunov2019unified})}\label{alg:EC-GDstar}
\begin{algorithmic}[1]
   \Require learning rate $\gamma>0$, initial vector $x^0 \in \R^d$
	\State Set $e_i^0 = 0$ for all $i=1,\ldots, n$   
   \For{$k=0,1,\dotsc$}
       \State Broadcast $x^{k}$ to all workers
        \For{$i=1,\dotsc,n$ in parallel}
            \State $g^{k}_i = \nabla f_{i}(x^k) - \nabla f_i(x^*)$
            \State $v_i^k = C(e_i^k + \gamma g_i^k)$
            \State $e_i^{k+1} = e_i^k + \gamma g_i^k - v_i^k$
        \EndFor
        \State $e^k = \frac{1}{n}\sum_{i=1}^ne_i^k$, $g^k = \frac{1}{n}\sum_{i=1}^ng_i^k$, $v^k = \frac{1}{n}\sum_{i=1}^nv_i^k$
       \State $x^{k+1} = x^k - v^k$
   \EndFor
\end{algorithmic}
\end{algorithm}

Assume that $f(x)$ is $\mu$-quasi strongly convex and each $f_i$ is convex and $L$-smooth. By definition of $g_i^k$ it trivially follows that
\begin{equation*}
	g^k = \frac{1}{n}\sum\limits_{i=1}^n g_i^k = \frac{1}{n}\sum\limits_{i=1}^n \left(\nabla f_i(x^k) - \nabla f_i(x^*)\right) = \nabla f(x^k) - \nabla f(x^*) = \nabla f(x^k),
\end{equation*}
$g_i^k = \bar{g}_i^k$, and
\begin{eqnarray*}
	\frac{1}{n}\sum\limits_{i=1}^n\|g_i^k\|^2 &=& \frac{1}{n}\sum\limits_{i=1}^n\|\nabla f_i(x^k) - \nabla f_i(x^*)\|^2\\
	&\overset{\eqref{eq:L_smoothness_cor}}{\le}& \frac{2L}{n}\sum\limits_{i=1}^n\left(f_i(x^k) - f_i(x^*) - \langle\nabla f_i(x^*), x^k - x^*\rangle\right) = 2L\left(f(x^k) - f(x^*)\right),\\
	\|g^k\|^2 &=& \|\nabla f(x^k)\|^2 \overset{\eqref{eq:L_smoothness_cor}}{\le} 2L\left(f(x^k) - f(x^*)\right).
\end{eqnarray*}

Applying Theorem~\ref{thm:ec_sgd_main_result_new} we get the following result.
\begin{theorem}\label{thm:ec_sgd_star}
	Assume that $f_i(x)$ is convex and $L$-smooth for all $i=1,\ldots, n$ and $f(x)$ is $\mu$-quasi strongly convex. Then {\tt EC-GDstar} satisfies Assumption~\ref{ass:key_assumption_finite_sums_new} with
	\begin{gather*}
		A = A' = L,\quad \widetilde{A} = 0,\quad B_1 = B_2 = \widetilde{B}_1 = \widetilde{B}_2 = B_1' = B_2' = 0,\\
		 D_1 = \widetilde{D}_1 = D_1' = 0,\quad \sigma_{1,k}^2 \equiv \sigma_{2,k}^2 \equiv 0,\\
		\rho_1 = \rho_2 = 1,\quad C_1 = C_2 = 0,\quad G = 0,\quad D_2 = 0,\quad F_1 = F_2 = 0,\quad D_3 = 0,
	\end{gather*}
	with $\gamma$ satisfying
	\begin{equation*}
		\gamma \le \frac{\delta}{8L\sqrt{3}}
	\end{equation*}
	and for all $K \ge 0$
	\begin{equation*}
		\EE\left[f(\bar{x}^K) - f(x^*)\right] \le \left(1 - \frac{\gamma\mu}{2}\right)^K\frac{4\|x^0 - x^*\|^2}{\gamma},
	\end{equation*}
	when $\mu > 0$ and
	\begin{equation*}
		\EE\left[f(\bar{x}^K) - f(x^*)\right] \le \frac{4\|x^0 - x^*\|^2}{K\gamma}
	\end{equation*}
	when $\mu=0$.
\end{theorem}
In other words, {\tt EC-GDstar} converges with linear rate $\cO\left(\frac{\kappa}{\delta}\ln\frac{1}{\varepsilon}\right)$ to the exact solution when $\mu > 0$ removing the drawback of {\tt EC-SGD} and {\tt EC-GD}. If $\mu = 0$ then the rate of convergence is $\cO\left(\frac{L\|x^0-x^*\|^2}{\delta\varepsilon}\right)$. However, {\tt EC-GDstar} relies on the fact that $i$-th node knows $\nabla f_i(x^*)$ which is not realistic.

\subsection{{\tt EC-SGD-DIANA}}\label{sec:ec_diana}
In this section we present a new method that converges to the exact optimum asymptotically but does not need to know $\nabla f_i(x^*)$ and instead of this it learns the gradients at the optimum. This method is inspired by another method called {\tt DIANA} (see \cite{mishchenko2019distributed, horvath2019stochastic}).
\begin{algorithm}[t]
   \caption{{\tt EC-SGD-DIANA}}\label{alg:EC-SGD-DIANA}
\begin{algorithmic}[1]
   \Require learning rates $\gamma>0$, $\alpha \in (0,1]$, initial vectors $x^0, h_1^0,\ldots, h_n^0 \in \R^d$
	\State Set $e_i^0 = 0$ for all $i=1,\ldots, n$   
	\State Set $h^0 = \frac{1}{n}\sum_{i=1}^n h_i^0$   
   \For{$k=0,1,\dotsc$}
       \State Broadcast $x^{k}, h^k$ to all workers
        \For{$i=1,\dotsc,n$ in parallel}
			\State Sample $\hat g_i^k$ such that $\EE[\hat g_i^k\mid x^k] = \nabla f_i(x^k)$ and $\EE\left[\|\hat g_i^k - \nabla f_i(x^k)\|^2\mid x^k\right] \le \widetilde{D}_{1,i}$ independently from other workers        
            \State $g^{k}_i = \hat g_i^k - h_i^k + h^k$
            \State $v_i^k = C(e_i^k + \gamma g_i^k)$
            \State $e_i^{k+1} = e_i^k + \gamma g_i^k - v_i^k$
            \State $h_i^{k+1} = h_i^k + \alpha Q(\hat g_i^k - h_i^k)$
        \EndFor
        \State $e^k = \frac{1}{n}\sum_{i=1}^ne_i^k$, $g^k = \frac{1}{n}\sum_{i=1}^ng_i^k$, $v^k = \frac{1}{n}\sum_{i=1}^nv_i^k$, $h^{k+1} = \frac{1}{n}\sum\limits_{i=1}^n h_i^{k+1} = h^k + \alpha\frac{1}{n}\sum\limits_{i=1}^n Q(\hat g_i^k - h_i^k)$
       \State $x^{k+1} = x^k - v^k$
   \EndFor
\end{algorithmic}
\end{algorithm}

We notice that master needs to gather only $C(e_i^k + \gamma g_i^k)$ and $Q(\hat g_i^k - h_i^k)$ from all nodes in order to perform an update.

\begin{lemma}\label{lem:ec_diana_second_moment_bound}
	Assume that $f_i(x)$ is convex and $L$-smooth for all $i=1,\ldots,n$. Then, for all $k\ge 0$ we have
	\begin{eqnarray}
		\EE\left[g^k\mid x^k\right] &=& \nabla f(x^k), \label{eq:ec_diana_unbiasedness}\\
		\frac{1}{n}\sum\limits_{i=1}^n\|\bar{g}_i^k\|^2 &\le& 4L\left(f(x^k) - f(x^*)\right) + 2\sigma_k^2, \label{eq:ec_diana_second_moment_bound}\\
		\frac{1}{n}\sum\limits_{i=1}^n\EE\left[\|g_i^k-\bar{g}_i^k\|^2\mid x^k\right] &\le& \widetilde{D}_1, \label{eq:ec_diana_variance_bound}\\
		\EE\left[\|g^k\|^2\mid x^k\right] &\le& 2L\left(f(x^k) - f(x^*)\right) + \frac{\widetilde{D}_1}{n} \label{eq:ec_diana_second_moment_bound_2}
	\end{eqnarray}
	where $\widetilde{D}_1 = \frac{1}{n}\sum_{i=1}^n \widetilde{D}_{1,i}$ and $\sigma_k^2 = \frac{1}{n}\sum_{i=1}^n\|h_i^k - \nabla f(x^*)\|^2$.
\end{lemma}
\begin{proof}
	First of all, we show unbiasedness of $g^k$:
	\begin{equation*}
		\EE\left[g^k\mid x^k\right] = \frac{1}{n}\sum\limits_{i=1}^n\EE\left[g_i^k\mid x^k\right] = \frac{1}{n}\sum\limits_{i=1}^n\left(\nabla f_i(x^k) - h_i^k + h^k\right) = \nabla f(x^k).
	\end{equation*}
	Next, we derive the upper bound for $\|\bar{g}_i^k\|^2$:
	\begin{eqnarray*}
		\|\bar{g}_i^k\|^2 &=& \|\nabla f_i(x^k) - h_i^k - h^k\|^2\\
		&\overset{\eqref{eq:a_b_norm_squared}}{\le}& 2\|\nabla f_i(x^k) - \nabla f_i(x^*)\|^2 + 2\left\|h_i^k - \nabla f_i(x^*) -\left(h^k + \nabla f(x^*)\right)\right\|^2	\\
		&\overset{\eqref{eq:L_smoothness_cor}}{\le}& 4L\left(f_i(x^k) - \nabla f_i(x^*) - \langle\nabla f_i(x^*), x^k - x^*\rangle\right)\\
		&&\quad +2\left\|h_i^k - \nabla f_i(x^*) -\left(h^k + \nabla f(x^*)\right)\right\|^2.
	\end{eqnarray*}
	Summing up previous inequality for $i = 1,\ldots, n$ we get
	\begin{eqnarray}
		\frac{1}{n}\sum\limits_{i=1}^n\|\bar{g}_i^k\|^2 &\le&4L(f(x^k) - f(x^*)) + \frac{2}{n}\sum\limits_{i=1}^n\left\|h_i^k - \nabla f_i(x^*) - \left(\frac{1}{n}\sum\limits_{i=1}^n(h_i^k - \nabla f_i(x^*))\right)\right\|^2\notag\\
		&\overset{\eqref{eq:variance_decomposition}}{\le}& 4L\left(f(x^k) - f(x^*)\right) + \frac{2}{n}\sum\limits_{i=1}^n\|h_i^k - \nabla f(x^*)\|^2.\label{eq:ec_sgd_diana_first_ineq}
	\end{eqnarray}
	Using the unbiasedness of $\hat g_i^k$ we derive
	\begin{equation*}
		\frac{1}{n}\sum\limits_{i=1}^n\EE\left[\|g_i^k-\bar{g}_i^k\|^2\mid x^k\right] = \frac{1}{n}\sum\limits_{i=1}^n\EE\left[\|\hat g_i^k-\nabla f_i(x^k)\|^2\mid x^k\right] \le \frac{1}{n}\sum\limits_{i=1}^n\widetilde{D}_{1,i} = \widetilde{D}_1.
	\end{equation*}
	Finally, we obtain the upper bound for the second moment of $g^k$ using the independence of $\hat g_1^k,\ldots, \hat g_n^k$:
	\begin{eqnarray*}
		\EE\left[\|g^k\|^2\mid x^k\right] &\overset{\eqref{eq:variance_decomposition}}{=}& \|\nabla f(x^k)\|^2 + \EE\left[\|g^k - \nabla f(x^k)\|^2\right]\\
		&\overset{\eqref{eq:L_smoothness_cor}}{\le}&  2L(f(x^k)-f(x^*)) + \EE\left[\left\|\frac{1}{n}\sum\limits_{i=1}^n(\hat g_i^k - \nabla f_i(x^k))\right\|^2\mid x^k\right]\\
		&=&2L(f(x^k)-f(x^*)) + \frac{1}{n^2}\sum\limits_{i=1}^n\EE\left[\left\|\hat g_i^k - \nabla f_i(x^k)\right\|^2\mid x^k\right]\\
		&\le& 2L(f(x^k)-f(x^*)) + \frac{1}{n^2}\sum\limits_{i=1}^n\widetilde{D}_{1,i}.
	\end{eqnarray*}
\end{proof}

\begin{lemma}\label{lem:ec_diana_sigma_k+1_bound}
	Let assumptions of Lemma~\ref{lem:ec_diana_second_moment_bound} hold and $\alpha\le\nicefrac{1}{(\omega+1)}$. Then, for all $k\ge 0$ we have
	\begin{equation}
		\EE\left[\sigma_{k+1}^2\mid x^k\right] \le (1 - \alpha)\sigma_k^2 + 2L\alpha(f(x^k) - f(x^*)) + \alpha^2(\omega+1) \widetilde{D}_1, \label{eq:ec_diana_sigma_k+1_bound}
	\end{equation}
	where $\sigma_k^2 = \frac{1}{n}\sum_{i=1}^n\|h_i^k - \nabla f_i(x^*)\|^2$ and $\widetilde{D}_1 = \frac{1}{n}\sum_{i=1}^n \widetilde{D}_{1,i}$.
\end{lemma}
\begin{proof}
%	We start with establishing a simple relation:
%	\begin{equation*}
%		\EE\left[h_i^{k+1}\mid x^k\right] = \EE\left[h_i^k + \alpha Q(\hat g_i^k - h_i^k)\right] = (1-\alpha)h_i^k + \alpha \nabla f_i(x^k).
%	\end{equation*}
	For simplicity, we introduce new notation: $h_i^* \eqdef \nabla f_i(x^*)$. Using this we derive an upper bound for the second moment of $h_i^{k+1} - h_i^*$:
	\begin{eqnarray*}
		\EE\left[\|h_i^{k+1} - h_i^*\|^2\mid x^k\right] &=& \EE\left[\left\|h_i^k - h_i^* + \alpha Q(\hat g_i^k - h_i^k) \right\|^2\mid x^k\right]\\
		&\overset{\eqref{eq:quantization_def}}{=}& \|h_i^k - h_i^*\|^2 +2\alpha\langle h_i^k - h_i^*, \nabla f_i(x^k) - h_i^k \rangle\\
		&&\quad + \alpha^2\EE\left[\|Q(\hat g_i^k - h_i^k)\|^2\mid x^k\right]\\
		&\overset{\eqref{eq:quantization_def},\eqref{eq:tower_property}}{\le}& \|h_i^k - h_i^*\|^2 +2\alpha\langle h_i^k - h_i^*, \nabla f_i(x^k) - h_i^k \rangle\\
		&&\quad + \alpha^2(\omega+1)\EE\left[\|\hat g_i^k - h_i^k\|^2\mid x^k\right].
	\end{eqnarray*}
	Using variance decomposition \eqref{eq:variance_decomposition} and $\alpha\le\nicefrac{1}{(\omega+1)}$ we get
	\begin{eqnarray*}
		\alpha^2(\omega+1)\EE\left[\|\hat g_i^k - h_i^k\|^2\mid x^k\right] &\overset{\eqref{eq:variance_decomposition}}{=}& \alpha^2(\omega+1)\EE\left[\|\hat g_i^k - \nabla f_i(x^k)\|^2\mid x^k\right]\\
		&&\quad + \alpha^2(\omega+1)\|\nabla f_i(x^k) - h_i^k\|^2\\
		&\le& \alpha^2(\omega+1) \widetilde{D}_{1,i} + \alpha\|\nabla f_i(x^k) - h_i^k\|^2.
	\end{eqnarray*}
	Putting all together we obtain
	\begin{eqnarray*}
		\EE\left[\|h_i^{k+1} - h_i^*\|^2\mid x^k\right] &\le& \|h_i^k - h_i^*\|^2 + \alpha\left\langle \nabla f_i(x^k) - h_i^k, f_i(x^k) + h_i^k - 2h_i^* \right\rangle + \alpha^2(\omega+1) \widetilde{D}_{1,i}\\
		&\overset{\eqref{eq:a-b_a+b}}{=}& \|h_i^k - h_i^*\|^2 + \alpha\|\nabla f_i(x^k) - h_i^*\|^2 - \alpha\|h_i^k - h_i^*\|^2 + \alpha^2(\omega+1) \widetilde{D}_{1,i}\\
		&\overset{\eqref{eq:L_smoothness_cor}}{\le}& (1-\alpha)\|h_i^k - h_i^*\|^2 + 2L\alpha\left(f_i(x^k) - f_i(x^*) - \langle\nabla f_i(x^*), x^k - x^* \rangle\right)\\
		&&\quad +\alpha^2(\omega+1) \widetilde{D}_{1,i}.
	\end{eqnarray*}
	Summing up the above inequality for $i=1,\ldots, n$ we derive
	\begin{equation*}
		\frac{1}{n}\sum\limits_{i=1}^n\EE\left[\|h_i^{k+1} - h_i^*\|^2\mid x^k\right] \le \frac{1-\alpha}{n}\sum\limits_{i=1}^n\|h_i^k - h_i^*\|^2 + 2L\alpha(f(x^k) - f(x^*)) + \frac{\alpha^2(\omega+1)}{n}\sum\limits_{i=1}^n\widetilde{D}_{1,i}.
	\end{equation*}
\end{proof}

Applying Theorem~\ref{thm:ec_sgd_main_result_new} we get the following result.
\begin{theorem}\label{thm:ec_diana}
	Assume that $f_i(x)$ is convex and $L$-smooth for all $i=1,\ldots, n$ and $f(x)$ is $\mu$-quasi strongly convex. Then {\tt EC-SGD-DIANA} satisfies Assumption~\ref{ass:key_assumption_finite_sums_new} with
	\begin{gather*}
		A = 2L,\quad \widetilde{A} = 0,\quad A' = L,\quad B_1 = 2,\quad \widetilde{D}_1 = \frac{1}{n}\sum\limits_{i=1}^n \widetilde{D}_{1,i},\quad \sigma_{1,k}^2 = \sigma_k^2 = \frac{1}{n}\sum\limits_{i=1}^n\|h_i^k - \nabla f_i(x^*)\|^2,\\
		B_1' = B_2' = B_2 = \widetilde{B}_1 = \widetilde{B}_2 = 0,\quad \sigma_{2,k}^2\equiv 0,\quad \rho_1 = \alpha,\quad \rho_2 = 1,\quad C_1 = L\alpha,\quad C_2 = 0,\quad D_1 = 0,\\
		D_2 = \alpha^2(\omega+1) \widetilde{D}_1,\quad D_1' = \frac{D_1}{n},\quad G = 0,\\
		F_1 = \frac{96L\gamma^2}{\delta^2\alpha\left(1-\min\left\{\frac{\gamma\mu}{2},\frac{\alpha}{4}\right\}\right)},\quad F_2 = 0,\quad D_3 = \frac{6L\gamma}{\delta}\left(\frac{4\alpha(\omega+1)}{\delta}+1\right)\widetilde{D}_1,
	\end{gather*}
	with $\gamma$ and $\alpha$ satisfying
	\begin{equation*}
		\gamma \le \min\left\{\frac{1}{4L}, \frac{\delta\sqrt{1-\alpha}}{8L\sqrt{6(3-\alpha)}}\right\},\quad \alpha \le \frac{1}{\omega+1},\quad M_1 = M_2 = 0
	\end{equation*}
	and for all $K \ge 0$
	\begin{equation*}
		\EE\left[f(\bar x^K) - f(x^*)\right] \le \left(1 - \min\left\{\frac{\gamma\mu}{2},\frac{\alpha}{4}\right\}\right)^K\frac{4(\|x^0 - x^*\|^2 + \gamma F_1 \sigma_0^2)}{\gamma} + 4\gamma\left(D_1' + D_3\right),
	\end{equation*}	
	when $\mu > 0$ and
	\begin{equation*}
		\EE\left[f(\bar{x}^K) - f(x^*)\right] \le \frac{4(\|x^0 - x^*\|^2 + \gamma F_1 \sigma_0^2)}{\gamma K} + 4\gamma\left(D_1' + D_3\right)
	\end{equation*}
	when $\mu=0$.
\end{theorem}
In other words, if 
\begin{equation*}
		\gamma = \min\left\{\frac{1}{4L}, \frac{\delta\sqrt{1-\alpha}}{8L\sqrt{6(3-\alpha)}}\right\},\quad \alpha = \min\left\{\frac{1}{\omega+1},\frac{1}{2}\right\}
\end{equation*}
and $\widetilde{D}_1 = 0$, i.e., $\hat g_i^k = \nabla f_i(x^k)$ almost surely (this is the setup of {\tt EC-GD-DIANA}), {\tt EC-SGD-DIANA} converges with the linear rate
\begin{equation*}
	\cO\left(\left(\omega + \frac{\kappa}{\delta}\right)\ln\frac{1}{\varepsilon}\right)
\end{equation*}
to the exact solution.  Applying Lemma~\ref{lem:lemma2_stich} we establish the rate of convergence to $\varepsilon$-solution in the case when $\mu > 0$.
\begin{corollary}\label{cor:ec_diana_str_cvx_cor}
	Let the assumptions of Theorem~\ref{thm:ec_diana} hold and $\mu > 0$. Then after $K$ iterations of {\tt EC-SGD-DIANA} with the stepsize
	\begin{eqnarray*}
		\gamma_0 &=& \min\left\{\frac{1}{4L}, \frac{\delta\sqrt{1-\alpha}}{8L\sqrt{6(3-\alpha)}}\right\},\quad R_0 = \|x^0-x^*\|,\quad  \tilde{F_1} = \frac{784L\gamma^2}{7\delta^2\alpha},\\
		\gamma &=& \min\left\{\gamma_0, \frac{\ln\left(\max\left\{2,\min\left\{\frac{n\left(R_0^2 + \tilde{F}_1\gamma_0\sigma_{1,0}^2\right)\mu^2K^2}{\widetilde{D}_1}, \frac{\delta\left(R_0^2 + \tilde{F}_1\gamma_0\sigma_{1,0}^2\right)\mu^3K^3}{6L\widetilde{D}_1(\nicefrac{4\alpha(\omega+1)}{\delta}+1)}\right\}\right\}\right)}{\mu K}\right\},
	\end{eqnarray*}
	and $\alpha \le \frac{1}{\omega+1}$ we have
	\begin{equation*}
		\EE\left[f(\bar{x}^K) - f(x^*)\right] = \widetilde\cO\left(\frac{L}{\delta}R_0^2\exp\left(-\min\left\{\frac{\delta\mu}{L},\alpha\right\}K\right) + \frac{\widetilde{D}_1}{n\mu K} + \frac{L\widetilde{D}_1\left(\nicefrac{\alpha(\omega+1)}{\delta}+1\right)}{\delta\mu^2 K^2}\right).
	\end{equation*}
	That is, to achive $\EE\left[f(\bar{x}^K) - f(x^*)\right] \le \varepsilon$ {\tt EC-SGD-DIANA} requires
	\begin{equation*}
		\widetilde{\cO}\left(\frac{1}{\alpha} + \frac{L}{\delta\mu} + \frac{D_1}{n\mu\varepsilon} + \frac{\sqrt{L\widetilde{D}_1\left(\nicefrac{\alpha(\omega+1)}{\delta}+1\right)}}{\mu\sqrt{\delta\varepsilon}}\right) \text{ iterations.}
	\end{equation*}
	In particular, if $\alpha = \frac{1}{\omega+1}$, then to achive $\EE\left[f(\bar{x}^K) - f(x^*)\right] \le \varepsilon$ {\tt EC-SGD-DIANA} requires
	\begin{equation*}
		\widetilde{\cO}\left(\omega + \frac{L}{\delta\mu} + \frac{\widetilde{D}_1}{n\mu\varepsilon} + \frac{\sqrt{L\widetilde{D}_1}}{\delta\mu\sqrt{\varepsilon}}\right) \text{ iterations,}
	\end{equation*}
	and if $\alpha = \frac{\delta}{\omega+1}$, then to achive $\EE\left[f(\bar{x}^K) - f(x^*)\right] \le \varepsilon$ {\tt EC-SGD-DIANA} requires
	\begin{equation*}
		\widetilde{\cO}\left(\frac{\omega+1}{\delta} + \frac{L}{\delta\mu} + \frac{\widetilde{D}_1}{n\mu\varepsilon} + \frac{\sqrt{L\widetilde{D}_1}}{\mu\sqrt{\delta\varepsilon}}\right) \text{ iterations.}
	\end{equation*}
\end{corollary}

Applying Lemma~\ref{lem:lemma_technical_cvx} we get the complexity result in the case when $\mu = 0$.
\begin{corollary}\label{cor:ec_diana_cvx_cor}
	Let the assumptions of Theorem~\ref{thm:ec_diana} hold and $\mu = 0$. Then after $K$ iterations of {\tt EC-SGD-DIANA} with the stepsize
	\begin{eqnarray*}
		\gamma_0 &=& \min\left\{\frac{1}{4L}, \frac{\delta\sqrt{1-\alpha}}{8L\sqrt{6(3-\alpha)}}\right\},\quad R_0 = \|x^0-x^*\|,\\
		\gamma &=& \min\left\{\gamma_0, \sqrt[3]{\frac{R_0^2\delta^2\alpha\left(1-\min\left\{\frac{\gamma_0\mu}{2},\frac{\alpha}{4}\right\}\right)}{96L\sigma_0^2}}, \sqrt{\frac{nR_0^2}{\widetilde{D}_1 K}}, \sqrt[3]{\frac{\delta R_0^2}{6L\widetilde{D}_1\left(\frac{4\alpha(\omega+1)}{\delta}+1\right) K}}\right\},
	\end{eqnarray*}	
	and $\alpha \le \frac{1}{\omega+1}$ we have $\EE\left[f(\bar{x}^K) - f(x^*)\right]$ of order
	\begin{equation*}
		\cO\left(\frac{LR_0^2}{\delta K} + \frac{\sqrt[3]{LR_0^4\sigma_0^2}}{K\sqrt[3]{\delta^2\alpha}} + \sqrt{\frac{R_0^2\widetilde{D}_1}{nK}}+ \sqrt[3]{\frac{LR_0^4\widetilde{D}_1\left(\frac{\alpha(\omega+1)}{\delta}+1\right)}{\delta K^2}}\right).
	\end{equation*}
	That is, to achive $\EE\left[f(\bar{x}^K) - f(x^*)\right] \le \varepsilon$ {\tt EC-SGD-DIANA} requires
	\begin{equation*}
		\cO\left(\frac{LR_0^2}{\delta \varepsilon} + \frac{\sqrt[3]{LR_0^4\sigma_0^2}}{\varepsilon\sqrt[3]{\delta^2\alpha}} + \frac{R_0^2\widetilde{D}_1}{n\varepsilon^2}+ \frac{R_0^2\sqrt{L\widetilde{D}_1\left(\frac{\alpha(\omega+1)}{\delta}+1\right)}}{\sqrt{\delta\varepsilon^3}}\right)
	\end{equation*}
	iterations. In particular, if $\alpha = \frac{1}{\omega+1}$, then to achive $\EE\left[f(\bar{x}^K) - f(x^*)\right] \le \varepsilon$ {\tt EC-SGD-DIANA} requires
	\begin{equation*}
		\cO\left(\frac{LR_0^2}{\delta \varepsilon} + \frac{\sqrt[3]{LR_0^4(\omega+1)\sigma_0^2}}{\varepsilon\sqrt[3]{\delta^2}} + \frac{R_0^2\widetilde{D}_1}{n\varepsilon^2}+ \frac{R_0^2\sqrt{L\widetilde{D}_1}}{\delta\sqrt{\varepsilon^3}}\right) \text{ iterations,}
	\end{equation*}
	and if $\alpha = \frac{\delta}{\omega+1}$, then to achive $\EE\left[f(\bar{x}^K) - f(x^*)\right] \le \varepsilon$ {\tt EC-SGD-DIANA} requires
	\begin{equation*}
		\cO\left(\frac{LR_0^2}{\delta \varepsilon} + \frac{\sqrt[3]{LR_0^4(\omega+1)\sigma_0^2}}{\delta\varepsilon} + \frac{R_0^2\widetilde{D}_1}{n\varepsilon^2}+ \frac{R_0^2\sqrt{L\widetilde{D}_1}}{\sqrt{\delta\varepsilon^3}}\right) \text{ iterations.}
	\end{equation*}
\end{corollary}

\subsection{{\tt EC-SGDsr-DIANA}}\label{sec:ec_sgdsr_DIANA}
In this section we consider the same setup as in Section~\ref{sec:diana_arbitrary_sampling} and consider {\tt EC-SGD-DIANA} adjusted to this setup. The resulting algorithm is called {\tt EC-SGDsr-DIANA}, see
\begin{algorithm}[t]
   \caption{{\tt EC-SGDsr-DIANA}}\label{alg:EC-SGDsr-DIANA}
\begin{algorithmic}[1]
   \Require learning rates $\gamma>0$, $\alpha \in (0,1]$, initial vectors $x^0, h_1^0,\ldots, h_n^0 \in \R^d$
	\State Set $e_i^0 = 0$ for all $i=1,\ldots, n$   
	\State Set $h^0 = \frac{1}{n}\sum_{i=1}^n h_i^0$   
   \For{$k=0,1,\dotsc$}
       \State Broadcast $x^{k}, h^k$ to all workers
        \For{$i=1,\dotsc,n$ in parallel}
			\State Sample $\hat g_i^{k} = \nabla f_{\xi_i^k}(x^k)$ satisfying Assumption~\ref{ass:exp_smoothness} independtently from other workers            
            \State $g^{k}_i = \hat g_i^k - h_i^k + h^k$
            \State $v_i^k = C(e_i^k + \gamma g_i^k)$
            \State $e_i^{k+1} = e_i^k + \gamma g_i^k - v_i^k$
            \State $h_i^{k+1} = h_i^k + \alpha Q(\hat g_i^k - h_i^k)$ \Comment{$Q(\cdot)$ is calculated independtly from other workers}
        \EndFor
        \State $e^k = \frac{1}{n}\sum_{i=1}^ne_i^k$, $g^k = \frac{1}{n}\sum_{i=1}^ng_i^k$, $v^k = \frac{1}{n}\sum_{i=1}^nv_i^k$, $h^{k+1} = \frac{1}{n}\sum\limits_{i=1}^n h_i^{k+1} = h^k + \alpha\frac{1}{n}\sum\limits_{i=1}^n Q(\hat g_i^k - h_i^k)$
       \State $x^{k+1} = x^k - v^k$
   \EndFor
\end{algorithmic}
\end{algorithm}
\begin{lemma}\label{lem:ec_sgdsr_diana_second_moment_bound}
	Let Assumption~\ref{ass:exp_smoothness} be satisfied and $f_i$ be convex and $L$-smooth for all $i\in[n]$. Then, for all $k\ge 0$ we have
	\begin{eqnarray}
		\EE\left[g^k\mid x^k\right] &=& \nabla f(x^k), \label{eq:ec_sgdsr_diana_unbiasedness}\\
		\frac{1}{n}\sum\limits_{i=1}^n\|\bar{g}_i^k\|^2 &\le& 4L\left(f(x^k) - f(x^*)\right) + 2\sigma_k^2, \label{eq:ec_sgdsr_diana_second_moment_bound}\\
		\frac{1}{n}\sum\limits_{i=1}^n\EE\left[\|g_i^k-\bar{g}_i^k\|^2\mid x^k\right] &\le& 6(\cL+L)\left(f(x^k) - f(x^*)\right) + \widetilde{D}_1, \label{eq:ec_sgdsr_diana_variance_bound}\\
		\EE\left[\|g^k\|^2\mid x^k\right] &\le& 4\cL\left(f(x^k) - f(x^*)\right) + D_1' \label{eq:ec_sgssr_diana_second_moment_bound_2}
	\end{eqnarray}
	where $\sigma_k^2 = \frac{1}{n}\sum_{i=1}^n\|h_i^k - \nabla f(x^*)\|^2$, $\widetilde{D}_1 = \frac{3}{n}\sum_{i=1}^n \EE_{\cD_i}\left[\|\nabla f_{\xi_i}(x^*) - \nabla f_i(x^*)\|^2\right]$ and\newline $D_1' = \frac{2}{n^2}\sum\limits_{i=1}^n\EE_{\cD_i}\left[\|\nabla f_{\xi_i}(x^*) - \nabla f_{i}(x^*)\|^2\right]$.
\end{lemma}
\begin{proof}
	First of all, we show unbiasedness of $g^k$:
	\begin{equation*}
		\EE\left[g^k\mid x^k\right] = \frac{1}{n}\sum\limits_{i=1}^n\EE\left[g_i^k\mid x^k\right] = \frac{1}{n}\sum\limits_{i=1}^n\left(\nabla f_i(x^k) - h_i^k + h^k\right) = \nabla f(x^k).
	\end{equation*}
	Following the same steps as in the proof of \eqref{eq:ec_sgd_diana_first_ineq} we derive \eqref{eq:ec_sgdsr_diana_second_moment_bound}. Next, we establish \eqref{eq:ec_sgdsr_diana_variance_bound}:
	\begin{eqnarray*}
		\frac{1}{n}\sum\limits_{i=1}^n\EE\left[\|g_i^k-\bar{g}_i^k\|^2\mid x^k\right] &=& \frac{1}{n}\sum\limits_{i=1}^n\EE_{\cD_i}\left[\|\nabla f_{\xi_i^k}(x^k)-\nabla f_i(x^k)\|^2\right]\\
		&\overset{\eqref{eq:a_b_norm_squared}}{\le}& \frac{3}{n}\sum\limits_{i=1}^n\EE_{\cD_i}\left[\|\nabla f_{\xi_i^k}(x^k)-\nabla f_{\xi_i^k}(x^*)\|^2\right]\\
		&&\quad + \frac{3}{n}\sum\limits_{i=1}^n\EE_{\cD_i}\left[\|\nabla f_{\xi_i^k}(x^*)-\nabla f_i(x^*)\|^2\right]\\
		&&\quad + \frac{3}{n}\sum\limits_{i=1}^n\|\nabla f_{i}(x^*)-\nabla f_i(x^k)\|^2\\
		&\overset{\eqref{eq:L_smoothness_cor},\eqref{eq:exp_smoothness}}{\le}& 6(\cL + L)\left(f(x^k)-f(x^*)\right) \\
		&&\quad + \frac{3}{n}\sum\limits_{i=1}^n\EE_{\cD_i}\left[\|\nabla f_{\xi_i}(x^*)-\nabla f_i(x^*)\|^2\right].
	\end{eqnarray*}
	Finally, we obtain the upper bound for the second moment of $g^k$ using the independence of $\xi_1^k,\ldots,\xi_n^k$:
	\begin{eqnarray*}
		\EE\left[\|g^k\|^2\mid x^k\right] &=& \EE\left[\left\|\frac{1}{n}\sum\limits_{i=1}^n(\nabla f_{\xi_i^k}(x^k) - \nabla f_{\xi_i^k}(x^*) + \nabla f_{\xi_i^k}(x^*) - \nabla f_i(x^*))\right\|^2\mid x^k\right]\\
		&\overset{\eqref{eq:a_b_norm_squared}}{\le}& \frac{2}{n}\sum\limits_{i=1}^n\EE\left[\|\nabla f_{\xi_i^k}(x^k) - \nabla f_{\xi_i^k}(x^*)\|^2\mid x^k\right]\\
		&&\quad + 2\EE\left[\left\|\frac{1}{n}\sum\limits_{i=1}^n(\nabla f_{\xi_i^k}(x^*) - \nabla f_i(x^*))\right\|^2\mid x^k\right]\\
		&\overset{\eqref{eq:exp_smoothness}}{\le}& 4\cL\left(f(x^k) - f(x^*)\right) + \frac{2}{n^2}\sum\limits_{i=1}^n\EE_{\cD_i}\left[\|\nabla f_{\xi_i}(x^*) - \nabla f_{i}(x^*)\|^2\right].
	\end{eqnarray*}
\end{proof}

\begin{lemma}\label{lem:ec_sgdsr_diana_sigma_k+1_bound}
	Let $f_i$ be convex and $L$-smooth, Assumption~\ref{ass:exp_smoothness} holds and $\alpha \le \nicefrac{1}{(\omega+1)}$. Then, for all $k\ge 0$ we have
	\begin{equation}
		\EE\left[\sigma_{k+1}^2\mid x^k\right] \le (1 - \alpha)\sigma_k^2 + 2\alpha(3\cL+4L)(f(x^k) - f(x^*)) + D_2, \label{eq:ec_sgdsr_diana_sigma_k+1_bound}
	\end{equation}
	where $\sigma_k^2 = \frac{1}{n}\sum_{i=1}^n\|h_i^k - \nabla f_i(x^*)\|^2$ and $D_2 = \alpha^2(\omega+1)\widetilde{D}_1$.
\end{lemma}
\begin{proof}
The proof is identical to the proof of Lemma~\ref{lem:diana_sigma_k+1_bound} up to the following changes in the notation: $\omega_1 = \omega$, $\Delta_i^k = Q(\hat g_i^k - h_i^k)$ and $\hat \Delta_i^k = \hat g_i^k - h_i^k$.
\end{proof}

Applying Theorem~\ref{thm:ec_sgd_main_result_new} we get the following result.
\begin{theorem}\label{thm:ec_sgdsr_diana}
	Assume that $f_i(x)$ is convex and $L$-smooth for all $i=1,\ldots, n$, $f(x)$ is $\mu$-quasi strongly convex and Assumption~\ref{ass:exp_smoothness} holds. Then {\tt EC-SGDsr-DIANA} satisfies Assumption~\ref{ass:key_assumption_finite_sums_new} with
	\begin{gather*}
		A = 2L,\quad \widetilde{A}=3(\cL+L),\quad A' = 2\cL,\quad B_1 = 2,\quad \widetilde{D}_1 = \frac{3}{n}\sum_{i=1}^n \EE_{\cD_i}\left[\|\nabla f_{\xi_i}(x^*) - \nabla f_i(x^*)\|^2\right],\\
		\sigma_{1,k}^2 = \sigma_k^2 = \frac{1}{n}\sum\limits_{i=1}^n\|h_i^k - \nabla f_i(x^*)\|^2,\quad D_1 = 0,\quad D_1' = \frac{2}{3n}\widetilde{D}_1,\quad D_2 = \alpha^2(\omega+1) \widetilde{D}_1\\
		\widetilde{B}_1 = B_1' = B_2' = B_2 = \widetilde{B}_2 = 0,\quad \sigma_{2,k}^2\equiv 0,\quad \rho_1 = \alpha,\quad \rho_2 = 1,\quad C_1 = 2\alpha(3\cL+4L),\quad C_2 = 0,\\
		G = 0,\quad F_1 = \frac{96L\gamma^2}{\delta^2\alpha\left(1-\min\left\{\frac{\gamma\mu}{2},\frac{\alpha}{4}\right\}\right)},\quad F_2 = 0,\quad D_3 = \frac{6L\gamma}{\delta}\left(\frac{4\alpha(\omega+1)}{\delta}+1\right)\widetilde{D}_1,
	\end{gather*}
	with $\gamma$ and $\alpha$ satisfying
	\begin{equation*}
		\gamma \le \min\left\{\frac{1}{4\cL}, \frac{\delta}{4\sqrt{6L\left(4L + 3\delta(\cL + L) + \frac{16(3\cL + 4L)}{1-\alpha}\right)}}\right\},\quad \alpha \le \frac{1}{\omega+1},\quad M_1 = M_2 = 0.
	\end{equation*}
	and for all $K \ge 0$
	\begin{equation*}
		\EE\left[f(\bar x^K) - f(x^*)\right] \le \left(1 - \min\left\{\frac{\gamma\mu}{2},\frac{\alpha}{4}\right\}\right)^K\frac{4(\|x^0 - x^*\|^2 + \gamma F_1 \sigma_0^2)}{\gamma} + 4\gamma\left(D_1' + D_3\right),
	\end{equation*}	
	when $\mu > 0$ and
	\begin{equation*}
		\EE\left[f(\bar{x}^K) - f(x^*)\right] \le \frac{4(\|x^0 - x^*\|^2 + \gamma F_1 \sigma_0^2)}{\gamma K} + 4\gamma\left(D_1' + D_3\right)
	\end{equation*}
	when $\mu=0$.
\end{theorem}
Applying Lemma~\ref{lem:lemma2_stich} we establish the rate of convergence to $\varepsilon$-solution in the case when $\mu > 0$.
\begin{corollary}\label{cor:ec_sgdsr_diana_str_cvx_cor}
	Let the assumptions of Theorem~\ref{thm:ec_sgdsr_diana} hold and $\mu > 0$. Then after $K$ iterations of {\tt EC-SGDsr-DIANA} with the stepsize
	\begin{eqnarray*}
		\gamma_0 &=& \min\left\{\frac{1}{4\cL}, \frac{\delta}{4\sqrt{6L\left(4L + 3\delta(\cL + L) + \frac{16(3\cL + 4L)}{1-\alpha}\right)}}\right\},\\
		R_0 &=& \|x^0-x^*\|,\quad  \tilde{F_1} = \frac{96L\gamma_0^2}{\delta^2\alpha\left(1-\min\left\{\frac{\gamma_0\mu}{2},\frac{\alpha}{4}\right\}\right)},\\
		\gamma &=& \min\left\{\gamma_0, \frac{\ln\left(\max\left\{2,\min\left\{\frac{3n\left(R_0^2 + \tilde{F}_1\gamma_0\sigma_{1,0}^2\right)\mu^2K^2}{2\widetilde{D}_1}, \frac{\delta\left(R_0^2 + \tilde{F}_1\gamma_0\sigma_{1,0}^2\right)\mu^3K^3}{6L\widetilde{D}_1\left(\frac{4\alpha(\omega+1)}{\delta}+1\right)}\right\}\right\}\right)}{\mu K}\right\},
	\end{eqnarray*} 
	and $\alpha \le \frac{1}{\omega+1}$ we have $\EE\left[f(\bar{x}^K) - f(x^*)\right]$ of order
	\begin{equation*}
		\widetilde\cO\left(\left(\cL +\frac{\sqrt{L\cL}}{\delta}\right)R_0^2\exp\left(-\min\left\{\frac{\mu}{\cL + \frac{\sqrt{L\cL}}{\delta}},\alpha\right\}K\right) + \frac{\widetilde{D}_1}{n\mu K} + \frac{L\widetilde{D}_1\left(\frac{\alpha(\omega+1)}{\delta}+1\right)}{\delta\mu^2 K^2}\right)
	\end{equation*}
	That is, to achive $\EE\left[f(\bar{x}^K) - f(x^*)\right] \le \varepsilon$ {\tt EC-SGDsr-DIANA} requires
	\begin{equation*}
		\widetilde{\cO}\left(\frac{1}{\alpha} + \frac{\cL}{\mu} + \frac{\sqrt{L\cL}}{\delta\mu} + \frac{\widetilde{D}_1}{n\mu\varepsilon} + \frac{\sqrt{L\widetilde{D}_1\left(\frac{\alpha(\omega+1)}{\delta}+1\right)}}{\mu\sqrt{\delta\varepsilon}}\right) \text{ iterations.}
	\end{equation*}
	In particular, if $\alpha = \frac{1}{\omega+1}$, then to achive $\EE\left[f(\bar{x}^K) - f(x^*)\right] \le \varepsilon$ {\tt EC-SGDsr-DIANA} requires
	\begin{equation*}
		\widetilde{\cO}\left(\omega + \frac{\cL}{\mu} + \frac{\sqrt{L\cL}}{\delta\mu} + \frac{\widetilde{D}_1}{n\mu\varepsilon} + \frac{\sqrt{L\widetilde{D}_1}}{\delta\mu\sqrt{\varepsilon}}\right) \text{ iterations,}
	\end{equation*}
	and if $\alpha = \frac{\delta}{\omega+1}$, then to achive $\EE\left[f(\bar{x}^K) - f(x^*)\right] \le \varepsilon$ {\tt EC-SGDsr-DIANA} requires
	\begin{equation*}
		\widetilde{\cO}\left(\frac{\omega+1}{\delta} + \frac{\cL}{\mu} + \frac{\sqrt{L\cL}}{\delta\mu} + \frac{\widetilde{D}_1}{n\mu\varepsilon} + \frac{\sqrt{L\widetilde{D}_1}}{\mu\sqrt{\delta\varepsilon}}\right) \text{ iterations.}
	\end{equation*}
\end{corollary}

Applying Lemma~\ref{lem:lemma_technical_cvx} we get the complexity result in the case when $\mu = 0$.
\begin{corollary}\label{cor:ec_sgdsr_diana_cvx_cor}
	Let the assumptions of Theorem~\ref{thm:ec_sgdsr_diana} hold and $\mu = 0$. Then after $K$ iterations of {\tt EC-SGDsr-DIANA} with the stepsize
	\begin{eqnarray*}
		\gamma_0 &=& \min\left\{\frac{1}{4\cL}, \frac{\delta}{4\sqrt{6L\left(4L + 3\delta(\cL + L) + \frac{16(3\cL + 4L)}{1-\alpha}\right)}}\right\},\quad R_0 = \|x^0-x^*\|,\\
		\gamma &=& \min\left\{\gamma_0, \sqrt[3]{\frac{R_0^2\delta^2\alpha\left(1-\min\left\{\frac{\gamma_0\mu}{2},\frac{\alpha}{4}\right\}\right)}{96L\sigma_0^2}}, \sqrt{\frac{3nR_0^2}{2\widetilde{D}_1 K}}, \sqrt[3]{\frac{\delta R_0^2}{6L\widetilde{D}_1\left(\frac{4\alpha(\omega+1)}{\delta}+1\right) K}}\right\},
	\end{eqnarray*}	
	and $\alpha \le \frac{1}{\omega+1}$ we have $\EE\left[f(\bar{x}^K) - f(x^*)\right]$ of order
	\begin{equation*}
		\cO\left(\frac{\cL R_0^2}{K} + \frac{\sqrt{\cL L} R_0^2}{\delta K} + \frac{\sqrt[3]{LR_0^4\sigma_0^2}}{K\sqrt[3]{\delta^2\alpha}} + \sqrt{\frac{R_0^2\widetilde{D}_1}{nK}}+ \sqrt[3]{\frac{LR_0^4\widetilde{D}_1\left(\frac{\alpha(\omega+1)}{\delta}+1\right)}{\delta K^2}}\right).
	\end{equation*}
	That is, to achive $\EE\left[f(\bar{x}^K) - f(x^*)\right] \le \varepsilon$ {\tt EC-SGDsr-DIANA} requires
	\begin{equation*}
		\cO\left(\frac{\cL R_0^2}{\varepsilon} + \frac{\sqrt{\cL L} R_0^2}{\delta \varepsilon} + \frac{\sqrt[3]{LR_0^4\sigma_0^2}}{\varepsilon\sqrt[3]{\delta^2\alpha}} + \frac{R_0^2\widetilde{D}_1}{n\varepsilon^2}+ \frac{R_0^2\sqrt{L\widetilde{D}_1\left(\frac{\alpha(\omega+1)}{\delta}+1\right)}}{\sqrt{\delta\varepsilon^3}}\right)
	\end{equation*}
	iterations. In particular, if $\alpha = \frac{1}{\omega+1}$, then to achive $\EE\left[f(\bar{x}^K) - f(x^*)\right] \le \varepsilon$ {\tt EC-SGDsr-DIANA} requires
	\begin{equation*}
		\cO\left(\frac{\cL R_0^2}{\varepsilon} + \frac{\sqrt{\cL L} R_0^2}{\delta \varepsilon} + \frac{\sqrt[3]{LR_0^4(\omega+1)\sigma_0^2}}{\varepsilon\sqrt[3]{\delta^2}} + \frac{R_0^2\widetilde{D}_1}{n\varepsilon^2}+ \frac{R_0^2\sqrt{L\widetilde{D}_1}}{\delta\sqrt{\varepsilon^3}}\right) \text{ iterations,}
	\end{equation*}
	and if $\alpha = \frac{\delta}{\omega+1}$, then to achive $\EE\left[f(\bar{x}^K) - f(x^*)\right] \le \varepsilon$ {\tt EC-SGDsr-DIANA} requires
	\begin{equation*}
		\cO\left(\frac{\cL R_0^2}{\varepsilon} + \frac{\sqrt{\cL L} R_0^2}{\delta \varepsilon} + \frac{\sqrt[3]{LR_0^4(\omega+1)\sigma_0^2}}{\delta\varepsilon} + \frac{R_0^2\widetilde{D}_1}{n\varepsilon^2}+ \frac{R_0^2\sqrt{L\widetilde{D}_1}}{\sqrt{\delta\varepsilon^3}}\right) \text{ iterations.}
	\end{equation*}
\end{corollary}

\subsection{{\tt EC-LSVRG}}\label{sec:ec_LSVRG}
In this section we consider problem \eqref{eq:main_problem} with $f(x)$ being $\mu$-quasi strongly convex and $f_i(x)$ satisfying \eqref{eq:f_i_sum} where functions $f_{ij}(x)$ are convex and $L$-smooth. For this problem we propose a new method called {\tt EC-LSVRG} which takes for the origin another method called {\tt LSVRG} (see \cite{hofmann2015variance,kovalev2019don}). 
\begin{algorithm}[t]
   \caption{{\tt EC-LSVRG}}\label{alg:ec-LSVRG}
\begin{algorithmic}[1]
   \Require learning rate $\gamma>0$, initial vector $x^0 \in \R^d$
   \State Set $e_i^0 = 0$ for all $i=1,\ldots, n$   
   \For{$k=0,1,\dotsc$}
       \State Broadcast $x^{k}$ to all workers
        \For{$i=1,\dotsc,n$ in parallel}
        	\State Pick $l$ uniformly at random from $[m]$
            \State Set $g^{k}_i = \nabla f_{il}(x^k) - \nabla f_{il}(w_i^k) + \nabla f_i(w_i^k)$
            \State $v_i^k = C(e_i^k + \gamma g_i^k)$
            \State $e_i^{k+1} = e_i^k + \gamma g_i^k - v_i^k$
            \State $w_i^{k+1} = \begin{cases}x^k,& \text{with probability } p,\\ w_i^k,& \text{with probability } 1-p\end{cases}$
        \EndFor
        \State $e^k = \frac{1}{n}\sum_{i=1}^ne_i^k$, $g^k = \frac{1}{n}\sum_{i=1}^ng_i^k$, $v^k = \frac{1}{n}\sum_{i=1}^nv_i^k$
       \State $x^{k+1} = x^k - v^k$
   \EndFor
\end{algorithmic}
\end{algorithm}

\begin{lemma}\label{lem:second_moment_bound_ec-LSVRG}
	For all $k\ge 0$, $i\in [n]$ we have
	\begin{equation}
		\bar{g}_i^k = \EE\left[g_i^k\mid x^k\right] = \nabla f_i(x^k) \label{eq:unbiasedness_g_i^k_ec-LSVRG}
	\end{equation}		
	and
	\begin{eqnarray}
		\frac{1}{n}\sum\limits_{i=1}^n\|\bar{g}_i^k\|^2 &\le& 4L\left(f(x^k) - f(x^*)\right) + D_1, \label{eq:second_moment_bound_ec-LSVRG}\\
		\frac{1}{n}\sum\limits_{i=1}^n\EE\left[\|g_i^k - \bar{g}_i^k\|^2\mid x^k\right] &\le& 12L\left(f(x^k) - f(x^*)\right) + 3\sigma_k^2, \label{eq:variance_bound_ec-LSVRG}\\
		\EE\left[\|g^k\|^2\mid x^k\right] &\le& 4L\left(f(x^k) - f(x^*)\right) + 2\sigma_k^2 \label{eq:second_moment_bound_ec-LSVRG_2}
	\end{eqnarray}
	where $\sigma_k^2 = \frac{1}{nm}\sum_{i=1}^n\sum_{j=1}^n\|\nabla f_{ij}(w_i^k) - \nabla f_{ij}(x^*)\|^2$ and $D_1 = \frac{2}{n}\sum_{i=1}^n\|\nabla f_{i}(x^*)\|^2$.
\end{lemma}
\begin{proof}
	First of all, we derive unbiasedness of $g_i^k$:
	\begin{equation*}
		\EE\left[g_i^k\mid x^k\right] = \frac{1}{m}\sum\limits_{j=1}^m\left(\nabla f_{ij}(x^k) - \nabla f_{ij}(w_i^k) + \nabla f_i(w_i^k)\right) = \nabla f_i(x^k).
	\end{equation*}
	Next, we get an upper bound for $\frac{1}{n}\sum\limits_{i=1}^n\|\bar{g}_i^k\|^2$:
	\begin{eqnarray*}
		\frac{1}{n}\sum\limits_{i=1}^n\|\bar{g}_i^k\|^2 &=& \frac{1}{n}\sum\limits_{i=1}^n\|\nabla f_i(x^k)\|^2\\
		&\overset{\eqref{eq:a_b_norm_squared}}{\le}& \frac{2}{n}\sum\limits_{i=1}^n\|\nabla f_i(x^k) - \nabla f_i(x^*)\|^2 + \frac{2}{n}\sum\limits_{i=1}^n\|\nabla f_i(x^*)\|^2\\
		&\overset{\eqref{eq:L_smoothness_cor}}{\le}& 4L\left(f(x^k)-f(x^*)\right) + \frac{2}{n}\sum\limits_{i=1}^n\|\nabla f_i(x^*)\|^2.
	\end{eqnarray*}
	Using \eqref{eq:unbiasedness_g_i^k_ec-LSVRG} we establish the following inequality:
	\begin{eqnarray*}
		\frac{1}{n}\sum\limits_{i=1}^n\EE\left[\|g_i^k - \bar{g}_i^k\|^2\mid x^k\right] &\overset{\eqref{eq:a_b_norm_squared}}{\le}& \frac{3}{n}\sum\limits_{i=1}^n\EE\left[\|\nabla f_{il}(x^k)- \nabla f_{il}(x^*)\|^2\mid x^k\right]\\
		&&\quad + \frac{3}{n}\sum\limits_{i=1}^n\EE\left[\left\|\nabla f_{il}(w_i^k)- \nabla f_{il}(x^*) - \left(\nabla f_i(w_i^k) - \nabla f_i(x^*)\right)\right\|^2\mid x^k\right] \\
		&&\quad + \frac{3}{n}\sum\limits_{i=1}^n\|\nabla f_{i}(x^*)-\nabla f_i(x^k)\|^2\\
		&\overset{\eqref{eq:L_smoothness_cor},\eqref{eq:variance_decomposition}}{\le}& 12L\left(f(x^k)-f(x^*)\right) + \frac{3}{nm}\sum\limits_{i=1}^n\sum\limits_{j=1}^m\|\nabla f_{ij}(w_i^k) - \nabla f_{ij}(x^*)\|^2.
	\end{eqnarray*}
	Finally, we derive \eqref{eq:second_moment_bound_ec-LSVRG_2}:
	\begin{eqnarray*}
		\EE\left[\|g^k\|^2\mid x^k\right] &=& \EE\left[\left\|\frac{1}{n}\sum\limits_{i=1}^n\left(\nabla f_{il}(x^k)-\nabla f_{il}(w_i^k) + \nabla f_i(w_i^k) - \nabla f_i(x^*)\right)\right\|^2\mid x^k\right]\\
		&\overset{\eqref{eq:a_b_norm_squared}}{\le}& \frac{2}{n}\sum\limits_{i=1}^n\EE\left[\|\nabla f_{il}(x^k)-\nabla f_{il}(x^*)\|^2\mid x^k\right]\\
		&&\quad + \frac{2}{n}\sum\limits_{i=1}^n\EE\left[\left\|\nabla f_{il}(w_i^k)- \nabla f_{il}(x^*) - \left(\nabla f_i(w_i^k) - \nabla f_i(x^*)\right)\right\|^2\mid x^k\right]\\
		&=& \frac{2}{nm}\sum\limits_{i=1}^n\sum\limits_{j=1}^m\|\nabla f_{ij}(x^k) - \nabla f_{ij}(x^*)\|^2\\
		&&\quad + \frac{2}{nm}\sum\limits_{i=1}^n\sum\limits_{j=1}^m\left\|\nabla f_{ij}(w_i^k) - \nabla f_{ij}(x^*) - \frac{1}{m}\sum\limits_{j=1}^m\left(\nabla f_{ij}(w_i^k) - \nabla f_{ij}(x^*)\right)\right\|^2\\
		&\overset{\eqref{eq:L_smoothness_cor},\eqref{eq:variance_decomposition}}{\le}& 4L\left(f(x^k)-f(x^*)\right) + \frac{2}{nm}\sum\limits_{i=1}^n\sum\limits_{j=1}^m\left\|\nabla f_{ij}(w_i^k) - \nabla f_{ij}(x^*)\right\|^2.
	\end{eqnarray*}
\end{proof}

\begin{lemma}\label{lem:sigma_k+1_bound_ec-LSVRG}
	For all $k\ge 0$, $i\in [n]$ we have
	\begin{equation}
		\EE\left[\sigma_{k+1}^2\mid x^k\right] \le (1-p)\sigma_k^2 + 2Lp\left(f(x^k) - f(x^*)\right), \label{eq:sigma_k+1_ec-LSVRG} 
	\end{equation}
	where $\sigma_k^2 = \frac{1}{nm}\sum_{i=1}^n\sum_{j=1}^n\|\nabla f_{ij}(w_i^k) - \nabla f_{ij}(x^*)\|^2$.
\end{lemma}
\begin{proof}
	By definition of $w_i^{k+1}$ we get
	\begin{eqnarray*}
		\EE\left[\sigma_{k+1}^2\mid x^k\right] &=& \frac{1}{nm}\sum\limits_{i=1}^n\sum\limits_{j=1}^m\EE\left[\|\nabla f_{ij}(w_i^{k+1}) - \nabla f_{ij}(x^*)\|^2\mid x^k\right]\\
		&=& \frac{1-p}{nm}\sum\limits_{i=1}^n\sum\limits_{j=1}^m\|\nabla f_{ij}(w_i^{k}) - \nabla f_{ij}(x^*)\|^2 + \frac{p}{nm}\sum\limits_{i=1}^n\sum\limits_{j=1}^m\|\nabla f_{ij}(x^{k}) - \nabla f_{ij}(x^*)\|^2\\
		&\overset{\eqref{eq:L_smoothness_cor}}{\le}& (1-p)\sigma_k^2 + \frac{2Lp}{nm}\sum\limits_{i=1}^n\sum\limits_{j=1}^m D_{f_{ij}}(x^k,x^*)\\
		&=& (1-p)\sigma_k^2 + 2Lp\left(f(x^k) - f(x^*)\right).
	\end{eqnarray*}
\end{proof}

Applying Theorem~\ref{thm:ec_sgd_main_result_new} we get the following result.
\begin{theorem}\label{thm:ec_LSVRG}
	Assume that $f(x)$ is $\mu$-quasi strongly convex and functions $f_{ij}$ are convex and $L$-smooth for all $i\in[n],j\in[m]$. Then {\tt EC-LSVRG} satisfies Assumption~\ref{ass:key_assumption_finite_sums_new} with
	\begin{gather*}
		A = 2L,\quad \widetilde{A} = 12L,\quad A' = 2L,\quad B_1 = \widetilde{B}_1 = B_1' = B_2 = 0,\quad D_1 = \frac{2}{n}\sum\limits_{i=1}^n\|\nabla f_{i}(x^*)\|^2,\\
		D_1' = \widetilde{D}_1 = 0,\quad \widetilde{B}_2 = 3,\quad B_2' = 2,\quad \sigma_{1,k}^2 \equiv 0,\quad C_1 = 0,\\
		\sigma_{2,k}^2 = \sigma_{k}^2 = \frac{1}{nm}\sum\limits_{i=1}^n\sum\limits_{j=1}^m\|\nabla f_{ij}(w_i^{k}) - \nabla f_{ij}(x^*)\|^2,\quad \rho_1 = 1,\quad \rho_2 = p,\quad C_2 = Lp,\quad D_2 = 0,\\
		G = 0,\quad F_1 = 0,\quad F_2 = \frac{72L\gamma^2}{\delta p\left(1-\min\left\{\frac{\gamma\mu}{2},\frac{p}{4}\right\}\right)},\quad D_3 = \frac{12L\gamma}{\delta^2}D_1,
	\end{gather*}
	with $\gamma$ satisfying
	\begin{equation*}
		\gamma \le \min\left\{\frac{1}{24L}, \frac{\delta}{8L\sqrt{3\left(2 + 3\delta\left(2+\frac{1}{1-p}\right)\right)}}\right\}, \quad M_2 = \frac{4}{p}.
	\end{equation*}
	and for all $K \ge 0$
	\begin{equation*}
		\EE\left[f(\bar x^K) - f(x^*)\right] \le \left(1 - \min\left\{\frac{\gamma\mu}{2},\frac{p}{4}\right\}\right)^K\frac{4(T^0 + \gamma F_2 \sigma_0^2)}{\gamma} + \frac{48L\gamma^2}{\delta^2}D_1
	\end{equation*}
	when $\mu > 0$ and
	\begin{equation*}
		\EE\left[f(\bar x^K) - f(x^*)\right] \le \frac{4(T^0 + \gamma F_2 \sigma_0^2)}{\gamma K} + \frac{48L\gamma^2}{\delta^2}D_1
	\end{equation*}
	when $\mu = 0$, where $T^k \eqdef \|x^k - x^*\|^2 + M_2\gamma^2 \sigma_k^2$.
\end{theorem}

In other words, {\tt EC-LSVRG} converges with linear rate $\cO\left(\left(\frac{1}{p} + \frac{\kappa}{\delta\sqrt{1-p}}\right)\ln\frac{1}{\varepsilon}\right)$ to the neighbourhood of the solution. If $m\ge 2$ then taking $p = \frac{1}{m}$ we get that in expectation the sample complexity of one iteration of {\tt EC-LSVRG} is $\cO(1)$ gradients calculations per node as for {\tt EC-SGDsr} with standard sampling and the rate of convergence to the neighbourhood becomes $\cO\left(\left(m + \frac{\kappa}{\delta}\right)\ln\frac{1}{\varepsilon}\right)$. We notice that the size of this neighbourhood is typically smaller than for {\tt EC-SGDsr}, but still the method fails to converge to the exact solution with linear rate. Applying Lemma~\ref{lem:lemma2_stich} we establish the rate of convergence to $\varepsilon$-solution in the case when $\mu > 0$.
\begin{corollary}\label{cor:ec_lsvrg_str_cvx_cor}
	Let the assumptions of Theorem~\ref{thm:ec_LSVRG} hold and $\mu > 0$. Then after $K$ iterations of {\tt EC-LSVRG} with the stepsize
	\begin{eqnarray*}
		\gamma_0 &=&  \min\left\{\frac{1}{24L}, \frac{\delta}{8L\sqrt{3\left(2 + 3\delta\left(2+\frac{1}{1-p}\right)\right)}}\right\},\\
		\tilde{T^0} &=& \|x^0-x^*\|^2 + M_2\gamma_0^2\sigma_0^2,\quad  \tilde{F_2} = \frac{72L\gamma_0^2}{\delta p\left(1-\min\left\{\frac{\gamma_0\mu}{2},\frac{p}{4}\right\}\right)},\\
		\gamma &=& \min\left\{\gamma_0, \frac{\ln\left(\max\left\{2,\frac{\delta^2\left(\tilde{T^0} + \tilde{F}_2\gamma_0\sigma_{0}^2\right)\mu^3K^3}{48LD_1}\right\}\right)}{\mu K}\right\},
	\end{eqnarray*}	 
	and $p = \frac{1}{m}$, $m\ge 2$ we have
	\begin{equation*}
		\EE\left[f(\bar{x}^K) - f(x^*)\right] = \widetilde\cO\left(\frac{L}{\delta}\left(\tilde{T^0} + \tilde{F}_2\gamma_0\sigma_{0}^2\right)\exp\left(-\min\left\{\frac{\delta\mu}{L},\frac{1}{m}\right\}K\right) + \frac{LD_1}{\delta^2\mu^2 K^2}\right).
	\end{equation*}
	That is, to achive $\EE\left[f(\bar{x}^K) - f(x^*)\right] \le \varepsilon$ {\tt EC-LSVRG} requires
	\begin{equation*}
		\widetilde{\cO}\left(m + \frac{L}{\delta\mu} + \frac{\sqrt{LD_1}}{\delta\mu\sqrt{\varepsilon}}\right) \text{ iterations.}
	\end{equation*}
\end{corollary}

Applying Lemma~\ref{lem:lemma_technical_cvx} we get the complexity result in the case when $\mu = 0$.
\begin{corollary}\label{cor:ec_lsvrg_cvx_cor}
	Let the assumptions of Theorem~\ref{thm:ec_LSVRG} hold and $\mu = 0$. Then after $K$ iterations of {\tt EC-LSVRG} with the stepsize
	\begin{eqnarray*}
		\gamma_0 &=& \min\left\{\frac{1}{24L}, \frac{\delta}{8L\sqrt{3\left(2 + 3\delta\left(2+\frac{1}{1-p}\right)\right)}}\right\},\quad R_0 = \|x^0-x^*\|,\\
		\gamma &=& \min\left\{\gamma_0, \sqrt{\frac{R_0^2p}{4\sigma_0^2}}, \sqrt[3]{\frac{R_0^2\delta p\left(1-\min\left\{\frac{\gamma_0\mu}{2},\frac{p}{4}\right\}\right)}{72L\sigma_0^2}}, \sqrt[3]{\frac{\delta^2 R_0^2}{12LD_1 K}}\right\},
	\end{eqnarray*}	
	and $p = \frac{1}{m}$, $m\ge 2$ we have $\EE\left[f(\bar{x}^K) - f(x^*)\right]$ of order
	\begin{equation*}
		\cO\left(\frac{L R_0^2}{\delta K} + \frac{\sqrt{m R_0^2\sigma_0^2}}{K} + \frac{\sqrt[3]{LR_0^4m\sigma_0^2}}{\sqrt[3]{\delta} K} + \frac{\sqrt[3]{LR_0^4}}{(\delta K)^{\nicefrac{2}{3}}}\sqrt[3]{\frac{1}{n}\sum\limits_{i=1}^n\|\nabla f_i(x^*)\|^2}\right).
	\end{equation*}
	That is, to achive $\EE\left[f(\bar{x}^K) - f(x^*)\right] \le \varepsilon$ {\tt EC-LSVRG} requires
	\begin{equation*}
		\cO\left(\frac{L R_0^2}{\delta\varepsilon} + \frac{\sqrt{m R_0^2\sigma_0^2}}{\varepsilon} + \frac{\sqrt[3]{LR_0^4m\sigma_0^2}}{\sqrt[3]{\delta} \varepsilon}+ \frac{R_0^2}{\delta\varepsilon^{\nicefrac{3}{2}}}\sqrt{\frac{L}{n}\sum\limits_{i=1}^n\|\nabla f_i(x^*)\|^2}\right)
	\end{equation*}
	iterations.
\end{corollary}

\subsection{{\tt EC-LSVRGstar}}\label{sec:ec_LSVRGstar}
In the setup of Section~\ref{sec:ec_LSVRG} we now assume that $i$-th node has an access to the $\nabla f_i(x^*)$. Under this unrealistic assumption we construct the method called {\tt EC-LSVRGstar} that asymptotically converges to the exact solution.

\begin{algorithm}[t]
   \caption{{\tt EC-LSVRGstar}}\label{alg:ec-LSVRGstar}
\begin{algorithmic}[1]
   \Require learning rate $\gamma>0$, initial vector $x^0 \in \R^d$
   \State Set $e_i^0 = 0$ for all $i=1,\ldots, n$   
   \For{$k=0,1,\dotsc$}
       \State Broadcast $x^{k}$ to all workers
        \For{$i=1,\dotsc,n$ in parallel}
        	\State Pick $l$ uniformly at random from $[m]$
            \State Set $g^{k}_i = \nabla f_{il}(x^k) - \nabla f_{il}(w_i^k) + \nabla f_i(w_i^k) - \nabla f_i(x^*)$
            \State $v_i^k = C(e_i^k + \gamma g_i^k)$
            \State $e_i^{k+1} = e_i^k + \gamma g_i^k - v_i^k$
            \State $w_i^{k+1} = \begin{cases}x^k,& \text{with probability } p,\\ w_i^k,& \text{with probability } 1-p\end{cases}$
        \EndFor
        \State $e^k = \frac{1}{n}\sum_{i=1}^ne_i^k$, $g^k = \frac{1}{n}\sum_{i=1}^ng_i^k$, $v^k = \frac{1}{n}\sum_{i=1}^nv_i^k$
       \State $x^{k+1} = x^k - v^k$
   \EndFor
\end{algorithmic}
\end{algorithm}

\begin{lemma}\label{lem:second_moment_bound_ec-LSVRGstar}
	For all $k\ge 0$, $i\in [n]$ we have
	\begin{equation}
		\EE\left[g^k\mid x^k\right] = \nabla f(x^k) \label{eq:unbiasedness_g^k_ec-LSVRGstar}
	\end{equation}		
	and
	\begin{equation}
		\frac{1}{n}\sum\limits_{i=1}^n\|\bar{g}_i^k\|^2 \le 2L\left(f(x^k) - f(x^*)\right), \label{eq:second_moment_bound_ec-LSVRGstar} 
	\end{equation}
	\begin{equation}
		\frac{1}{n}\sum\limits_{i=1}^n\EE\left[\|g_i^k-\bar{g}_i^k\|^2\mid x^k\right] \le 4L\left(f(x^k) - f(x^*)\right) + 2\sigma_k^2, \label{eq:variance_bound_ec-LSVRGstar} 
	\end{equation}
	\begin{equation}
		\EE\left[\|g^k\|^2\mid x^k\right] \le 4L\left(f(x^k) - f(x^*)\right) + 2\sigma_k^2, \label{eq:second_moment_bound_ec-LSVRGstar_2} 
	\end{equation}
	where $\sigma_k^2 = \frac{1}{nm}\sum_{i=1}^n\sum_{j=1}^n\|\nabla f_{ij}(w_i^k) - \nabla f_{ij}(x^*)\|^2$.
\end{lemma}
\begin{proof}
	First of all, we derive unbiasedness of $g^k$:
	\begin{eqnarray*}
		\EE\left[g^k\mid x^k\right] &=& \frac{1}{n}\sum\limits_{i=1}^n\EE\left[\nabla f_{il}(x^k) - \nabla f_{il}(w_i^k) + \nabla f_i(w_i^k) - \nabla f_i(x^*)\mid x^k\right]\\
		&=& \frac{1}{nm}\sum\limits_{i=1}^n\sum\limits_{j=1}^m\left(\nabla f_{ij}(x^k) - \nabla f_{ij}(w_i^k) + \nabla f_i(w_i^k) - \nabla f_i(x^*)\right)\\
		&=& \nabla f(x^k) + \frac{1}{n}\sum\limits_{i=1}^n\left(-\nabla f_i(w_i^k) + \nabla f_i(w_i^k)\right) - \nabla f(x^*) = \nabla f(x^k).
	\end{eqnarray*}
	Next, we get an upper bound for $\frac{1}{n}\sum\limits_{i=1}^n\|\bar{g}_i^k\|^2$:
	\begin{equation*}
		\frac{1}{n}\sum\limits_{i=1}^n\|\bar{g}_i^k\|^2 = \frac{1}{n}\sum\limits_{i=1}^n\|\nabla f_i(x^k)-\nabla f_i(x^*)\|^2 \overset{\eqref{eq:L_smoothness_cor}}{\le} 2L\left(f(x^k)-f(x^*)\right).
	\end{equation*}		
	Since the variance of random vector is not greater than its second moment we obtain:
	\begin{eqnarray}
		\frac{1}{n}\sum\limits_{i=1}^n\EE\left[\|g_i^k-\bar{g}_i^k\|^2\mid x^k\right] &\overset{\eqref{eq:variance_decomposition}}{\le}& \frac{1}{n}\sum\limits_{i=1}^n\EE\left[\|g_i^k\|^2\mid x^k\right]\notag\\
		&\overset{\eqref{eq:a_b_norm_squared}}{\le}& \frac{2}{n}\sum\limits_{i=1}^n\EE\left[\|\nabla f_{il}(x^k)- \nabla f_{il}(x^*)\|^2\mid x^k\right]\notag\\
		&&\quad + \frac{2}{n}\sum\limits_{i=1}^n\EE\left[\left\|\nabla f_{il}(w_i^k)- \nabla f_{il}(x^*) - \left(\nabla f_i(w_i^k) - \nabla f_i(x^*)\right)\right\|^2\mid x^k\right]\notag\\
		&\overset{\eqref{eq:L_smoothness_cor},\eqref{eq:variance_decomposition}}{\le}& 4L\left(f(x^k)-f(x^*)\right) + \frac{2}{nm}\sum\limits_{i=1}^n\sum\limits_{j=1}^m\|\nabla f_{ij}(w_i^k) - \nabla f_{ij}(x^*)\|^2.\notag
	\end{eqnarray}
	Inequality \eqref{eq:second_moment_bound_ec-LSVRGstar_2} trivially follows from the inequality above by Jensen's inequality and convexity of $\|\cdot\|^2$. 
\end{proof}

\begin{lemma}\label{lem:sigma_k+1_bound_ec-LSVRGstar}
	For all $k\ge 0$, $i\in [n]$ we have
	\begin{equation}
		\EE\left[\sigma_{k+1}^2\mid x^k\right] \le (1-p)\sigma_k^2 + 2Lp\left(f(x^k) - f(x^*)\right), \label{eq:sigma_k+1_ec-LSVRGstar} 
	\end{equation}
	where $\sigma_k^2 = \frac{1}{nm}\sum_{i=1}^n\sum_{j=1}^n\|\nabla f_{ij}(w_i^k) - \nabla f_{ij}(x^*)\|^2$.
\end{lemma}
\begin{proof} The proof of this lemma is identical to the proof of Lemma~\ref{lem:sigma_k+1_bound_ec-LSVRG}.
\end{proof}

Applying Theorem~\ref{thm:ec_sgd_main_result_new} we get the following result.
\begin{theorem}\label{thm:ec_LSVRGstar}
	Assume that $f(x)$ is $\mu$-quasi strongly convex and functions $f_{ij}$ are convex and $L$-smooth for all $i\in[n],j\in[m]$. Then {\tt EC-LSVRGstar} satisfies Assumption~\ref{ass:key_assumption_finite_sums_new} with
	\begin{gather*}
		A = L,\quad \widetilde{A} = A' = 2L,\quad B_1 = \widetilde{B}_1 = B_1' = B_2 = 0,\quad \widetilde{B}_2 = B_2' = 2,\quad D_1 = D_1' = 0,\\
		\sigma_{1,k}^2 \equiv 0,,\quad C_1 = 0,\quad\sigma_{2,k}^2 = \sigma_{k}^2 = \frac{1}{nm}\sum\limits_{i=1}^n\sum\limits_{j=1}^m\|\nabla f_{ij}(w_i^{k}) - \nabla f_{ij}(x^*)\|^2,\quad \rho_1 = 1,\\
		\rho_2 = p,\quad C_2 = Lp,\quad D_2 = 0,\quad G = 0,\quad F_1 = 0,\quad F_2 = \frac{48L\gamma^2(2+p)}{\delta p},\quad D_3 = 0,
	\end{gather*}
	with $\gamma$ satisfying
	\begin{equation*}
		\gamma \le \min\left\{\frac{3}{56L}, \frac{\delta}{8L\sqrt{3\left(1+\delta\left(1 + \frac{2}{1-p}\right)\right)}}\right\}, \quad M_2 = \frac{8}{3p}.
	\end{equation*}
	and for all $K \ge 0$
	\begin{equation*}
		\EE\left[f(\bar x^K) - f(x^*)\right] \le \left(1 - \min\left\{\frac{\gamma\mu}{2},\frac{p}{4}\right\}\right)^K\frac{4(T^0 + \gamma F_2 \sigma_0^2)}{\gamma}
	\end{equation*}	
	when $\mu > 0$ and
	\begin{equation*}
		\EE\left[f(\bar x^K) - f(x^*)\right] \le \frac{4(T^0 + \gamma F_2 \sigma_0^2)}{\gamma K}
	\end{equation*}
	when $\mu = 0$, where $T^k \eqdef \|x^k - x^*\|^2 + M_2\gamma^2 \sigma_k^2$.
\end{theorem}

In other words, {\tt EC-LSVRGstar} converges with linear rate $\cO\left(\left(\frac{1}{p} + \frac{\kappa}{\delta\sqrt{1-p}}\right)\ln\frac{1}{\varepsilon}\right)$ exactly to the solution when $\mu > 0$. If $m\ge 2$ then taking $p = \frac{1}{m}$ we get that in expectation the sample complexity of one iteration of {\tt EC-LSVRGstar} is $\cO(1)$ gradients calculations per node as for {\tt EC-SGDsr} with standard sampling and the rate of convergence becomes $\cO\left(\left(m + \frac{\kappa}{\delta}\right)\ln\frac{1}{\varepsilon}\right)$. 

Applying Lemma~\ref{lem:lemma_technical_cvx} we get the complexity result in the case when $\mu = 0$.
\begin{corollary}\label{cor:ec_lsvrg_star_cvx_cor}
	Let the assumptions of Theorem~\ref{thm:ec_LSVRGstar} hold and $\mu = 0$. Then after $K$ iterations of {\tt EC-LSVRGstar} with the stepsize
	\begin{eqnarray*}
		\gamma_0 &=& \min\left\{\frac{3}{56L}, \frac{\delta}{8L\sqrt{3\left(1+\delta\left(1 + \frac{2}{1-p}\right)\right)}}\right\},\quad R_0 = \|x^0-x^*\|,\\
		\gamma &=& \min\left\{\gamma_0, \sqrt{\frac{3pR_0^2}{8\sigma_0^2}}, \sqrt[3]{\frac{R_0^2\delta p\left(1-\min\left\{\frac{\gamma_0\mu}{2},\frac{p}{4}\right\}\right)}{72L\sigma_0^2}}\right\},
	\end{eqnarray*}	
	and $p = \frac{1}{m}$, $m\ge 2$ we have $\EE\left[f(\bar{x}^K) - f(x^*)\right]$ of order
	\begin{equation*}
		\cO\left(\frac{L R_0^2}{\delta K} + \frac{\sqrt{R_0^2m\sigma_0^2}}{K} + \frac{\sqrt[3]{LR_0^4m\sigma_0^2}}{\sqrt[3]{\delta}K}\right).
	\end{equation*}
	That is, to achive $\EE\left[f(\bar{x}^K) - f(x^*)\right] \le \varepsilon$ {\tt EC-LSVRGstar} requires
	\begin{equation*}
		\cO\left(\frac{L R_0^2}{\delta \varepsilon} + \frac{\sqrt{R_0^2m\sigma_0^2}}{\varepsilon} + \frac{\sqrt[3]{LR_0^4m\sigma_0^2}}{\sqrt[3]{\delta}\varepsilon}\right)
	\end{equation*}
	iterations.
\end{corollary}
However, such convergence guarantees are obtained under very restrictive assumption: the method requires to know vectors $\nabla f_i(x^*)$.

\subsection{{\tt EC-LSVRG-DIANA}}\label{sec:ec_LSVRG-diana}
In the setup of Section~\ref{sec:ec_LSVRG} we construct a new method called {\tt EC-LSVRG-DIANA} which does not require to know $\nabla f_i(x^*)$ and has linear convergence to the exact solution.
\begin{algorithm}[t]
   \caption{{\tt EC-LSVRG-DIANA}}\label{alg:ec-LSVRG-diana}
\begin{algorithmic}[1]
   \Require learning rates $\gamma>0$, $\alpha \in (0,1]$, initial vectors $x^0, h_1^0,\ldots, h_n^0 \in \R^d$
	\State Set $e_i^0 = 0$ for all $i=1,\ldots, n$   
	\State Set $h^0 = \frac{1}{n}\sum_{i=1}^n h_i^0$   
   \For{$k=0,1,\dotsc$}
       \State Broadcast $x^{k}, h^k$ to all workers
        \For{$i=1,\dotsc,n$ in parallel}
			\State Pick $l$ uniformly at random from $[m]$
            \State Set $\hat g^{k}_i = \nabla f_{il}(x^k) - \nabla f_{il}(w_i^k) + \nabla f_i(w_i^k)$           
            \State $g^{k}_i = \hat g_i^k - h_i^k + h^k$
            \State $v_i^k = C(e_i^k + \gamma g_i^k)$
            \State $e_i^{k+1} = e_i^k + \gamma g_i^k - v_i^k$
            \State $h_i^{k+1} = h_i^k + \alpha Q(\hat g_i^k - h_i^k)$
            \State $w_i^{k+1} = \begin{cases}x^k,& \text{with probability } p,\\ w_i^k,& \text{with probability } 1-p\end{cases}$
        \EndFor
        \State $e^k = \frac{1}{n}\sum_{i=1}^ne_i^k$, $g^k = \frac{1}{n}\sum_{i=1}^ng_i^k$, $v^k = \frac{1}{n}\sum_{i=1}^nv_i^k$, $h^{k+1} = \frac{1}{n}\sum\limits_{i=1}^n h_i^{k+1} = h^k + \alpha\frac{1}{n}\sum\limits_{i=1}^n Q(\hat g_i^k - h_i^k)$
       \State $x^{k+1} = x^k - v^k$
   \EndFor
\end{algorithmic}
\end{algorithm}
As in {\tt EC-SGD-DIANA} the master needs to gather only $C(e_i^k + \gamma g_i^k)$ and $Q(\hat g_i^k - h_i^k)$ from all nodes in order to perform an update.

\begin{lemma}\label{lem:ec_LSVRG-diana_second_moment_bound}
	Assume that $f_{ij}(x)$ is convex and $L$-smooth for all $i=1,\ldots,n$, $j=1,\ldots,m$. Then, for all $k\ge 0$ we have
	\begin{eqnarray}
		\EE\left[g^k\mid x^k\right] &=& \nabla f(x^k), \label{eq:ec_LSVRG-diana_unbiasedness}\\
		\frac{1}{n}\sum\limits_{i=1}^n\|\bar{g}_i^k\|^2 &\le& 4L\left(f(x^k) - f(x^*)\right) + 2\sigma_{1,k}^2, \label{eq:ec_LSVRG-diana_second_moment_bound}\\
		\frac{1}{n}\sum\limits_{i=1}^n\EE\left[\|g_i^k-\bar{g}_i^k\|^2\mid x^k\right] &\le& 6L\left(f(x^k) - f(x^*)\right) + 3\sigma_{1,k}^2 + 3\sigma_{2,k}^2, \label{eq:ec_LSVRG-diana_variance_bound}\\
		\EE\left[\|g^k\|^2\mid x^k\right] &\le& 4L\left(f(x^k) - f(x^*)\right) + 2\sigma_{2,k}^2 \label{eq:ec_LSVRG-diana_second_moment_bound_2}
	\end{eqnarray}
	where $$\sigma_{1,k}^2 = \frac{1}{n}\sum_{i=1}^n\|h_i^k - \nabla f(x^*)\|^2,\quad \sigma_{2,k}^2 =  \frac{1}{nm}\sum_{i=1}^n\sum_{j=1}^m\|\nabla f_{ij}(w_i^k) - \nabla f_{ij}(x^*)\|^2.$$
\end{lemma}
\begin{proof}
	First of all, we show unbiasedness of $g^k$:
	\begin{eqnarray*}
		\EE\left[g^k\mid x^k\right] &=& \frac{1}{n}\sum\limits_{i=1}^n\EE\left[\hat g_i^k - h_i^k + h^k\mid x^k\right]\\
		&=& \frac{1}{nm}\sum\limits_{i=1}^n\sum\limits_{j=1}^m\left(\nabla f_{ij}(x^k) - \nabla f_{ij}(w_i^k) + \nabla f_i(w_i^k) - h_i^k + h^k\right) = \nabla f(x^k).	
	\end{eqnarray*}
	Next, we derive the upper bound for $\frac{1}{n}\sum\limits_{i=1}^n\|\bar{g}_i^k\|^2$:
	\begin{eqnarray*}
		\frac{1}{n}\sum\limits_{i=1}^n\|\bar{g}_i^k\|^2 &=& \frac{1}{n}\sum\limits_{i=1}^n\|\nabla f_i(x^k)-h_i^k + h^k\|^2\\
		&\overset{\eqref{eq:a_b_norm_squared}}{\le}& \frac{2}{n}\sum\limits_{i=1}^n\|\nabla f_i(x^k)-\nabla f_i(x^*)\|^2 + \frac{2}{n}\sum\limits_{i=1}^n\left\|h_i^k - \nabla f_i(x^*)-\left(h^k-\nabla f(x^*)\right)\right\|^2\\
		&\overset{\eqref{eq:L_smoothness_cor},\eqref{eq:variance_decomposition}}{\le}& 4L\left(f(x^k)-f(x^*)\right) + \frac{2}{n}\sum\limits_{i=1}^n\|h_i^k-\nabla f_i(x^*)\|^2.
	\end{eqnarray*}
	Since the variance of random vector is not greater than its second moment we obtain:
	\begin{eqnarray*}
		\frac{1}{n}\sum\limits_{i=1}^n\EE\left[\|g_i^k-\bar{g}_i^k\|^2\mid x^k\right] &\overset{\eqref{eq:variance_decomposition}}{\le}& \frac{1}{n}\sum\limits_{i=1}^n\EE\left[\|g_i^k\|^2\mid x^k\right]\\
		&=& \frac{1}{n}\sum\limits_{i=1}^n\EE\left[\|\nabla f_{il}(x^k) - \nabla f_{il}(w_i^k) + \nabla f_i(w_i^k) - h_i^k + h^k\|^2\mid x^k\right]\\
		&\overset{\eqref{eq:a_b_norm_squared}}{\le}& \frac{3}{n}\sum\limits_{i=1}^n\EE\left[\left\|\nabla f_{il}(x^k) - \nabla f_{il}(x^*)\right\|^2\mid x^k\right]\\
		&&\quad + \frac{3}{n}\sum\limits_{i=1}^n\EE\left[\left\|\nabla f_{il}(w_i^k) - \nabla f_{il}(x^*) -\left(\nabla f_i(w_i^k) - \nabla f_i(x^*)\right)\right\|^2\mid x^k\right]\\
		&&\quad + \frac{3}{n}\sum\limits_{i=1}^n\left\|h_i^k - \nabla f_{i}(x^*) -\left(h^k - \nabla f(x^*)\right)\right\|^2\\
		&\overset{\eqref{eq:L_smoothness_cor},\eqref{eq:variance_decomposition}}{\le}& 6L\left(f(x^k)-f(x^*)\right) + \frac{3}{nm}\sum\limits_{i=1}^n\sum\limits_{j=1}^m\|\nabla f_{ij}(w_i^k) - \nabla f_{ij}(x^*)\|^2\\
		&&\quad + \frac{3}{n}\sum\limits_{i=1}^n\left\|h_i^k - \nabla f_{i}(x^*)\right\|^2.
	\end{eqnarray*}
	Finally, we obtain an upper boud for the second moment of $g^k$:
	\begin{eqnarray*}
		\EE\left[\|g^k\|^2\mid x^k\right] &=& \EE\left[\left\|\frac{1}{n}\sum\limits_{i=1}^n\left(\nabla f_{il}(x^k)-\nabla f_{il}(w_i^k) + \nabla f_i(w_i^k) - \nabla f_i(x^*)\right)\right\|^2\mid x^k\right]\\
		&\overset{\eqref{eq:a_b_norm_squared}}{\le}& \frac{2}{n}\sum\limits_{i=1}^n\EE\left[\|\nabla f_{il}(x^k)-\nabla f_{il}(x^*)\|^2\mid x^k\right]\\
		&&\quad + \frac{2}{n}\sum\limits_{i=1}^n\EE\left[\left\|\nabla f_{il}(w_i^k)- \nabla f_{il}(x^*) - \left(\nabla f_i(w_i^k) - \nabla f_i(x^*)\right)\right\|^2\mid x^k\right]\\
		&=& \frac{2}{nm}\sum\limits_{i=1}^n\sum\limits_{j=1}^m\|\nabla f_{ij}(x^k) - \nabla f_{ij}(x^*)\|^2\\
		&&\quad + \frac{2}{nm}\sum\limits_{i=1}^n\sum\limits_{j=1}^m\left\|\nabla f_{ij}(w_i^k) - \nabla f_{ij}(x^*) - \frac{1}{m}\sum\limits_{j=1}^m\left(\nabla f_{ij}(w_i^k) - \nabla f_{ij}(x^*)\right)\right\|^2\\
		&\overset{\eqref{eq:L_smoothness_cor},\eqref{eq:variance_decomposition}}{\le}& 4L\left(f(x^k)-f(x^*)\right) + \frac{2}{nm}\sum\limits_{i=1}^n\sum\limits_{j=1}^m\left\|\nabla f_{ij}(w_i^k) - \nabla f_{ij}(x^*)\right\|^2.
	\end{eqnarray*}
\end{proof}

\begin{lemma}\label{lem:ec_LSVRG-diana_sigma_k+1_bound}
	Assume that $\alpha \le \nicefrac{1}{(\omega+1)}$. Then, for all $k\ge 0$ we have
	\begin{equation}
		\EE\left[\sigma_{1,k+1}^2\mid x^k\right] \le (1 - \alpha)\sigma_{1,k}^2 + 6L\alpha(f(x^k) - f(x^*)) + 2\alpha\sigma_{2,k}^2, \label{eq:ec_LSVRG-diana_sigma_k+1_bound}
	\end{equation}
	\begin{equation}
		\EE\left[\sigma_{2,{k+1}}^2\mid x^k\right] \le (1 - p)\sigma_{k,2}^2 + 2Lp\left(f(x^k)-f(x^*)\right)
	\end{equation}
	where $\sigma_{1,k}^2 = \frac{1}{n}\sum_{i=1}^n\|h_i^k - \nabla f_i(x^*)\|^2$ and $\sigma_{2,k}^2= \frac{1}{nm}\sum_{i=1}^n\sum_{j=1}^m\|\nabla f_{ij}(w_i^k) - \nabla f_{ij}(x^*)\|^2$.
\end{lemma}
\begin{proof}
	First of all, we derive an upper bound for the second moment of $h_i^{k+1} - h_i^*$:
	\begin{eqnarray*}
		\EE\left[\|h_i^{k+1} - h_i^*\|^2\mid x^k\right] &=& \EE\left[\left\|h_i^k - h_i^* + \alpha Q(\hat g_i^k - h_i^k) \right\|^2\mid x^k\right]\\
		&\overset{\eqref{eq:quantization_def}}{=}& \|h_i^k - h_i^*\|^2 +2\alpha\langle h_i^k - h_i^*, \nabla f_i(x^k) - h_i^k \rangle\\
		&&\quad + \alpha^2\EE\left[\|Q(\hat g_i^k - h_i^k)\|^2\mid x^k\right]\\
		&\overset{\eqref{eq:quantization_def},\eqref{eq:tower_property}}{\le}& \|h_i^k - h_i^*\|^2 +2\alpha\langle h_i^k - h_i^*, \nabla f_i(x^k) - h_i^k \rangle\\
		&&\quad + \alpha^2(\omega+1)\EE\left[\|\hat g_i^k - h_i^k\|^2\mid x^k\right].
	\end{eqnarray*}
	Using variance decomposition \eqref{eq:variance_decomposition} and $\alpha \le \nicefrac{1}{(\omega+1)}$ we get
	\begin{eqnarray*}
		\alpha^2(\omega+1)\EE\left[\|\hat g_i^k - h_i^k\|^2\mid x^k\right] &\overset{\eqref{eq:variance_decomposition}}{=}& \alpha^2(\omega+1)\EE\left[\|\hat g_i^k - \nabla f_i(x^k)\|^2\mid x^k\right] + \alpha^2(\omega+1)\|\nabla f_i(x^k) - h_i^k\|^2\\
		&\le& \alpha\EE\left[\|\hat g_i^k - \nabla f_i(x^k)\|^2\mid x^k\right] + \alpha\|\nabla f_i(x^k) - h_i^k\|^2\\
		&\overset{\eqref{eq:a_b_norm_squared}}{\le}& 2\alpha\EE\left[\left\|\nabla f_{il}(x^k) - \nabla f_{il}(x^*) -\left(\nabla f_i(x^k) - \nabla f_i(x^*)\right)\right\|^2\mid x^k\right]\\
		&&\quad + 2\alpha\EE\left[\left\|\nabla f_{il}(w_i^k) - \nabla f_{il}(x^*) -\left(\nabla f_i(w_i^k) - \nabla f_i(x^*)\right)\right\|^2\mid x^k\right]\\
		&&\quad+ \alpha\|\nabla f_i(x^k) - h_i^k\|^2\\
		&\overset{\eqref{eq:variance_decomposition}}{\le}&	2\alpha\EE\left[\left\|\nabla f_{il}(x^k) - \nabla f_{il}(x^*)\right\|^2\mid x^k\right]\\
		&&\quad + 2\alpha\EE\left[\left\|\nabla f_{il}(w_i^k) - \nabla f_{il}(x^*)\right\|^2\mid x^k\right] + \alpha\|\nabla f_i(x^k) - h_i^k\|^2\\
		&\overset{\eqref{eq:L_smoothness_cor}}{\le}& 4L\alpha D_{f_i}(x^k,x^*) + \frac{2\alpha}{m}\sum\limits_{j=1}^m\|\nabla f_{ij}(w_i^k) - \nabla f_{ij}(x^*)\|^2\\
		&&\quad + \alpha\|\nabla f_i(x^k) - h_i^k\|^2
	\end{eqnarray*}
	Putting all together we obtain
	\begin{eqnarray*}
		\EE\left[\|h_i^{k+1} - h_i^*\|^2\mid x^k\right] &\le& \|h_i^k - h_i^*\|^2 + \alpha\left\langle \nabla f_i(x^k) - h_i^k, f_i(x^k) + h_i^k - 2h_i^* \right\rangle\\
		&&\quad + 4L\alpha D_{f_i}(x^k,x^*) + \frac{2\alpha}{m}\sum\limits_{j=1}^m\|\nabla f_{ij}(w_i^k) - \nabla f_{ij}(x^*)\|^2\\
		&\overset{\eqref{eq:a-b_a+b}}{=}& \|h_i^k - h_i^*\|^2 + \alpha\|\nabla f_i(x^k) - h_i^*\|^2 - \alpha\|h_i^k - h_i^*\|^2\\
		&&\quad + 4L\alpha D_{f_i}(x^k,x^*) + \frac{2\alpha}{m}\sum\limits_{j=1}^m\|\nabla f_{ij}(w_i^k) - \nabla f_{ij}(x^*)\|^2\\
		&\overset{\eqref{eq:L_smoothness_cor}}{\le}& (1-\alpha)\|h_i^k - h_i^*\|^2 + 6L\alpha D_{f_i}(x^k,x^*)\\
		&&\quad + \frac{2\alpha}{m}\sum\limits_{j=1}^m\|\nabla f_{ij}(w_i^k) - \nabla f_{ij}(x^*)\|^2.
	\end{eqnarray*}
	Summing up the above inequality for $i=1,\ldots, n$ we derive
	\begin{eqnarray}
		\EE\left[\sigma_{1,k+1}^2\mid x^k\right] &\le& (1-\alpha)\sigma_{1,k}^2 + 6L\alpha(f(x^k) - f(x^*)) + 2\alpha\sigma_{2,k}^2.\notag
	\end{eqnarray}
	Similarly to the proof of Lemma~\ref{lem:sigma_k+1_bound_ec-LSVRG} we get
	\begin{eqnarray}
		\EE\left[\sigma_{2,k+1}^2\mid x^k\right] &=& \frac{1}{nm}\sum\limits_{i=1}^n\sum\limits_{j=1}^m\EE\left[\|\nabla f_{ij}(w_i^{k+1}) - \nabla f_{ij}(x^*)\|^2\mid x^k\right]\notag\\
		&=& \frac{1-p}{nm}\sum\limits_{i=1}^n\sum\limits_{j=1}^m\|\nabla f_{ij}(w_i^{k}) - \nabla f_{ij}(x^*)\|^2\notag\\
		&&\quad + \frac{p}{nm}\sum\limits_{i=1}^n\sum\limits_{j=1}^m\|\nabla f_{ij}(x^{k}) - \nabla f_{ij}(x^*)\|^2\notag\\
		&\overset{\eqref{eq:L_smoothness_cor}}{\le}& (1-p)\sigma_{2,k}^2 + \frac{2Lp}{nm}\sum\limits_{i=1}^n\sum\limits_{j=1}^m D_{f_{ij}}(x^k,x^*)\notag\\
		&=& (1-p)\sigma_{2,k}^2 + 2Lp\left(f(x^k) - f(x^*)\right).\notag
	\end{eqnarray}
\end{proof}

Applying Theorem~\ref{thm:ec_sgd_main_result_new} we get the following result.
\begin{theorem}\label{thm:ec_LSVRG-diana}
	Assume that $f_{ij}(x)$ is convex and $L$-smooth for all $i=1,\ldots, n$, $j=1,\ldots,m$ and $f(x)$ is $\mu$-quasi strongly convex. Then {\tt EC-LSVRG-DIANA} satisfies Assumption~\ref{ass:key_assumption_finite_sums_new} with
	\begin{gather*}
		A = A' = 2L,\quad B_1' = B_2 = 0,\quad B_1 = B_2' = 2,\quad D_1 = \widetilde{D}_1 = D_1' = D_2 = D_3 = 0,\\
		\widetilde{A} = 3L,\quad \widetilde{B}_1 = \widetilde{B}_2 = 3,\quad \sigma_{1,k}^2 = \frac{1}{n}\sum\limits_{i=1}^n\|h_i^k - \nabla f_i(x^*)\|^2,\quad \rho_1 = \alpha,\\
		\sigma_{2,k}^2 = \frac{1}{nm}\sum\limits_{i=1}^n\sum_{j=1}^m\|\nabla f_{ij}(w_i^k) - \nabla f_{ij}(x^*)\|^2,\quad \rho_2 = p,\quad C_1 = 3L\alpha,\quad C_2 = Lp,\\
		G = 2,\quad F_1 = \frac{24L\gamma^2\left(\frac{4}{\delta}+3\right)}{\delta\alpha\left(1-\min\left\{\frac{\gamma\mu}{2},\frac{\alpha}{4},\frac{p}{4}\right\}\right)},\quad F_2 = \frac{24L\gamma^2\left(\frac{4}{1-\alpha}\left(\frac{4}{\delta}+3\right) + 3\right)}{\delta p\left(1-\min\left\{\frac{\gamma\mu}{2},\frac{\alpha}{4},\frac{p}{4}\right\}\right)},
	\end{gather*}
	with $\gamma$ and $\alpha$ satisfying
	\begin{equation*}
		\gamma \le \min\left\{\frac{9}{296L}, \frac{\delta}{4L\sqrt{6\left(4+3\delta+\frac{2}{1-\alpha}\left(3+\frac{4}{1-p}\right)(4+3\delta)+\frac{6\delta}{1-p}\right)}}\right\},\quad \alpha \le \frac{1}{\omega+1}
	\end{equation*}
	with $M_1 = 0$ and $M_2 = \frac{8}{3p} + \frac{32}{9p}$ and for all $K \ge 0$
	\begin{equation*}
		\EE\left[f(\bar x^K) - f(x^*)\right] \le \left(1 - \min\left\{\frac{\gamma\mu}{2},\frac{\alpha}{4},\frac{p}{4}\right\}\right)^K\frac{4(T^0 + \gamma F_1 \sigma_{1,0}^2 + \gamma F_2 \sigma_{2,0}^2)}{\gamma},
	\end{equation*}	
	when $\mu > 0$ and
	\begin{equation*}
		\EE\left[f(\bar x^K) - f(x^*)\right] \le \frac{4(T^0 + \gamma F_1 \sigma_{1,0}^2 + \gamma F_2 \sigma_{2,0}^2)}{K\gamma}
	\end{equation*}
	when $\mu = 0$, where $T^k \eqdef \|x^k - x^*\|^2+ M_2\gamma^2 \sigma_{2,k}^2$.
\end{theorem}
In other words, if $p = \nicefrac{1}{m}$, $m \ge 2$ and
\begin{equation*}
		\gamma = \min\left\{\frac{9}{296L}, \frac{\delta}{4L\sqrt{6\left(4+3\delta+\frac{2}{1-\alpha}\left(3+\frac{4}{1-p}\right)(4+3\delta)+\frac{6\delta}{1-p}\right)}}\right\},\quad \alpha = \min\left\{\frac{1}{\omega+1},\frac{1}{2}\right\},
\end{equation*}
then {\tt EC-LSVRG-DIANA} converges with the linear rate
\begin{equation*}
	\cO\left(\left(\omega + m + \frac{\kappa}{\delta}\right)\ln\frac{1}{\varepsilon}\right)
\end{equation*}
to the exact solution when $\mu > 0$.

Applying Lemma~\ref{lem:lemma_technical_cvx} we get the complexity result in the case when $\mu = 0$.
\begin{corollary}\label{cor:ec_lsvrg_diana_cvx_cor}
	Let the assumptions of Theorem~\ref{thm:ec_LSVRG-diana} hold and $\mu = 0$. Then after $K$ iterations of {\tt EC-LSVRG-DIANA} with the stepsize
	\begin{eqnarray*}
		\gamma_0 &=& \min\left\{\frac{9}{296L}, \frac{\delta}{4L\sqrt{6\left(4+3\delta+\frac{2}{1-\alpha}\left(3+\frac{4}{1-p}\right)(4+3\delta)+\frac{6\delta}{1-p}\right)}}\right\},\quad R_0 = \|x^0-x^*\|,\\
		\gamma &=& \min\left\{\gamma_0, \sqrt{\frac{9pR_0^2}{56\sigma_{2,0}^2}}, \sqrt[3]{\frac{R_0^2}{\frac{24L\left(\frac{4}{\delta}+3\right)}{\delta\alpha\left(1-\min\left\{\frac{\gamma_0\mu}{2},\frac{\alpha}{4},\frac{p}{4}\right\}\right)}\sigma_{1,0}^2 + \frac{24L\left(\frac{4}{1-\alpha}\left(\frac{4}{\delta}+3\right) + 3\right)}{\delta p\left(1-\min\left\{\frac{\gamma_0\mu}{2},\frac{\alpha}{4},\frac{p}{4}\right\}\right)}\sigma_{2,0}^2}}\right\},
	\end{eqnarray*}	
	and $p = \frac{1}{m}$, $m\ge 2$, $\alpha = \min\left\{\frac{1}{\omega+1},\frac{1}{2}\right\}$ we have $\EE\left[f(\bar{x}^K) - f(x^*)\right]$ of order
	\begin{equation*}
		\cO\left(\frac{L R_0^2}{\delta K} + \frac{\sqrt{R_0^2 m\sigma_{2,0}^2}}{K} + \frac{\sqrt[3]{LR_0^4((\omega+1)\sigma_{1,0}^2 + m\sigma_{2,0}^2)}}{\delta^{\nicefrac{2}{3}}K}\right).
	\end{equation*}
	That is, to achive $\EE\left[f(\bar{x}^K) - f(x^*)\right] \le \varepsilon$ {\tt EC-LSVRG-DIANA} requires
	\begin{equation*}
		\cO\left(\frac{L R_0^2}{\delta \varepsilon} + \frac{\sqrt{R_0^2 m\sigma_{2,0}^2}}{\varepsilon} + \frac{\sqrt[3]{LR_0^4((\omega+1)\sigma_{1,0}^2 + m\sigma_{2,0}^2)}}{\delta^{\nicefrac{2}{3}}\varepsilon}\right)
	\end{equation*}
	iterations.
\end{corollary}

\clearpage

\section{Special Cases: Delayed Updates Methods}\label{sec:special_cases2}

\begin{table*}[!t]
\caption{Complexity of SGD methods with delayed updates established in this paper. Symbols: $\varepsilon = $ error tolerance; $\delta = $ contraction factor of compressor $\cC$; $\omega = $ variance parameter of compressor $\cQ$; $\kappa = \nicefrac{L}{\mu}$; $\cL =$ expected smoothness constant; $\sigma_*^2 = $ variance of the stochastic gradients in the solution; $\zeta_*^2 =$ average of $\|\nabla f_i(x^*)\|^2$; $\sigma^2 =$ average of the uniform bounds for the variances of stochastic gradients of workers; $\cM_{2,q} = (\omega+1)\sigma^2 + \omega\zeta_*^2$; $\sigma^2_q = (1+\omega)\left(1+\frac{\omega}{n}\right)\sigma^2$. $^\dagger${\tt D-QGDstar} is a special case of {\tt D-QSGDstar} where each worker $i$ computes the full gradient  $\nabla f_i(x^k)$; $^\ddagger${\tt D-GD-DIANA} is a special case of {\tt D-SGD-DIANA} where each worker $i$ computes the full gradient  $\nabla f_i(x^k)$.
}
\label{tbl:special_cases_delayed_methods}
\begin{center}
\footnotesize
\begin{tabular}{|c|l|c|c|c|c|}
\hline
\bf Problem & \bf Method &   \bf Alg \# &  \bf Citation &  \bf  Sec \#  
%&  \bf Thm \# 
& \bf Rate (constants ignored)\\
\hline
\eqref{eq:main_problem}+\eqref{eq:f_i_sum} & {\tt D-SGDsr}  & Alg \ref{alg:d-SGDsr} & {\color{red}\bf new} & \ref{sec:d_SGDsr} 
%&  \ref{thm:ec_SGDsr} 
& {\color{red}$\widetilde{\cO}\left(\frac{\cL + \sqrt{L^2\tau^2 + L\cL\tau}}{\mu} + \frac{\sigma_*^2}{n\mu\varepsilon} + \frac{\sqrt{L\tau \sigma_*^2}}{\mu\sqrt{n\varepsilon}}\right)$}\\
%%%%%%%%%%%%%%%%%%%%
%%%%%%%%%%%%%%%%%%%%
\hline
\eqref{eq:main_problem}+\eqref{eq:f_i_expectation} & {\tt D-SGD}  & Alg \ref{alg:d-sgd} & {\cite{stich2019error}}  & \ref{sec:d_sgd_pure} 
%&  \ref{thm:ec_sgd_pure} 
& $\widetilde{\cO}\left(\tau\kappa + \frac{\sigma_*^2}{n\mu\varepsilon} + \frac{\sqrt{L\tau\sigma_*^2}}{\mu\sqrt{n\varepsilon}}\right)$\\
%%%%%%%%%%%%%%%%%%%%
%%%%%%%%%%%%%%%%%%%%
%\hline
%\eqref{eq:main_problem}+\eqref{eq:f_i_expectation} & {\tt D-SGD}  & Alg \ref{alg:d-sgd} & {\cite{agarwal2011distributed}}  & \xmark &  \xmark & \xmark 
%% &  \xmark 
%& \ref{sec:d_sgd_pure} &  \ref{thm:d_sgd_pure} \\
%%%%%%%%%%%%%%%%%%%%
%%%%%%%%%%%%%%%%%%%%
\hline
\eqref{eq:main_problem}+\eqref{eq:f_i_expectation} & {\tt D-QSGD}  & Alg \ref{alg:d-qsgd} & {\color{red}\bf new}
& \ref{sec:d_qsgd} 
%& \ref{thm:ec_sgd_star} 
& {\color{red} $\widetilde{\cO}\left( \kappa\left(\tau + \frac{\omega}{n}\right) + \frac{\cM_{2,q}}{n\mu\varepsilon} + \frac{\sqrt{L\tau\cM_{2,q}}}{\mu\sqrt{n\varepsilon}} \right)$ } \\
%%%%%%%%%%%%%%%%%%%%
%%%%%%%%%%%%%%%%%%%%
\hline
\eqref{eq:main_problem}+\eqref{eq:f_i_expectation} & {\tt D-QSGDstar}  & Alg \ref{alg:d-qSGDstar} & {\color{red}\bf new}
& \ref{sec:d_qsgd_star} 
%& \ref{thm:ec_sgd_star} 
& {\color{red} $\widetilde{\cO}\left( \kappa\left(\tau + \frac{\omega}{n}\right) + \frac{\sigma^2}{n\mu\varepsilon} + \frac{\sqrt{L\tau\sigma^2}}{\mu\sqrt{n\varepsilon}} \right)$ } \\
%%%%%%%%%%%%%%%%%%%%
%%%%%%%%%%%%%%%%%%%%
\hline
\eqref{eq:main_problem}+\eqref{eq:f_i_expectation} & {\tt D-QGDstar}$^\dagger$  & Alg \ref{alg:d-qSGDstar} & {\color{red}\bf new}
& \ref{sec:d_qsgd_star} 
%& \ref{thm:ec_sgd_star} 
& {\color{red} $\cO\left( \kappa\left(\tau + \frac{\omega}{n}\right)\log\frac{1}{\varepsilon}\right)$ } \\
%%%%%%%%%%%%%%%%%%%%
%%%%%%%%%%%%%%%%%%%%
\hline
\eqref{eq:main_problem}+\eqref{eq:f_i_expectation} & {\tt D-SGD-DIANA}  & Alg \ref{alg:d-diana} & {\color{red}\bf new}
& \ref{sec:d_diana} 
%&  \ref{thm:ec_diana} 
& {\color{red}$\widetilde{\cO}\left(\omega +\kappa\left(\tau + \frac{\omega}{n}\right) + \frac{\sigma^2}{n\mu\varepsilon} + \frac{\sqrt{L\tau\sigma_q^2}}{\mu\sqrt{n\varepsilon}}\right)$}\\
%%%%%%%%%%%%%%%%%%%%
%%%%%%%%%%%%%%%%%%%%
\hline
\eqref{eq:main_problem}+\eqref{eq:f_i_expectation} & {\tt D-GD-DIANA}$^\ddagger$  & Alg \ref{alg:d-diana} & {\color{red}\bf new}
& \ref{sec:d_diana} 
%&  \ref{thm:ec_diana} 
& {\color{red}$\cO\left(\left(\omega + \kappa\left(\tau + \frac{\omega}{n}\right)\right) \log \frac{1}{\varepsilon}\right)$} \\
%%%%%%%%%%%%%%%%%%%%
%%%%%%%%%%%%%%%%%%%%
%\hline
%\eqref{eq:main_problem}+\eqref{eq:f_i_expectation} & {\tt D-QSGD}  & Alg \ref{alg:d-qsgd} & NEW???  & \xmark &  \xmark & \cmark 
%% &  \xmark 
%& \ref{sec:d_qsgd} &  \ref{thm:d_qsgd} \\
%%%%%%%%%%%%%%%%%%%%
%%%%%%%%%%%%%%%%%%%%
%\hline
%\eqref{eq:main_problem}+\eqref{eq:f_i_expectation} & {\tt D-QSGDstar}  & Alg \ref{alg:d-qSGDstar} & NEW  & \cmark &  \cmark & \xmark 
%% &  \xmark 
%& \ref{sec:d_qsgd_star} & \ref{thm:d_qsgd_star} \\
%%%%%%%%%%%%%%%%%%%%%
%%%%%%%%%%%%%%%%%%%%%
%\hline
%\eqref{eq:main_problem}+\eqref{eq:f_i_expectation} & {\tt D-DIANA}  & Alg \ref{alg:d-diana} & NEW  & \xmark\cmark &  \xmark & \cmark 
%% &  \xmark 
%& \ref{sec:d_diana} & \ref{thm:d_diana} \\
%%%%%%%%%%%%%%%%%%%%
%%%%%%%%%%%%%%%%%%%%
%\hline
%\eqref{eq:main_problem}+\eqref{eq:f_i_sum} & {\tt D-SGDsr}  & Alg \ref{alg:d-SGDsr} & NEW  & \xmark & \cmark & \xmark 
%% &  \xmark 
%& \ref{sec:d_SGDsr} &  \ref{thm:d_SGDsr} \\
%%%%%%%%%%%%%%%%%%%%
%%%%%%%%%%%%%%%%%%%%
\hline
\eqref{eq:main_problem}+\eqref{eq:f_i_sum} & {\tt D-LSVRG}  & Alg \ref{alg:d-LSVRG} & {\color{red}\bf new}
& \ref{sec:d_LSVRG} 
%&  \ref{thm:ec_LSVRG} 
& {\color{red}$\cO\left(\left(m + \kappa\tau\right)\log\frac{1}{\varepsilon}\right)$}\\
%%%%%%%%%%%%%%%%%%%%
%%%%%%%%%%%%%%%%%%%%
\hline
\eqref{eq:main_problem}+\eqref{eq:f_i_sum} & {\tt D-QLSVRG}  & Alg \ref{alg:d-qLSVRG} & {\color{red}\bf new}
& \ref{sec:d_qLSVRG} 
%&  \ref{thm:ec_LSVRG} 
& {\color{red}$\widetilde{\cO}\left(m + \kappa\left(\tau+\frac{\omega}{n}\right) + \frac{\zeta_*^2}{n\mu\varepsilon} + \frac{\sqrt{L\tau\zeta_*^2}}{\mu\sqrt{n\varepsilon}} \right)$}\\
%%%%%%%%%%%%%%%%%%%%
%%%%%%%%%%%%%%%%%%%%
\hline
\eqref{eq:main_problem}+\eqref{eq:f_i_sum} & {\tt D-QLSVRGstar}  & Alg \ref{alg:d-qLSVRGstar} & {\color{red}\bf new}
 & \ref{sec:d_qLSVRGstar} 
% & \ref{thm:ec_LSVRGstar} 
 & {\color{red}$\cO\left(\left(m + \kappa\left(\tau+\frac{\omega}{n}\right)\right)\log\frac{1}{\varepsilon}\right)$} \\
%%%%%%%%%%%%%%%%%%%%
%%%%%%%%%%%%%%%%%%%%
\hline
\eqref{eq:main_problem}+\eqref{eq:f_i_sum} & {\tt D-LSVRG-DIANA}  & Alg \ref{alg:d-LSVRG-diana} & {\color{red}\bf new}
& \ref{sec:d_LSVRG-diana} 
%& \ref{thm:ec_LSVRG-diana} 
& {\color{red}$\cO \left(\left( \omega + m + \kappa\left(\tau+\frac{\omega}{n}\right) \right) \log \frac{1}{\varepsilon}\right)$} \\
%%%%%%%%%%%%%%%%%%%%
%%%%%%%%%%%%%%%%%%%%
%\hline
%\eqref{eq:main_problem}+\eqref{eq:f_i_sum} & {\tt D-LSVRG}  & Alg \ref{alg:d-LSVRG} & NEW  & \cmark & \xmark & \xmark 
%% &  \xmark 
%& \ref{sec:d_LSVRG} & \ref{thm:d_LSVRG} \\
%%%%%%%%%%%%%%%%%%%%%
%%%%%%%%%%%%%%%%%%%%%
%\hline
%\eqref{eq:main_problem}+\eqref{eq:f_i_sum} & {\tt D-QLSVRG}  & Alg \ref{alg:d-qLSVRG} & NEW  & \xmark\cmark & \xmark & \cmark  
%% & \xmark 
%& \ref{sec:d_qLSVRG} & \ref{thm:d_qLSVRG} \\
%%%%%%%%%%%%%%%%%%%%%
%%%%%%%%%%%%%%%%%%%%%
%\hline
%\eqref{eq:main_problem}+\eqref{eq:f_i_sum} & {\tt D-QLSVRGstar}  & Alg \ref{alg:d-qLSVRGstar} & NEW  & \cmark & \xmark & \cmark 
%% &  \xmark 
%& \ref{sec:d_qLSVRGstar} & \ref{thm:d_qLSVRGstar} \\
%%%%%%%%%%%%%%%%%%%%%
%%%%%%%%%%%%%%%%%%%%%
%\hline
%\eqref{eq:main_problem}+\eqref{eq:f_i_sum} & {\tt D-LSVRG-DIANA}  & Alg \ref{alg:d-LSVRG-diana} & NEW  & \cmark & \xmark & \cmark 
%% &  \xmark 
%& \ref{sec:d_LSVRG-diana} & \ref{thm:d_LSVRG-diana} \\
%%%%%%%%%%%%%%%%%%%%%
%%%%%%%%%%%%%%%%%%%%%
\hline
\end{tabular}
\end{center}
\end{table*}

\subsection{{\tt D-SGD}}\label{sec:d_sgd_pure}
In this section we consider the same setup as in Section~\ref{sec:ec_sgd_pure}.
\begin{algorithm}[t]
   \caption{{\tt D-SGD}}\label{alg:d-sgd}
\begin{algorithmic}[1]
   \Require learning rate $\gamma>0$, initial vector $x^0 \in \R^d$
	\State Set $e_i^0 = 0$ for all $i=1,\ldots, n$   
   \For{$k=0,1,\dotsc$}
       \State Broadcast $x^{k}$ to all workers
        \For{$i=1,\dotsc,n$ in parallel}
            \State Sample $g^{k}_i = \nabla f_{\xi_i}(x^k) - \nabla f_i(x^*)$
            \State $v_i^k = \begin{cases}\gamma g_i^{k-\tau},& \text{if } k \ge \tau,\\ 0,& \text{if } k < \tau \end{cases}$
            \State $e_i^{k+1} = e_i^k + \gamma g_i^k - v_i^k$
        \EndFor
        \State $e^k = \frac{1}{n}\sum_{i=1}^n e_i^k$, $g^k = \frac{1}{n}\sum_{i=1}^ng_i^k$, $v^k = \frac{1}{n}\sum_{i=1}^nv_i^k = \frac{1}{n}\sum_{i=1}^n \nabla f_{\xi_i}(x^{k-\tau})$
       \State $x^{k+1} = x^k - v^k$
   \EndFor
\end{algorithmic}
\end{algorithm}
We notice that vectors $e_i^k$ appear only in the analysis and there is no need to compute them. Moreover, we use $\nabla f_i(x^*)$ in the definition of $g_i^k$ which is problematic at the firt glance. Indeed, workers do not know $\nabla f_i(x^*)$. However, since $0 = \nabla f(x^*) = \frac{1}{n}\nabla f_i(x^*)$ and master node uses averages of $g_i^k$ for the updates one can ignore $\nabla f_i(x^*)$ in $g_i^k$ in the implementation of {\tt D-SGD} and get exactly the same method. We define $g_i^k$ in such a way only for the theoretical analysis.
\begin{lemma}[see also Lemmas 1,2 from~\citep{nguyen2018sgd}]\label{lem:lemma_d_sgd}
    Assume that $f_{\xi_i}(x)$ are convex in $x$ for every $\xi_i$, $i=1,\ldots,n$. Then for every $x\in\R^d$ and $i=1,\ldots, n$
    \begin{equation}\label{eq:lemma_d_sgd_1}
        \EE\left[\|g^k\|^2\mid x^k\right] \le 4L(f(x^k) - f(x^*)) + \frac{2}{n^2}\sum\limits_{i=1}^n \Var\left[\nabla f_{\xi_i}(x^*)\right].
    \end{equation}
    If further $f(x)$ is $\mu$-quasi strongly convex with possibly non-convex $f_i,f_{\xi_i}$ and $\mu > 0$, then for every $x\in\R^d$ and $i = 1,\ldots, n$
    \begin{equation}\label{eq:lemma_d_sgd_2}
        \EE\left[\|g^k\|^2\mid x^k\right] \le 4L\kappa(f(x^k) - f(x^*)) + \frac{2}{n^2}\sum\limits_{i=1}^n \Var\left[\nabla f_{\xi_i}(x^*)\right],
    \end{equation}
    where $\kappa = \frac{L}{\mu}$.
\end{lemma}
\begin{proof}
	By definition of $g^k$ we have
	\begin{eqnarray}
		\EE\left[\|g^k\|^2\mid x^k\right] &=& \EE\left[\left\|\frac{1}{n}\sum\limits_{i=1}^n\left(\nabla f_{\xi_i}(x^k) - \nabla f_{\xi_i}(x^*) + \nabla f_{\xi_i}(x^*) - \nabla f_i(x^*)\right)\right\|^2\mid x^k\right]\notag \\
		&\overset{\eqref{eq:a_b_norm_squared}}{\le}& 2\EE\left[\left\|\frac{1}{n}\sum\limits_{i=1}^n\left(\nabla f_{\xi_i}(x^k) - \nabla f_{\xi_i}(x^*)\right)\right\|^2\mid x^k\right]\notag\\
		&&\quad + 2\underbrace{\EE\left[\left\|\frac{1}{n}\sum\limits_{i=1}^n\left(\nabla f_{\xi_i}(x^*) - \nabla f_i(x^*)\right)\right\|^2\right]}_{\Var\left[\frac{1}{n}\sum\limits_{i=1}^n\nabla f_{\xi_i}(x^*)\right]}\notag\\
		&\overset{\eqref{eq:a_b_norm_squared}}{\le}& \frac{2}{n}\sum\limits_{i=1}^n\EE\left[\|\nabla f_{\xi_i}(x^k)-\nabla f_{\xi_i}(x^*)\|^2\mid x^k\right]\notag\\
		&&\quad + \frac{2}{n^2}\sum\limits_{i=1}^n \underbrace{\EE\left[\|\nabla f_{\xi_i}(x^*) - \nabla f_i(x^*)\|^2\right]}_{\Var\left[\nabla f_{\xi_i}(x^*)\right]}, \label{eq:d_sgd_pure_technical_1}
	\end{eqnarray}
	where in the last inequality we use independence of $\nabla f_{\xi_i}(x^*)$, $i=1,\ldots,n$. Using this we derive inequality \eqref{eq:lemma_d_sgd_1}:
	\begin{eqnarray*}
		\EE\left[\|g^k\|^2\mid x^k\right] &\overset{\eqref{eq:d_sgd_pure_technical_1},\eqref{eq:L_smoothness_cor}}{\le}& \frac{4L}{n}\sum\limits_{i=1}^n \EE\left[D_{f_{\xi_i}}(x^k,x^*)\mid x^k\right] + \frac{2}{n^2}\sum\limits_{i=1}^n\Var\left[\nabla f_{\xi_i}(x^*)\right]\\
		&=& \frac{4L}{n}\sum\limits_{i=1}^n D_{f_i}(x^k,x^*) + \frac{2}{n^2}\sum\limits_{i=1}^n\Var\left[\nabla f_{\xi_i}(x^*)\right]\\
		&=& 4L\left(f(x^k) - f(x^*)\right) + \frac{2}{n^2}\sum\limits_{i=1}^n\Var\left[\nabla f_{\xi_i}(x^*)\right].
	\end{eqnarray*}
	Next, if $f(x)$ is $\mu$-quasi strongly convex, but $f_i,f_{\xi_i}$ are not necessary convex, we obtain
	\begin{eqnarray*}
		\EE\left[\|g^k\|^2\mid x^k\right] &\overset{\eqref{eq:d_sgd_pure_technical_1},\eqref{eq:L_smoothness}}{\le}& \frac{2L^2}{n}\sum\limits_{i=1}^n \|x^k - x^*\|^2 + \frac{2}{n^2}\sum\limits_{i=1}^n\Var\left[\nabla f_{\xi_i}(x^*)\right]\\
		&\overset{\eqref{eq:str_quasi_cvx}}{\le}& \frac{4L^2}{\mu}\left(f(x^k) - f(x^*)\right) + \frac{2}{n^2}\sum\limits_{i=1}^n\Var\left[\nabla f_{\xi_i}(x^*)\right].
	\end{eqnarray*}
\end{proof}

\begin{theorem}\label{thm:d_sgd_pure}
	Assume that $f_\xi(x)$ is convex in $x$ for every $\xi$. Then {\tt D-SGD} satisfies Assumption~\ref{ass:key_assumption_new} with
	\begin{gather*}
		A' = 2L,\quad B_1' = B_2' = 0,\quad D_1' = \frac{2}{n^2}\sum\limits_{i=1}^n\Var\left[\nabla f_{\xi_i}(x^*)\right],\quad \sigma_{1,k}^2 \equiv \sigma_{2,k}^2 \equiv 0\\
		\rho_1 = \rho_2 = 1,\quad C_1 = C_2 = 0,\quad D_2 = 0\\
		F_1 = F_2 = 0,\quad D_3 = \frac{6\gamma\tau L}{n^2}\sum\limits_{i=1}^n\Var\left[\nabla f_{\xi_i}(x^*)\right]	
	\end{gather*}
	with $\gamma$ satisfying
	\begin{equation*}
		\gamma \le \frac{1}{8L\sqrt{2\tau\left(\tau + 2\right)}}
	\end{equation*}
	and for all $K \ge 0$
	\begin{equation*}
		\EE\left[f(\bar{x}^K) - f(x^*)\right] \le \left(1 - \frac{\gamma\mu}{2}\right)^K\frac{4\|x^0 - x^*\|^2}{\gamma} + \frac{8\gamma}{n^2}\left(1 + 3L\gamma\tau\right)\sum\limits_{i=1}^n\Var\left[\nabla f_{\xi_i}(x^*)\right]
	\end{equation*}
	when $\mu > 0$ and
	\begin{equation*}
		\EE\left[f(\bar{x}^K) - f(x^*)\right] \le \frac{4\|x^0 - x^*\|^2}{\gamma K} + \frac{8\gamma}{n^2}\left(1 + 3L\gamma\tau\right)\sum\limits_{i=1}^n\Var\left[\nabla f_{\xi_i}(x^*)\right]
	\end{equation*}
	when $\mu = 0$. If further $f_i(x)$ are $\mu$-strongly convex with possibly non-convex $f_{\xi_i}$ and $\mu > 0$, then
	{\tt D-SGD} satisfies Assumption~\ref{ass:key_assumption_new} with
	\begin{gather*}
		A' = 2\kappa L,\quad B_1' = B_2' = 0,\quad D_1' = \frac{2}{n^2}\sum\limits_{i=1}^n\Var\left[\nabla f_{\xi_i}(x^*)\right],\quad \sigma_{1,k}^2 \equiv \sigma_{2,k}^2 \equiv 0,\\
		\rho_1 = \rho_2 = 1,\quad C_1 = C_2 = 0,\quad D_2 = 0,\quad G = 0,\\
		F_1 = F_2 = 0,\quad D_3 = \frac{6\gamma\tau L}{n^2}\sum\limits_{i=1}^n\Var\left[\nabla f_{\xi_i}(x^*)\right]	
	\end{gather*}
	with $\gamma$ satisfying
	\begin{equation*}
		\gamma \le \min\left\{\frac{1}{8\kappa L}, \frac{1}{8L\sqrt{2\tau\left(\tau + 2\kappa\right)}}\right\}
	\end{equation*}
	and for all $K \ge 0$
	\begin{equation*}
		\EE\left[f(\bar{x}^K) - f(x^*)\right] \le \left(1 - \frac{\gamma\mu}{2}\right)^K\frac{4\|x^0 - x^*\|^2}{\gamma} + \frac{8\gamma}{n^2}\left(1 + 3L\gamma\tau\right)\sum\limits_{i=1}^n\Var\left[\nabla f_{\xi_i}(x^*)\right].
	\end{equation*}
\end{theorem}

In other words, {\tt D-SGD} converges with linear rate $\cO\left(\tau\kappa\ln\frac{1}{\varepsilon}\right)$ to the neighbourhood of the solution when $\mu > 0$. Applying Lemma~\ref{lem:lemma2_stich} we establish the rate of convergence to $\varepsilon$-solution.
\begin{corollary}\label{cor:d_SGD_pure_str_cvx_cor}
	Let the assumptions of Theorem~\ref{thm:d_sgd_pure} hold, $f_{\xi}(x)$ are convex for each $\xi$ and $\mu > 0$. Then after $K$ iterations of {\tt D-SGD} with the stepsize
	\begin{equation*}
		\gamma = \min\left\{\frac{1}{8L\sqrt{2\tau\left(\tau + 2\right)}}, \frac{\ln\left(\max\left\{2,\min\left\{\frac{\|x^0-x^*\|^2\mu^2K^2}{D_1'}, \frac{\|x^0-x^*\|^2\mu^3K^3}{3\tau LD_1}\right\}\right\}\right)}{\mu K}\right\}
	\end{equation*}	 
	we have
	\begin{equation*}
		\EE\left[f(\bar{x}^K) - f(x^*)\right] = \widetilde\cO\left(L\tau\|x^0 - x^*\|^2\exp\left(-\frac{\mu}{\tau L}K\right) + \frac{D_1'}{\mu K} + \frac{L\tau D_1'}{\mu^2 K^2}\right).
	\end{equation*}
	That is, to achive $\EE\left[f(\bar{x}^K) - f(x^*)\right] \le \varepsilon$ {\tt D-SGD} requires
	\begin{equation*}
		\widetilde{\cO}\left(\frac{\tau L}{\mu} + \frac{D_1'}{\mu\varepsilon} + \frac{\sqrt{L\tau D_1'}}{\mu\sqrt{\varepsilon}}\right) \text{ iterations.}
	\end{equation*}
\end{corollary}
\begin{corollary}\label{cor:d_SGD_pure_str_cvx_cor_2}
	Let the assumptions of Theorem~\ref{thm:d_sgd_pure} hold and $f(x)$ is $\mu$-strongly convex with $\mu > 0$ and possibly non-convex $f_i,f_{\xi_i}$. Then after $K$ iterations of {\tt D-SGD} with the stepsize
	\begin{equation*}
		\gamma = \min\left\{\frac{1}{8\kappa L}, \frac{1}{8L\sqrt{2\tau\left(\tau + 2\kappa\right)}}, \frac{\ln\left(\max\left\{2,\min\left\{\frac{\|x^0-x^*\|^2\mu^2K^2}{D_1'}, \frac{\|x^0-x^*\|^2\mu^3K^3}{L\tau D_1'}\right\}\right\}\right)}{\mu K}\right\}
	\end{equation*}	 
	we have $\EE\left[f(\bar{x}^K) - f(x^*)\right]$ of order
	\begin{equation*}
		\widetilde\cO\left(L\left(\kappa+\tau\sqrt{\kappa}\right)\|x^0 - x^*\|^2\exp\left(-\min\left\{\frac{\mu}{\tau L\sqrt{\kappa}}, \frac{1}{\kappa^2}\right\}K\right) + \frac{D_1'}{\mu K} + \frac{L\tau D_1'}{\mu^2 K^2}\right).
	\end{equation*}
	That is, to achive $\EE\left[f(\bar{x}^K) - f(x^*)\right] \le \varepsilon$ {\tt D-SGD} requires
	\begin{equation*}
		\widetilde{\cO}\left(\kappa^2 + \tau\kappa^{\nicefrac{3}{2}} + \frac{D_1'}{\mu\varepsilon} + \frac{\sqrt{L\tau D_1'}}{\mu\sqrt{\varepsilon}}\right) \text{ iterations.}
	\end{equation*}
\end{corollary}

Applying Lemma~\ref{lem:lemma_technical_cvx} we get the complexity result in the case when $\mu = 0$.
\begin{corollary}\label{cor:d_sgd_cvx_cor}
	Let the assumptions of Theorem~\ref{thm:d_sgd_pure} hold, $f_{\xi}(x)$ are convex for each $\xi$ and $\mu = 0$. Then after $K$ iterations of {\tt D-SGD} with the stepsize
	\begin{eqnarray*}
		\gamma &=& \min\left\{\frac{1}{8L\sqrt{2\tau\left(\tau + 2\right)}}, \sqrt{\frac{\|x^0 - x^*\|^2}{D_1' K}}, \sqrt[3]{\frac{\|x^0 - x^*\|^2}{3L\tau D_1' K}}\right\}	\end{eqnarray*}		
	we have $\EE\left[f(\bar{x}^K) - f(x^*)\right]$ of order
	\begin{equation*}
		\cO\left(\frac{\tau LR_0^2}{K} + \sqrt{\frac{R_0^2 \tau D_1'}{K}} + \frac{\sqrt[3]{LR_0^4\tau D_1'}}{K^{\nicefrac{2}{3}}}\right)
	\end{equation*}
	where $R_0 = \|x^0 - x^*\|$. That is, to achive $\EE\left[f(\bar{x}^K) - f(x^*)\right] \le \varepsilon$ {\tt D-SGD} requires
	\begin{equation*}
		\cO\left(\frac{\tau LR_0^2}{\varepsilon} + \frac{R_0^2 D_1'}{\varepsilon^2} + \frac{R_0^2\sqrt{L\tau D_1'}}{\varepsilon^{\nicefrac{3}{2}}}\right)
	\end{equation*}
	iterations.
\end{corollary}

\subsection{{\tt D-QSGD}}\label{sec:d_qsgd}
In this section we show how one can combine delayed updates with quantization using our scheme.
\begin{algorithm}[h!]
   \caption{{\tt D-QSGD}}\label{alg:d-qsgd}
\begin{algorithmic}[1]
   \Require learning rate $\gamma>0$, initial vector $x^0 \in \R^d$
	\State Set $e_i^0 = 0$ for all $i=1,\ldots, n$   
   \For{$k=0,1,\dotsc$}
       \State Broadcast $x^{k-\tau}$ to all workers
        \For{$i=1,\dotsc,n$}
			\State Sample $\hat g_i^{k-\tau}$ independently from other nodes such that $\EE[\hat g_i^{k-\tau}\mid x^{k-\tau}] = \nabla f_i(x^{k-\tau})$ and $\EE\left[\|\hat g_i^{k-\tau} - \nabla f_i(x^{k-\tau})\|^2\mid x^{k-\tau}\right] \le D_{i}$            
            \State $g^{k-\tau}_i = Q(\hat g_i^{k-\tau}) - \nabla f_i(x^*)$ (quantization is performed independently from other nodes)
            \State $v_i^k = \gamma g_i^{k-\tau}$
            \State $e_i^{k+1} = e_i^k + \gamma g_i^k - v_i^k$
        \EndFor
        \State $e^k = \frac{1}{n}\sum_{i=1}^ne_i^k$, $g^k = \frac{1}{n}\sum_{i=1}^ng_i^k$, $v^k = \frac{1}{n}\sum_{i=1}^nv_i^k = \frac{\gamma}{n}\sum_{i=1}^ng_i^{k-\tau} = \frac{\gamma}{n}\sum_{i=1}^nQ(\hat g_i^{k-\tau})$
       \State $x^{k+1} = x^k - v^k$
   \EndFor
\end{algorithmic}
\end{algorithm}
\begin{lemma}\label{lem:d_qsgd_second_moment_bound}
	Assume that $f_i(x)$ is convex and $L$-smooth for all $i=1,\ldots,n$. Then, for all $k\ge 0$ we have
	\begin{eqnarray*}
		\EE\left[g^k\mid x^k\right] &=& \nabla f(x^k), 
%		\label{eq:d_qsgd_unbiasedness}
		\\
		\EE\left[\|g^k\|^2\mid x^k\right] &\le& 2L\left(1 + \frac{2\omega}{n}\right)\left(f(x^k) - f(x^*)\right) + \frac{(\omega+1)D}{n} + \frac{2\omega}{n^2}\sum\limits_{i=1}^n\|\nabla f_i(x^*)\|^2 
%		\label{eq:d_qsgd_second_moment_bound}
	\end{eqnarray*}
	where $D = \frac{1}{n}\sum_{i=1}^n D_{i}$.
\end{lemma}
\begin{proof}
	First of all, we show unbiasedness of $g^k$:
	\begin{eqnarray*}
		\EE\left[g^k\mid x^k\right] &=& \frac{1}{n}\sum\limits_{i=1}^n\EE\left[g_i^k\mid x^k\right] = \frac{1}{n}\sum\limits_{i=1}^n\EE\left[\EE_Q\left[Q(\hat g_i^k) - \nabla f_i(x^*)\right]\mid x^k\right]\\
		&\overset{\eqref{eq:quantization_def}}{=}& \frac{1}{n}\sum\limits_{i=1}^n\left(\nabla f_i(x^k) - \nabla f_i(x^*)\right) = \nabla f(x^k),
	\end{eqnarray*}
	where $\EE_{Q}\left[\cdot\right]$ denotes mathematical expectation w.r.t.\ the randomness coming only from the quantization. Next, we derive the upper bound for the second moment of $g^k$:
	\begin{eqnarray}
		\EE_Q\left[\|g^k\|^2\right] &=& \EE_Q\left[\left\|\frac{1}{n}\sum\limits_{i=1}^n\left(Q(\hat g_i^k) - \nabla f_i(x^*)\right)\right\|^2\right]\notag\\
		&\overset{\eqref{eq:variance_decomposition}}{=}& \EE_Q\left[\left\|\frac{1}{n}\sum\limits_{i=1}^n\left(Q(\hat g_i^k) - \hat g_i^k\right)\right\|^2\right] + \left\|\frac{1}{n}\sum\limits_{i=1}^n\left(\hat g_i^k - \nabla f_i(x^*)\right)\right\|^2.\label{eq:d_qsgd_technical_1}
	\end{eqnarray}
	Since $Q(\hat g_1^k),\ldots, Q(\hat g_n^k)$ are independent quantizations, we get
	\begin{eqnarray*}
		\EE_Q\left[\|g^k\|^2\right] &\overset{\eqref{eq:d_qsgd_technical_1}}{\le}& \frac{1}{n^2}\sum\limits_{i=1}^n\EE_Q\left[\left\|Q(\hat g_i^k) - \hat g_i^k\right\|^2\right] + \left\|\frac{1}{n}\sum\limits_{i=1}^n\left(\hat g_i^k - \nabla f_i(x^*)\right)\right\|^2\\
		&\overset{\eqref{eq:quantization_def}}{\le}& \frac{\omega}{n^2}\sum\limits_{i=1}^n\|\hat g_i^k\|^2 + \left\|\frac{1}{n}\sum\limits_{i=1}^n\left(\hat g_i^k - \nabla f_i(x^*)\right)\right\|^2.
	\end{eqnarray*}
	Taking conditional expectation $\EE\left[\cdot\mid x^k\right]$ from the both sides of the previous inequality we obtain
	\begin{eqnarray}
		\EE\left[\|g^k\|^2\mid x^k\right] &\le& \frac{\omega}{n^2}\sum\limits_{i=1}^n\EE\left[\|\hat g_i^k\|^2\mid x^k\right] + \EE\left[\left\|\frac{1}{n}\sum\limits_{i=1}^n\left(\hat g_i^k - \nabla f_i(x^*)\right)\right\|^2\mid x^k\right]\notag\\
		&\overset{\eqref{eq:variance_decomposition}}{\le}& \frac{\omega}{n^2}\sum\limits_{i=1}^n\|\nabla f_i(x^k)\|^2 + \frac{\omega}{n^2}\sum\limits_{i=1}^n\EE\left[\|\hat g_i^k - \nabla f_i(x^k)\|^2\mid x^k\right]\notag\\
		&&\quad + \underbrace{\left\|\frac{1}{n}\sum\limits_{i=1}^n\left(\nabla f_i(x^k) - \nabla f_i(x^*)\right)\right\|^2}_{\|\nabla f(x^k)-\nabla f(x^*)\|^2} + \EE\left[\left\|\frac{1}{n}\sum\limits_{i=1}^n\left(\hat g_i^k - \nabla f_i(x^k)\right)\right\|^2\mid x^k\right].\notag
	\end{eqnarray}
	It remains to estimate terms in the second and the third lines of the previous inequality:
	\begin{eqnarray*}
		\frac{\omega}{n^2}\sum\limits_{i=1}^n\|\nabla f_i(x^k)\|^2 &\overset{\eqref{eq:a_b_norm_squared}}{\le}& \frac{2\omega}{n^2}\sum\limits_{i=1}^n \|\nabla f_i(x^k) - \nabla f_i(x^*)\|^2 + \frac{2\omega}{n^2}\sum\limits_{i=1}^n \|\nabla f_i(x^*)\|^2\\
		&\overset{\eqref{eq:L_smoothness_cor}}{\le}& \frac{4\omega L}{n}\left(f(x^k) - f(x^*)\right) + \frac{2\omega}{n^2}\sum\limits_{i=1}^n \|\nabla f_i(x^*)\|^2,\\
		\frac{\omega}{n}\sum\limits_{i=1}^n\EE\left[\|\hat g_i^k - \nabla f_i(x^k)\|^2\mid x^k\right] &\le& \frac{\omega}{n^2}\sum\limits_{i=1}^nD_{i} = \frac{\omega D}{n},\\
		\|\nabla f(x^k) - \nabla f(x^*)\|^2 &\overset{\eqref{eq:L_smoothness_cor}}{\le}& 2L\left(f(x^k) - f(x^*)\right),\\
		\EE\left[\left\|\frac{1}{n}\sum\limits_{i=1}^n\left(\hat g_i^k - \nabla f_i(x^k)\right)\right\|^2\mid x^k\right] &=& \frac{1}{n^2}\sum\limits_{i=1}^n\EE\left[\|\hat g_i^k - \nabla f_i(x^k)\|^2\mid x^k\right]\\
		&\le& \frac{1}{n^2}\sum\limits_{i=1}^n D_i = \frac{D}{n}.
	\end{eqnarray*}
	Putting all together we get
	\begin{eqnarray*}
		\EE\left[\|g^k\|^2\mid x^k\right] &\le& 2L\left(1 + \frac{2\omega}{n}\right)\left(f(x^k) - f(x^*)\right) + \frac{(\omega+1)D}{n} + \frac{2\omega}{n^2}\sum\limits_{i=1}^n \|\nabla f_i(x^*)\|^2.
	\end{eqnarray*}
\end{proof}

\begin{theorem}\label{thm:d_qsgd}
	Assume that $f_i(x)$ is convex and $L$-smooth for all $i=1,\ldots, n$ and $f(x)$ is $\mu$-quasi strongly convex. Then {\tt D-QSGD} satisfies Assumption~\ref{ass:key_assumption_new} with
	\begin{gather*}
		A' = L\left(1 + \frac{2\omega}{n}\right),\quad B_1' = B_2' = 0,\quad D_1' = \frac{(\omega+1)D}{n} + \frac{2\omega}{n^2}\sum\limits_{i=1}^n \|\nabla f_i(x^*)\|^2,\\
		\sigma_{1,k}^2 \equiv \sigma_{2,k}^2 \equiv 0,\quad \rho_1 = \rho_2 = 1,\quad C_1 = C_2 = 0,\quad D_2 = 0\\
		F_1 = F_2 = 0,\quad G = 0,\quad D_3 = \frac{3\gamma\tau L}{n}\left((\omega+1)D + \frac{2\omega}{n}\sum\limits_{i=1}^n\|\nabla f_i(x^*)\|^2\right)
	\end{gather*}
	with $\gamma$ satisfying
	\begin{equation*}
		\gamma \le \min\left\{\frac{1}{4L(1+\nicefrac{2\omega}{n})}, \frac{1}{8L\sqrt{2\tau\left(\tau + 1 + \nicefrac{2\omega}{n}\right)}}\right\}
	\end{equation*}
	and for all $K \ge 0$
	\begin{equation*}
		\EE\left[f(\bar x^K) - f(x^*)\right] \le \left(1 - \frac{\gamma\mu}{2}\right)^K\frac{4\|x^0 - x^*\|^2}{\gamma} + \gamma\left(D_1' + D_3\right)
	\end{equation*}
	when $\mu > 0$ and
	\begin{equation*}
		\EE\left[f(\bar x^K) - f(x^*)\right] \le \frac{4\|x^0 - x^*\|^2}{\gamma K} + \gamma\left(D_1' + D_3\right)
	\end{equation*}
	when $\mu = 0$.
\end{theorem}
In other words, {\tt D-QSGD} converges with the linear rate
\begin{equation*}
	\cO\left(\left(\kappa\left(1+\frac{\omega}{n}\right) + \kappa\sqrt{\tau\left(\tau + \frac{\omega}{n}\right)}\right)\ln\frac{1}{\varepsilon}\right)
\end{equation*}
to the neighbourhood of the solution when $\mu > 0$. Applying Lemma~\ref{lem:lemma2_stich} we establish the rate of convergence to $\varepsilon$-solution.
\begin{corollary}\label{cor:d_QSGD_str_cvx_cor}
	Let the assumptions of Theorem~\ref{thm:d_qsgd} hold, $f_{\xi}(x)$ are convex for each $\xi$ and $\mu > 0$. Then after $K$ iterations of {\tt D-QSGD} with the stepsize
	\begin{eqnarray*}
		\gamma_0 &=& \min\left\{\frac{1}{4L(1+\nicefrac{2\omega}{n})}, \frac{1}{8L\sqrt{2\tau\left(\tau + 1 + \nicefrac{2\omega}{n}\right)}}\right\},\quad R_0 = \|x^0 - x^*\|,\\
		\gamma &=& \min\left\{\gamma_0, \frac{\ln\left(\max\left\{2,\min\left\{\frac{R_0^2\mu^2K^2}{D_1'}, \frac{R_0^2\mu^3K^3}{3\tau LD_1'}\right\}\right\}\right)}{\mu K}\right\}
	\end{eqnarray*}
	we have $\EE\left[f(\bar{x}^K) - f(x^*)\right]$ of order
	\begin{equation*}
		 \widetilde\cO\left(LR_0^2\left(1+\frac{\omega}{n}+\sqrt{\tau\left(\tau + \frac{\omega}{n}\right)}\right)\exp\left(-\frac{\mu}{L\left(1+\frac{\omega}{n}+\sqrt{\tau\left(\tau + \frac{\omega}{n}\right)}\right)}K\right) + \frac{D_1'}{\mu K} + \frac{L\tau D_1'}{\mu^2 K^2}\right).
	\end{equation*}
	That is, to achive $\EE\left[f(\bar{x}^K) - f(x^*)\right] \le \varepsilon$ {\tt D-QSGD} requires
	\begin{equation*}
		\widetilde{\cO}\left(\frac{L}{\mu}\left(1+\frac{\omega}{n}\right) + \frac{L}{\mu}\sqrt{\tau\left(\tau + \frac{\omega}{n}\right)} + \frac{D_1'}{\mu\varepsilon} + \frac{\sqrt{L\tau D_1'}}{\mu\sqrt{\varepsilon}}\right) \text{ iterations.}
	\end{equation*}
\end{corollary}

Applying Lemma~\ref{lem:lemma_technical_cvx} we get the complexity result in the case when $\mu = 0$.
\begin{corollary}\label{cor:d_QSGD_cvx_cor}
	Let the assumptions of Theorem~\ref{thm:d_qsgd} hold and $\mu = 0$. Then after $K$ iterations of {\tt D-QSGD} with the stepsize
	\begin{eqnarray*}
		\gamma_0 &=& \min\left\{\frac{1}{4L(1+\nicefrac{2\omega}{n})}, \frac{1}{8L\sqrt{2\tau\left(\tau + 1 + \nicefrac{2\omega}{n}\right)}}\right\},\\	
		\gamma &=& \min\left\{\gamma_0, \sqrt{\frac{\|x^0 - x^*\|^2}{D_1' K}}, \sqrt[3]{\frac{\|x^0 - x^*\|^2}{3L\tau D_1' K}}\right\}	
	\end{eqnarray*}		
	we have $\EE\left[f(\bar{x}^K) - f(x^*)\right]$ of order
	\begin{equation*}
		\cO\left(\frac{LR_0^2\left(1+\frac{\omega}{n}\right)}{K} + \frac{LR_0^2\sqrt{\tau\left(\tau + \frac{\omega}{n}\right)}}{K} + \sqrt{\frac{R_0^2 D_1'}{K}} + \frac{\sqrt[3]{LR_0^4\tau D_1'}}{K^{\nicefrac{2}{3}}}\right)
	\end{equation*}
	where $R_0 = \|x^0 - x^*\|$. That is, to achive $\EE\left[f(\bar{x}^K) - f(x^*)\right] \le \varepsilon$ {\tt D-QSGD} requires
	\begin{equation*}
		\cO\left(\frac{LR_0^2\left(1+\frac{\omega}{n}\right)}{\varepsilon} + \frac{LR_0^2\sqrt{\tau\left(\tau + \frac{\omega}{n}\right)}}{\varepsilon} + \frac{R_0^2 D_1'}{\varepsilon^2} + \frac{R_0^2\sqrt{L\tau D_1'}}{\varepsilon^{\nicefrac{3}{2}}}\right)
	\end{equation*}
	iterations.
\end{corollary}

\subsection{{\tt D-QSGDstar}}\label{sec:d_qsgd_star}
As we saw in Section~\ref{sec:d_qsgd} {\tt D-QSGD} fails to converge to the exact optimum asymptotically even if $\hat g_i^k = \nabla f_i(x^k)$ for all $i=1,\ldots,n$ almost surely, i.e., all $D_i = 0$ for all $i=1,\ldots,n$. As for {\tt EC-GDstar} we assume now that $i$-th worker has an access to $\nabla f_i(x^*)$. Using this one can construct the method with delayed updates that converges asymptotically to the exact solution when the full gradients are available.
\begin{algorithm}[h!]
   \caption{{\tt D-QSGDstar}}\label{alg:d-qSGDstar}
\begin{algorithmic}[1]
   \Require learning rate $\gamma>0$, initial vector $x^0 \in \R^d$
	\State Set $e_i^0 = 0$ for all $i=1,\ldots, n$   
   \For{$k=0,1,\dotsc$}
       \State Broadcast $x^{k-\tau}$ to all workers
        \For{$i=1,\dotsc,n$}
			\State Sample $\hat g_i^{k-\tau}$ independently from other nodes such that $\EE[\hat g_i^{k-\tau}\mid x^{k-\tau}] = \nabla f_i(x^{k-\tau})$ and $\EE\left[\|\hat g_i^{k-\tau} - \nabla f_i(x^{k-\tau})\|^2\mid x^{k-\tau}\right] \le D_{i}$            
            \State $g^{k-\tau}_i = Q(\hat g_i^{k-\tau}- \nabla f_i(x^*))$ (quantization is performed independently from other nodes)
            \State $v_i^k = \gamma g_i^{k-\tau}$
            \State $e_i^{k+1} = e_i^k + \gamma g_i^k - v_i^k$
        \EndFor
        \State $e^k = \frac{1}{n}\sum_{i=1}^ne_i^k$, $g^k = \frac{1}{n}\sum_{i=1}^ng_i^k$, $v^k = \frac{1}{n}\sum_{i=1}^nv_i^k = \frac{\gamma}{n}\sum_{i=1}^ng_i^{k-\tau} = \frac{\gamma}{n}\sum_{i=1}^nQ(\hat g_i^{k-\tau}-\nabla f_i(x^*))$
       \State $x^{k+1} = x^k - v^k$
   \EndFor
\end{algorithmic}
\end{algorithm}
\begin{lemma}\label{lem:d_qsgd_star_second_moment_bound}
	Assume that $f_i(x)$ is convex and $L$-smooth for all $i=1,\ldots,n$. Then, for all $k\ge 0$ we have
	\begin{eqnarray}
		\EE\left[g^k\mid x^k\right] &=& \nabla f(x^k), \label{eq:d_qsgd_star_unbiasedness}\\
		\EE\left[\|g^k\|^2\mid x^k\right] &\le& 2L\left(1 + \frac{\omega}{n}\right)\left(f(x^k) - f(x^*)\right) + \frac{(\omega+1)D}{n} \label{eq:d_qsgd_star_second_moment_bound}
	\end{eqnarray}
	where $D = \frac{1}{n}\sum_{i=1}^n D_{i}$.
\end{lemma}
\begin{proof}
	First of all, we show unbiasedness of $g^k$:
	\begin{eqnarray*}
		\EE\left[g^k\mid x^k\right] &=& \frac{1}{n}\sum\limits_{i=1}^n\EE\left[g_i^k\mid x^k\right] = \frac{1}{n}\sum\limits_{i=1}^n\EE\left[\EE_Q\left[Q(\hat g_i^k - \nabla f_i(x^*))\right]\mid x^k\right]\\
		&\overset{\eqref{eq:quantization_def}}{=}& \frac{1}{n}\sum\limits_{i=1}^n\left(\nabla f_i(x^k) - \nabla f_i(x^*)\right) = \nabla f(x^k),
	\end{eqnarray*}
	where $\EE_{Q}\left[\cdot\right]$ denotes mathematical expectation w.r.t.\ the randomness coming only from the quantization. Next, we derive the upper bound for the second moment of $g^k$:
	\begin{eqnarray}
		\EE_Q\left[\|g^k\|^2\right] &=& \EE_Q\left[\left\|\frac{1}{n}\sum\limits_{i=1}^n\left(Q\left(\hat g_i^k - \nabla f_i(x^*)\right)\right)\right\|^2\right]\notag\\
		&\overset{\eqref{eq:variance_decomposition}}{=}& \EE_Q\left[\left\|\frac{1}{n}\sum\limits_{i=1}^n\left(Q\left(\hat g_i^k - \nabla f_i(x^*)\right) - \left(\hat g_i^k - \nabla f_i(x^*)\right)\right)\right\|^2\right]\notag\\
		&&\quad + \left\|\frac{1}{n}\sum\limits_{i=1}^n\left(\hat g_i^k - \nabla f_i(x^*)\right)\right\|^2.\label{eq:d_qsgd_star_technical_1}
	\end{eqnarray}
	Since $Q\left(\hat g_1^k - \nabla f_1(x^*)\right),\ldots, Q\left(\hat g_n^k - \nabla f_n(x^*)\right)$ are independent quantizations, we get
	\begin{eqnarray*}
		\EE_Q\left[\|g^k\|^2\right] &\overset{\eqref{eq:d_qsgd_star_technical_1}}{\le}& \frac{1}{n^2}\sum\limits_{i=1}^n\EE_Q\left[\left\|Q\left(\hat g_i^k - \nabla f_i(x^*)\right) - \left(\hat g_i^k - \nabla f_i(x^*)\right)\right\|^2\right]\notag\\
		&&\quad + \left\|\frac{1}{n}\sum\limits_{i=1}^n\left(\hat g_i^k - \nabla f_i(x^*)\right)\right\|^2\\
		&\overset{\eqref{eq:quantization_def}}{\le}& \frac{\omega}{n^2}\sum\limits_{i=1}^n\|\hat g_i^k - \nabla f_i(x^*)\|^2 + \left\|\frac{1}{n}\sum\limits_{i=1}^n\left(\hat g_i^k - \nabla f_i(x^*)\right)\right\|^2.
	\end{eqnarray*}
	Taking conditional expectation $\EE\left[\cdot\mid x^k\right]$ from the both sides of the previous inequality we obtain
	\begin{eqnarray}
		\EE\left[\|g^k\|^2\mid x^k\right] &\le& \frac{\omega}{n^2}\sum\limits_{i=1}^n\EE\left[\|\hat g_i^k - \nabla f_i(x^*)\|^2\mid x^k\right] + \EE\left[\left\|\frac{1}{n}\sum\limits_{i=1}^n\left(\hat g_i^k - \nabla f_i(x^*)\right)\right\|^2\mid x^k\right]\notag\\
		&\overset{\eqref{eq:variance_decomposition}}{\le}& \frac{\omega}{n^2}\sum\limits_{i=1}^n\|\nabla f_i(x^k)- \nabla f_i(x^*)\|^2 + \frac{\omega}{n^2}\sum\limits_{i=1}^n\EE\left[\|\hat g_i^k - \nabla f_i(x^k)\|^2\mid x^k\right]\notag\\
		&&\quad + \underbrace{\left\|\frac{1}{n}\sum\limits_{i=1}^n\left(\nabla f_i(x^k) - \nabla f_i(x^*)\right)\right\|^2}_{\|\nabla f(x^k)-\nabla f(x^*)\|^2} + \EE\left[\left\|\frac{1}{n}\sum\limits_{i=1}^n\left(\hat g_i^k - \nabla f_i(x^k)\right)\right\|^2\mid x^k\right].\notag
	\end{eqnarray}
	It remains to estimate terms in the second and the third lines of the previous inequality:
	\begin{eqnarray*}
		\frac{\omega}{n^2}\sum\limits_{i=1}^n\|\nabla f_i(x^k)-\nabla f_i(x^*)\|^2 &\overset{\eqref{eq:L_smoothness_cor}}{\le}& \frac{2\omega L}{n}\left(f(x^k) - f(x^*)\right),\\
		\frac{\omega}{n}\sum\limits_{i=1}^n\EE\left[\|\hat g_i^k - \nabla f_i(x^k)\|^2\mid x^k\right] &\le& \frac{\omega}{n^2}\sum\limits_{i=1}^nD_{i} = \frac{\omega D}{n},\\
		\|\nabla f(x^k) - \nabla f(x^*)\|^2 &\overset{\eqref{eq:L_smoothness_cor}}{\le}& 2L\left(f(x^k) - f(x^*)\right),\\
		\EE\left[\left\|\frac{1}{n}\sum\limits_{i=1}^n\left(\hat g_i^k - \nabla f_i(x^k)\right)\right\|^2\mid x^k\right] &=& \frac{1}{n^2}\sum\limits_{i=1}^n\EE\left[\|\hat g_i^k - \nabla f_i(x^k)\|^2\mid x^k\right]\\
		&\le& \frac{1}{n^2}\sum\limits_{i=1}^n D_i = \frac{D}{n}.
	\end{eqnarray*}
	Putting all together we get
	\begin{eqnarray*}
		\EE\left[\|g^k\|^2\mid x^k\right] &\le& 2L\left(1 + \frac{\omega}{n}\right)\left(f(x^k) - f(x^*)\right) + \frac{(\omega+1)D}{n}.
	\end{eqnarray*}
\end{proof}

\begin{theorem}\label{thm:d_qsgd_star}
	Assume that $f_i(x)$ is convex and $L$-smooth for all $i=1,\ldots, n$ and $f(x)$ is $\mu$-quasi strongly convex. Then {\tt D-QSGDstar} satisfies Assumption~\ref{ass:key_assumption_new} with
	\begin{gather*}
		A' = L\left(1 + \frac{\omega}{n}\right),\quad B_1' = B_2' = 0,\quad D_1' = \frac{(\omega+1)D}{n},\quad \sigma_{1,k}^2 \equiv \sigma_{2,k}^2 \equiv 0,\\
		\rho_1 = \rho_2 = 1,\quad C_1 = C_2 = 0,\quad D_2 = 0,\quad G = 0,\\
		F_1 = F_2 = 0,\quad D_3 = \frac{3\gamma\tau L(\omega+1)D}{n}
	\end{gather*}
	with $\gamma$ satisfying
	\begin{equation*}
		\gamma \le \min\left\{\frac{1}{4L(1+\nicefrac{\omega}{n})}, \frac{1}{8L\sqrt{\tau\left(\tau + 1 + \nicefrac{\omega}{n}\right)}}\right\}.
	\end{equation*}
	and for all $K \ge 0$
	\begin{equation*}
		\EE\left[f(\bar x^K) - f(x^*)\right] \le \left(1 - \frac{\gamma\mu}{2}\right)^K\frac{4\|x^0 - x^*\|^2}{\gamma} + 4\gamma\left(D_1' + D_3\right)
	\end{equation*}	
	when $\mu > 0$ and
	\begin{equation*}
		\EE\left[f(\bar x^K) - f(x^*)\right] \le \frac{4\|x^0 - x^*\|^2}{\gamma K} + 4\gamma\left(D_1' + D_3\right)
	\end{equation*}
	when $\mu = 0$.
\end{theorem}
In other words, {\tt D-QSGDstar} converges with the linear rate
\begin{equation*}
	\cO\left(\left(\tau + \kappa\left(1+\frac{\omega}{n}\right) + \kappa\sqrt{\tau\left(\tau + \frac{\omega}{n}\right)}\right)\ln\frac{1}{\varepsilon}\right)
\end{equation*}
to the exact solution when $\mu > 0$ and $D = 0$, i.e., $\hat g_i^k = \nabla f_i(x^k)$ for all $i=1,\ldots,n$ almost surely. Applying Lemma~\ref{lem:lemma2_stich} we establish the rate of convergence to $\varepsilon$-solution.
\begin{corollary}\label{cor:d_QSGDstar_str_cvx_cor}
	Let the assumptions of Theorem~\ref{thm:d_qsgd_star} hold and $\mu > 0$. Then after $K$ iterations of {\tt D-QSGDstar} with the stepsize
	\begin{eqnarray*}
		\gamma_0 &=& \min\left\{\frac{1}{4L(1+\nicefrac{\omega}{n})}, \frac{1}{8L\sqrt{\tau\left(\tau + 1 + \nicefrac{\omega}{n}\right)}}\right\},\quad R_0 = \|x^0 - x^*\|,\\
		\gamma &=& \min\left\{\gamma_0, \frac{\ln\left(\max\left\{2,\min\left\{\frac{nR_0^2\mu^2K^2}{D}, \frac{nR_0^2\mu^3K^3}{3\tau LD}\right\}\right\}\right)}{\mu K}\right\}
	\end{eqnarray*}	 
	we have $\EE\left[f(\bar{x}^K) - f(x^*)\right]$ of order
	\begin{equation*}
		 \widetilde\cO\left(LR_0^2\left(1+\frac{\omega}{n}+\sqrt{\tau\left(\tau + \frac{\omega}{n}\right)}\right)\exp\left(-\frac{\mu}{L\left(1+\frac{\omega}{n}+\sqrt{\tau\left(\tau + \frac{\omega}{n}\right)}\right)}K\right) + \frac{D}{n\mu K} + \frac{L\tau D}{n\mu^2 K^2}\right).
	\end{equation*}
	That is, to achive $\EE\left[f(\bar{x}^K) - f(x^*)\right] \le \varepsilon$ {\tt D-QSGDstar} requires
	\begin{equation*}
		\widetilde{\cO}\left(\frac{L}{\mu}\left(1+\frac{\omega}{n}\right) + \frac{L}{\mu}\sqrt{\tau\left(\tau + \frac{\omega}{n}\right)} + \frac{D}{n\mu\varepsilon} + \frac{\sqrt{L\tau D}}{\mu\sqrt{n\varepsilon}}\right) \text{ iterations.}
	\end{equation*}
\end{corollary}

Applying Lemma~\ref{lem:lemma_technical_cvx} we get the complexity result in the case when $\mu = 0$.
\begin{corollary}\label{cor:d_QSGDstar_cvx_cor}
	Let the assumptions of Theorem~\ref{thm:d_qsgd_star} hold and $\mu = 0$. Then after $K$ iterations of {\tt D-QSGDstar} with the stepsize
	\begin{eqnarray*}
		\gamma_0 &=& \min\left\{\frac{1}{4L(1+\nicefrac{2\omega}{n})}, \frac{1}{8L\sqrt{\tau\left(\tau + 1 + \nicefrac{\omega}{n}\right)}}\right\},\\	
		\gamma &=& \min\left\{\gamma_0, \sqrt{\frac{n\|x^0 - x^*\|^2}{D K}}, \sqrt[3]{\frac{n\|x^0 - x^*\|^2}{3L\tau D K}}\right\}	
	\end{eqnarray*}		
	we have $\EE\left[f(\bar{x}^K) - f(x^*)\right]$ of order
	\begin{equation*}
		\cO\left(\frac{LR_0^2\left(1+\frac{\omega}{n}\right)}{K} + \frac{LR_0^2\sqrt{\tau\left(\tau + \frac{\omega}{n}\right)}}{K} + \sqrt{\frac{R_0^2 D}{nK}} + \frac{\sqrt[3]{LR_0^4\tau D}}{n^{\nicefrac{1}{3}}K^{\nicefrac{2}{3}}}\right)
	\end{equation*}
	where $R_0 = \|x^0 - x^*\|$. That is, to achive $\EE\left[f(\bar{x}^K) - f(x^*)\right] \le \varepsilon$ {\tt D-QSGDstar} requires
	\begin{equation*}
		\cO\left(\frac{LR_0^2\left(1+\frac{\omega}{n}\right)}{\varepsilon} + \frac{LR_0^2\sqrt{\tau\left(\tau + \frac{\omega}{n}\right)}}{\varepsilon} + \frac{R_0^2 D}{n\varepsilon^2} + \frac{R_0^2\sqrt{L\tau D}}{\sqrt{n}\varepsilon^{\nicefrac{3}{2}}}\right)
	\end{equation*}
	iterations.
\end{corollary}

\subsection{{\tt D-SGD-DIANA}}\label{sec:d_diana}
In this section we present a practical version of {\tt D-QSGDstar}: {\tt D-SGD-DIANA}.
\begin{algorithm}[h!]
   \caption{{\tt D-SGD-DIANA}}\label{alg:d-diana}
\begin{algorithmic}[1]
   \Require learning rates $\gamma>0, \alpha\in(0,1]$, initial vectors $x^0, h_1^0,\ldots,h_n^0 \in \R^d$
	\State Set $e_i^0 = 0$ for all $i=1,\ldots, n$   
	\State Set $h^0 = \frac{1}{n}\sum_{i=1}^n h_i^0$   
   \For{$k=0,1,\dotsc$}
       \State Broadcast $x^{k-\tau}$ to all workers
        \For{$i=1,\dotsc,n$}
			\State Sample $\hat g_i^{k-\tau}$ independently from other nodes such that $\EE[\hat g_i^{k-\tau}\mid x^{k-\tau}] = \nabla f_i(x^{k-\tau})$ and $\EE\left[\|\hat g_i^{k-\tau} - \nabla f_i(x^{k-\tau})\|^2\mid x^{k-\tau}\right] \le D_{i}$            
			\State $\hat \Delta_i^{k-\tau} = Q(\hat g_i^{k-\tau}- h_i^{k-\tau})$ (quantization is performed independently from other nodes)            
            \State $g^{k-\tau}_i = h_i^{k-\tau} + \hat \Delta_i^{k-\tau}$
            \State $v_i^k = \gamma g_i^{k-\tau}$
            \State $e_i^{k+1} = e_i^k + \gamma g_i^k - v_i^k$
            \State $h_i^{k-\tau+1} = h_i^{k-\tau} + \alpha\hat \Delta_i^{k-\tau}$
        \EndFor
        \State $h^{k-\tau} = \frac{1}{n}\sum_{i=1}^nh_i^{k-\tau}$, $e^k = \frac{1}{n}\sum_{i=1}^ne_i^k$, $g^k = \frac{1}{n}\sum_{i=1}^ng_i^k$, $v^k = \frac{1}{n}\sum_{i=1}^nv_i^k = \frac{\gamma}{n}\sum_{i=1}^ng_i^{k-\tau} = \gamma h^{k-\tau} + \frac{\gamma}{n}\sum_{i=1}^n\hat{\Delta}_i^{k-\tau}$
       \State $x^{k+1} = x^k - v^k$
       \State $h^{k-\tau+1} = h^{k-\tau}+\frac{\alpha}{n}\sum_{i=1}^n\hat{\Delta}_i^{k-\tau}$
   \EndFor
\end{algorithmic}
\end{algorithm}

\begin{lemma}[Lemmas 1 and 2 from \cite{horvath2019stochastic}]\label{lem:d_diana_second_moment_bound}
	Assume that $f_i(x)$ is convex and $L$-smooth for all $i=1,\ldots,n$ and $\alpha \le \nicefrac{1}{(\omega+1)}$. Then, for all $k\ge 0$ we have
	\begin{eqnarray}
		\EE\left[g^k\mid x^k\right] &=& \nabla f(x^k), \label{eq:d_diana_unbiasedness}\\
		\EE\left[\|g^k\|^2\mid x^k\right] &\le& 2L\left(1 + \frac{2\omega}{n}\right)\left(f(x^k) - f(x^*)\right) + \frac{2\omega\sigma_k^2}{n} + \frac{(\omega+1)D}{n} \label{eq:d_diana_second_moment_bound}\\
		\EE\left[\sigma_{k+1}^2\mid x^k\right] &\le& (1-\alpha)\sigma_k^2 + 2L\alpha\left(f(x^k) - f(x^*)\right) + \alpha D \label{eq:d_diana_sigma_k+1_bound}
	\end{eqnarray}
	where $\sigma_k^2 = \frac{1}{n}\sum_{i=1}^n\|h_i^k - \nabla f_i(x^*)\|^2$ and $D = \frac{1}{n}\sum_{i=1}^n D_{i}$.
\end{lemma}

\begin{theorem}\label{thm:d_diana}
	Assume that $f_i(x)$ is convex and $L$-smooth for all $i=1,\ldots, n$ and $f(x)$ is $\mu$-quasi strongly convex. Then {\tt D-SGD-DIANA} satisfies Assumption~\ref{ass:key_assumption_new} with
	\begin{gather*}
		A' = L\left(1 + \frac{2\omega}{n}\right),\quad B_1' = \frac{2\omega}{n},\quad D_1' = \frac{(\omega+1)D}{n},\quad \sigma_{1,k}^2 = \sigma_k^2 = \frac{1}{n}\sum_{i=1}^n\|h_i^k - \nabla f_i(x^*)\|^2,\\
		B_2' = 0,\quad \rho_1 = \alpha,\quad \rho_2 = 1,\quad C_1 = L\alpha,\quad C_2 = 0,\quad D_2 = \frac{\alpha(\omega+1)D}{n},\quad G = 0,\\
		F_1 = \frac{12\gamma^2L\omega \tau(2+\alpha)}{n\alpha},\quad F_2 = 0,\quad D_3 = 3\gamma\tau L\left(1 + \frac{4\omega}{n}\right)\frac{(\omega+1)D}{n}
	\end{gather*}
	with $\gamma$ and $\alpha$ satisfying
	\begin{equation*}
		\gamma \le \min\left\{\frac{1}{4L(1+\nicefrac{14\omega}{3n})}, \frac{1}{8L\sqrt{2\tau\left(1+\tau + \nicefrac{2\omega}{n} + \nicefrac{4\omega}{n(1-\alpha)} \right)}}\right\},\quad \alpha \le \frac{1}{\omega+1}, \quad M_1 = \frac{8\omega}{3n\alpha}
	\end{equation*}
	and for all $K \ge 0$
	\begin{equation*}
		\EE\left[f(\bar x^K) - f(x^*)\right] \le \left(1 - \min\left\{\frac{\gamma\mu}{2},\frac{\alpha}{4}\right\}\right)^K\frac{4(T^0 + \gamma F_1 \sigma_0^2)}{\gamma} + 4\gamma\left(D_1' + M_1D_2 + D_3\right)
	\end{equation*}	
	when $\mu > 0$ and 
	\begin{equation*}
		\EE\left[f(\bar x^K) - f(x^*)\right] \le \frac{4(T^0 + \gamma F_1 \sigma_0^2)}{\gamma K} + 4\gamma\left(D_1' + M_1D_2 + D_3\right)
	\end{equation*}
	when $\mu = 0$, where $T^k \eqdef \|\tx^k - x^*\|^2 + M_1\gamma^2 \sigma_k^2$.	
\end{theorem}
In other words, if
\begin{equation*}
	\gamma \le \min\left\{\frac{1}{4L(1+\nicefrac{14\omega}{3n})}, \frac{1}{8L\sqrt{2\tau\left(1+\tau + \nicefrac{10\omega}{n} \right)}}\right\},\quad \alpha \le \min\left\{\frac{1}{\omega+1},\frac{1}{2}\right\}
\end{equation*}
then {\tt D-SGD-DIANA} converges with the linear rate
\begin{equation*}
	\cO\left(\left(\omega + \kappa\left(1+\frac{\omega}{n}\right) + \kappa\sqrt{\tau\left(\tau + \frac{\omega}{n}\right)}\right)\ln\frac{1}{\varepsilon}\right)
\end{equation*}
to the exact solution when $\mu > 0$. Applying Lemma~\ref{lem:lemma2_stich} we establish the rate of convergence to $\varepsilon$-solution.
\begin{corollary}\label{cor:d_diana_str_cvx_cor}
	Let the assumptions of Theorem~\ref{thm:d_diana} hold and $\mu > 0$. Then after $K$ iterations of {\tt D-SGD-DIANA} with the stepsize
	\begin{eqnarray*}
		\gamma_0 &=& \min\left\{\frac{1}{4L(1+\nicefrac{14\omega}{3n})}, \frac{1}{8L\sqrt{2\tau\left(1+\tau + \nicefrac{10\omega}{n} \right)}}\right\},\quad R_0 = \|x^0 - x^*\|,\\
		 \tilde{F}_1 &=& \frac{12L\omega\tau(2+\alpha)\gamma_0^2}{n\alpha},\quad \tilde{T}^0 = R_0^2 + M_1\gamma_0^2\sigma_0^2,\\
		\gamma &=& \min\left\{\gamma_0, \frac{\ln\left(\max\left\{2,\min\left\{\frac{\left(\tilde{T}^0+\gamma_0\tilde{F}_1\sigma_0^2\right)\mu^2K^2}{D_1'+M_1D_2}, \frac{\left(\tilde{T}^0 +\gamma_0\tilde{F}_1\sigma_0^2\right)\mu^3K^3}{3\tau L\left(D_1' + \frac{2B_1'D_2}{\alpha}\right)}\right\}\right\}\right)}{\mu K}\right\}
	\end{eqnarray*}
	and $\alpha \le \min\left\{\frac{1}{\omega+1},\frac{1}{2}\right\}$ we have $\EE\left[f(\bar{x}^K) - f(x^*)\right]$ of order
	\begin{gather*}
		\widetilde\cO\left(LR_0^2\left(1+\frac{\omega}{n}+\sqrt{\tau\left(\tau + \frac{\omega}{n}\right)}\right)\exp\left(-\min\left\{\frac{\mu}{L\left(1+\frac{\omega}{n}+\sqrt{\tau\left(\tau + \frac{\omega}{n}\right)}\right)},\frac{1}{1+\omega}\right\}K\right)\right)\\
		 + \widetilde{O}\left(\frac{D_1'+M_1D_2}{\mu K} + \frac{\tau L\left(D_1' + \frac{B_1'D_2}{\alpha}\right)}{\mu^2 K^2}\right).
	\end{gather*}
	That is, to achive $\EE\left[f(\bar{x}^K) - f(x^*)\right] \le \varepsilon$ {\tt D-SGD-DIANA} requires
	\begin{equation*}
		\widetilde{\cO}\left(\omega+\frac{L}{\mu}\left(1+\frac{\omega}{n}\right) + \frac{L}{\mu}\sqrt{\tau\left(\tau + \frac{\omega}{n}\right)} + \frac{(\omega+1)\left(1+\frac{\omega}{n}\right)D}{n\mu\varepsilon} + \frac{\sqrt{L\tau(\omega+1)\left(1+\frac{\omega}{n}\right) D}}{\mu\sqrt{n\varepsilon}}\right)
	\end{equation*}
	iterations.
\end{corollary}

Applying Lemma~\ref{lem:lemma_technical_cvx} we get the complexity result in the case when $\mu = 0$.
\begin{corollary}\label{cor:d_diana_cvx_cor}
	Let the assumptions of Theorem~\ref{thm:d_diana} hold and $\mu = 0$. Then after $K$ iterations of {\tt D-SGD-DIANA} with the stepsize
	\begin{eqnarray*}
		\gamma_0 &=& \min\left\{\frac{1}{4L(1+\nicefrac{14\omega}{3n})}, \frac{1}{8L\sqrt{2\tau\left(1+\tau + \nicefrac{10\omega}{n} \right)}}\right\},\quad R_0 = \|x^0 - x^*\|,\\
		\gamma &=& \min\left\{\gamma_0, \sqrt{\frac{R_0^2}{M_1\sigma_0^2}}, \sqrt[3]{\frac{R_0^2n\alpha}{12L\omega\tau(2+\alpha)\sigma_0^2}}, \sqrt{\frac{R_0^2}{(D_1' + M_1D_2)K}}, \sqrt[3]{\frac{R_0^2}{3\tau L\left(D_1' + \frac{2B_1'D_2}{\alpha}\right)K}}\right\}	
	\end{eqnarray*}		
	we have $\EE\left[f(\bar{x}^K) - f(x^*)\right]$ of order
	\begin{gather*}
		\cO\left(\frac{L\left(1+\frac{\omega}{n}\right)R_0^2}{K} + \frac{L\sqrt{\tau\left(\tau + \frac{\omega}{n}\right)}R_0^2}{K} + \frac{\sqrt{R_0^2\omega(1+\omega)\sigma_0^2}}{\sqrt{n}K} + \frac{\sqrt[3]{R_0^4 L\tau\omega(1+\omega)\sigma_0^2}}{\sqrt[3]{n}K}\right)\\
		+\cO\left(\sqrt{\frac{(1+\omega)\left(1+\frac{\omega}{n}\right)R_0^2D}{nK}} + \frac{\sqrt[3]{R_0^4\tau L(1+\omega)\left(1+\frac{\omega}{n}\right)D}}{n^{\nicefrac{1}{3}}K^{\nicefrac{2}{3}}}\right).
	\end{gather*}
	That is, to achive $\EE\left[f(\bar{x}^K) - f(x^*)\right] \le \varepsilon$ {\tt D-SGD-DIANA} requires
	\begin{gather*}
		\cO\left(\frac{L\left(1+\frac{\omega}{n}\right)R_0^2}{\varepsilon} + \frac{L\sqrt{\tau\left(\tau + \frac{\omega}{n}\right)}R_0^2}{\varepsilon} + \frac{\sqrt{R_0^2\omega(1+\omega)\sigma_0^2}}{\sqrt{n}\varepsilon} + \frac{\sqrt[3]{R_0^4 L\tau\omega(1+\omega)\sigma_0^2}}{\sqrt[3]{n}\varepsilon}\right)\\
		+\cO\left(\frac{(1+\omega)\left(1+\frac{\omega}{n}\right)R_0^2D}{n\varepsilon^2} + \frac{R_0^2\sqrt{\tau L(1+\omega)\left(1+\frac{\omega}{n}\right)D}}{n^{\nicefrac{1}{2}}\varepsilon^{\nicefrac{3}{2}}}\right)\quad \text{iterations.}
	\end{gather*}
\end{corollary}

\subsection{{\tt D-SGDsr}}\label{sec:d_SGDsr}
In this section we consider the same settings as in Section~\ref{sec:ec_SGDsr}, but this time we consider delayed updates. Moreover, in this section we need slightly weaker assumption.
\begin{assumption}[Expected smoothness]\label{ass:exp_smoothness_f}
	We assume that function $f$ is $\cL$-smooth in expectation w.r.t.\ distribution $\cD$, i.e., there exists constant $\cL = \cL(f,\cD)$ such that
	\begin{equation}
		\EE_{\cD}\left[\|\nabla f_{\xi}(x) - \nabla f_{\xi}(x^*)\|^2\right] \le 2\cL\left(f(x) - f(x^*)\right) \label{eq:exp_smoothness_f}
	\end{equation}
	for all $i\in [n]$ and $x\in\R^d$.
\end{assumption}
\begin{algorithm}[t]
   \caption{{\tt D-SGDsr}}\label{alg:d-SGDsr}
\begin{algorithmic}[1]
   \Require learning rate $\gamma>0$, initial vector $x^0 \in \R^d$
	\State Set $e_i^0 = 0$ for all $i=1,\ldots, n$   
   \For{$k=0,1,\dotsc$}
       \State Broadcast $x^{k-\tau}$ to all workers
        \For{$i=1,\dotsc,n$ in parallel}
            \State Sample $g^{k-\tau}_i = \nabla f_{\xi_i}(x^{k-\tau}) - \nabla f_i(x^*)$
            \State $v_i^k = \gamma g_i^{k-\tau}$
            \State $e_i^{k+1} = e_i^k + \gamma g_i^k - v_i^k$
        \EndFor
        \State $e^k = \frac{1}{n}\sum_{i=1}^ne_i^k$, $g^k = \frac{1}{n}\sum_{i=1}^ng_i^k$, $v^k = \frac{1}{n}\sum_{i=1}^nv_i^k = \frac{1}{n}\sum_{i=1}^n\nabla f_{\xi_i}(x^{k-\tau})$
       \State $x^{k+1} = x^k - v^k$
   \EndFor
\end{algorithmic}
\end{algorithm}

\begin{lemma}\label{lem:key_lemma_d-SGDsr}
	For all $k\ge 0$ we have
	\begin{equation}
		\EE\left[\|g^k\|^2\mid x^k\right] \le 4\cL\left(f(x^k) - f(x^*)\right) + 2\EE_{\cD}\left[\|\nabla f_{\xi}(x^*)\|^2\right]. \label{eq:key_lemma_d-SGDsr} 
	\end{equation}
\end{lemma}
\begin{proof}
	Applying straightforward inequality $\|a+b\|^2 \le 2\|a\|^2 + 2\|b\|^2$ for $a,b\in\R^d$ we get
	\begin{eqnarray*}
		\EE\left[\|g^k\|^2\mid x^k\right] &=& \EE\left[\left\|\frac{1}{n}\sum\limits_{i=1}^n\left(\nabla f_{\xi_i}(x^k) - \nabla f_i(x^*)\right)\right\|^2\mid x^k\right] \\
		&\overset{\eqref{eq:a_b_norm_squared}}{\le}& 2\EE_{\cD}\left[\|\nabla f_{\xi}(x^k) - \nabla f_{\xi}(x^*)\|^2\right] + 2\EE_{\cD}\left[\|\nabla f_{\xi}(x^*) - \nabla f(x^*)\|^2\right]\\
		&\overset{\eqref{eq:exp_smoothness_f}}{\le}& 4\cL\left(f(x^k) - f(x^*)\right) + 2\EE_{\cD}\left[\|\nabla f_{\xi}(x^*)\|^2\right].
	\end{eqnarray*}
\end{proof}

\begin{theorem}\label{thm:d_SGDsr}
	Assume that $f(x)$ is $\mu$-quasi strongly convex, $L$-smooth and Assumption~\ref{ass:exp_smoothness_f} holds. Then {\tt D-SGDsr} satisfies Assumption~\ref{ass:key_assumption_new} with
	\begin{gather*}
		A' = 2\cL,\quad B_1' = B_2' = 0,\quad D_1' = 2\EE_{\cD}\|\nabla f_{\xi}(x^*)\|^2,\quad \sigma_{1,k}^2 \equiv \sigma_{2,k}^2 \equiv 0\\
		\rho_1 = \rho_2 = 1,\quad C_1 = C_2 = 0,\quad D_2 = 0,\quad G = 0,\\
		F_1 = F_2 = 0,\quad D_3 = 6\gamma\tau L\EE_{\cD}\|\nabla f_{\xi}(x^*)\|^2
	\end{gather*}
	with $\gamma$ satisfying
	\begin{equation*}
		\gamma \le \min\left\{\frac{1}{8\cL}, \frac{1}{8\sqrt{L\tau\left(L\tau + 2\cL\right)}}\right\}
	\end{equation*}
	and for all $K \ge 0$
	\begin{equation*}
		\EE\left[f(\bar{x}^K) - f(x^*)\right] \le \left(1 - \frac{\gamma\mu}{2}\right)^K\frac{4\|x^0 - x^*\|^2}{\gamma} + 8\gamma(1 + 3\gamma\tau L)\EE_{\cD}\|\nabla f_\xi(x^*)\|^2
	\end{equation*}
	when $\mu > 0$ and
	\begin{equation*}
		\EE\left[f(\bar{x}^K) - f(x^*)\right] \le \frac{4\|x^0 - x^*\|^2}{\gamma K} + 8\gamma(1 + 3\gamma\tau L)\EE_{\cD}\|\nabla f_\xi(x^*)\|^2
	\end{equation*}
	when $\mu = 0$.
\end{theorem}
In other words, {\tt D-SGDsr} converges with linear rate $\cO\left(\left(\frac{\cL}{\mu} + \frac{\sqrt{L\cL\tau + L^2\tau^2}}{\mu}\right)\ln\frac{1}{\varepsilon}\right)$ to the neighbourhood of the solution when $\mu > 0$. Applying Lemma~\ref{lem:lemma2_stich} we establish the rate of convergence to $\varepsilon$-solution.
\begin{corollary}\label{cor:d_SGDsr_str_cvx_cor}
	Let the assumptions of Theorem~\ref{thm:d_SGDsr} hold and $\mu > 0$. Then after $K$ iterations of {\tt D-SGDsr} with the stepsize
	\begin{eqnarray*}
		\gamma_0 &=& \min\left\{\frac{1}{8\cL}, \frac{1}{8\sqrt{L\tau\left(L\tau + 2\cL\right)}}\right\},\quad R_0 = \|x^0 - x^*\|,\\
		\gamma &=& \min\left\{\gamma_0, \frac{\ln\left(\max\left\{2,\min\left\{\frac{R_0^2\mu^2K^2}{D_1'}, \frac{R_0^2\mu^3K^3}{3\tau LD_1'}\right\}\right\}\right)}{\mu K}\right\}
	\end{eqnarray*}	 
	we have $\EE\left[f(\bar{x}^K) - f(x^*)\right]$ of order
	\begin{equation*}
		\widetilde\cO\left(R_0^2\left(\cL + \sqrt{L^2\tau^2 + L\cL\tau}\right)\exp\left(-\frac{\mu}{\tau L}K\right) + \frac{\EE_{\cD}\|\nabla f_{\xi}(x^*)\|^2}{\mu K} + \frac{L\tau \EE_{\cD}\|\nabla f_{\xi}(x^*)\|^2}{\mu^2 K^2}\right).
	\end{equation*}
	That is, to achive $\EE\left[f(\bar{x}^K) - f(x^*)\right] \le \varepsilon$ {\tt D-SGDsr} requires
	\begin{equation*}
		\widetilde{\cO}\left(\frac{\cL + \sqrt{L^2\tau^2 + L\cL\tau}}{\mu} + \frac{\EE_{\cD}\|\nabla f_{\xi}(x^*)\|^2}{\mu\varepsilon} + \frac{\sqrt{L\tau \EE_{\cD}\|\nabla f_{\xi}(x^*)\|^2}}{\mu\sqrt{\varepsilon}}\right) \text{ iterations.}
	\end{equation*}
\end{corollary}

Applying Lemma~\ref{lem:lemma_technical_cvx} we get the complexity result in the case when $\mu = 0$.
\begin{corollary}\label{cor:d_sgdsr_cvx_cor}
	Let the assumptions of Theorem~\ref{thm:d_SGDsr} hold and $\mu = 0$. Then after $K$ iterations of {\tt D-SGDsr} with the stepsize
	\begin{eqnarray*}
		\gamma &=& \min\left\{\frac{1}{8\cL}, \frac{1}{8\sqrt{L\tau\left(L\tau + 2\cL\right)}}, \sqrt{\frac{\|x^0 - x^*\|^2}{D_1' K}}, \sqrt[3]{\frac{\|x^0 - x^*\|^2}{3L\tau D_1' K}}\right\}
		\end{eqnarray*}		
	we have $\EE\left[f(\bar{x}^K) - f(x^*)\right]$ of order
	\begin{equation*}
		\cO\left(\frac{\cL R_0^2}{K} + \frac{\sqrt{L^2\tau^2 + L\cL\tau} R_0^2}{K} + \sqrt{\frac{R_0^2 \tau \EE_\cD\|\nabla f_{\xi}(x^*)\|^2}{K}} + \frac{\sqrt[3]{LR_0^4\tau \EE_\cD\|\nabla f_{\xi}(x^*)\|^2}}{K^{\nicefrac{2}{3}}}\right)
	\end{equation*}
	where $R_0 = \|x^0 - x^*\|$. That is, to achive $\EE\left[f(\bar{x}^K) - f(x^*)\right] \le \varepsilon$ {\tt D-SGDsr} requires
	\begin{equation*}
		\cO\left(\frac{\cL R_0^2}{\varepsilon} + \frac{\sqrt{L^2\tau^2 + L\cL\tau} R_0^2}{\varepsilon} + \frac{R_0^2 \EE_\cD\|\nabla f_{\xi}(x^*)\|^2}{\varepsilon^2} + \frac{R_0^2\sqrt{L\tau \EE_\cD\|\nabla f_{\xi}(x^*)\|^2}}{\varepsilon^{\nicefrac{3}{2}}}\right)
	\end{equation*}
	iterations.
\end{corollary}

\subsection{{\tt D-LSVRG}}\label{sec:d_LSVRG}
In the same settings as in Section~\ref{sec:ec_LSVRG} we now consider a new method called {\tt D-LSVRG} which is another modification of {\tt LSVRG} that works with delayed updates. 
\begin{algorithm}[t]
   \caption{{\tt D-LSVRG}}\label{alg:d-LSVRG}
\begin{algorithmic}[1]
   \Require learning rate $\gamma>0$, initial vector $x^0 \in \R^d$
   \State Set $e_i^0 = 0$ for all $i=1,\ldots, n$   
   \For{$k=0,1,\dotsc$}
       \State Broadcast $x^{k-\tau}$ to all workers
        \For{$i=1,\dotsc,n$ in parallel}
        	\State Pick $l$ uniformly at random from $[m]$
            \State Set $g^{k-\tau}_i = \nabla f_{il}(x^{k-\tau}) - \nabla f_{il}(w_i^{k-\tau}) + \nabla f_i(w_i^{k-\tau})$
            \State $v_i^k = \gamma g_{i}^{k-\tau}$
            \State $e_i^{k+1} = e_i^k + \gamma g_i^k - v_i^k$
            \State $w_i^{k-\tau+1} = \begin{cases}x^{k-\tau},& \text{with probability } p,\\ w_i^{k-\tau},& \text{with probability } 1-p\end{cases}$
        \EndFor
        \State $e^k = \frac{1}{n}\sum_{i=1}^ne_i^k$, $g^k = \frac{1}{n}\sum_{i=1}^ng_i^k$, $v^k = \frac{1}{n}\sum_{i=1}^nv_i^k$
       \State $x^{k+1} = x^k - v^k$
   \EndFor
\end{algorithmic}
\end{algorithm}

\begin{lemma}\label{lem:second_moment_bound_d-LSVRG}
	For all $k\ge 0$, $i\in [n]$ we have
	\begin{equation}
		\EE\left[g_i^k\mid x^k\right] = \nabla f_i(x^k) \label{eq:unbiasedness_g_i^k_d-LSVRG}
	\end{equation}		
	and
	\begin{equation}
		\EE\left[\|g^k\|^2\mid x^k\right] \le 4L\left(f(x^k) - f(x^*)\right) + 2\sigma_k^2, \label{eq:second_moment_bound_d-LSVRG} 
	\end{equation}
	where $\sigma_k^2 = \frac{1}{nm}\sum_{i=1}^n\sum_{j=1}^n\|\nabla f_{ij}(w_i^k) - \nabla f_{ij}(x^*)\|^2$.
\end{lemma}
\begin{proof}
	First of all, we derive unbiasedness of $g_i^k$:
	\begin{equation*}
		\EE\left[g_i^k\mid x^k\right] = \frac{1}{m}\sum\limits_{j=1}^m\left(\nabla f_{ij}(x^k) - \nabla f_{ij}(w_i^k) + \nabla f_i(w_i^k)\right) = \nabla f_i(x^k).
	\end{equation*}
	Next, we estimate the second moment of $g^k$:
	\begin{eqnarray*}
		\EE\left[\|g^k\|^2\mid x^k\right] &=& \EE\left[\left\|\frac{1}{n}\sum\limits_{i=1}^n\left(\nabla f_{il}(x^k) - \nabla f_{il}(w_i^k) + \nabla f_{i}(w_i^k)\right)\right\|^2\right]\\
		&=& \EE\left[\left\|\frac{1}{n}\sum\limits_{i=1}^n\left(\nabla f_{il}(x^k) - \nabla f_{il}(x^*) + \nabla f_{il}(x^*) - \nabla f_{il}(w_i^k) + \nabla f_{i}(w_i^k) - \nabla f_i(x^*)\right)\right\|^2\right]\\
		&\overset{\eqref{eq:a_b_norm_squared}}{\le}& \frac{2}{n}\sum\limits_{i=1}^n\EE\left[\|\nabla f_{il}(x^k) - \nabla f_{il}(x^*)\|^2\mid x^k\right]\\
		&&\quad + \frac{2}{n}\sum\limits_{i=1}^n\EE\left[\left\|\nabla f_{il}(w_i^k)- \nabla f_{il}(x^*) - \left(\nabla f_{i}(w_i^k) - \nabla f_i(x^*)\right)\right\|^2\mid x^k\right]\\
		&\overset{\eqref{eq:variance_decomposition}}{\le}& \frac{2}{nm}\sum\limits_{i=1}^n\sum\limits_{j=1}^m\|\nabla f_{ij}(x^k) - \nabla f_{ij}(x^*)\|^2 + \frac{2}{n}\EE\left[\left\|\nabla f_{il}(w_i^k)- \nabla f_{il}(x^*)\right\|^2\mid x^k\right]\\
		&\overset{\eqref{eq:L_smoothness_cor}}{\le}& \frac{4L}{nm}\sum\limits_{i=1}^n\sum\limits_{j=1}^mD_{f_{ij}}(x^k,x^*) + \frac{2}{nm}\sum\limits_{i=1}^n\sum\limits_{j=1}^m\|\nabla f_{ij}(w_i^k) - \nabla f_{ij}(x^*)\|^2\\
		&=& 4L\left(f(x^k) - f(x^*)\right) + 2\sigma_k^2.
	\end{eqnarray*}
\end{proof}

\begin{lemma}\label{lem:sigma_k+1_bound_d-LSVRG}
	For all $k\ge 0$, $i\in [n]$ we have
	\begin{equation}
		\EE\left[\sigma_{k+1}^2\mid x^k\right] \le (1-p)\sigma_k^2 + 2Lp\left(f(x^k) - f(x^*)\right), \label{eq:sigma_k+1_d-LSVRG} 
	\end{equation}
	where $\sigma_k^2 = \frac{1}{nm}\sum_{i=1}^n\sum_{j=1}^n\|\nabla f_{ij}(w_i^k) - \nabla f_{ij}(x^*)\|^2$.
\end{lemma}
\begin{proof}
	The proof is identical to the proof of Lemma~\ref{lem:sigma_k+1_bound_ec-LSVRG}.
\end{proof}

\begin{theorem}\label{thm:d_LSVRG}
	Assume that $f(x)$ is $\mu$-quasi strongly convex and functions $f_{ij}$ are convex and $L$-smooth for all $i\in[n],j\in[m]$. Then {\tt D-LSVRG} satisfies Assumption~\ref{ass:key_assumption_new} with
	\begin{gather*}
		A' = 2L,\quad B_1' = 0,\quad B_2' = 2,\quad D_1' = 0,\quad \sigma_{2,k}^2 = \sigma_k^2 = \frac{1}{nm}\sum\limits_{i=1}^n\sum\limits_{j=1}^m\|\nabla f_{ij}(w_i^{k}) - \nabla f_{ij}(x^*)\|^2,\\
		\sigma_{1,k}^2 \equiv 0,\quad \rho_1 = 1,\quad\rho_2 = p,\quad C_1 = 0,\quad C_2 = Lp,\quad D_2 = 0,\\
		G = 0,\quad F_1 = 0,\quad F_2 = \frac{12\gamma^2L \tau(2+p)}{p},\quad D_3 = 0
	\end{gather*}
	with $\gamma$ satisfying
	\begin{equation*}
		\gamma \le \min\left\{\frac{3}{56L}, \frac{1}{8L\sqrt{\tau\left(2+\tau + \nicefrac{4}{(1-p)}\right)}}\right\}, \quad M_2 = \frac{8}{3p}
	\end{equation*}
	and for all $K \ge 0$
	\begin{equation*}
		\EE\left[f(\bar x^K) - f(x^*)\right] \le \left(1 - \min\left\{\frac{\gamma\mu}{2},\frac{p}{4}\right\}\right)^K\frac{4(T^0 + \gamma F_2 \sigma_0^2)}{\gamma}
	\end{equation*}	
	when $\mu > 0$ and
	\begin{equation*}
		\EE\left[f(\bar x^K) - f(x^*)\right] \le \frac{4(T^0 + \gamma F_2 \sigma_0^2)}{\gamma K}
	\end{equation*}		
	when $\mu = 0$, where $T^k \eqdef \|\tx^k - x^*\|^2 + M_2\gamma^2 \sigma_k^2$.
\end{theorem}

In other words, {\tt D-LSVRG} converges with linear rate $\cO\left(\left(\frac{1}{p} + \kappa\sqrt{\tau\left(\tau + \frac{1}{(1-p)}\right)}\right)\ln\frac{1}{\varepsilon}\right)$ to the exact solution when $\mu > 0$. If $m\ge 2$ then taking $p = \frac{1}{m}$ we get that in expectation the sample complexity of one iteration of {\tt D-LSVRG} is $\cO(1)$ gradients calculations per node as for {\tt D-SGDsr} with standard sampling and the rate of convergence to the exact solution becomes $\cO\left(\left(m + \kappa\tau\right)\ln\frac{1}{\varepsilon}\right)$.

Applying Lemma~\ref{lem:lemma_technical_cvx} we get the complexity result in the case when $\mu = 0$.
\begin{corollary}\label{cor:d_lsvrg_cvx_cor}
	Let the assumptions of Theorem~\ref{thm:d_LSVRG} hold and $\mu = 0$. Then after $K$ iterations of {\tt D-LSVRG} with the stepsize
	\begin{eqnarray*}
		\gamma &=& \min\left\{\frac{3}{56L}, \frac{1}{8L\sqrt{\tau\left(2+\tau + \nicefrac{4}{(1-p)}\right)}}, \sqrt{\frac{\|x^0 - x^*\|^2}{M_2\sigma_0^2}}, \sqrt[3]{\frac{\|x^0 - x^*\|^2p}{12L\tau(2+p)\sigma_0^2}}\right\}
	\end{eqnarray*}		
	and $p = \frac{1}{m}$, $m\ge 2$ we have $\EE\left[f(\bar{x}^K) - f(x^*)\right]$ of order
	\begin{equation*}
		\cO\left(\frac{L\tau R_0^2}{K} + \frac{\sqrt{R_0^2m\sigma_0^2}}{K} + \frac{\sqrt[3]{R_0^4 L\tau \sigma_0^2}}{K}\right)
	\end{equation*}
	where $R_0 = \|x^0 - x^*\|$. That is, to achive $\EE\left[f(\bar{x}^K) - f(x^*)\right] \le \varepsilon$ {\tt D-LSVRG} requires
	\begin{equation*}
		\cO\left(\frac{L\tau R_0^2}{\varepsilon} + \frac{\sqrt{R_0^2m\sigma_0^2}}{\varepsilon} + \frac{\sqrt[3]{R_0^4 L\tau \sigma_0^2}}{\varepsilon}\right)
	\end{equation*}
	iterations.
\end{corollary}

\subsection{{\tt D-QLSVRG}}\label{sec:d_qLSVRG}
In this section we add a quantization to {\tt D-LSVRG}.  
\begin{algorithm}[t]
   \caption{{\tt D-QLSVRG}}\label{alg:d-qLSVRG}
\begin{algorithmic}[1]
   \Require learning rate $\gamma>0$, initial vector $x^0 \in \R^d$
   \State Set $e_i^0 = 0$ for all $i=1,\ldots, n$   
   \For{$k=0,1,\dotsc$}
       \State Broadcast $x^{k-\tau}$ to all workers
        \For{$i=1,\dotsc,n$ in parallel}
        	\State Pick $l$ uniformly at random from $[m]$
            \State Set $\hat g^{k-\tau}_i = \nabla f_{il}(x^{k-\tau}) - \nabla f_{il}(w_i^{k-\tau}) + \nabla f_i(w_i^{k-\tau})$
			\State Set $g_i^{k-\tau} = Q(\hat g_i^{k-\tau})$ (quantization is performed independently from other nodes)        
            \State $v_i^k = \gamma g_{i}^{k-\tau}$
            \State $e_i^{k+1} = e_i^k + \gamma g_i^k - v_i^k$
            \State $w_i^{k-\tau+1} = \begin{cases}x^{k-\tau},& \text{with probability } p,\\ w_i^{k-\tau},& \text{with probability } 1-p\end{cases}$
        \EndFor
        \State $e^k = \frac{1}{n}\sum_{i=1}^ne_i^k$, $g^k = \frac{1}{n}\sum_{i=1}^ng_i^k$, $v^k = \frac{1}{n}\sum_{i=1}^nv_i^k$
       \State $x^{k+1} = x^k - v^k$
   \EndFor
\end{algorithmic}
\end{algorithm}

\begin{lemma}\label{lem:second_moment_bound_d-qLSVRG}
	For all $k\ge 0$, $i\in [n]$ we have
	\begin{equation*}
		\EE\left[g_i^k\mid x^k\right] = \nabla f_i(x^k)
	\end{equation*}		
	and
	\begin{equation*}
		\EE\left[\|g^k\|^2\mid x^k\right] \le 4L\left(1 + \frac{2\omega}{n}\right)\left(f(x^k) - f(x^*)\right) + 2\left(1 + \frac{2\omega}{n}\right)\sigma_k^2 + \frac{2\omega}{n^2}\sum\limits_{i=1}^n\|\nabla f_i(x^*)\|^2,  
	\end{equation*}
	where $\sigma_k^2 = \frac{1}{nm}\sum_{i=1}^n\sum_{j=1}^n\|\nabla f_{ij}(w_i^k) - \nabla f_{ij}(x^*)\|^2$.
\end{lemma}
\begin{proof}
	First of all, we derive unbiasedness of $g_i^k$:
	\begin{eqnarray*}
		\EE\left[g_i^k\mid x^k\right] &\overset{\eqref{eq:tower_property}}{=}& \EE\left[\EE_{Q}\left[Q(\hat g_i^k)\right]\mid x^k\right] \overset{\eqref{eq:quantization_def}}{=} \EE\left[\hat g_i^k\mid x^k\right]\\
		&=& \frac{1}{m}\sum\limits_{j=1}^m\left(\nabla f_{ij}(x^k) - \nabla f_{ij}(w_i^k) + \nabla f_i(w_i^k)\right) = \nabla f_i(x^k).
	\end{eqnarray*}
	Next, we estimate the second moment of $g^k$:
	\begin{eqnarray*}
		\EE_{Q}\left[\|g^k\|^2\right] &=& \EE_{Q}\left[\left\|\frac{1}{n}\sum\limits_{i=1}^nQ(\hat g_i^k)\right\|^2\right]\\
		&\overset{\eqref{eq:variance_decomposition}}{=}& \EE_{Q}\left[\left\|\frac{1}{n}\sum\limits_{i=1}^n\left(Q(\hat g_i^k)-\hat g_i^k\right)\right\|^2\right] + \left\|\frac{1}{n}\sum\limits_{i=1}^n\hat g_i^k\right\|^2.
	\end{eqnarray*}
	Since quantization on nodes is performed independently we can decompose the first term from the last row of the previous inequality into the sum of variances:
	\begin{eqnarray*}
		\EE_{Q}\left[\|g^k\|^2\right] &=& \frac{1}{n^2}\sum\limits_{i=1}^n\EE_{Q}\left\|Q(\hat g_i^k)-\hat g_i^k\right\|^2 + \left\|\frac{1}{n}\sum\limits_{i=1}^n\hat g_i^k\right\|^2\\
		&\overset{\eqref{eq:quantization_def}}{\le}& \frac{\omega}{n^2}\sum\limits_{i=1}^n\|\hat g_i^k\|^2 + \left\|\frac{1}{n}\sum\limits_{i=1}^n\left(\hat g_i^k-\nabla f_i(x^*)\right)\right\|^2\\
		&\overset{\eqref{eq:a_b_norm_squared}}{\le}& \left(1 + \frac{2\omega}{n}\right)\frac{1}{n}\sum\limits_{i=1}^n\|\hat g_i^k - \nabla f_i(x^*)\|^2 + \frac{2\omega}{n^2}\sum\limits_{i=1}^n\|\nabla f_i(x^*)\|^2.
	\end{eqnarray*}
	Taking conditional mathematical expectation $\EE\left[\cdot\mid x^k\right]$ from the both sides of previous inequality we get
	\begin{eqnarray*}
		\EE\left[\|g^k\|^2\mid x^k\right] &\le& \left(1 + \frac{2\omega}{n}\right)\frac{2}{n}\sum\limits_{i=1}^n\EE\left[\|\nabla f_{il}(x^k) - \nabla f_{il}(x^*)\|^2\mid x^k\right]\\
		&&\quad + \left(1 + \frac{2\omega}{n}\right)\frac{2}{n}\sum\limits_{i=1}^n\EE\left[\left\|\nabla f_{il}(w_i^k) - \nabla f_{il}(x^*) - \left(\nabla f_i(w_i^k) - \nabla f_i(x^*)\right) \right\|^2\mid x^k\right]\\
		&&\quad  + \frac{2\omega}{n^2}\sum\limits_{i=1}^n\|\nabla f_i(x^*)\|^2\\
		&\le& \left(1 + \frac{2\omega}{n}\right)\frac{2}{nm}\sum\limits_{i=1}^n\sum\limits_{j=1}^m\|\nabla f_{ij}(x^k) - \nabla f_{ij}(x^*)\|^2\\
		&&\quad + \left(1 + \frac{2\omega}{n}\right)\frac{2}{n}\sum\limits_{i=1}^n\EE\left[\left\|\nabla f_{il}(w_i^k) - \nabla f_{il}(x^*)\right\|^2\mid x^k\right] + \frac{2\omega}{n^2}\sum\limits_{i=1}^n\|\nabla f_i(x^*)\|^2\\
		&\overset{\eqref{eq:L_smoothness_cor}}{\le}& \left(1 + \frac{2\omega}{n}\right)\frac{4L}{nm}\sum\limits_{i=1}^n\sum\limits_{j=1}^mD_{f_{ij}}(x^k,x^*)\\
		&&\quad + \left(1 + \frac{2\omega}{n}\right)\frac{2}{nm}\sum\limits_{i=1}^n\sum\limits_{j=1}^m\left\|\nabla f_{ij}(w_i^k) - \nabla f_{ij}(x^*)\right\|^2 + \frac{2\omega}{n^2}\sum\limits_{i=1}^n\|\nabla f_i(x^*)\|^2\\
		&=& 4L\left(1 + \frac{2\omega}{n}\right)\left(f(x^k) - f(x^*)\right) + 2\left(1 + \frac{2\omega}{n}\right)\sigma_k^2 + \frac{2\omega}{n^2}\sum\limits_{i=1}^n\|\nabla f_i(x^*)\|^2.
	\end{eqnarray*}
\end{proof}

\begin{lemma}\label{lem:sigma_k+1_bound_d-qLSVRG}
	For all $k\ge 0$, $i\in [n]$ we have
	\begin{equation}
		\EE\left[\sigma_{k+1}^2\mid x^k\right] \le (1-p)\sigma_k^2 + 2Lp\left(f(x^k) - f(x^*)\right), \label{eq:sigma_k+1_d-qLSVRG} 
	\end{equation}
	where $\sigma_k^2 = \frac{1}{nm}\sum_{i=1}^n\sum_{j=1}^n\|\nabla f_{ij}(w_i^k) - \nabla f_{ij}(x^*)\|^2$.
\end{lemma}
\begin{proof}
	The proof is identical to the proof of Lemma~\ref{lem:sigma_k+1_bound_ec-LSVRG}.
\end{proof}

\begin{theorem}\label{thm:d_qLSVRG}
	Assume that $f(x)$ is $\mu$-quasi strongly convex and functions $f_{ij}$ are convex and $L$-smooth for all $i\in[n],j\in[m]$. Then {\tt D-QLSVRG} satisfies Assumption~\ref{ass:key_assumption_new} with
	\begin{gather*}
		A' = 2L\left(1+\frac{2\omega}{n}\right),\quad B_1' = 0,\quad B_2' = 2\left(1+\frac{2\omega}{n}\right),\quad D_1' = \frac{2\omega}{n^2}\sum\limits_{i=1}^n\|\nabla f_i(x^*)\|^2,\quad \sigma_{1,0}^2 \equiv 0,\\
		\sigma_{2,k}^2 = \sigma_k^2 = \frac{1}{nm}\sum\limits_{i=1}^n\sum\limits_{j=1}^m\|\nabla f_{ij}(w_i^{k}) - \nabla f_{ij}(x^*)\|^2,\quad \rho_1 = 1,\quad \rho_2 = p,\quad C_2 = Lp,\quad D_2 = 0,\\
		C_1 = 0,\quad G = 0,\quad F_1 = 0,\quad F_2 = \frac{12\gamma^2L \tau\left(1+\frac{2\omega}{n}\right)(2+p)}{p},\quad D_3 = \frac{6\gamma\tau L\omega}{n^2}\sum\limits_{i=1}^n\|\nabla f_i(x^*)\|^2
	\end{gather*}
	with $\gamma$ satisfying
	\begin{equation*}
		\gamma \le \min\left\{\frac{3}{56L(1+\nicefrac{2\omega}{n})}, \frac{1}{8L\sqrt{\tau\left(\tau+2\left(1+\nicefrac{2\omega}{n}\right)\left(1+\nicefrac{2}{(1-p)}\right)\right)}}\right\}, \quad M_2 = \frac{8\left(1+\frac{2\omega}{n}\right)}{3p}
	\end{equation*}
	and for all $K \ge 0$
	\begin{equation*}
		\EE\left[f(\bar x^K) - f(x^*)\right] \le \left(1 - \min\left\{\frac{\gamma\mu}{2},\frac{p}{4}\right\}\right)^K\frac{4(T^0 + \gamma F_2 \sigma_0^2)}{\gamma} + 4\gamma\left(D_1' + D_3\right)
	\end{equation*}	
	when $\mu > 0$ and
	\begin{equation*}
		\EE\left[f(\bar x^K) - f(x^*)\right] \le \frac{4(T^0 + \gamma F_2 \sigma_0^2)}{\gamma K} + 4\gamma\left(D_1' + D_3\right)
	\end{equation*}
	when $\mu = 0$, where $T^k \eqdef \|\tx^k - x^*\|^2 + M_2\gamma^2 \sigma_k^2$.
\end{theorem}

In other words, {\tt D-QLSVRG} converges with linear rate 
$$
\cO\left(\left(\frac{1}{p} + \kappa\left(1+\frac{\omega}{n}\right) + \kappa\sqrt{\tau\left(\tau + \left(1+\frac{\omega}{n}\right)\left(1+\frac{1}{(1-p)}\right)\right)}\right)\ln\frac{1}{\varepsilon}\right)
$$
to neighbourhood the solution when $\mu > 0$. If $m\ge 2$ then taking $p = \frac{1}{m}$ we get that in expectation the sample complexity of one iteration of {\tt D-QLSVRG} is $\cO(1)$ gradients calculations per node as for {\tt D-QSGDsr} with standard sampling and the rate of convergence to the neighbourhood of the solution becomes
$$
\cO\left(\left(m + \kappa\left(1+\frac{\omega}{n}\right) + \kappa\sqrt{\tau\left(\tau + \frac{\omega}{n}\right)}\right)\ln\frac{1}{\varepsilon}\right).
$$
Applying Lemma~\ref{lem:lemma2_stich} we establish the rate of convergence to $\varepsilon$-solution.
\begin{corollary}\label{cor:d_QLSVRG_str_cvx_cor}
	Let the assumptions of Theorem~\ref{thm:d_qLSVRG} hold, $f_{\xi}(x)$ are convex for each $\xi$ and $\mu > 0$. Then after $K$ iterations of {\tt D-QLSVRG} with the stepsize
	\begin{eqnarray*}
		\gamma_0 &=& \min\left\{\frac{3}{56L(1+\nicefrac{2\omega}{n})}, \frac{1}{8L\sqrt{\tau\left(\tau+2\left(1+\nicefrac{2\omega}{n}\right)\left(1+\nicefrac{2}{(1-p)}\right)\right)}}\right\},\quad R_0 = \|x^0 - x^*\|,\\
		\gamma &=& \min\left\{\gamma_0, \frac{\ln\left(\max\left\{2,\min\left\{\frac{R_0^2\mu^2K^2}{D_1'}, \frac{R_0^2\mu^3K^3}{3\tau LD_1'}\right\}\right\}\right)}{\mu K}\right\}
	\end{eqnarray*}
	and $p = \frac{1}{m}$, $m \ge 2$ we have $\EE\left[f(\bar{x}^K) - f(x^*)\right]$ of order
	\begin{equation*}
		 \widetilde\cO\left(LR_0^2\left(1+\frac{\omega}{n}+\sqrt{\tau\left(\tau + \frac{\omega}{n}\right)}\right)\exp\left(-\frac{\mu}{L\left(1+\frac{\omega}{n}+\sqrt{\tau\left(\tau + \frac{\omega}{n}\right)}\right)}K\right) + \frac{D_1'}{\mu K} + \frac{L\tau D_1'}{\mu^2 K^2}\right).
	\end{equation*}
	That is, to achive $\EE\left[f(\bar{x}^K) - f(x^*)\right] \le \varepsilon$ {\tt D-QLSVRG} requires
	\begin{equation*}
		\widetilde{\cO}\left(\frac{L}{\mu}\left(1+\frac{\omega}{n}\right) + \frac{L}{\mu}\sqrt{\tau\left(\tau + \frac{\omega}{n}\right)} + \frac{D_1'}{\mu\varepsilon} + \frac{\sqrt{L\tau D_1'}}{\mu\sqrt{\varepsilon}}\right) \text{ iterations.}
	\end{equation*}
\end{corollary}

Applying Lemma~\ref{lem:lemma_technical_cvx} we get the complexity result in the case when $\mu = 0$.
\begin{corollary}\label{cor:d_QLSVRG_cvx_cor}
	Let the assumptions of Theorem~\ref{thm:d_qLSVRG} hold and $\mu = 0$. Then after $K$ iterations of {\tt D-QLSVRG} with the stepsize
	\begin{eqnarray*}
		\gamma_0 &=& \min\left\{\frac{3}{56L(1+\nicefrac{2\omega}{n})}, \frac{1}{8L\sqrt{\tau\left(\tau+2\left(1+\nicefrac{2\omega}{n}\right)\left(1+\nicefrac{2}{(1-p)}\right)\right)}}\right\},\quad R_0 = \|x^0 - x^*\|,\\	
		\gamma &=& \min\left\{\gamma_0, \sqrt{\frac{R_0^2}{M_2\sigma_{0}^2}}, \sqrt[3]{\frac{R_0^2p}{12L\tau\left(1+\frac{2\omega}{n}\right)(2+p)}}, \sqrt{\frac{R_0^2}{D_1' K}}, \sqrt[3]{\frac{R_0^2}{3L\tau D_1' K}}\right\}	
	\end{eqnarray*}		
	and $p = \frac{1}{m}$, $m\ge 2$ we have $\EE\left[f(\bar{x}^K) - f(x^*)\right]$ of order
	\begin{gather*}
		\cO\left(\frac{LR_0^2\left(1+\frac{\omega}{n}+\sqrt{\tau\left(\tau + \frac{\omega}{n}\right)}\right)}{K} + \frac{\sqrt{R_0^2 m\left(1+\frac{\omega}{n}\right)\sigma_0^2}}{K} + \frac{\sqrt[3]{R_0^4 L\tau m\left(1+\frac{\omega}{n}\right)}}{K}\right)\\
		 + \cO\left(\sqrt{\frac{R_0^2 D_1'}{K}} + \frac{\sqrt[3]{LR_0^4\tau D_1'}}{K^{\nicefrac{2}{3}}}\right).
	\end{gather*}
	That is, to achive $\EE\left[f(\bar{x}^K) - f(x^*)\right] \le \varepsilon$ {\tt D-QLSVRG} requires
	\begin{gather*}
		\cO\left(\frac{LR_0^2\left(1+\frac{\omega}{n}+\sqrt{\tau\left(\tau + \frac{\omega}{n}\right)}\right)}{\varepsilon} + \frac{\sqrt{R_0^2 m\left(1+\frac{\omega}{n}\right)\sigma_0^2}}{\varepsilon} + \frac{\sqrt[3]{R_0^4 L\tau m\left(1+\frac{\omega}{n}\right)}}{\varepsilon}\right)\\
		 + \cO\left(\frac{R_0^2 D_1'}{\varepsilon^2} + \frac{R_0^2\sqrt{L\tau D_1'}}{\varepsilon^{\nicefrac{3}{2}}}\right)
	\end{gather*}
	iterations.
\end{corollary}

\subsection{{\tt D-QLSVRGstar}}\label{sec:d_qLSVRGstar}
Now we assume that $i$-th node has an access to $\nabla f_i(x^*)$ and modify {\tt D-QLSVRG} in order to get convergence asymptotically to the exact optimum.   
\begin{algorithm}[t]
   \caption{{\tt D-QLSVRGstar}}\label{alg:d-qLSVRGstar}
\begin{algorithmic}[1]
   \Require learning rate $\gamma>0$, initial vector $x^0 \in \R^d$
   \State Set $e_i^0 = 0$ for all $i=1,\ldots, n$   
   \For{$k=0,1,\dotsc$}
       \State Broadcast $x^{k-\tau}$ to all workers
        \For{$i=1,\dotsc,n$ in parallel}
        	\State Pick $l$ uniformly at random from $[m]$
            \State Set $\hat g^{k-\tau}_i = \nabla f_{il}(x^{k-\tau}) - \nabla f_{il}(w_i^{k-\tau}) + \nabla f_i(w_i^{k-\tau})$
			\State Set $g_i^{k-\tau} = Q(\hat g_i^{k-\tau} - \nabla f_i(x^*))$ (quantization is performed independently from other nodes)        
            \State $v_i^k = \gamma g_{i}^{k-\tau}$
            \State $e_i^{k+1} = e_i^k + \gamma g_i^k - v_i^k$
            \State $w_i^{k-\tau+1} = \begin{cases}x^{k-\tau},& \text{with probability } p,\\ w_i^{k-\tau},& \text{with probability } 1-p\end{cases}$
        \EndFor
        \State $e^k = \frac{1}{n}\sum_{i=1}^ne_i^k$, $g^k = \frac{1}{n}\sum_{i=1}^ng_i^k$, $v^k = \frac{1}{n}\sum_{i=1}^nv_i^k$
       \State $x^{k+1} = x^k - v^k$
   \EndFor
\end{algorithmic}
\end{algorithm}

\begin{lemma}\label{lem:second_moment_bound_d-qLSVRGstar}
	For all $k\ge 0$, $i\in [n]$ we have
	\begin{equation}
		\EE\left[g^k\mid x^k\right] = \nabla f(x^k) \label{eq:unbiasedness_g^k_d-qLSVRGstar}
	\end{equation}		
	and
	\begin{equation}
		\EE\left[\|g^k\|^2\mid x^k\right] \le 2L\left(1 + \frac{\omega}{n}\right)\left(f(x^k) - f(x^*)\right) + 2\left(1 + \frac{\omega}{n}\right)\sigma_k^2, \label{eq:second_moment_bound_d-qLSVRGstar} 
	\end{equation}
	where $\sigma_k^2 = \frac{1}{nm}\sum_{i=1}^n\sum_{j=1}^n\|\nabla f_{ij}(w_i^k) - \nabla f_{ij}(x^*)\|^2$.
\end{lemma}
\begin{proof}
	First of all, we derive unbiasedness of $g_i^k$:
	\begin{eqnarray*}
		\EE\left[g^k\mid x^k\right] &\overset{\eqref{eq:tower_property}}{=}& \EE\left[\EE_{Q}\left[\frac{1}{n}\sum\limits_{i=1}^nQ(\hat g_i^k - \nabla f_i(x^*))\right]\mid x^k\right] \overset{\eqref{eq:quantization_def}}{=} \EE\left[\frac{1}{n}\sum\limits_{i=1}^n\left(\hat g_i^k-\nabla f_i(x^*)\right)\mid x^k\right]\\
		&=& \frac{1}{nm}\sum\limits_{i=1}^n\sum\limits_{j=1}^m\left(\nabla f_{ij}(x^k) - \nabla f_{ij}(w_i^k) + \nabla f_i(w_i^k)\right) = \nabla f(x^k).
	\end{eqnarray*}
	Next, we estimate the second moment of $g^k$:
	\begin{eqnarray*}
		\EE_{Q}\left[\|g^k\|^2\right] &=& \EE_{Q}\left[\left\|\frac{1}{n}\sum\limits_{i=1}^nQ(\hat g_i^k - \nabla f_i(x^*))\right\|^2\right]\\
		&\overset{\eqref{eq:variance_decomposition}}{=}& \EE_{Q}\left[\left\|\frac{1}{n}\sum\limits_{i=1}^n\left(Q(\hat g_i^k- \nabla f_i(x^*))-\left(\hat g_i^k- \nabla f_i(x^*)\right)\right)\right\|^2\right] + \left\|\frac{1}{n}\sum\limits_{i=1}^n\hat g_i^k- \nabla f_i(x^*)\right\|^2.
	\end{eqnarray*}
	Since quantization on nodes is performed independently we can decompose the first term from the last row of the previous inequality into the sum of variances:
	\begin{eqnarray*}
		\EE_{Q}\left[\|g^k\|^2\right] &=& \frac{1}{n^2}\sum\limits_{i=1}^n\EE_{Q}\left\|Q(\hat g_i^k- \nabla f_i(x^*))-\left(\hat g_i^k- \nabla f_i(x^*)\right)\right\|^2 + \left\|\frac{1}{n}\sum\limits_{i=1}^n\hat g_i^k- \nabla f_i(x^*)\right\|^2\\
		&\overset{\eqref{eq:quantization_def}}{\le}& \frac{\omega}{n^2}\sum\limits_{i=1}^n\|\hat g_i^k- \nabla f_i(x^*)\|^2 + \left\|\frac{1}{n}\sum\limits_{i=1}^n\left(\hat g_i^k-\nabla f_i(x^*)\right)\right\|^2\\
		&\overset{\eqref{eq:a_b_norm_squared}}{\le}& \left(1 + \frac{\omega}{n}\right)\frac{1}{n}\sum\limits_{i=1}^n\|\hat g_i^k - \nabla f_i(x^*)\|^2.
	\end{eqnarray*}
	Taking conditional mathematical expectation $\EE\left[\cdot\mid x^k\right]$ from the both sides of previous inequality and using the bound
	\begin{equation*}
		\frac{1}{n}\sum\limits_{i=1}^n\EE\left[\|\hat g_i^k - \nabla f_i(x^*)\|^2\mid x^k\right] \le 4L\left(f(x^k)- f(x^*)\right) + 2\sigma_k^2
	\end{equation*}
	implicitly obtained in the proof of Lemma~\ref{lem:second_moment_bound_d-qLSVRG} we get \eqref{eq:second_moment_bound_d-qLSVRGstar}.
\end{proof}

\begin{lemma}\label{lem:sigma_k+1_bound_d-qLSVRGstar}
	For all $k\ge 0$, $i\in [n]$ we have
	\begin{equation}
		\EE\left[\sigma_{k+1}^2\mid x^k\right] \le (1-p)\sigma_k^2 + 2Lp\left(f(x^k) - f(x^*)\right), \label{eq:sigma_k+1_d-qLSVRGstar} 
	\end{equation}
	where $\sigma_k^2 = \frac{1}{nm}\sum_{i=1}^n\sum_{j=1}^n\|\nabla f_{ij}(w_i^k) - \nabla f_{ij}(x^*)\|^2$.
\end{lemma}
\begin{proof}
	The proof is identical to the proof of Lemma~\ref{lem:sigma_k+1_bound_ec-LSVRG}.
\end{proof}

\begin{theorem}\label{thm:d_qLSVRGstar}
	Assume that $f(x)$ is $\mu$-quasi strongly convex and functions $f_{ij}$ are convex and $L$-smooth for all $i\in[n],j\in[m]$. Then {\tt D-QLSVRGstar} satisfies Assumption~\ref{ass:key_assumption_new} with
	\begin{gather*}
		A' = 2L\left(1+\frac{2\omega}{n}\right),\quad B_1' = 0,\quad B_2' = 2\left(1+\frac{2\omega}{n}\right),\quad D_1' = 0,\quad \sigma_{1,0}^2 \equiv 0,\\
		\sigma_{2,k}^2 = \sigma_k^2 = \frac{1}{nm}\sum\limits_{i=1}^n\sum\limits_{j=1}^m\|\nabla f_{ij}(w_i^{k}) - \nabla f_{ij}(x^*)\|^2,\quad \rho_1 = 1,\quad \rho_2 = p,\quad C_2 = Lp,\quad D_2 = 0,\\
		C_1 = 0,\quad G = 0,\quad F_1 = 0,\quad F_2 = \frac{12\gamma^2L \tau\left(1+\frac{2\omega}{n}\right)(2+p)}{p},\quad D_3 = 0
	\end{gather*}
	with $\gamma$ satisfying
	\begin{equation*}
		\gamma \le \min\left\{\frac{3}{56L(1+\nicefrac{2\omega}{n})}, \frac{1}{8L\sqrt{\tau\left(\tau+2\left(1+\nicefrac{2\omega}{n}\right)\left(1+\nicefrac{2}{(1-p)}\right)\right)}}\right\}, \quad M_2 = \frac{8\left(1+\frac{2\omega}{n}\right)}{3p}
	\end{equation*}
	and for all $K \ge 0$
	\begin{equation*}
		\EE\left[f(\bar x^K) - f(x^*)\right] \le \left(1 - \min\left\{\frac{\gamma\mu}{2},\frac{p}{4}\right\}\right)^K\frac{4(T^0 + \gamma F_2 \sigma_0^2)}{\gamma}
	\end{equation*}	
	when $\mu > 0$ and
	\begin{equation*}
		\EE\left[f(\bar x^K) - f(x^*)\right] \le \frac{4(T^0 + \gamma F_2 \sigma_0^2)}{\gamma K}
	\end{equation*}
	when $\mu = 0$, where $T^k \eqdef \|\tx^k - x^*\|^2 + M_2\gamma^2 \sigma_k^2$.
\end{theorem}

In other words, {\tt D-QLSVRGstar} converges with linear rate 
$$
\cO\left(\left(\frac{1}{p} + \kappa\left(1+\frac{\omega}{n}\right) + \kappa\sqrt{\tau\left(\tau + \left(1+\frac{\omega}{n}\right)\left(1+\frac{1}{(1-p)}\right)\right)}\right)\ln\frac{1}{\varepsilon}\right)
$$
to the exact solution when $\mu > 0$. If $m\ge 2$ then taking $p = \frac{1}{m}$ we get that in expectation the sample complexity of one iteration of {\tt D-QLSVRGstar} is $\cO(1)$ gradients calculations per node as for {\tt D-QSGDsr} with standard sampling and the rate of convergence to the exact solution becomes
$$
\cO\left(\left(m + \kappa\left(1+\frac{\omega}{n}\right) + \kappa\sqrt{\tau\left(\tau + \frac{\omega}{n}\right)}\right)\ln\frac{1}{\varepsilon}\right).
$$

Applying Lemma~\ref{lem:lemma_technical_cvx} we get the complexity result in the case when $\mu = 0$.
\begin{corollary}\label{cor:d_QLSVRGstar_cvx_cor}
	Let the assumptions of Theorem~\ref{thm:d_qLSVRGstar} hold and $\mu = 0$. Then after $K$ iterations of {\tt D-QLSVRGstar} with the stepsize
	\begin{eqnarray*}
		\gamma_0 &=& \min\left\{\frac{3}{56L(1+\nicefrac{2\omega}{n})}, \frac{1}{8L\sqrt{\tau\left(\tau+2\left(1+\nicefrac{2\omega}{n}\right)\left(1+\nicefrac{2}{(1-p)}\right)\right)}}\right\},\quad R_0 = \|x^0 - x^*\|,\\	
		\gamma &=& \min\left\{\gamma_0, \sqrt{\frac{R_0^2}{M_2\sigma_{0}^2}}, \sqrt[3]{\frac{R_0^2p}{12L\tau\left(1+\frac{2\omega}{n}\right)(2+p)}}\right\}	
	\end{eqnarray*}		
	and $p = \frac{1}{m}$, $m\ge 2$ we have $\EE\left[f(\bar{x}^K) - f(x^*)\right]$ of order
	\begin{gather*}
		\cO\left(\frac{LR_0^2\left(1+\frac{\omega}{n}+\sqrt{\tau\left(\tau + \frac{\omega}{n}\right)}\right)}{K} + \frac{\sqrt{R_0^2 m\left(1+\frac{\omega}{n}\right)\sigma_0^2}}{K} + \frac{\sqrt[3]{R_0^4 L\tau m\left(1+\frac{\omega}{n}\right)}}{K}\right).
	\end{gather*}
	That is, to achive $\EE\left[f(\bar{x}^K) - f(x^*)\right] \le \varepsilon$ {\tt D-QLSVRGstar} requires
	\begin{gather*}
		\cO\left(\frac{LR_0^2\left(1+\frac{\omega}{n}+\sqrt{\tau\left(\tau + \frac{\omega}{n}\right)}\right)}{\varepsilon} + \frac{\sqrt{R_0^2 m\left(1+\frac{\omega}{n}\right)\sigma_0^2}}{\varepsilon} + \frac{\sqrt[3]{R_0^4 L\tau m\left(1+\frac{\omega}{n}\right)}}{\varepsilon}\right)
	\end{gather*}
	iterations.
\end{corollary}

However, such convergence guarantees are obtained under very restrictive assumption: the method requires to know vectors $\nabla f_i(x^*)$.

\subsection{{\tt D-LSVRG-DIANA}}\label{sec:d_LSVRG-diana}
In the setup of Section~\ref{sec:d_LSVRG} we construct a new method with delayed updates and quantization called {\tt D-LSVRG-DIANA} which does not require to know $\nabla f_i(x^*)$ and has linear convergence to the exact solution.

\begin{algorithm}[t]
   \caption{{\tt D-LSVRG-DIANA}}\label{alg:d-LSVRG-diana}
\begin{algorithmic}[1]
   \Require learning rates $\gamma>0$, $\alpha \in (0,1]$, initial vectors $x^0, h_1^0,\ldots, h_n^0 \in \R^d$
	\State Set $e_i^0 = 0$ for all $i=1,\ldots, n$   
	\State Set $h^0 = \frac{1}{n}\sum_{i=1}^n h_i^0$   
   \For{$k=0,1,\dotsc$}
       \State Broadcast $x^{k-\tau}$ to all workers
        \For{$i=1,\dotsc,n$ in parallel}
			\State Pick $l$ uniformly at random from $[m]$
            \State Set $\hat g^{k-\tau}_i = \nabla f_{il}(x^{k-\tau}) - \nabla f_{il}(w_i^{k-\tau}) + \nabla f_i(w_i^{k-\tau})$           
			\State $\hat \Delta_i^{k-\tau} = Q(\hat g^{k-\tau}_i - h_i^{k-\tau})$ (quantization is performed independently from other nodes)           
			\State $g_i^{k-\tau} = h_i^{k-\tau} + \hat \Delta_i^{k-\tau}$            
            \State $v_i^k = \gamma g_i^{k-\tau}$
            \State $e_i^{k+1} = e_i^k + \gamma g_i^k - v_i^k$
            \State $h_i^{k-\tau+1} = h_i^{k-\tau} + \alpha \hat \Delta_i^{k-\tau}$
        \EndFor
        \State $e^k = \frac{1}{n}\sum_{i=1}^ne_i^k$, $g^k = \frac{1}{n}\sum_{i=1}^ng_i^k = h^k + \frac{1}{n}\sum\limits_{i=1}^n\hat\Delta_i^k$, $v^k = \frac{1}{n}\sum_{i=1}^nv_i^k = \gamma h^{k-\tau} + \frac{\gamma}{n}\sum_{i=1}^n \hat\Delta_i^{k-\tau}$
        \State $h^{k-\tau+1} = \frac{1}{n}\sum\limits_{i=1}^n h_i^{k-\tau+1} = h^{k-\tau} + \alpha\frac{1}{n}\sum\limits_{i=1}^n \hat \Delta_i^{k-\tau}$
       \State $x^{k+1} = x^k - v^k$
   \EndFor
\end{algorithmic}
\end{algorithm}

\begin{lemma}\label{lem:d_LSVRG-diana_second_moment_bound}
	Assume that $f_{ij}(x)$ is convex and $L$-smooth for all $i=1,\ldots,n$, $j=1,\ldots,m$. Then, for all $k\ge 0$ we have
	\begin{eqnarray}
		\EE\left[g^k\mid x^k\right] &=& \nabla f(x^k), \label{eq:d_LSVRG-diana_unbiasedness}\\
		\EE\left[\|g^k\|^2\mid x^k\right] &\le& 4L\left(1+\frac{2\omega}{n}\right)\left(f(x^k) - f(x^*)\right) + \frac{2\omega}{n}\sigma_{1,k}^2 + 2\left(1+\frac{2\omega}{n}\right)\sigma_{2,k}^2 \label{eq:d_LSVRG-diana_second_moment_bound}
	\end{eqnarray}
	where $\sigma_{1,k}^2 = \frac{1}{n}\sum_{i=1}^n\|h_i^k - \nabla f(x^*)\|^2$ and $\sigma_{2,k}^2 = \frac{1}{nm}\sum_{i=1}^n\sum_{j=1}^m\|\nabla f_{ij}(w_i^k) - \nabla f_{ij}(x^*)\|^2$.
\end{lemma}
\begin{proof}
	First of all, we show unbiasedness of $g^k$:
	\begin{eqnarray*}
		\EE\left[g^k\mid x^k\right] &\overset{\eqref{eq:tower_property}}{=}& h^k + \frac{1}{n}\sum\limits_{i=1}^n\EE\left[\EE_Q\left[\hat \Delta_i^k\right]\mid x^k\right] \overset{\eqref{eq:quantization_def}}{=} h^k + \frac{1}{n}\sum\limits_{i=1}^n\EE\left[\hat g_i^k - h_i^k\mid x^k\right]	\\
		&=& \frac{1}{nm}\sum\limits_{i=1}^n\sum\limits_{j=1}^m\left(\nabla f_{ij}(x^k) - \nabla f_{ij}(w_i^k) + \nabla f_i(w_i^k)\right) = \nabla f(x^k).
	\end{eqnarray*}
	Next, we derive the upper bound for the second moment of $g^k$:
	\begin{eqnarray*}
		\EE_{Q}\left[\|g^k\|^2\right] &=& \EE_{Q}\left[\left\|h^k + \frac{1}{n}\sum\limits_{i=1}^n\hat{\Delta}_i^k\right\|^2\right]\\
		&\overset{\eqref{eq:variance_decomposition}}{=}& \EE_{Q}\left[\left\|\frac{1}{n}\sum\limits_{i=1}^n\left(\hat{\Delta}_i^k - \hat g_i^k + h_i^k\right)\right\|^2\right] + \left\|\frac{1}{n}\sum\limits_{i=1}^n\hat g_i^k\right\|^2.
	\end{eqnarray*}
	Since quantization on nodes is performed independently we can decompose the first term from the last row of the previous inequality into the sum of variances:
	\begin{eqnarray*}
		\EE_{Q}\left[\|g^k\|^2\right] &\le& \frac{1}{n^2}\sum\limits_{i=1}^n\EE_{Q}\left[\|\hat{\Delta}_i^k - \hat g_i^k + h_i^k\|^2\right] + \left\|\frac{1}{n}\sum\limits_{i=1}^n\left(\hat g_i^k - \nabla f_i(x^*)\right)\right\|^2\\
		&\overset{\eqref{eq:quantization_def},\eqref{eq:a_b_norm_squared}}{\le}& \frac{\omega}{n^2}\sum\limits_{i=1}^n\|\hat g_i^k - h_i^k\|^2 + \frac{1}{n}\sum\limits_{i=1}^n\|\hat g_i^k - \nabla f_i(x^*)\|^2\\
		&\overset{\eqref{eq:a_b_norm_squared}}{\le}& \left(1+\frac{2\omega}{n}\right)\frac{1}{n}\sum\limits_{i=1}^n\|\hat g_i^k - \nabla f_i(x^*)\|^2 + \frac{2\omega}{n^2}\sum\limits_{i=1}^n\|h_i^k - f_i(x^*)\|^2.
	\end{eqnarray*}
	Taking mathematical expectation $\EE\left[\cdot\mid x^k\right]$ from the both sides of the previous inequality and using the bound
	\begin{equation*}
		\frac{1}{n}\sum\limits_{i=1}^n\EE\left[\|\hat g_i^k - \nabla f_i(x^*)\|^2\mid x^k\right] \le 4L\left(f(x^k)- f(x^*)\right) + \frac{2}{nm}\sum\limits_{i=1}^n\sum\limits_{j=1}^m\|\nabla f_{ij}(w_i^k) - \nabla f_{ij}(x^*)\|^2
	\end{equation*}
	implicitly obtained in the proof of Lemma~\ref{lem:second_moment_bound_d-qLSVRG} we get \eqref{eq:d_LSVRG-diana_second_moment_bound}.
\end{proof}

\begin{lemma}\label{lem:d_LSVRG-diana_sigma_k+1_bound}
	Assume that $\alpha \le \nicefrac{1}{(\omega+1)}$. Then, for all $k\ge 0$ we have
	\begin{equation*}
		\EE\left[\sigma_{1,k+1}^2\mid x^k\right] \le (1 - \alpha)\sigma_{1,k}^2 + 6L\alpha(f(x^k) - f(x^*)) + 2\alpha\sigma_{2,k}^2,
	\end{equation*}
	\begin{equation*}
		\EE\left[\sigma_{2,{k+1}}^2\mid x^k\right] \le (1 - p)\sigma_{k,2}^2 + 2Lp\left(f(x^k)-f(x^*)\right)
	\end{equation*}
	where $\sigma_{1,k}^2 = \frac{1}{n}\sum_{i=1}^n\|h_i^k - \nabla f_i(x^*)\|^2$ and $\sigma_{2,k}^2= \frac{1}{nm}\sum_{i=1}^n\sum_{j=1}^m\|\nabla f_{ij}(w_i^k) - \nabla f_{ij}(x^*)\|^2$.
\end{lemma}
\begin{proof}
	The proof is identical to the proof of Lemma~\ref{lem:ec_LSVRG-diana_sigma_k+1_bound}.
\end{proof}

\begin{theorem}\label{thm:d_LSVRG-diana}
	Assume that $f_{ij}(x)$ is convex and $L$-smooth for all $i=1,\ldots, n$, $j=1,\ldots,m$ and $f(x)$ is $\mu$-quasi strongly convex. Then {\tt D-LSVRG-DIANA} satisfies Assumption~\ref{ass:key_assumption_new} with
	\begin{gather*}
		A' = 2L\left(1+\frac{2\omega}{n}\right),\quad B_1' = \frac{2\omega}{n},\quad B_2' = 2\left(1+\frac{2\omega}{n}\right),\quad D_1' = 0,\\
		\sigma_{1,k}^2 = \frac{1}{n}\sum\limits_{i=1}^n\|h_i^k - \nabla f_i(x^*)\|^2,\quad \sigma_{2,k}^2 = \frac{1}{nm}\sum\limits_{i=1}^n\sum_{j=1}^m\|\nabla f_{ij}(w_i^k) - \nabla f_{ij}(x^*)\|^2,\\
		\rho_1 = \alpha,\quad \rho_2 = p,\quad C_1 = 3L\alpha,\quad C_2 = Lp,\quad D_2 = 0,\quad G = 2,\\
		F_1 = \frac{12\gamma^2 L\omega\tau(2+\alpha)}{n\alpha},\quad F_2 = \frac{12\gamma^2\tau L(2+p)}{p}\left(\frac{4\omega}{n(1-\alpha)} + 1 + \frac{2\omega}{n}\right),\quad D_3 = 0
	\end{gather*}
	with $\gamma$ and $\alpha$ satisfying
	\begin{equation*}
		\gamma \le \min\left\{\frac{1}{8L\left(\frac{37}{9} + \frac{24\omega}{3n}\right)}, \frac{1}{8L\sqrt{\tau\left(2+\tau+ \frac{4}{1-p}+\frac{4\omega}{n}\left(1+\frac{3}{1-\alpha}+\frac{2}{1-p}+\frac{4}{(1-\alpha)(1-p)}\right)\right)}}\right\},
	\end{equation*}
	\begin{equation*}
		\alpha \le \frac{1}{\omega+1},\quad M_1 = \frac{8\omega}{3n\alpha},\quad M_2 = \frac{8\left(7 + \frac{6\omega}{n}\right)}{9p}.
	\end{equation*}
	and for all $K \ge 0$
	\begin{equation*}
		\EE\left[f(\bar x^K) - f(x^*)\right] \le \left(1 - \min\left\{\frac{\gamma\mu}{2},\frac{\alpha}{4},\frac{p}{4}\right\}\right)^K\frac{4(T^0 + \gamma F_1 \sigma_{1,0}^2 + \gamma F_2 \sigma_{2,0}^2)}{\gamma}
	\end{equation*}	
	when $\mu > 0$ and
	\begin{equation*}
		\EE\left[f(\bar x^K) - f(x^*)\right] \le \frac{4(T^0 + \gamma F_1 \sigma_{1,0}^2 + \gamma F_2 \sigma_{2,0}^2)}{\gamma K}
	\end{equation*}	
	when $\mu = 0$, where $T^k \eqdef \|\tx^k - x^*\|^2 + M_1\gamma^2 \sigma_{1,k}^2 + M_2\gamma^2 \sigma_{2,k}^2$.
\end{theorem}
In other words, if $m\ge 2$, $p = \nicefrac{1}{m}$, $\alpha = \min\left\{\frac{1}{\omega+1},\frac{1}{2}\right\}$ and 
\begin{equation*}
		\gamma \le \min\left\{\frac{1}{8L\left(\frac{37}{9} + \frac{24\omega}{3n}\right)}, \frac{1}{8L\sqrt{\tau\left(2+\tau+ \frac{4}{1-p}+\frac{4\omega}{n}\left(1+\frac{3}{1-\alpha}+\frac{2}{1-p}+\frac{4}{(1-\alpha)(1-p)}\right)\right)}}\right\},
\end{equation*}
{\tt D-LSVRG-DIANA} converges with the linear rate
\begin{equation*}
	\cO\left(\left(\omega + m + \kappa\left(1+\frac{\omega}{n}\right) + \kappa\sqrt{\tau\left(\tau+\frac{\omega}{n}\right)}\right)\ln\frac{1}{\varepsilon}\right)
\end{equation*}
to the exact solution when $\mu > 0$.

Applying Lemma~\ref{lem:lemma_technical_cvx} we get the complexity result in the case when $\mu = 0$.
\begin{corollary}\label{cor:d_LSVRG_DIANA_cvx_cor}
	Let the assumptions of Theorem~\ref{thm:d_LSVRG-diana} hold and $\mu = 0$. Then after $K$ iterations of {\tt D-LSVRG-DIANA} with the stepsize
	\begin{eqnarray*}
		\gamma_0 &=& \min\left\{\frac{1}{8L\left(\frac{37}{9} + \frac{24\omega}{3n}\right)}, \frac{1}{8L\sqrt{\tau\left(2+\tau+ \frac{4}{1-p}+\frac{4\omega}{n}\left(1+\frac{3}{1-\alpha}+\frac{2}{1-p}+\frac{4}{(1-\alpha)(1-p)}\right)\right)}}\right\},\\	
		\gamma &=& \min\left\{\gamma_0, \sqrt{\frac{R_0^2}{M_1\sigma_{1,0}^2 + M_2\sigma_{2,0}^2}}, \sqrt[3]{\frac{R_0^2}{12\tau L\left(\frac{\omega(2+\alpha)}{n\alpha} + \frac{2+p}{p}\left(1+\frac{2\omega}{n}+\frac{4\omega}{n(1-\alpha)}\right)\right)}}\right\},
	\end{eqnarray*}		
	where $R_0 = \|x^0 - x^*\|$, $\alpha = \min\left\{\frac{1}{\omega+1},\frac{1}{2}\right\}$ and $p = \frac{1}{m}$, $m\ge 2$ we have $\EE\left[f(\bar{x}^K) - f(x^*)\right]$ of order
	\begin{gather*}
		\cO\left(\frac{LR_0^2\left(1+\frac{\omega}{n}+\sqrt{\tau\left(\tau + \frac{\omega}{n}\right)}\right)}{K} + \frac{\sqrt{R_0^2\omega(\omega+1)\sigma_{1,0}^2}}{\sqrt{n}K} + \frac{\sqrt{R_0^2m\left(1+\frac{\omega}{n}\right)\sigma_{2,0}^2}}{K}\right)\\
		+\cO\left(\frac{\sqrt[3]{R_0^4\tau L \omega(\omega+1)\sigma_{1,0}^2}}{\sqrt[3]{n}K} + \frac{\sqrt[3]{R_0^4\tau L m\left(1+\frac{\omega}{n}\right)\sigma_{2,0}^2}}{K}\right)
	\end{gather*}
	That is, to achive $\EE\left[f(\bar{x}^K) - f(x^*)\right] \le \varepsilon$ {\tt D-LSVRG-DIANA} requires
	\begin{gather*}
		\cO\left(\frac{LR_0^2\left(1+\frac{\omega}{n}+\sqrt{\tau\left(\tau + \frac{\omega}{n}\right)}\right)}{\varepsilon} + \frac{\sqrt{R_0^2\omega(\omega+1)\sigma_{1,0}^2}}{\sqrt{n}\varepsilon} + \frac{\sqrt{R_0^2m\left(1+\frac{\omega}{n}\right)\sigma_{2,0}^2}}{\varepsilon}\right)\\
		+\cO\left(\frac{\sqrt[3]{R_0^4\tau L \omega(\omega+1)\sigma_{1,0}^2}}{\sqrt[3]{n}\varepsilon} + \frac{\sqrt[3]{R_0^4\tau L m\left(1+\frac{\omega}{n}\right)\sigma_{2,0}^2}}{\varepsilon}\right)
	\end{gather*}
	iterations.
\end{corollary}

\clearpage

\begin{table*}[!t]
\caption{The parameters for which the methods from Tables~\ref{tbl:special_cases2} and \ref{tbl:special_cases_delayed_methods} satisfy Assumption~\ref{ass:key_assumption_new}. The meaning of the expressions appearing in the table, as well as their justification is defined in details in the Sections~\ref{sec:special_cases} and \ref{sec:special_cases2}. Symbols: $\varepsilon = $ error tolerance; $\delta = $ contraction factor of compressor $\cC$; $\omega = $ variance parameter of compressor $\cQ$; $\kappa = \nicefrac{L}{\mu}$; $\cL =$ expected smoothness constant; $\sigma_*^2 = $ variance of the stochastic gradients in the solution; $\zeta_*^2 =$ average of $\|\nabla f_i(x^*)\|^2$; $\sigma^2 =$ average of the uniform bounds for the variances of stochastic gradients of workers.}
\label{tbl:special_cases-parameters}
\begin{center}
 \tiny
\begin{adjustbox}{angle=90}
\begin{tabular}{|c|c|c|c|c|c|c|c|c|c|c|}
\hline
 Method &   $A'$ & $B_1'$ & $B_2'$ & $\rho_1$ & $\rho_2$ & $C_1$ & $C_2$ & $F_1,\quad F_2$ & $G$ & $D_1'$, $D_2$, $D_3$\\
 \hline
\hline
 {\tt EC-SGDsr}   &  $2\cL $ & $0$ & $0$ & $1$ & $1$ & $0$ & $0$ & $0,\quad 0$ & $0$ & $\frac{2\sigma_*^2}{n},\quad 0,\quad\frac{6L\gamma}{\delta}\left(\frac{4\zeta_*^2}{\delta}+3\sigma_*^2\right)$\\
 %%%%%%%%%%%%%%%%%%%%
 %%%%%%%%%%%%%%%%%%%%
 \hline
 {\tt EC-SGD}   &  $2L $ & $0$ & $0$ & $1$ & $1$ & $0$ & $0$ & $0,\quad 0$ & $0$ & $\frac{2\sigma_*^2}{n},\quad 0,\quad\frac{12L\gamma}{\delta}\left(\frac{2\zeta_*^2}{\delta}+\sigma_*^2\right)$\\
 %%%%%%%%%%%%%%%%%%%%
 %%%%%%%%%%%%%%%%%%%%
 \hline
 {\tt EC-GDstar}   &  $L $ & $0$ & $0$ & $1$ & $1$ & $0$ & $0$ & $0,\quad 0$ & $0$ & $0,\quad 0,\quad 0$\\
 %%%%%%%%%%%%%%%%%%%%
 %%%%%%%%%%%%%%%%%%%%
 \hline
 {\tt EC-SGD-DIANA}   &  $L $ & $0$ & $0$ & $\alpha$ & $1$ & $L\alpha$ & $0$ & $\frac{96L\gamma^2}{\delta^2\alpha(1-\eta)},\quad 0$ & $0$ & \makecell{$\frac{\sigma^2}{n},\quad \alpha^2(\omega+1)\sigma^2,$\\ $\frac{6L\gamma}{\delta}\left(\frac{4\alpha(\omega+1)}{\delta}+1\right)\sigma^2$}\\
 %%%%%%%%%%%%%%%%%%%%
 %%%%%%%%%%%%%%%%%%%%
 \hline
 {\tt EC-SGDsr-DIANA}   &  $2\cL $ & $0$ & $0$ & $\alpha$ & $1$ & $2\alpha(3\cL + 4L)$ & $0$ & $\frac{96L\gamma^2}{\delta^2\alpha(1-\eta)},\quad 0$ & $0$ & \makecell{$\frac{2\sigma_*^2}{n},\quad \alpha^2(\omega+1)\sigma_*^2,$\\ $\frac{18L\gamma}{\delta}\left(\frac{4\alpha(\omega+1)}{\delta}+1\right)\sigma_*^2$}\\
 %%%%%%%%%%%%%%%%%%%%
 %%%%%%%%%%%%%%%%%%%%
 \hline
 {\tt EC-LSVRG}   &  $2L $ & $0$ & $2$ & $1$ & $p$ & $0$ & $Lp$ & $0,\quad \frac{72L\gamma^2}{\delta p(1-\eta)}$ & $0$ & $0,\quad 0,\quad \frac{24L\gamma}{\delta^2}\zeta_*^2$\\
 %%%%%%%%%%%%%%%%%%%%
 %%%%%%%%%%%%%%%%%%%%
 \hline
  {\tt EC-LSVRGstar}   &  $2L $ & $0$ & $2$ & $1$ & $p$ & $0$ & $Lp$ & $0,\quad \frac{48L\gamma^2}{\delta p}$ & $0$ & $0,\quad 0,\quad 0$\\
 %%%%%%%%%%%%%%%%%%%%
 %%%%%%%%%%%%%%%%%%%%
 \hline
  {\tt EC-LSVRG-DIANA}   &  $2L $ & $0$ & $2$ & $\alpha$ & $p$ & $3L\alpha$ & $Lp$ & \makecell{$\frac{24L\gamma^2\left(\frac{4}{\delta}+3\right)}{\delta\alpha(1-\eta)},$\\ $\frac{24L\gamma^2\left(\frac{4}{1-\alpha}\left(\frac{4}{\delta}+3\right)+3\right)}{\delta p(1-\eta)}$} & $2$ & $0,\quad 0,\quad 0$\\
 %%%%%%%%%%%%%%%%%%%%
 %%%%%%%%%%%%%%%%%%%%
 \hline
 {\tt D-SGDsr}   &  $2\cL $ & $0$ & $0$ & $1$ & $1$ & $0$ & $0$ & $0,\quad 0$ & $0$ & $\frac{2\sigma_*^2}{n},\quad 0,\quad \frac{6L\tau\gamma\sigma_*^2}{n}$\\
 %%%%%%%%%%%%%%%%%%%%
 %%%%%%%%%%%%%%%%%%%%
 \hline
 {\tt D-SGD}   &  $2L $ & $0$ & $0$ & $1$ & $1$ & $0$ & $0$ & $0,\quad 0$ & $0$ & $\frac{2\sigma_*^2}{n},\quad 0,\quad \frac{6L\tau\gamma\sigma_*^2}{n}$\\
 %%%%%%%%%%%%%%%%%%%%
 %%%%%%%%%%%%%%%%%%%%
 \hline
 {\tt D-QSGD}   &  $L\left(1+\frac{2\omega}{n}\right) $ & $0$ & $0$ & $1$ & $1$ & $0$ & $0$ & $0,\quad 0$ & $0$ & \makecell{$\frac{(\omega+1)\sigma^2}{n}+\frac{2\omega\zeta_*^2}{n},\quad 0,$\\ $\frac{3\gamma\tau L}{n}\left((\omega+1)\sigma^2 + 2\omega\zeta_*^2\right)$}\\
 %%%%%%%%%%%%%%%%%%%%
 %%%%%%%%%%%%%%%%%%%%
 \hline
 {\tt D-QSGDstar}   &  $L\left(1+\frac{\omega}{n}\right) $ & $0$ & $0$ & $1$ & $1$ & $0$ & $0$ & $0,\quad 0$ & $0$ & $\frac{(\omega+1)\sigma^2}{n},\quad 0,\quad \frac{3\gamma\tau L(\omega+1)\sigma^2}{n}$\\
 %%%%%%%%%%%%%%%%%%%%
 %%%%%%%%%%%%%%%%%%%%
 \hline
 {\tt D-QGDstar}   &  $L\left(1+\frac{\omega}{n}\right) $ & $0$ & $0$ & $1$ & $1$ & $0$ & $0$ & $0,\quad 0$ & $0$ & $0,\quad 0,\quad 0$\\
 %%%%%%%%%%%%%%%%%%%%
 %%%%%%%%%%%%%%%%%%%%
 \hline
 {\tt D-SGD-DIANA}   &  $L\left(1+\frac{2\omega}{n}\right) $ & $\frac{2\omega}{n}$ & $0$ & $\alpha$ & $1$ & $L\alpha$ & $0$ & $\frac{12\gamma^2 L\omega\tau(2+\alpha)}{n\alpha},\quad 0$ & $0$ & \makecell{$\frac{(\omega+1)\sigma^2}{n},\quad \frac{\alpha(\omega+1)\sigma^2}{n},$\\ $3\gamma\tau L\left(1+\frac{4\omega}{n}\right)\frac{(\omega+1)\sigma^2}{n}$}\\
 %%%%%%%%%%%%%%%%%%%%
 %%%%%%%%%%%%%%%%%%%%
 \hline
 {\tt D-LSVRG}   &  $2L$ & $0$ & $2$ & $1$ & $p$ & $0$ & $Lp$ & $0,\quad \frac{12\gamma^2 L\tau(2+p)}{np}$ & $0$ & $0,\quad 0,\quad 0$\\
 %%%%%%%%%%%%%%%%%%%%
 %%%%%%%%%%%%%%%%%%%%
 \hline
 {\tt D-QLSVRG}   &  $2L\left(1+\frac{2\omega}{n}\right)$ & $0$ & $2\left(1+\frac{2\omega}{n}\right)$ & $1$ & $p$ & $0$ & $Lp$ & $0,\quad \frac{12\gamma^2 L\tau\left(1+\frac{2\omega}{n}\right)\tau(2+p)}{p}$ & $0$ & \makecell{$\frac{2\omega\zeta_*^2}{n},\quad 0,$\\ $\frac{6\gamma\tau L\omega\zeta_*^2}{n}$}\\
 %%%%%%%%%%%%%%%%%%%%
 %%%%%%%%%%%%%%%%%%%%
 \hline
 {\tt D-QLSVRGstar}   &  $2L\left(1+\frac{2\omega}{n}\right)$ & $0$ & $2\left(1+\frac{2\omega}{n}\right)$ & $1$ & $p$ & $0$ & $Lp$ & $0,\quad \frac{12\gamma^2 L\left(1+\frac{2\omega}{n}\right)\tau(2+p)}{p}$ & $0$ & $0,\quad 0,\quad 0$\\
 %%%%%%%%%%%%%%%%%%%%
 %%%%%%%%%%%%%%%%%%%%
 \hline
 {\tt D-LSVRG-DIANA}   &  $2L\left(1+\frac{2\omega}{n}\right)$ & $\frac{2\omega}{n}$ & $2\left(1+\frac{2\omega}{n}\right)$ & $\alpha$ & $p$ & $3L\alpha$ & $Lp$ & \begin{tabular}{c}
 $\frac{12\gamma^2 L\omega\tau(2+\alpha)}{n\alpha}$,\\$\frac{12\gamma^2\tau L(2+p)}{p}\left(1 + \frac{2\omega(3-\alpha)}{n(1-\alpha)}\right)$
 \end{tabular} & $0$ & $0,\quad 0,\quad 0$\\
 %%%%%%%%%%%%%%%%%%%%
 %%%%%%%%%%%%%%%%%%%%
 \hline
\end{tabular}
\end{adjustbox}
\end{center}
\end{table*}

\end{document}